\documentclass[12pt]{amsart}
\usepackage[left=20mm,right=20mm]{geometry}
\usepackage{graphicx} 
\usepackage{amsmath}
\usepackage{amssymb}
\usepackage{amscd}
\usepackage{amsthm}
\usepackage{needspace}
\usepackage{enumerate}
\usepackage{xcolor}
\definecolor{linkgold}{HTML}{966e09}
\usepackage[all]{xy}
\usepackage{url}
\usepackage[colorlinks=true,allcolors=linkgold,breaklinks,pagebackref]{hyperref}
\usepackage[lite,abbrev,alphabetic]{amsrefs}
\usepackage{ragged2e}
\AddToHook{env/biblist/begin}{\RaggedRight}
\usepackage{cleveref}
\usepackage{comment}
\usepackage{todonotes}
\usepackage{tikz-cd}
\usepackage{quiver}
\usepackage{enumitem}

\makeatletter
\renewcommand{\subsection}{\@startsection{subsection}{2}%
  \z@{.7\linespacing\@plus.5\linespacing}{.35\linespacing}%
  {\normalfont\bfseries\raggedright}}
\makeatother

\theoremstyle{definition}
\newtheorem{definition}{Definition}[section]
\newtheorem{example}[definition]{Example}

\theoremstyle{plain}
\newtheorem{theorem}[definition]{Theorem}
\newtheorem{proposition}[definition]{Proposition}
\newtheorem{lemma}[definition]{Lemma}
\newtheorem{corollary}[definition]{Corollary}

\newtheorem*{theorem*}{Theorem}

\theoremstyle{remark}
\newtheorem{remark}[definition]{Remark}
\newtheorem*{Ac}{Acknowledgement}
\newtheorem*{AIdisclosure}{AI disclosure}

\crefname{theorem}{Theorem}{Theorems}
\crefname{proposition}{Proposition}{Propositions}
\crefname{lemma}{Lemma}{Lemmas}
\crefname{corollary}{Corollary}{Corollaries}
\crefname{conjecture}{Conjecture}{Conjectures}
\crefname{hypothesis}{Hypothesis}{Hypotheses}
\crefname{remark}{Remark}{Remarks}
\crefname{condition}{Condition}{Conditions}
\crefname{example}{Example}{Examples}

\def\C{\mathbb{C}}
\def\R{\mathbb{R}}
\def\Q{\mathbb{Q}}
\def\Z{\mathbb{Z}}
\def\A{\mathbb{A}}

\def\GL{\mathrm{GL}}
\def\PGL{\mathrm{PGL}}
\def\SL{\mathrm{SL}}

\def\O{\textnormal{O}}
\def\SO{\mathrm{SO}}
\def\GSO{\mathrm{GSO}}

\def\SU{\textnormal{SU}}
\def\Sp{\mathrm{Sp}}
\def\GSp{\mathrm{GSp}}
\def\Spin{\mathrm{Spin}}
\def\PGSO{\mathrm{PGSO}}
\def\PGSp{\mathrm{PGSp}}
\def\PGO{\mathrm{PGO}}

\def\sl{\mathfrak{sl}}

\def\rd{\,\mathrm{d}}

\DeclareMathOperator{\diag}{diag} 

\DeclareMathOperator{\rank}{rank}
\DeclareMathOperator{\Gal}{Gal}
\DeclareMathOperator{\Res}{Res}
\DeclareMathOperator{\Ind}{Ind}
\DeclareMathOperator{\Hom}{Hom}
\DeclareMathOperator{\Aut}{Aut}
\DeclareMathOperator{\Lie}{Lie}

\DeclareMathOperator{\Cent}{Cent}
\DeclareMathOperator{\Irr}{Irr}
\DeclareMathOperator{\sgn}{sgn}
\DeclareMathOperator{\Ad}{Ad}
\DeclareMathOperator{\Std}{Std}
\DeclareMathOperator{\Ker}{Ker}
\DeclareMathOperator{\Image}{Im}
\DeclareMathOperator{\tr}{tr}
\DeclareMathOperator{\Sym}{Sym}
\DeclareMathOperator{\Int}{Int}

\DeclareMathOperator{\Out}{Out}
\DeclareMathOperator{\der}{der}
\DeclareMathOperator{\Ann}{Ann}
\DeclareMathOperator{\AV}{AV}

\newcommand{\fa}{\mathfrak{a}}

\newcommand{\fe}{\mathfrak{e}}

\newcommand{\fg}{\mathfrak{g}}

\newcommand{\fm}{\mathfrak{m}}
\newcommand{\fn}{\mathfrak{n}}

\newcommand{\bF}{\mathbb{F}}
\newcommand{\bG}{\mathbb{G}}

\newcommand{\cA}{\mathcal{A}}

\newcommand{\cD}{\mathcal{D}}
\newcommand{\cE}{\mathcal{E}}

\newcommand{\cH}{\mathcal{H}}
\newcommand{\cI}{\mathcal{I}}

\newcommand{\cL}{\mathcal{L}}
\newcommand{\cM}{\mathcal{M}}
\newcommand{\cN}{\mathcal{N}}
\newcommand{\cO}{\mathcal{O}}
\newcommand{\cP}{\mathcal{P}}

\newcommand{\cR}{\mathcal{R}}
\newcommand{\cS}{\mathcal{S}}

\newcommand{\cU}{\mathcal{U}}
\newcommand{\cV}{\mathcal{V}}

\newcommand{\cZ}{\mathcal{Z}}

\numberwithin{equation}{section}
\newcommand{\disc}{\mathrm{disc}}

\newcommand{\elliptic}{\mathrm{ell}}

\newcommand{\cusp}{\mathrm{cusp}}
\newcommand{\semi}{\mathrm{ss}}
\newcommand{\reg}{\mathrm{reg}}
\newcommand{\temp}{\mathrm{temp}}
\newcommand{\gen}{\mathrm{gen}}

\newcommand{\rG}{\mathrm{G}}

\newcommand{\abs}[1]{\lvert #1 \rvert}

\newcommand{\pr}{\mathrm{pr}}
\newcommand{\GSpin}{\mathrm{GSpin}}

\newcommand{\Norm}{\mathrm{Norm}}

\newcommand{\Trans}{\mathrm{Trans}}
\newcommand{\St}{\mathrm{St}}

\DefineSimpleKey{bib}{arxiv}
\newcommand\myurl[1]{\url{#1}}
\newcommand\arxiv[1]{available at \href{https://arxiv.org/abs/#1}{arXiv:#1}}

\BibSpec{article}{%
  +{}  {\PrintAuthors}                {author}
  +{,} { \textit}                     {title}
  +{.} { }                            {part}
  +{:} {\textit}                     {subtitle}
  +{,} {\PrintContributions}         {contribution}
  +{,} {\PrintConference}            {conference}
  +{}  {\PrintBook}                   {book}
  +{,} { }                            {booktitle}
  +{,} { }                            {journal}
  +{,} { \textbf}                     {volume}
  +{,} { \issuetext}                  {number}
  +{,} { \PrintDateB}                 {date}
  +{,} { pp.~}                       {pages}
  +{,} { available at \eprint}        {eprint}
  +{;} { \arxiv}               {arxiv}
  +{.\!\!}  {\SentenceSpace \PrintReviews} {review}
  +{,} { \PrintDOI}                   {doi}
}

\title[Endoscopy for $\rG_2$]{Endoscopic description of the local Langlands correspondence for $\rG_2$}
\author{Yugo Takanashi}
\address{Department of Mathematics, Graduate School of Science, Kyoto University, Kitashirakawa Oiwake-cho, Sakyo-ku, Kyoto 606-8502, Japan}
\email{takanashi.yugo.7s@kyoto-u.ac.jp}
\date{}

\let\oldtableofcontents\tableofcontents
\renewcommand{\tableofcontents}{%
  {\fontsize{12}{12}\selectfont\oldtableofcontents}%
}

\hypersetup{bookmarksdepth=2,bookmarksnumbered=true,bookmarksopen=false}
\makeatletter
\renewcommand{\l@part}{\@tocline{-1}{3pt}{0pt}{}{\bfseries}}
\renewcommand{\tocpart}[3]{%
  \@ifempty{#2}{#3}{\indentlabel{\makebox[5.3em][l]{#1 #2.}}#3}%
}
\AddToHook{env/bibdiv/begin}{%
  \def\toclevel@section{0}%
  \def\@tocwrite#1{\@tocwriteb\tocpart{part}}%
}
\newcommand{\majorcontentsentry}[1]{%
  \ifnum\c@tocdepth<1
    \addtocontents{toc}{\protect\contentsline{part}%
      {\protect\tocpart{}{}{#1}}{\thepage}{\@currentHref}}%
  \fi
}
\makeatother

\begin{document}

\begin{abstract}
  We prove that the $L$-packets for $p$-adic $\rG_2$ constructed by Gan and Savin satisfy the endoscopic character identities.
  We also prove twisted endoscopic character identities for triality-invariant functorial lifts from $\rG_2$ to $\PGSO_8$, and a global multiplicity formula for the part of the discrete automorphic spectrum of $\rG_2$ arising from cuspidal automorphic representations of $\PGL_3$.
\end{abstract}

\maketitle
\setcounter{tocdepth}{1}
\tableofcontents
\bigskip

\section*{Introduction}\label{section:introduction}
\majorcontentsentry{Introduction}

Let $F$ be a local field of characteristic $0$.
Let $G$ be a reductive group over $F$.
The local Langlands correspondence for $G$ is a conjectural finite-to-one correspondence between the set $\Pi(G)$ of isomorphism classes of smooth admissible irreducible representations of $G$ and the set $\Phi(G)$ of $L$-parameters for $G$.
Harish-Chandra and Langlands established the archimedean correspondence.
For general linear groups it is due to Harris--Taylor, Henniart, and Scholze, and for classical groups to Arthur, Mok, and others.

For $p$-adic fields, Harris--Khare--Thorne \cite{HarrisKhareThorne2023-LocalLanglandsParameterizationGenericSupercuspidalRepresentations} constructed a Langlands parameterization of generic supercuspidal representations of $\rG_2(F)$.
Gan and Savin established the local Langlands correspondence for $\rG_2(F)$ in \cite{GanSavin2023-LocalLanglandsConjectureG2}, using the exceptional theta correspondences studied in \cite{GanSavin2023-HoweDualityDichotomyExceptionalThetaCorrespondences}.
Aubert and Xu \cite{AubertXu2023-ExplicitLocalLanglandsCorrespondenceG2} give an explicit construction when the residue characteristic of $F$ is different from $2$ and $3$.

An endoscopic description of the Gan--Savin packets is a step toward the classification of automorphic representations of $\rG_2$ over number fields.

\subsection{The first main theorem: Endoscopic character identities for \texorpdfstring{$\rG_2$}{G2}}

Our first main theorem describes the Gan--Savin correspondence by endoscopic character identities relating $\rG_2(F)$, $\PGSO_8(F)\rtimes\theta$ for a triality automorphism $\theta$, and their elliptic endoscopic groups.
The proof is given in Section~\ref{section:canonical-extensions} and Sections~\ref{section:local-LIR-G2}--\ref{section:ECR-G2-SO4}.

\begin{theorem}[Theorem~\ref{theorem:summary}]
   Let $F$ be a non-archimedean local field of characteristic zero.
   Let $\phi^{\rG_2}$ be a tempered $L$-parameter for $\rG_2(F)$.
   Let $\Pi_{\phi^{\rG_2}}^{\mathrm{GS}}$ be the $L$-packet of $\rG_2(F)$ constructed by Gan and Savin.
   Let $\Pi_{\iota_{\rG_2}^{\PGSO_8} \circ \phi^{\rG_2}}$ denote the distinguished $L$-packet of $\PGSO_8(F)$ constructed in \cite{Takanashi2026-HiragaIchinoIkedaConjectureFormalDegrees} based on the work of Xu \cite{Xu2018-LPacketsQuasisplitGSp2nGO2n}.

   Then the following assertions hold.
   \begin{enumerate}
      \item
      Each member $\pi$ of the distinguished $\PGSO_8(F)$-packet has a canonical extension $\widetilde\pi$ to $\PGSO_8(F) \rtimes \theta$.
      \item
      For matching test functions,
      \begin{align*}
         \sum_{\pi \in \Pi_{\iota_{\rG_2}^{\PGSO_8} \circ \phi^{\rG_2}}} \tr\widetilde{\pi}(f^{\PGSO_8 \rtimes \theta})
         =
          \sum_{\sigma \in \Pi^{\mathrm{GS}}_{\phi^{\rG_2}}} \dim \rho_{\sigma} \tr\sigma(f^{\PGSO_8 \rtimes \theta}_{\rG_2}),
      \end{align*}
      where $\rho_{\sigma}$ is the irreducible representation of the component group $\cS_{\phi^{\rG_2}}$ corresponding to $\sigma$ and $f^{\PGSO_8\rtimes\theta}_{\rG_2}$ is the transfer of $f^{\PGSO_8\rtimes\theta}$.
      Furthermore, the twisted endoscopic character identities between $\PGSO_8(F) \rtimes \theta$ and $\SO_4(F)$ are also satisfied.
      \item
      The $L$-packet $\Pi_{\phi^{\rG_2}}^{\mathrm{GS}}$ satisfies the endoscopic character identities for all the elliptic endoscopic groups of $\rG_2$. For any element $s \in \cS_{\phi^{\rG_2}}$, we have
      \begin{align*}
         \sum_{\sigma \in \Pi^{\mathrm{GS}}_{\phi^{\rG_2}}} \tr\rho_{\sigma}(s) \tr\sigma(f^{\rG_2})
         =
         S^{H_s}_{\phi^{H_s}}(f^{\rG_2}_{H_s}),
      \end{align*}
      where $H_s$ is the elliptic endoscopic group of $\rG_2$ associated to $s$, $\phi^{H_s}$ is the $L$-parameter for $H_s$ such that $\iota_{H_s}^{\rG_2} \circ \phi^{H_s} = \phi^{\rG_2}$, and $S^{H_s}_{\phi^{H_s}}$ is the stable character of the $L$-packet $\Pi_{\phi^{H_s}}$.
   \end{enumerate}
\end{theorem}

By \cite{Takanashi2026-HiragaIchinoIkedaConjectureFormalDegrees}, this implies the Hiraga--Ichino--Ikeda conjecture for formal degrees in the Gan--Savin $L$-packets.
We also remove the exception in residue characteristic $3$ in the internal parametrization of the Gan--Savin $L$-packets. See Remark~\ref{remark:completion-LLC}.

\subsection{The second main theorem: A global application}

Let $F$ be a number field.
The elliptic endoscopic group $\PGL_3$ of $\rG_2$ has an $L$-embedding $\iota^{\rG_2}_{\PGL_3} \colon {}^L\PGL_3 \to {}^L\rG_2$.
We determine the global $A$-packets of $\rG_2$ associated to cuspidal representations of $\PGL_3(\A_F)$.

\begin{theorem}[Theorem~\ref{theorem:Arthur-multiplicity-PGL3}]
  Let $\rho$ be a cuspidal representation of $\PGL_3(\A_F)$.
  Let $\psi$ be a formal parameter with associated Hecke--Satake family $c(\psi) = \iota^{\rG_2}_{\PGL_3} \circ c(\rho)$.
  We define the global $S$-group as
  \begin{align*}
    \begin{cases}
      \cS_{\psi} = A_3 \subset S_3 & \text{if $\rho$ is not self-dual},\\
      \cS_{\psi} = S_3 & \text{if $\rho$ is self-dual}.
    \end{cases}
  \end{align*}

  Then, the discrete spectrum $\cA^2(\rG_2)_{\psi}$ of $\rG_2(\A_F)$ associated to the Hecke--Satake family $c(\psi)$ is given by
  \begin{align*}
    \cA^2(\rG_2)_{\psi}
    =
    \sum_{\sigma \in \Pi_{\psi}} m_{\mathrm{A}}(\sigma) \sigma,
  \end{align*}
  where $\Pi_{\psi} = \otimes_{v} ' \Pi_{\iota_{\PGL_3}^{\rG_2} \circ \phi_{\rho_v}}$ is the global $A$-packet of $\rG_2(\A_F)$ associated to $\psi$ and $m_{\mathrm{A}}(\sigma)$ is the multiplicity of $\sigma$ conjectured by Arthur and defined using the global $S$-group $\cS_{\psi}$.
\end{theorem}

This resolves the problem in \cite{GanGurevich2005-NonTemperedArthurPacketsG2}*{p.142} and gives representations in tempered $L$-packets of $\rG_2$ with unbounded multiplicities, as in \cite{GanGurevich2005-NonTemperedArthurPacketsG2}*{Section 4}.
An analogous formula for discrete generic parameters lifted from $\SO_4$ is proved in Theorem~\ref{theorem:Arthur-multiplicity-SO4}.
Further study of local and global $A$-packets of $\rG_2$ will appear in a collaborative work with Wee Teck Gan and Sug Woo Shin as announced in \cite{Gan2025-TrialityAdjointLiftingGL3}.

\begin{remark}\label{remark:weighted-fundamental-lemma-hypothesis}
  We assume the twisted weighted fundamental lemma for the unramified twisted endoscopic Levi data that occur for $\widetilde{\PGSO_8}$.
  Under this hypothesis we use the stabilization of M{\oe}glin--Waldspurger \cites{MoeglinWaldspurger2016-StabilisationFormuleTracesTordue1,MoeglinWaldspurger2016-StabilisationFormuleTracesTordue2}.
  For the weighted fundamental lemma and its reduction, see \cite{MoeglinWaldspurger2016-StabilisationFormuleTracesTordue1}*{II.4.4} and \cite{AtobeGanEtAl2024-LocalIntertwiningRelationsCoTemperedPackets}*{Section 0.4}.
  Halleck-Dub\'e proves a cohomological version for quasi-split groups in \cite{HalleckDube2026-CohomologicalWeightedFundamentalLemmaNonSplitGroups}*{Theorem 15.14}. The subsequent manuscript \cite{HalleckDube2026-WeightedFundamentalLemmaNonSplitGroups}*{Theorems 16.6 and 17.9} proves the Lie algebra and nonstandard weighted fundamental lemmas under the stated characteristic assumptions, and hence the twisted weighted fundamental lemma in good characteristic.
\end{remark}

\begin{Ac}
  This project began while the author was a graduate student.
  The author is grateful to Yoichi Mieda, his advisor at the time, for his guidance and encouragement.
  The author would like to express his sincere gratitude to Wee Teck Gan and Sug Woo Shin for their support throughout the preparation of this paper.
  The author is also grateful to the Korea Institute for Advanced Study (KIAS).
  The author also thanks Chonnam National University for its support during a visit hosted by Yeansu Kim.
  This work was supported by JSPS KAKENHI Grant Number JP26KJ0167. 
\end{Ac}

\begin{AIdisclosure}
  After writing the initial draft, the author used GPT-6 Astra only for proofreading and reviewing the mathematical content of this paper, and for assistance with drafting in Sections~\ref{theta-A2:sec:detection}--\ref{theta-A2:sec:regularization} and~\ref{theta-A2:sec:real-models}--\ref{theta-A2:sec:local-AV}.
  The reference material used for this drafting assistance included \cite{GanSavin2022-TwistedCompositionAlgebrasArthurPacketsTrialitySpin8} and \cite{Gan2005-MultiplicityFormulaCubicUnipotentArthurPackets} for Sections~\ref{theta-A2:sec:detection}, \ref{theta-A2:sec:real-models}, and~\ref{theta-A2:sec:local-AV}, and \cite{GanSavin2020-ExceptionalSiegelWeilFormulaPolesSpin} for Section~\ref{theta-A2:sec:regularization}.

  The author is responsible for all statements and proofs in this paper.
\end{AIdisclosure}

\part{Preliminaries}\label{part:preliminaries}

\section{General notation and assumptions}\label{section:general-notation}

\subsection{Reductive groups over local fields}

Let $F$ be a local field of characteristic zero.
Let $G$ be a connected reductive group defined over $F$.
We mainly follow the notation in \cite{Takanashi2026-HiragaIchinoIkedaConjectureFormalDegrees}.
Let $\widehat{G}$ denote the complex dual group of $G$ and ${}^LG$ denote the $L$-group of $G$.
Let $\Phi(G)$ denote the set of $L$-parameters for $G$ and $\Pi(G)$ (resp. $\Pi_{\temp}(G)$, $\Pi_{\disc}(G)$) denote the set of isomorphism classes of irreducible (resp. tempered irreducible, discrete irreducible) smooth representations of $G(F)$.

In the endoscopic character identities, $\sigma$ denotes a member of a $\rG_2$-packet and $\rho_\sigma$ its component-group representation. We write $\tr\sigma$ for the distribution character of $\sigma$. For packet distributions the superscript records the group. $S^H_{\phi^H}$ is the stable packet character on $H(F)$, and $I^{G}_{\phi^H}$ denotes its endoscopic transfer to $G(F)$.

For a Levi subgroup or an endoscopic group $H$ of $G$, let
\begin{align*}
    \iota_{H}^{G} \colon {}^LH \to {}^LG
\end{align*}
denote the $L$-embedding associated to $H$.

Let $W_F$ be the Weil group of $F$. Put $L_F=W_F\times\SL_2(\C)$ for $p$-adic $F$ and $L_F=W_F$ for archimedean $F$. We identify $\Phi(\GL_n)$ with the isomorphism classes of semisimple $n$-dimensional representations of $L_F$ that are continuous on $W_F$ and algebraic on the $\SL_2(\C)$ factor when present.
For $p$-adic $F$, the Deligne $\SL_2$ part of a parameter is its restriction to $\{1\}\times\SL_2(\C)$ in $L_F$. For an $A$-parameter with domain $L_F\times\SL_2(\C)$, the last factor is the Arthur $\SL_2$. Write $S_d^A$ for its $d$-dimensional irreducible representation.

\subsection{Reductive groups over global fields}

Let $F$ be a number field and $\A = \A_F$ be the ring of adeles of $F$.
Let $S_{\infty}$ (resp. $S^{\infty}$) denote the set of infinite (resp. finite) places of $F$.
Let $G$ be a connected reductive group defined over $F$.

Let $S$ be a set of places of $F$.
Let $\A_S = \A_{F, S}$ denote the restricted direct product of $F_v$ for $v \in S$.
For an irreducible admissible smooth representation $\Pi$ of $G(\A_S)$, let $\Pi_v$ denote the local component of $\Pi$ at the place $v$.

Let $\Pi$ be an automorphic representation of $G(\A)$.
Let $S$ be a finite set of places of $F$ containing $S_{\infty}$ such that $\Pi_v$ is unramified for any $v \notin S$.
Let $c(\Pi) = \{c(\Pi_v)\}_{v \notin S}$ be its Hecke--Satake family of semisimple conjugacy classes in $\widehat{G}$ defined by the Satake isomorphism.

\subsection{Orbital integrals and transfer factors}\label{section:orbital-integrals-transfer-factors}

Let $F$ be a local field of characteristic zero.
We use the orbital integrals and compatible measures of \cite{MoeglinWaldspurger2016-StabilisationFormuleTracesTordue1}*{I.2.4}.
For a strongly regular semisimple element $\gamma\in\widetilde G(F)$, put $G_\gamma=Z_G(\gamma)^\circ$ and $D^{\widetilde G}(\gamma)=\det(1-\operatorname{Ad}(\gamma)\mid\mathfrak g/\mathfrak g_\gamma)$.
We set
\begin{align*}
  I^{\widetilde G}(\gamma,f)
  =\abs{D^{\widetilde G}(\gamma)}^{1/2}
  \int_{G_\gamma(F)\backslash G(F)}f(g^{-1}\gamma g)\rd g.
\end{align*}
The measures on $G_\gamma(F)$ and $H_\delta(F)$ correspond locally under the norm map for related strongly regular elements $\gamma$ and $\delta$.
We use the Whittaker-normalized absolute Kottwitz--Shelstad factor of \cite{Kaletha2022-LocalLanglandsConjecturesDisconnectedGroups}*{Section 4.9}, with the factor $\Delta_{\mathrm{IV}}$ omitted, and denote it by $\Delta(\delta,\gamma)$.
Thus transfer is defined by
\begin{align*}
  S^H(\delta,f_H)
  =d(\theta)^{1/2}\sum_{\gamma\leftrightarrow\delta}
  \frac{\Delta(\delta,\gamma)}{[Z_G(\gamma)(F):G_\gamma(F)]}
  I^{\widetilde G}(\gamma,f),
\end{align*}
where the sum is over rational conjugacy classes and
\begin{align*}
  d(\theta)=\abs{\det(1-\theta\mid\mathfrak t/\mathfrak t^\theta)}_F.
\end{align*}
Here $\theta$ is the automorphism of a pinned maximal torus $T$ of the quasi-split inner form induced by the twisting.
In particular, $d(\theta)=\abs{3}_F$ for $\widetilde{\PGSO_8}$.

\subsection{Specific groups}

Let $F$ be any field.
Let $\rG_2$ denote the split exceptional group of type $\rG_2$ defined over $F$.
Let $\PGSO_8$ (resp. $\PGSp_6$) denote the split simple adjoint group of type $D_4$ (resp. $C_3$) defined over $F$.
Let $(T, B, \{X_{\alpha}\}_\alpha)$ denote a pinning of $\PGSO_8$ defined over $F$.
We write the set of simple roots of $\PGSO_8$ with respect to $(T, B)$ as $\Delta = \{\alpha_1, \alpha_2, \alpha_3, \alpha_4\}$, where $\alpha_2$ is the central root in the Dynkin diagram of type $D_4$, following the labeling in the book of Bourbaki \cite{Bourbaki2002-LieGroupsLieAlgebrasChapters4-6}.
Let $\theta$ denote one of the two outer automorphisms of $\PGSO_8$ of order $3$ defined over $F$ that preserve the pinning $(T, B, \{X_{\alpha}\}_\alpha)$.
These structures are defined over $\Z$.

For any classical or similitude classical group $G$, let $\Std_{G}$ denote the standard representation of ${}^LG$ in some $\GL_N$.
We also denote the $7$-dimensional standard representation of ${}^L\rG_2$ by $\Std_{\rG_2}$.
For $G=\PGSp_6$ or $\PGSO_8$, write $\operatorname{Lift}_{G}^{\GL_N}(\pi)$ for the transfer attached to $\Std_G$, with $N=7$ or $8$, respectively.

\subsection{Topological groups}
Unless otherwise stated, all the representations of topological groups are assumed to be defined on complex vector spaces.
For a locally compact group $X$, let $X^{\cD}$ denote its unitary dual.

\section{Exceptional theta correspondences for \texorpdfstring{$\rG_2$}{G2}}\label{section:exceptional-theta-correspondences}

Let $F$ be a $p$-adic field of characteristic zero and $D$ a cubic division algebra over $F$.
Let $\PGL_3^{+}$ be the semidirect product of $\PGL_3$ and $\Z/2\Z$ defined by the outer automorphism of $\PGL_3$ of order $2$ defined over $F$ and preserving the standard pinning.

The Gan--Savin correspondence \cite{GanSavin2023-LocalLanglandsConjectureG2}
\begin{align*}
  \Pi(\rG_2(F)) \to \Phi(\rG_2(F)) \colon \pi \mapsto \phi_{\pi},
\end{align*}
is compatible with parabolic induction and satisfies the generic packet conjecture.
Its relation to exceptional theta correspondences is as follows.
\begin{theorem}\label{theorem:properties-exceptional-theta}
  We have exceptional theta liftings
  \begin{itemize}
    \item from $\rG_2$ to $\PGSp_6$
          \begin{align*}
            \theta_{\rG_2}^{\PGSp_6} \colon \Pi(\rG_2(F)) \to \Pi(\PGSp_6(F)) \cup \{0\},
          \end{align*}
    \item from $\rG_2$ to $\mathrm{PD}^{\times}$
          \begin{align*}
            \theta_{\rG_2}^{\mathrm{PD}^{\times}} \colon \Pi(\rG_2(F)) \to \Pi(\mathrm{PD}^{\times}(F)) \cup \{0\},
          \end{align*}
    \item and from $\rG_2$ to $\PGL_3^{+}$
          \begin{align*}
            \theta_{\rG_2}^{\PGL_3^{+}} \colon \Pi(\rG_2(F)) \to \Pi(\PGL_3^{+}(F)) \cup \{0\}.
          \end{align*}
  \end{itemize}

  These liftings satisfy the following properties.
  \begin{enumerate}
    \item
    These liftings preserve temperedness and genericity.
    \item
    Any generic representation $\pi \in \Pi(\rG_2(F))$ lifts to a nonzero generic representation of $\PGSp_6(F)$.
    \item
    For $\pi \in \Pi(\rG_2(F))$, exactly one of $\theta_{\rG_2}^{\PGSp_6}(\pi)$ and $\theta_{\rG_2}^{\mathrm{PD}^{\times}}(\pi)$ is nonzero.
    \item
    For any $\pi \in \Pi_{\disc}(\rG_2(F))$, there exists a unique group $H \in \{\PGSp_6, \mathrm{PD}^{\times}, \PGL_3^{+}\}$ such that $\theta_{\rG_2}^H(\pi)$ is a nonzero discrete series representation.
    \item
    A representation $\sigma \in \Pi(\PGSp_6(F))$ lies in the image of $\theta_{\rG_2}^{\PGSp_6}$ if and only if the spin $L$-parameter defined by Gan--Savin contains a trivial representation of $L_F$.
    \item
    For any $\phi \in \Phi_{\disc}(\rG_2(F))$, all the elements of the $L$-packet $\Pi_{\phi}(\rG_2(F))$ lift to $\PGSp_6(F)$ if and only if the standard $L$-parameter $\Std_{\rG_2} \circ \phi$ does not contain the trivial representation of $L_F$.
    \item
    The liftings are compatible with the local Langlands correspondences, in the following sense.
      \begin{itemize}
        \item
        If $\theta_{\rG_2}^{\PGSp_6}(\pi) \neq 0$, then the standard $L$-parameter of $\pi$ is equal to the standard $L$-parameter of $\theta_{\rG_2}^{\PGSp_6}(\pi)$.
        If, in addition, $\phi_{\pi}$ is a discrete parameter of $\PGSp_6$-type, then we have the canonical isomorphism of component groups
        \begin{align*}
          \cS_{\phi_{\pi}} \cong \cS_{\iota_{\rG_2}^{\PGSp_6} \circ \phi_{\pi}}.
        \end{align*}
        Here the component group on the $\PGSp_6$ side is taken modulo $Z(\Spin_7)\cong\mu_2$. See \cite{GanSavin2023-LocalLanglandsConjectureG2}*{Lemma 2.4(ii)}.
        \item
        If $\theta_{\rG_2}^{\mathrm{PD}^{\times}}(\pi) \neq 0$, then the standard $L$-parameter of $\pi$ is equal to
        \begin{align*}
          \phi_{\theta_{\rG_2}^{\mathrm{PD}^{\times}}(\pi)} \oplus \phi_{\theta_{\rG_2}^{\mathrm{PD}^{\times}}(\pi)}^{\vee} \oplus \mathbf{1}.
        \end{align*}
        \item
        If $\theta_{\rG_2}^{\PGL_3^{+}}(\pi) \neq 0$, then the standard $L$-parameter of $\pi$ is equal to
        \begin{align*}
          \phi_{\tau} \oplus \phi_{\tau}^{\vee} \oplus \mathbf{1},
        \end{align*}
        where $\tau$ is an irreducible representation of $\PGL_3(F)$ which appears in the restriction of $\theta_{\rG_2}^{\PGL_3^{+}}(\pi)$ to $\PGL_3(F)$.
      \end{itemize}

    We also have the canonical injection for any $\phi \in \Phi(\rG_2(F))$
    \begin{align*}
      \Pi_{\phi} \hookrightarrow \cS_{\phi}^{\cD},
    \end{align*}
    which is a bijection unless the residue characteristic of $F$ is
    $3$ and $\phi$ is a discrete parameter which factors through
    ${}^L\PGL_3$.  In particular, if
    $\phi\in\Phi_{\disc}(\rG_2(F))_{\PGSp_6}$, then this map is a
    bijection.  In this case $\cS_{\phi}$ is an elementary abelian
    $2$-group, and hence
    \begin{align*}
      \abs{\Pi_{\phi}}=\abs{\cS_{\phi}}.
    \end{align*}
  \end{enumerate}
\end{theorem}

\begin{remark}
  The backward liftings have similar properties and are denoted, for example, by $\theta_{\PGSp_6}^{\rG_2}$.
\end{remark}

\begin{definition}
  Let $\Phi_{\disc}(\rG_2(F))_{\PGSp_6}$ (resp. $\Phi_{\disc}(\rG_2(F))_{\PGL_3^{+}}$) denote the set of discrete $L$-parameters $\phi \in \Phi_{\disc}(\rG_2(F))$ such that the standard $L$-parameter $\Std_{\rG_2}\circ\phi$ does not contain the trivial representation of $L_F$ (resp. contains the trivial representation).
\end{definition}

\begin{lemma}[\cite{Takanashi2026-HiragaIchinoIkedaConjectureFormalDegrees}*{Lemma 3.57}]\label{lemma:injectivity-G_2-to-GL7}
    The canonical map
    \begin{align*}
        \Phi(\rG_2(F)) \to \Phi(\GL_7(F)),
    \end{align*}
    induced by the standard representation of $\rG_2(\C)$ is injective.
    Hence, the morphism
    \begin{align*}
        \Phi(\rG_2(F)) \to \Phi(\GL_7(F)) \to \Phi(\GL_8(F)),
    \end{align*}
    obtained by adding the trivial representation is also injective.
\end{lemma}

\begin{proof}
    This follows from \cite{Larsen1994-ConjugacyElementConjugateHomomorphisms}*{Proposition 2.8} and \cite{Griess1995-BasicConjugacyTheoremsG2}*{1.1 Theorem 1}.
\end{proof}

\begin{corollary}
  The maps
  \begin{align*}
    \Phi(\rG_2(F)) &\to \Phi(\PGSp_6(F)), \\
    \Phi(\rG_2(F)) &\to \Phi(\PGSO_8(F))
  \end{align*}
  are injective.
\end{corollary}

\begin{proof}
  Their composites with the standard representations into $\GL_7$ and $\GL_8$, respectively, are injective by Lemma~\ref{lemma:injectivity-G_2-to-GL7}.
\end{proof}

\section{Theta correspondence between \texorpdfstring{$\PGSp_6$}{PGSp6} and \texorpdfstring{$\PGSO_8$}{PGSO8}}

Let $\iota^{\GSO_8}_{\GSp_6}$ denote the $L$-embedding ${}^L \GSp_6 \to {}^L \GSO_8$.
We use Xu's local Langlands correspondences
\begin{align*}
        {\cL_{\GSO_{2n}}}_{/ \O_{2n} \times \GL_1}
        \colon
        \Pi_{\temp}(\GSO_{2n}(F))_{/ \mathrm{GO}_{2n}}
        \to
        \Phi_{\temp}(\GSO_{2n}(F))_{/ \O_{2n} \times \GL_1}
\end{align*}
and
\begin{align*}
        {\cL_{\GSp_{2n}}}_{/ \Sp_{2n} \times \GL_1}
        \colon
        \Pi_{\temp}(\GSp_{2n}(F))
        \to
        \Phi_{\temp}(\GSp_{2n}(F))_{/ \Sp_{2n} \times \GL_1},
\end{align*}
for $\GSO_{2n}$ and $\GSp_{2n}$, as recalled in \cite{Takanashi2026-HiragaIchinoIkedaConjectureFormalDegrees}*{Theorems 3.35 and 3.39}.

\begin{theorem}[cf. \cite{Takanashi2026-HiragaIchinoIkedaConjectureFormalDegrees}*{Theorem 3.51}]\label{theorem:similitude-theta}
    There exists a map
    \begin{align*}
        \theta^{\PGSO_8}_{\PGSp_6} \colon \Pi_{\temp}(\PGSp_{6}(F)) \to \Pi_{\temp}(\PGSO_8(F)) \cup \{ 0 \}.
    \end{align*}
    The map satisfies the following properties.

    \begin{enumerate}
    \item
    The image is equal to the disjoint union of $\Pi_{\temp}(\PGSO_8(F))_{\theta} =  \{ \pi \in \Pi_{\temp}(\PGSO_8(F)) \mid \gamma(0, \iota^{\GL_{8}}_{\SO_8} \circ \phi^{\SO_8}_{\pi}, \psi) = 0  \}$ and $\{ 0 \}$.

    \item
    The map is compatible with the local Langlands correspondence in the following sense.
    For any tempered irreducible representation $\sigma$ of $\PGSp_6(F)$ such that $\theta^{\PGSO_8}_{\PGSp_6}(\sigma) \neq 0$, we have
    \begin{align*}
           {\cL_{\mathrm{GSO}_8}}_{/ \O_8 \times \GL_1}
           (\theta^{\PGSO_8}_{\mathrm{PGSp}_6}(\sigma))
           =
           [
           \iota^{\GSO_8}_{\GSp_6}
           \circ
           {\cL_{\GSp_6}}_{/ \Sp_6 \times \GL_1} (\sigma)]_{\O_8 \times \GL_1}.
    \end{align*}

    \item
    If we set
    \begin{align*}
        \Pi_{\temp}(\PGSp_6(F))_{\theta} = \{ \sigma \in \Pi_{\temp}(\PGSp_6(F)) \mid  \theta^{\PGSO_8}_{\PGSp_6}(\sigma) \neq 0 \},
    \end{align*}
    then the map $\theta^{\PGSO_8}_{\PGSp_6}$ induces a bijection
    \begin{align*}
        \theta^{\PGSO_8}_{\PGSp_6} \colon
        \Pi_{\temp}(\PGSp_{6}(F))_{\theta} \xrightarrow{\sim} \Pi_{\temp}(\PGSO_8(F))_{\theta}.
    \end{align*}

    \item
    The map $\theta^{\PGSO_8}_{\PGSp_6}$ maps generic representations to generic representations.
    The inverse image of any generic representation is generic.
    \end{enumerate}
\end{theorem}

\section{On the Howe duality property}

\begin{theorem}[\cite{GanTakeda2016-ProofHoweDualityConjecture}, \cite{BakicGanSavin2023-SimilitudeExceptionalThetaCorrespondences}, \cite{GanSavin2023-HoweDualityDichotomyExceptionalThetaCorrespondences}]
  Let $F$ be a non-archimedean local field of characteristic zero.
  The Howe duality property holds for the following dual pairs.
  \begin{itemize}
    \item $(\rG_2(F), \PGSp_6(F))$,
    \item $(\GSp_{2n}(F), \GSO_{2n+2}(F))$ for any $n \geq 1$,
    \item $(\PGL_3^+(F), \rG_2(F))$,
    \item $(\mathrm{PD}^{\times}(F), \rG_2(F))$ for any cubic division algebra $\mathrm{D}$ over $F$.
  \end{itemize}
\end{theorem}

\begin{theorem}[\cite{Howe1989-TranscendingClassicalInvariantTheory}, \cite{BakicGanSavin2023-SimilitudeExceptionalThetaCorrespondences}, \cite{LokeSavin2019-DualitySphericalRepresentationsExceptionalThetaCorrespondences}]
  Let $F = \R$.
  The Howe duality property holds for $(\GSp_{2n}(F), \GSO_{2n+2}(F))$ for every $n \geq 1$, and for $(\rG_2(F), \PGSp_6(F))$ and $(\PGL_3^+(F), \rG_2(F))$ on spherical representations.
  In general, any theta lifting is compatible with the formation of infinitesimal characters.
\end{theorem}

\begin{theorem}\label{theorem:spherical-theta-parameters}
  Let $F$ be a local field of characteristic zero.
  The theta liftings for the following dual pairs are compatible with the morphism on the $L$-groups for spherical representations.
  \begin{itemize}
    \item $(\rG_2(F), \PGSp_6(F))$,
    \item $(\GSp_{2n}(F), \GSO_{2n+2}(F))$ for any $n \geq 1$,
    \item $(\PGL_3^+(F), \rG_2(F))$.
  \end{itemize}
\end{theorem}

\begin{proof}
  For the exceptional pairs over non-archimedean fields, see \cite{SavinWoodbury2015-MatchingHeckeOperatorsExceptionalDualPairs}*{Theorem 1.1}. For $\PGL_3^+$, restrict to $\PGL_3$.
  Over archimedean fields, the correspondence of infinitesimal characters is given in \cite{Li1999-CorrespondencesInfinitesimalCharactersReductiveDualPairs}*{Lemma 2.1, Theorem 8.1 and Table 1}. For spherical real representations, see also \cite{LokeSavin2019-DualitySphericalRepresentationsExceptionalThetaCorrespondences}*{Corollaries 4.5 and 4.8}.
  For the classical pair over non-archimedean fields, \cite{ChenevierGan2025-TrialityFunctoriality}*{Proposition 8.8(ii), (iii)}, with $m=n+1$, gives compatibility with the natural embedding $\GSpin_{2n+1}(\C)\hookrightarrow\GSpin_{2n+2}(\C)$. We use the normalization of \cite{ChenevierGan2025-TrialityFunctoriality}*{Section 8.1}.
  Over archimedean fields, see \cite{Przebinda1996-DualityCorrespondenceInfinitesimalCharacters}*{Theorems 1.13 and 1.19}. In this case, the passage to similitudes is as in \cite{BakicGanSavin2023-SimilitudeExceptionalThetaCorrespondences}*{Theorem 5.1 and Section 6.1}.
\end{proof}

\part{Local representation theory}\label{part:local-representation-theory}

\section{The distinguished \texorpdfstring{$L$-packets}{L-packets} for \texorpdfstring{$\PGSO_8(F)$}{PGSO8(F)}}

Let $F$ be a local field of characteristic zero.
First suppose that $F = \R$.
\begin{lemma}\label{lemma:lifting-discrete-parameter-archimedean-G2-twisted-PGSO8}
  Let $\phi^{\rG_2}$ be a discrete $L$-parameter for $\rG_2(\R)$.
  Then, the lift $\iota_{\rG_2}^{\PGSO_8} \circ \phi^{\rG_2}$ is a discrete $L$-parameter for $\PGSO_8(\R)$.
\end{lemma}

\begin{proof}
  The lift of a discrete $L$-parameter for $\rG_2(\R)$ to $\PGSp_6(\R)$ is discrete, as follows from \cite{GrossSavin1998-MotivesGaloisGroupTypeG2Exceptional}*{Theorem 3.5, Corollary 3.9}.
  It remains to show that the standard $L$-parameter of $\iota_{\rG_2}^{\PGSp_6} \circ \phi^{\rG_2}$ contains no trivial representation of $W_{\R}$.
  If the trivial representation occurs, the standard $L$-parameter of $\phi^{\rG_2}$ also contains a trivial representation of $W_{\R}$ since the standard representation of $\rG_2(\C)$ is contained in the standard representation of $\Spin_7(\C)$.
  Then $\phi^{\rG_2}$ factors through $\SL_3(\C) \subset \rG_2(\C)$, contradicting the discreteness of $\iota_{\rG_2}^{\PGSp_6} \circ \phi^{\rG_2}$.
\end{proof}

\begin{lemma}\label{lemma:invariance-lifting-archimedean}
  Let $\phi^{\rG_2}$ be a discrete $L$-parameter for $\rG_2(\R)$.
  Then, every representation in the lifted $L$-packet
  \[
    \Pi_{\iota_{\rG_2}^{\PGSO_8} \circ \phi^{\rG_2}}(\PGSO_8(\R))
  \]
  is invariant under the outer automorphisms of $\PGSO_8(\R)$.
\end{lemma}

\begin{proof}
  Put $\psi=\iota_{\rG_2}^{\PGSO_8}\circ\phi^{\rG_2}$.
  By the preceding lemma, $\psi$ is discrete.
  Its restriction to $\C^\times\subset W_{\R}$ has centralizer a maximal torus
  $\widehat T\subset\Spin_8(\C)$, and $\psi(j)$ acts on $\widehat T$ by inversion.
  Consequently,
  \begin{align*}
    \cS_\psi=\widehat T[2]/Z(\Spin_8(\C)).
  \end{align*}
  Let $Q$ and $P$ be the root and weight lattices of $D_4$.
  Using the simply laced identification with the coroot and coweight lattices, we have
  \begin{align*}
    \widehat T[2]=Q/2Q,\quad
    Z(\Spin_8(\C))=2P/2Q,\quad
    \cS_\psi=Q/2P.
  \end{align*}
  Number the simple roots as in Section~\ref{section:general-notation}, so that
  $\alpha_2$ is the central node.
  If $\omega_1,\omega_3$ are the corresponding fundamental weights, then
  \begin{align*}
    2\omega_1&=2\alpha_1+2\alpha_2+\alpha_3+\alpha_4,\\
    2\omega_3&=\alpha_1+2\alpha_2+2\alpha_3+\alpha_4.
  \end{align*}
  Thus $[\alpha_1]=[\alpha_3]=[\alpha_4]$ in $Q/2P$, which is generated by
  $[\alpha_1]$ and $[\alpha_2]$.
  The diagram automorphism group $S_3$ permutes the three outer nodes and fixes
  $\alpha_2$, so its action on $\cS_\psi$ is trivial.
  The image of $\psi$ is fixed by these automorphisms, and the pinned automorphisms
  preserve the Whittaker datum.
  The Tate--Nakayama description of the packet pairing in
  \cite{Shelstad2008-TemperedEndoscopyRealGroupsIII}*{Section 8, Corollaries 11.1 and 11.6}
  is natural under these automorphisms. Since they fix the generic base point,
  their action on enhancements is dual to their action on $\cS_\psi$.
  Thus every representation in the packet is fixed.
\end{proof}

Now suppose that $F$ is $p$-adic.
\begin{theorem}[\cite{Takanashi2026-HiragaIchinoIkedaConjectureFormalDegrees}*{Proposition 6.6}]\label{theorem:invariance-lifting-non-archimedean}
  Let $\phi^{\rG_2}$ be a tempered $L$-parameter for $\rG_2(F)$.
  Let $\sigma$ be a representation in the $L$-packet $\Pi_{\phi^{\rG_2}}(\rG_2(F))$ such that $\theta_{\PGSp_6}^{\PGSO_8} \circ \theta_{\rG_2}^{\PGSp_6}(\sigma) \neq 0$.
  Then, the $L$-packet constructed by Xu containing $\theta_{\PGSp_6}^{\PGSO_8} \circ \theta_{\rG_2}^{\PGSp_6}(\sigma)$ consists of representations which are invariant under all the outer automorphisms of $\PGSO_8(F)$.
\end{theorem}

\begin{definition}
   Let $F$ be a local field of characteristic zero.
   Let $\phi^{\rG_2}$ be an $L$-parameter for $\rG_2(F)$.
   In any case, we also call the $L$-packet a distinguished $L$-packet for $\PGSO_8(F)$ and let $\Pi_{\iota_{\rG_2}^{\PGSO_8} \circ \phi^{\rG_2}}(\PGSO_8(F))$ denote the distinguished $L$-packet for $\PGSO_8(F)$.
   We often write it as $\Pi_{\iota_{\rG_2}^{\PGSO_8} \circ \phi^{\rG_2}}$ if the context is clear.
\end{definition}

\begin{theorem}[\cite{Takanashi2026-HiragaIchinoIkedaConjectureFormalDegrees}*{Proposition 3.58, pp.~25--26}]\label{theorem:characterization-distinguished-L-packets}
  Let $F$ be a non-archimedean local field of characteristic zero, let $\psi$ be a non-trivial additive character of $F$, and let $s_1$ be a pinning-preserving outer automorphism of $\PGSO_8$ of order two exchanging the standard representation and a half-spin representation.
  Let $\pi$ be a tempered irreducible representation of $\PGSO_8(F)$ such that
  \begin{align*}
    \gamma(0, \Std_{\PGSO_8} \circ \phi_{\pi}^{\PGSO_8}, \psi) &= 0, \\
    \gamma(0, \Std_{\PGSO_8} \circ \phi_{s_1\pi}^{\PGSO_8}, \psi) &= 0.
  \end{align*}
  Then, there exists a unique tempered irreducible representation $\pi^{\rG_2}$ of $\rG_2(F)$ such that
  \begin{align*}
    \theta_{\PGSp_6}^{\PGSO_8} \circ \theta_{\rG_2}^{\PGSp_6}(\pi^{\rG_2}) = \pi.
  \end{align*}
  The resulting map to the set of tempered $L$-parameters for $\rG_2(F)$ is surjective, and each fiber is a distinguished Xu $L$-packet.
\end{theorem}

\section{Canonical extensions to \texorpdfstring{$\widetilde{\PGSO_8}(F)$}{widetilde{PGSO8(F)}}}\label{section:canonical-extensions}

We fix a pinning $(T, B, \{X_{\alpha}\}_\alpha)$ of $\PGSO_8$ defined over $F$.
Let $\PGSO_8^{+}(F)$ denote the semidirect product of $\PGSO_8(F)$ and the automorphism group $\Aut_{\mathrm{pin}}(\PGSO_8)$ of $\PGSO_8(F)$ defined over $F$ which preserves the pinning $(T, B, \{X_{\alpha}\}_\alpha)$.
We set $\widetilde{\PGSO_8} = \PGSO_8 \rtimes \theta$.

\begin{lemma}
  We have
  \begin{align*}
    H^2(S_3, \C^{\times}) = 0.
  \end{align*}
\end{lemma}

\begin{proof}
  This follows from the Lyndon-Hochschild-Serre spectral sequence associated to the short exact sequence
  \begin{align*}
    1 \to A_3 \to S_3 \to S_3/A_3 \to 1
  \end{align*}
\end{proof}

\begin{corollary}
  Let $F$ be a local field of characteristic zero.
  Let $\pi$ be an irreducible representation of $\PGSO_8(F)$ which is invariant under the outer automorphisms of $\PGSO_8(F)$.
  Then, there exists an extension of $\pi$ to a representation of $\PGSO_8^{+}(F)$.
\end{corollary}

\begin{lemma}
  Let $F$ be a local field of characteristic zero.
  If $\pi$ is an irreducible representation of $\PGSO_8(F)$ which is invariant under the outer automorphisms of $\PGSO_8(F)$, then the restriction of any extension of $\pi$ to $\PGSO_8^{+}(F)$ to $\widetilde{\PGSO_8}(F)$ defines the same representation of $\widetilde{\PGSO_8}(F)$.
\end{lemma}

\begin{proof}
  The sign character of $S_3$ satisfies $\sgn|_{A_3} = 1$.
\end{proof}

\begin{definition}
  Let $F$ be a local field of characteristic zero.
  Let $\pi$ be an irreducible representation of $\PGSO_8(F)$ which is invariant under the outer automorphisms of $\PGSO_8(F)$.
  We call the restriction to $\widetilde{\PGSO_8}(F)$ of any extension of $\pi$ to $\PGSO_8^{+}(F)$ the canonical extension, denoted by $\widetilde{\pi}$.
\end{definition}

\begin{corollary}\label{corollary:G2-canonical-extension}
  Let $\phi^{\rG_2}$ be an $L$-parameter for $\rG_2(F)$.
  Then, any representation in the distinguished $L$-packet $\Pi_{\iota_{\rG_2}^{\PGSO_8} \circ \phi^{\rG_2}}(\PGSO_8(F))$ has a canonical extension to a representation of $\widetilde{\PGSO_8}(F)$.
  If the representation $\pi$ is generic (resp. unramified), then the canonical extension $\widetilde{\pi}$ is equal to the Whittaker extension (resp. unramified extension) of $\pi$ to $\widetilde{\PGSO_8}(F)$.
\end{corollary}

\begin{corollary}\label{corollary:canonical-extension-global}
  Let $F$ be a number field.
  Let $\Pi^{+}$ be a smooth admissible representation of $\PGSO_8(\A_F) \rtimes \Aut_{\mathrm{pin}}(\PGSO_8)$ such that the restriction of $\Pi^{+}$ to $\PGSO_8(\A_F)$ is irreducible.
  Then, the restriction of each local component $\Pi^{+}_v$ to $\widetilde{\PGSO_8}(F_v)$ is the canonical extension of the restriction of $\Pi^{+}_v$ to $\PGSO_8(F_v)$.

  In particular, let $\Pi \subset \cA(\PGSO_8(\A_F))$ be an automorphic representation of $\PGSO_8(\A_F)$ which is invariant under the canonical action of $\Aut_{\mathrm{pin}}(\PGSO_8)(F)$.
  Then, each local component $\Pi_v$ has a canonical extension to a representation of $\widetilde{\PGSO_8}(F_v)$ and the restricted tensor product of the action of $\theta$ coincides with the canonical action of $\theta$ on $\Pi$.
\end{corollary}

\section{Preliminaries on relative Weyl groups}\label{section:relative-Weyl-groups}

\subsection{A method of their computation}

We recall the Weyl-group normalizer computations in \cite{Takanashi2026-HiragaIchinoIkedaConjectureFormalDegrees}*{Section 8.2}.
Let $\Phi = (\Phi, \Phi^{\vee}, V)$ be a root system.
Let $\Phi_1$ be a root subsystem of $\Phi$.
We denote the Weyl group of $\Phi$ (resp. $\Phi_1$) by $W$ (resp. $W_1$).
We fix a set of simple roots $I_1$ of $\Phi_1$.

We compute $N_{W}(W_1) = \Norm_{W}(W_1)$ following \cite{Carter1972-ConjugacyClassesWeylGroup}.

\begin{lemma}[\cite{Carter1972-ConjugacyClassesWeylGroup}*{Proposition 28}]\label{lemma:Weyl-group-normalizer}
    Let $\Phi_2$ be the root subsystem in $\Phi_1^{\perp} =  \bigcap_{\alpha \in \Phi_1} \Ker(\alpha^{\vee})$ defined as $\Phi_1^{\perp} \cap \Phi$.
    \begin{enumerate}
        \item
        The Weyl group $W_2$ of $\Phi_2$ is contained in $N_{W}(W_1)$.
        The element $w \in W$ is in $W_2$ if and only if it fixes the roots in $\Phi_1$.
        \item
        The product $W_1 \times W_2$ is a normal subgroup of $N_{W}(W_1)$.
        The quotient $N_{W}(W_1)/(W_1 \times W_2)$ is isomorphic to the group of automorphisms $\Aut_{W}(I_1)$ of the Dynkin diagram of $\Phi_1$ induced by $W$.
    \end{enumerate}
\end{lemma}

 \begin{lemma}\label{lemma:computation-normalizer}
     We fix a set of simple roots $I_2$ of $\Phi_2$.
     Then, the normalizer $N_{W}(I_1, I_2) = \Norm_{W}(I_1) \cap \Norm_{W}(I_2)$ of the subsets $I_1$ and $I_2$ in $W$ gives a splitting of the map
     \begin{align*}
          N_{W}(W_1) \twoheadrightarrow  N_{W}(W_1)/(W_1 \times W_2).
     \end{align*}
     Hence
     \begin{align*}
        W_2 \rtimes N_{W}(I_1, I_2) \xrightarrow{\sim} N_{W}(W_1)/W_1
     \end{align*}
     as subquotients of $W$.
 \end{lemma}

\begin{remark}\label{remark:levi-weyl}
    Let $G$ be a reductive group over $F$.
    Let $(P_0, A_0)$ be a minimal parabolic pair of $G$ over $F$.
    For a Levi subgroup $M$ containing $A_0$, the roots $\Phi(M, A_0)$ form a subsystem of $\Phi(G, A_0)$, and there is a canonical isomorphism
    \begin{align*}
        N_{W(G, A_0)}(W(M, A_0))/W(M, A_0) \cong W(G, M).
    \end{align*}
    Thus, we can apply Lemmas \ref{lemma:Weyl-group-normalizer} and
    \ref{lemma:computation-normalizer} to this situation.
\end{remark}

\subsection{Compatibility with the dual setting}

\begin{lemma}\label{lemma:tits-lift-and-dual}
  Let $F$ be a non-archimedean local field of characteristic zero, and $G$ be a quasi-split connected reductive group over $F$.
   Let $\mathrm{pin}_{G}$ (resp. $\mathrm{pin}_{\widehat{G}}$) be an $F$-pinning of $G$ (resp. $\widehat{G}$).
  Let $I, J$ be subsets of the set $\Delta$ of simple roots of $G$.
  Let $w$ be an element of the Weyl group of $G$ such that $w(I) = J$.
  Let $\widehat{w}$ be the corresponding element in $W(\widehat{G}, \widehat{T}) \cong W(G, T)$.
  Let $\widetilde{w}$ (resp. $\widetilde{\widehat{w}}$) be the Tits lift of $w$ (resp. $\widehat{w}$).

  Then, we have
  \begin{align*}
    \begin{cases}
      \Ad(\widetilde{w}) X_{\alpha} = X_{w(\alpha)} \\
      \Ad(\widetilde{\widehat{w}}) X_{\alpha^{\vee}} = X_{\widehat{w}(\alpha^{\vee})}
    \end{cases}
  \end{align*}
  for any $\alpha$ which restricts to an element of $I$.
  Hence
  \begin{align*}
    \widehat{\Ad(\widetilde{w})} = \Ad(\widetilde{\widehat{w}})^{-1} \colon M_{J}^{\vee} \xrightarrow{\sim} M_{I}^{\vee}.
  \end{align*}
\end{lemma}

\begin{proof}
  This follows from \cite{Springer1998-LinearAlgebraicGroups}*{Proposition 9.3.5}.
\end{proof}

\section{Normalized intertwining operators for \texorpdfstring{$\rG_2(F)$}{G2(F)}}

Let $F$ be a local field of characteristic zero.
The lemma follows from the representation theory of $\sl_2$.

\begin{lemma}\label{lemma:dual-nilradicals-G2}
  Let $(B, T)$ be a Borel pair of $\rG_2$ and $\alpha, \beta$ be the short root and the long root, respectively.
  Let $P_{s} = P_{\alpha} = M_s N_s$ be the short root parabolic subgroup of $\rG_2$ and $P_{l} = P_{\beta} = M_l N_l$ be the long root parabolic subgroup of $\rG_2$.
  \begin{enumerate}
    \item
    The action of $\widehat{M_s} = \widehat{M}_{\alpha^{\vee}}$ on the Lie algebra $\widehat{N}_{\alpha^{\vee}}$ is isomorphic as a representation of $\GL_2$ to $\Std_{\GL_2} \oplus \det \oplus \Std_{\GL_2} \otimes \det$.
    \item
    The action of $\widehat{M_l} = \widehat{M}_{\beta^{\vee}}$ on the Lie algebra $\widehat{N}_{\beta^{\vee}}$ is isomorphic as a representation of $\GL_2$ to $\Sym^3(\Std_{\GL_2}) \otimes {\det}^{-1} \oplus \det$.
  \end{enumerate}
\end{lemma}

\begin{lemma}
    As representations of $\GL_2$, we have
    \begin{align*}
      \Sym^3(\Std_{\GL_2}) \otimes {\det}^{-1} \oplus \Std_{\GL_2} = \Ad_{\GL_2} \otimes \Std_{\GL_2}.
    \end{align*}
\end{lemma}

\begin{definition}
  Let $G$ be a reductive group over $F$.
  Let $M, M'$ be standard Levi subgroups of $G$ and $w$ be an element in the relative Weyl group $W(M, M')$.
  Let $P$ (resp. $P'$) be a parabolic subgroup of $G$ with Levi component $M$ (resp. $M'$).
  Let $\psi \in \Psi(M)$ be an $A$-parameter for $M$ and $\pi \in \Pi_{\psi}$ (if defined).
  Let $\lambda \in \mathfrak{a}_M^*$.
  An intertwining operator $J_P(w, \pi_{\lambda}) \colon I_P(\pi_{\lambda}) \to I_{P'}(w\pi_{\lambda})$ for $\pi_{\lambda}$ is defined as a meromorphic continuation of the integral
  \begin{align*}
    (J_P(w, \pi_{\lambda})f_{\lambda})(g) = \int_{(\widetilde{w}N\widetilde{w}^{-1} \cap N') \backslash N'} f_{\lambda}(\widetilde{w}^{-1}ug) du
  \end{align*}

  The normalized intertwining operator $R_P(w, \pi_{\lambda}, \psi)$ is defined by
  \begin{align*}
    R_P(w, \pi_{\lambda}, \psi) = r_P(w, \psi_{\lambda})^{-1} J_P(w, \pi_{\lambda}),
  \end{align*}
  with
  \begin{align*}
    r_P(w, \psi_{\lambda}) = \lambda(w) \gamma_{A}(0, \psi_{\lambda}, \rho^{\vee}_{w^{-1}P' | P}, \psi_F)^{-1}.
  \end{align*}
  We recall the notation from \cite{AtobeGanEtAl2024-LocalIntertwiningRelationsCoTemperedPackets}*{Section 1.7}.
  \begin{enumerate}
    \item
    The $L$-factor and $\epsilon$-factor are defined by
    \begin{align*}
    L(s, \psi, \rho) = L(s, \rho \circ \phi_{\psi}), \quad \epsilon(s, \psi, \rho, \psi_F) = \epsilon(s, \rho \circ \phi_{\psi}, \psi_F),
    \end{align*}
    for a representation $\rho$ of $\widehat{M}$.
    Arthur's $\gamma$-factor is defined by
    \begin{align*}
      \gamma_{A}(s, \psi, \rho, \psi_F)
      = \epsilon(s, \psi, \rho, \psi_F) \frac{L(1+s, \psi, \rho)}{L(s, \psi, \rho)}.
    \end{align*}

    \item
    The representation $\rho_{w^{-1}P' | P}$ is defined by
    \begin{align*}
      \rho_{w^{-1}P' | P} = \Ad(\widetilde{w})^{-1} \widehat{\fn'} / (\Ad(\widetilde{w})^{-1} \widehat{\fn'}  \cap \widehat{\fn}),
    \end{align*}
    with $\widehat{\fn}$ (resp. $\widehat{\fn'}$) being the Lie algebra of $\widehat{N}$ (resp. $\widehat{N'}$).

    \item
    We only use the case $\lambda(w) = 1$. For the general definition, see \cite{AtobeGanEtAl2024-LocalIntertwiningRelationsCoTemperedPackets}*{Section 1.7}.
  \end{enumerate}

\end{definition}

\begin{proposition}\label{proposition:normalized-intertwining-composition}
  Assume that $G=\rG_2$ and $\psi$ is tempered.
  With Whittaker normalization, we have
  \begin{align*}
    R_{P'}(w', w\pi, w\psi) R_P(w, \pi, \psi) = R_P(w'w, \pi, \psi)
  \end{align*}
  for any $w \in W(M, M'), w' \in W(M', M'')$.
\end{proposition}

\begin{proof}
  If $M=G$, the assertion is trivial. Otherwise, $M$ is a torus or is isomorphic to $\GL_2$, so $\pi$ is generic. For non-archimedean $F$, the adjoint-cube normalization is given by \cite{Shahidi1989-ThirdSymmetricPowerLFunctionsGL2}*{Propositions 2.2 and 6.2}. The factors defined there as ratios of Rankin--Selberg and standard factors \cite{Shahidi1989-ThirdSymmetricPowerLFunctionsGL2}*{Section 1} agree with the Artin factors by compatibility of the local Langlands correspondence and the above decomposition of $\Ad_{\GL_2}\otimes\Std_{\GL_2}$. The remaining factors are standard or Tate factors. For archimedean $F$, the agreement follows from \cite{Shahidi1985-LocalCoefficientsArtinFactorsRealGroups}*{Theorem 3.1}. The assertion therefore follows from \cite{Arthur2013-EndoscopicClassificationRepresentations}*{Theorem 2.5.1(a), pp.~112--113}.
\end{proof}

\begin{example}\label{example:normalizing-factor-G2}
  For $G=\rG_2$, take the maximal parabolics $P_s=M_sN_s$ and $P_l=M_lN_l$ of Lemma~\ref{lemma:dual-nilradicals-G2}. Their normalizing factors are as follows.
  \begin{enumerate}
    \item
    Let $w\in W(G, M_s)$ be the nontrivial relative Weyl element, represented in $W(G,T)$ by $w_0^G w_0^{M_s}$, where $w_0^G$ and $w_0^{M_s}$ are the longest elements of $W(G,T)$ and $W(M_s,T)$, respectively.

    Then, we have
    \begin{align*}
      r_P(w, \psi_{\lambda})
      = \gamma_{A}(0, \pi_{\lambda}, \widehat{\overline{\fn}}^{\vee}, \psi_F)^{-1} = \gamma_{A}(0, \pi_{\lambda}, \widehat{\fn}, \psi_F)^{-1}
    \end{align*}
    and
    \begin{align*}
      \gamma_{A}(s, \pi, \widehat{\overline{\fn}}^{\vee}, \psi_F)
      =
      \gamma_{A}(s, \pi, \psi_F) \gamma_{A}(s, \omega_{\pi}, \psi_F)  \gamma_{A}(s, \pi \otimes \omega_{\pi}, \psi_F).
    \end{align*}

    \item
    Let $w\in W(G, M_l)$ be the nontrivial relative Weyl element, represented in $W(G,T)$ by $w_0^G w_0^{M_l}$, where $w_0^{M_l}$ is the longest element of $W(M_l,T)$.

    Then, we have
    \begin{align*}
      r_P(w, \psi_{\lambda})
      &= \gamma_{A}(0, \pi_{\lambda}, \widehat{\overline{\fn}}^{\vee}, \psi_F)^{-1} = \gamma_{A}(0, \pi_{\lambda}, \widehat{\fn}, \psi_F)^{-1}
    \end{align*}
    and
    \begin{align*}
      \gamma_{A}(s, \pi, \widehat{\overline{\fn}}^{\vee}, \psi_F)
      =
      \gamma_{A}(s, \Sym^3(\pi) \otimes \omega_{\pi}^{-1}, \psi_F) \gamma_{A}(s, \omega_{\pi}, \psi_F).
    \end{align*}
    We also have
    \begin{align*}
      \gamma_{A}(s, \pi, \Sym^3 \otimes {\det}^{-1}, \psi_F)
      = 
      \frac{\gamma_{A}(s, \pi, \Std \otimes \Ad, \psi_F)}
      {\gamma_{A}(s, \pi, \psi_F)}.
    \end{align*}
  \end{enumerate}
\end{example}

\begin{lemma}
  Let $F$ be a local field of characteristic zero.
  We set $G = \rG_2$ and let $P$ be a proper parabolic subgroup of $G$ with a Levi component $M$.
  Let $w$ be a non-trivial element in the relative Weyl group $W(M)$.
  Let $\pi$ be an irreducible tempered representation of $M$ such that $w \pi \cong \pi$.
  Put $\rho_w = \rho_{w^{-1}P \mid P}^{\vee}$.
  Then, we have
  \begin{align*}
    \lim_{s \to 0} \frac{\gamma_A(s, \pi, \rho_w, \psi_F)}{\gamma(s, \pi, \rho_w, \psi_F)} = 1.
  \end{align*}
\end{lemma}

\begin{proof}
  Let $\phi_{\pi}$ be the parameter of $\pi$ and put
  $V_w = \rho_w \circ \phi_{\pi}$.
  Since $M$ is a torus or isomorphic to $\GL_2$, the condition
  $w\pi \cong \pi$ implies $w\phi_{\pi} \cong \phi_{\pi}$.
  Put $\widehat{\fn}_w = \Ad(\widetilde{w})^{-1}\widehat{\fn}$ and
  $\widehat{\fn}_{\cap} = \widehat{\fn}\cap\widehat{\fn}_w$.
  The two nilradicals have isomorphic restrictions along $\phi_{\pi}$,
  while the Killing form identifies their quotients by $\widehat{\fn}_{\cap}$ as dual
  $\widehat{M}$-modules.  Cancelling the common submodule shows that
  $V_w \cong V_w^{\vee}$.  Hence
  \begin{align*}
    \lim_{s \to 0}
    \frac{\gamma_A(s,V_w,\psi_F)}{\gamma(s,V_w,\psi_F)}
    = \lim_{s \to 0}\frac{L(1+s,V_w)}{L(1-s,V_w^{\vee})}=1,
  \end{align*}
  since temperedness implies that $L(s,V_w)$ is holomorphic and
  nonzero at $s=1$.
\end{proof}

\begin{corollary}
  If $w\pi = \pi$, then the operator $R_P(w, \pi, \psi)$ preserves the Whittaker functional on $I_P(\pi)$ if $\pi$ is tempered.
\end{corollary}

\section{Elliptic triplets for \texorpdfstring{$\rG_2(F)$}{G2(F)}}

Set $G = \rG_2$ and fix a Borel pair $(B, T)$ of $G$ and let $\Delta = \{\alpha, \beta\}$ be the set of simple roots with respect to $(B, T)$, where $\alpha$ is the short root and $\beta$ is the long root.
We set $M_s = M_{\alpha}$ and $M_l = M_{\beta}$.

\subsection{Elliptic triplets for proper Levi subgroups of \texorpdfstring{$\rG_2$}{G2}}

Since every proper Levi subgroup of $\rG_2$ is isomorphic to $\GL_2$ or $\GL_1^2$, we have the following lemma.
\begin{lemma}
   Let $M$ be a proper Levi subgroup of $\rG_2$.
   Then, any elliptic triplet $(L, \sigma_L, r)$ is given by $L = M$, $\sigma_L$ is a discrete series representation of $M(F)$, and $r$ is the trivial element in $R_{\sigma_L}$.
\end{lemma}

\subsection{Classification of regular elements in the Weyl group for \texorpdfstring{$\rG_2$}{G2}}

\begin{lemma}\label{lemma:regular-weyl-G2-non-toric-Levi}
   Let $M$ be a proper Levi subgroup of $\rG_2$ which is not a torus.
   Then, the unique regular element in $W(\rG_2, M)$ is given by the class of the element $w_0^G w_0^M$ of the Weyl group of $G$.
\end{lemma}

\begin{lemma}\label{lemma:regular-weyl-G2-toric-Levi}
   The regular elements in $W(\rG_2, T)$ are given by the classes of the following elements.
   \begin{align*}
      \begin{cases}
         s_{\alpha}s_{\beta}, \\
         s_{\beta}s_{\alpha}, \\
         s_{\alpha}s_{\beta}s_{\alpha}s_{\beta}, \\
         s_{\beta}s_{\alpha}s_{\beta}s_{\alpha}, \\
         s_{\alpha}s_{\beta}s_{\alpha}s_{\beta}s_{\alpha}s_{\beta}.
      \end{cases}
   \end{align*}
\end{lemma}

\begin{proof}
  Put $c=s_\alpha s_\beta$. In the basis $(\alpha,\beta)$, we have
  \begin{align*}
    c=\begin{pmatrix}2&-3\\1&-1\end{pmatrix},\quad c^3=-1.
  \end{align*}
  The Weyl group is dihedral of order twelve. Its five nontrivial rotations
  $c,c^2,c^3,c^4,c^5$ have no nonzero fixed vector, while every reflection fixes a line.
\end{proof}

\subsection{Discrete series representations of Levi subgroups fixed by regular elements in the Weyl group}

\begin{lemma}\label{lemma:G2-Levi-Weyl-invariance}
   Let $M$ be a proper Levi subgroup of $\rG_2$ which is not a torus.
   Then, any discrete series representation $\sigma$ of $M(F)$ is fixed by the unique regular element in $W(\rG_2, M)$ if and only if $\sigma$ is self-dual.

   If $M = M_{\alpha}$, then we have $W_0(\sigma) = \{1\}$ if and only if $\omega_{\sigma} \neq \mathbf{1}$.
   If $M = M_{\beta}$, then we have $W_0(\sigma) = \{1\}$ if and only if $\omega_{\sigma} \neq \mathbf{1}$ and $\Sym^3 \phi_{\sigma} \otimes \det \phi_{\sigma}^{-1}$ does not contain the trivial representation.
\end{lemma}

\begin{proof}
  The element $w=w_0^G w_0^M$ acts on $M\cong\GL_2$ by inverse transpose
  up to an inner automorphism, so $w\sigma\cong\sigma^\vee$.
  In relative rank one, $W_0(\sigma)$ is generated by the reflection exactly
  when the Plancherel measure vanishes.
  By Proposition~\ref{proposition:normalized-intertwining-composition}, unitarity,
  and Example~\ref{example:normalizing-factor-G2}, this measure is a positive
  constant times $|\gamma_A(0,\widehat{\fn}\circ\phi_\sigma,\psi_F)|^2$.
  Since $\widehat{\fn}\circ\phi_\sigma$ is tempered, its $L$-factor is finite and nonzero at $1$ and has a pole at $0$ exactly when the representation contains $\mathbf1$. This is precisely when the measure vanishes.
  The discrete parameter $\phi_\sigma$ and its twist
  $\phi_\sigma\otimes\det\phi_\sigma$ contain no trivial summand.
  The two decompositions in Lemma~\ref{lemma:dual-nilradicals-G2}
  therefore give the assertions.
\end{proof}

\begin{lemma}\label{lemma:quadratic-character-invariant-discrete-series}
  Let $F$ be a local field of characteristic zero.
  Let $\pi$ be a self-dual discrete series representation of $\GL_2(F)$ with a non-trivial quadratic central character $\omega_{\pi}$.
  Let $K/F$ be a quadratic extension corresponding to $\omega_{\pi}$.
  Then, $\pi$ is of the form $\mathrm{AI}_{K/F}(\chi)$ for a character $\chi$ of $K^{\times}$ such that $\chi_{\restriction_{F^{\times}}} = 1$.
\end{lemma}

\begin{proof}
  We have $\pi \cong \pi^{\vee} \cong \pi \otimes \omega_{\pi}$ and $\omega_{\pi} \neq 1$.
  Thus, $\pi$ is dihedral and we have $\pi = \mathrm{AI}_{K/F}(\chi)$ for a character $\chi$ of $K^{\times}$ such that $\chi_{\restriction_{F^{\times}}} = 1$.
\end{proof}

\begin{lemma}\label{lemma:symmetric-cube-dihedral}
  Let $F$ be a local field of characteristic zero.
  Let $\pi = \mathrm{AI}_{K/F}(\chi)$ be a dihedral discrete series representation of $\GL_2(F)$, where $K/F$ is quadratic and $\chi$ is a character of $K^{\times}$ such that $\chi_{\restriction_{F^{\times}}} = 1$.
  Then,
  \begin{align*}
    \Sym^3(\phi_{\pi}) \otimes \det \phi_{\pi}^{-1}
    =
    \mathrm{AI}_{K/F}(\chi^3) \oplus \mathrm{AI}_{K/F}(\chi).
  \end{align*}
\end{lemma}

\begin{proof}
  Choose an induced basis $e_1,e_2$ of $\phi_\pi$.
  Since $\chi_{\restriction_{F^\times}}=1$, the group $W_K$ acts by
  $\chi,\chi^{-1}$, a representative of $W_F/W_K$ exchanges the basis vectors,
  and $\det\phi_\pi=\omega_{K/F}$.
  The stable subspaces
  \begin{align*}
    \langle e_1^3,e_2^3\rangle,\quad
    \langle e_1^2e_2,e_1e_2^2\rangle
  \end{align*}
  give $\Sym^3\phi_\pi=\mathrm{AI}_{K/F}(\chi^3)\oplus\mathrm{AI}_{K/F}(\chi)$.
  Tensoring either induced representation by $\omega_{K/F}$ preserves it,
  which proves the formula.
\end{proof}

\begin{lemma}\label{G2-elliptic-GL2-Levi-representations}
  Let $F$ be a local field of characteristic zero.
  We set $G = \rG_2$ and let $P$ be a proper parabolic subgroup of $G$ with a Levi component $M$.
  Let $w$ be a non-trivial element in the relative Weyl group $W(M)$.
  Let $\sigma$ be a discrete series representation of $M$ such that $w \sigma \cong \sigma$.
  \begin{enumerate}
    \item
    If $P$ is a short-root parabolic subgroup of $G$, then the $R$-group $R^{G}(\sigma)$ is non-trivial and $W_0(M)_{\sigma} = 1$ if and only if $\sigma=\mathrm{AI}_{K/F}(\chi)$ for a quadratic extension $K/F$ and a character $\chi$ with $\chi_{\restriction_{F^\times}}=1$.
    In this case, there exists a unique discrete series parameter $\phi^{\SO_4}(\sigma)$ of $\SO_4$ such that we have
    \begin{align*}
        \Std_{\SO_4} \circ \phi^{\SO_4}(\sigma) \oplus \wedge^2_+ \circ \phi^{\SO_4}(\sigma)
        =
        \phi_{\sigma} \oplus \phi_{\sigma} \oplus \omega_{K/F} \oplus \omega_{K/F} \oplus \mathbf{1}.
    \end{align*}

    \item
    If $P$ is a long-root parabolic subgroup of $G$, then the $R$-group $R^{G}(\sigma)$ is non-trivial and $W_0(M)_{\sigma} = 1$ if and only if $\sigma = \mathrm{AI}_{K/F}(\chi)$ is dihedral with respect to a quadratic extension $K$ of $F$ such that $\chi_{\restriction_{F^{\times}}} = 1$ and $\chi^3 \neq 1$.
    In this case, there exists a unique discrete series parameter $\phi^{\SO_4}(\sigma)$ of $\SO_4$ such that we have
    \begin{align*}
        \Std_{\SO_4} \circ \phi^{\SO_4}(\sigma) \oplus \wedge^2_+ \circ \phi^{\SO_4}(\sigma)
        =
        \phi_{\sigma} \oplus \phi_{\sigma} \oplus \Ad(\phi_{\sigma}).
    \end{align*}
  \end{enumerate}
\end{lemma}

\begin{proof}
  This follows from Example~\ref{example:normalizing-factor-G2} and the fact that
  \begin{align*}
    R^{G}(\sigma) = W(M)_{\sigma}/W_0(M)_{\sigma},
  \end{align*}
  where $W(M)_{\sigma}$ is the stabilizer of $\sigma$ in $W(M)$ and $W_0(M)_{\sigma}$ consists of the elements $w$ of $W(M)_{\sigma}$ for which $R_P(w, \sigma, \psi_F)$ is scalar. Scalarity of $R_P(w, \sigma, \psi_F)$ is equivalent to $\gamma_{A}(0, \sigma, \widehat{\fn}, \psi_F) = 0$.
\end{proof}

\begin{lemma}\label{G2-elliptic-toral-representations}
  Let $F$ be a local field of characteristic zero.
  We set $G = \rG_2$ and let $B = TU$ be a Borel subgroup of $G$ with a Levi component $T$.
  \begin{enumerate}
    \item
    The set of regular elements in $W(G, T)$ is given by
    \begin{align*}
      \{s_{\alpha} s_{\beta}, (s_{\alpha} s_{\beta})^2, s_{\beta} s_{\alpha}, (s_{\beta} s_{\alpha})^2, (s_{\alpha} s_{\beta})^3 \}.
    \end{align*}
    The characteristic polynomial of each element is given as follows.
    \begin{enumerate}
      \item
      For $s_{\alpha} s_{\beta}$ and $s_{\beta} s_{\alpha}$, the characteristic polynomial is $x^2 - x + 1$.
      \item
      For $(s_{\alpha} s_{\beta})^2$ and $(s_{\beta} s_{\alpha})^2$, the characteristic polynomial is $x^2 + x + 1$.
      \item
      For $(s_{\alpha} s_{\beta})^3$, the characteristic polynomial is $(x+1)^2$.
    \end{enumerate}
    \item
    Let $\sigma$ be a character of $T$ and let $w$ be a regular element in $W(G, T)$ such that $w \sigma \cong \sigma$.
    \begin{itemize}
      \item
      In case $(a)$, we have $\sigma = \mathbf{1}$ and $W(T)_{\sigma} = W_0(T)_{\sigma} = W(G, T)$.
      \item
      In case $(b)$, $\sigma$ is given by a cubic character of $F^{\times}$ and we have $W_0(T)_{\sigma} \neq 1$.
      \item
      In case $(c)$, we have $\sigma^2=1$, so
      $\sigma=(\omega_1\circ\alpha)(\omega_2\circ\beta)$ for
      quadratic characters $\omega_1,\omega_2$, possibly trivial.
      We have $W_0(T)_{\sigma}=1$ if and only if $\omega_1,\omega_2$
      are independent.  In this case they determine a biquadratic
      extension $L/F$, and $R^{G}(\sigma)=\Z/2\Z$.
    \end{itemize}
    \item
    In case $(c)$, assume that $W_0(T)_{\sigma}=1$, and let
    $\omega_1, \omega_2, \omega_3$ be the three non-trivial quadratic characters of $F^{\times}$ corresponding to the three quadratic extensions of $F$ contained in $L$.
    Then, there exists a unique discrete series parameter $\phi^{\SO_4}(\sigma)$ of $\SO_4$ such that we have
    \begin{align*}
        \Std_{\SO_4} \circ \phi^{\SO_4}(\sigma) \oplus \wedge^2_+ \circ \phi^{\SO_4}(\sigma)
        =
        \omega_1^{\oplus 2} \oplus \omega_2^{\oplus 2} \oplus \omega_3^{\oplus 2} \oplus \mathbf{1}.
    \end{align*}
    We have
    \begin{align*}
      \Std_{\SO_4} \circ \phi^{\SO_4}(\sigma) = \omega_1 \oplus \omega_2 \oplus \omega_3 \oplus \mathbf{1}
    \end{align*}
    and
    \begin{align*}
      \wedge^2_+ \circ \phi^{\SO_4}(\sigma) = \omega_1 \oplus \omega_2 \oplus \omega_3.
    \end{align*}
  \end{enumerate}
\end{lemma}

\begin{proof}
  Put $c=s_\alpha s_\beta$ and use the matrix in the proof of
  Lemma~\ref{lemma:regular-weyl-G2-toric-Levi}.
  Its powers give the characteristic polynomials in (1).
  Write $\sigma=(\chi_1\circ\alpha)(\chi_2\circ\beta)$.
  In these coordinates,
  \begin{align*}
    c\sigma=(\chi_1^2\chi_2^{-3},\chi_1\chi_2^{-1}),\quad
    c^2\sigma=(\chi_1\chi_2^{-3},\chi_1\chi_2^{-2}).
  \end{align*}
  Hence $c\sigma=\sigma$ forces $\chi_1=\chi_2=1$, and
  $c^2\sigma=\sigma$ forces $\chi_1=1$ and $\chi_2^3=1$.
  As in the proof of Lemma~\ref{lemma:G2-Levi-Weyl-invariance}, the rank-one
  normalizing factors show that $W_0(T)_\sigma$ is generated by the reflections
  $s_\gamma$ such that $\sigma\circ\gamma^\vee=1$.
  This proves (a), and in (b) we have
  $\sigma\circ\alpha^\vee=\chi_2^{-3}=1$.

  Since $c^3=-1$, case (c) gives $\chi_i^2=1$.
  The characters $\sigma\circ\gamma^\vee$ for the six positive roots are
  $\omega_1,\omega_2,\omega_1\omega_2$, each occurring twice.
  Thus $W_0(T)_\sigma=1$ exactly when $\omega_1,\omega_2$ are independent.
  The action of $W(G,T)$ on $X^*(T)/2X^*(T)$ has image
  $\GL_2(\mathbb F_2)$ and kernel $\{1,c^3\}$.
  Independence therefore gives $W(T)_\sigma=\{1,c^3\}$ and
  $R^G(\sigma)\cong\Z/2\Z$.

  Finally, $V=\mathbf1\oplus\omega_1\oplus\omega_2\oplus\omega_3$ is orthogonal,
  has determinant one, and has finite centralizer in $\SO_4$.
  It defines a discrete parameter, and both halves of $\wedge^2 V$ are
  $\omega_1\oplus\omega_2\oplus\omega_3$.
  Any discrete parameter satisfying the stated seven-dimensional identity
  has a multiplicity-free standard representation, which must therefore be $V$.
  Reflection in the trivial line centralizes $V$ and has determinant $-1$,
  so its $\O_4$-conjugacy class is a single $\SO_4$-conjugacy class.
\end{proof}

\begin{remark}\label{remark:LIR-and-ellipticity}
  Let $(M, \sigma, w)$ be a triple appearing in the previous lemmas such that $w \sigma \cong \sigma$,  $R^{G}(\sigma)$ is non-trivial and $W_0(M)_{\sigma}$ is trivial.
  Then, the datum $(M, \sigma, w)$ is elliptic.
  Furthermore, in all such cases, we have obtained a unique discrete series parameter $\phi^{\SO_4}(\sigma)$ of $\SO_4$.
\end{remark}

\section{Elliptic triplets for \texorpdfstring{$\widetilde{\PGSO_8}(F)$}{widetilde{PGSO8(F)}}}

\begin{definition}[\cite{MoeglinWaldspurger2018-FormuleTracesLocaleTordue}*{2.9, 2.11}]
  Let $F$ be a local field of characteristic zero.
  Let $\widetilde{G} = (G, \widetilde{G})$ be a reductive space over $F$.
  An essential triplet for $(G, \widetilde{G})$ is a triplet $(M, \sigma, \widetilde{r})$ where
  \begin{itemize}
    \item
    $M$ is a Levi subgroup of $G$,
    \item
    $\sigma$ is an irreducible discrete series representation of $M(F)$ such that $W^{\widetilde{G}}(\sigma) \neq \emptyset$, and
    \item
    $\widetilde{r} \in R^{\widetilde{G}}(\sigma)$ is an element of the $R$-group of $\sigma$ in $\widetilde{G}$ that is essential in the sense of \cite{MoeglinWaldspurger2018-FormuleTracesLocaleTordue}*{2.9}.
  \end{itemize}
  An essential triplet $(M, \sigma, \widetilde{r})$ is said to be discrete if we have
  \begin{align*}
    \widetilde{r} \cap W^{\widetilde{G}}(\sigma)_{\reg} \neq \emptyset.
  \end{align*}
  A discrete essential triplet $(M, \sigma, \widetilde{r})$ is said to be elliptic if we have
  \begin{align*}
    W_0^{G}(\sigma) = \{1\}.
  \end{align*}

  Let $E(\widetilde{G})$ (resp. $E_{\disc}(\widetilde{G})$, $E_{\elliptic}(\widetilde{G})$) denote the set of essential triplets (resp. discrete essential triplets, elliptic triplets) for $\widetilde{G}$ modulo the conjugation by $G$.
\end{definition}

\begin{lemma}[\cite{MoeglinWaldspurger2018-FormuleTracesLocaleTordue}*{2.11, Lemme}]
  Let $F$ be a local field of characteristic zero.
  Let $(M, \sigma, \widetilde{r})$ be an essential triplet for $(G, \widetilde{G})$.
  The following conditions are equivalent.
  \begin{enumerate}
    \item
    There exists a Levi subspace $\widetilde{L} = (L, \widetilde{L})$ of $\widetilde{G}$ such that $M \subset L \subsetneq G$ and $\widetilde{r} \in R^{\widetilde{L}}(\sigma)$.
    \item
    We have
    \begin{align*}
      W_0^{G}(\sigma) \neq \{1\},
    \end{align*}
    or $W_0^{G}(\sigma) = \{1\}$ and
    \begin{align*}
      \widetilde{r} \not \in W^{\widetilde{G}}(\sigma)_{\reg}.
    \end{align*}
    \item
    The triplet $(M, \sigma, \widetilde{r})$ is not elliptic.
  \end{enumerate}
\end{lemma}

Henceforth in this section, set $G = \PGSO_8$ and $\widetilde{G} = \widetilde{\PGSO_8}$, with an $F$-pinning $(T, B, \{X_{\alpha}\}_\alpha)$ preserved by an automorphism $\theta$ of order $3$.
Let $\cL_{G}$ be the set of conjugacy (associate) classes of Levi subgroups of $G$.

\begin{remark}\label{remark:Whittaker-normalization-Levi}
   The proper Levi subgroups of $\PGSO_8$ and $\rG_2$ are products of general linear groups modulo connected split central tori.
   For an irreducible tempered representation $\sigma$ of such a Levi $M$, fix a nonzero Whittaker functional $\lambda_M$.
   An isomorphism $A_a$ implementing an automorphism $a$ that preserves the Whittaker datum is uniquely normalized by
   \begin{align*}
     \lambda_M\circ A_a=\lambda_M.
   \end{align*}
   Permutations of the general linear factors preserve the product Whittaker functional, and this normalization descends through the central quotient.
   We use the normalized intertwining operators of \cite{Arthur2013-EndoscopicClassificationRepresentations}*{(2.3.27), Theorem 2.5.1}.
   At a self-intertwining point they preserve the induced Whittaker functional.
   Their multiplicativity and compatibility with induction in stages identify the resulting operators without an additional scalar.
   These normalizations apply at real places as well, and the essentiality condition is automatic for the triples considered here.
\end{remark}

\begin{lemma}\label{lemma:Levi-conjugacy-class-D_4}
  Let $\Delta = \{ \alpha_1, \alpha_2, \alpha_3, \alpha_4 \}$ be the set of simple roots of $\PGSO_8$.
  For $I \subset \Delta$, let $M_I$ be the standard Levi subgroup of $\PGSO_8$ corresponding to $I$.
  A complete set of representatives of $\cL_{\PGSO_8}$ is given by the following.
  \[
    \renewcommand{\arraystretch}{1.15}
    \begin{array}{@{}r@{\quad}l@{\quad}l@{\quad}l@{}}
      (1) & T & & \\
      (2) & M_{\alpha_2} & & \\
      (3) & M_{\alpha_1, \alpha_2} & & \\
      (4) & M_{\alpha_1, \alpha_3, \alpha_4} & & \\
      (5) & \PGSO_8 & & \\
      (6) & M_{\alpha_1, \alpha_3} & M_{\alpha_3, \alpha_4} & M_{\alpha_1, \alpha_4} \\
      (7) & M_{\alpha_1, \alpha_2, \alpha_3} & M_{\alpha_2, \alpha_3, \alpha_4} & M_{\alpha_1, \alpha_2, \alpha_4}
    \end{array}
  \]
  The entries in each row are pairwise non-conjugate and permuted by $\theta$.
\end{lemma}

\begin{remark}\label{remark:explicit-description-theta-invariant-Levi}
  We have
  \begin{align*}
    M_{\alpha_2} \cong (\GL_2 \times \GL_1^3) / \GL_1,
  \end{align*}
  where the embedding $\GL_1 \to \GL_2 \times \GL_1^3$ is given by
  \begin{align*}
    t \mapsto (t 1_{\GL_2}, t, t, t).
  \end{align*}
  Indeed, the coweights $\varpi_{\alpha_1}^{\vee}, \varpi_{\alpha_3}^{\vee}, \varpi_{\alpha_4}^{\vee}$ give a basis of $X_*(Z(M))$, and these cocharacters together with $M^{\der}_{\alpha_2} \cong \SL_2$ give an isomorphism
  \begin{align*}
      M_{\alpha_2} = \SL_2 \times \GL_1^3 / \Delta \mu_2 \cong (\GL_2 \times \GL_1^3) / \GL_1.
  \end{align*}
  If we write an element of the maximal torus of $M_{\alpha_2}$ as the class of $t = (\diag(s_1, s_2), t_1, t_2, t_3) \in \GL_2 \times \GL_1^3$, then we have 
  \begin{align*}
    \alpha_1(t) = t_1 s_1^{-1}, \quad
    \alpha_2(t) = s_1 s_2^{-1}, \quad
    \alpha_3(t) = t_2 s_1^{-1}, \quad
    \alpha_4(t) = t_3 s_1^{-1}.
  \end{align*}
  Thus, the embedding $M_{\alpha_2} \hookrightarrow \PGSO_8$ is induced from the map 
  \begin{align*}
    (g, t_1, t_2, t_3) \mapsto 
    \diag(t_1, g, t_2^{-1} \det(g), t_3^{-1} \det(g), t_2^{-1} t_3^{-1} \det(g)g, t_1^{-1} t_2^{-1} t_3^{-1} \det(g)^2).
  \end{align*} 

  Similarly, we have
   \begin{align*}
     M_{\alpha_1, \alpha_3, \alpha_4} \cong (\GL_2^3 \times \GL_1) / \GL_1^3,
   \end{align*}
   where the embedding $\GL_1^3 \to \GL_2^3 \times \GL_1$ is given by
  \begin{align*}
    (t_1, t_2, t_3) \mapsto (t_1 1_{\GL_2}, t_2 1_{\GL_2}, t_3 1_{\GL_2}, t_1 t_2 t_3).
  \end{align*}

  The action of $\theta$ on $M_{\alpha_1, \alpha_3, \alpha_4}$ is
  given by the permutation of the three $\GL_2$-factors.  Its action on
  $M_{\alpha_2}$ is given by the permutation of the three $\GL_1$-factors.
\end{remark}

\begin{lemma}
   Let $M = M_I$ be a proper standard Levi subgroup of $G$. Then, $\Norm_{\widetilde{G}}(M)(F) = \emptyset$ if and only if $M_I$ is in the sixth or seventh row of Lemma~\ref{lemma:Levi-conjugacy-class-D_4}.
\end{lemma}

\subsection{Classification of elliptic triplets for proper Levi subspaces of \texorpdfstring{$\widetilde{\PGSO_8}$}{twisted PGSO8}}

For any standard Levi subgroup $M$, let $w_0^M$ be the longest element of the Weyl group $W(M, T)$.

\begin{lemma}
  If $(L, \sigma_L, \widetilde{r}^M_L)$ is an elliptic triplet for $(M, \widetilde{M})$ for a $\theta$-stable standard Levi subgroup $M \neq G$, then $L=M$, $\sigma_L$ is a $\theta$-stable discrete series representation of $M(F)$, and $\widetilde{r}^M_L$ is given by the class of $\theta$.
\end{lemma}

\begin{proof}
   If $M = T$, this is trivial.
   Suppose that $M \neq T$ and $L \subsetneq M$. The subset of the simple roots of $M$ defining $L$ must be stable under the action induced by $\theta$. For $M=M_{\alpha_2}$ or $M=M_{\alpha_1,\alpha_3,\alpha_4}$, the only proper such subset is the empty set. Hence, after conjugation in $M$, we may assume that $L=T$.
   Then, $\sigma_L$ must satisfy $W^{M}(\sigma_L) = 1$, i.e., $\sigma_L$ is a regular character of $T(F)$.

    Assume that $I = \{\alpha_2\}$.
    A basis for $X_*(T)$ is given by $\{\varpi_{\alpha_i}^{\vee}\}_{i=1}^{4}$.
    We have
    \begin{align*}
      s_{\alpha_2} \varpi_{\alpha_2}^{\vee} = - \varpi_{\alpha_2}^{\vee} +  \sum_{i=1, 3, 4} \varpi_{\alpha_i}^{\vee}.
    \end{align*}
    Thus, the action of a regular element $w_0^{M_{\alpha_2}} \rtimes \theta$ can be identified with the automorphism
    \begin{align*}
      (t_1, t_2, t_3, t_4) \mapsto (t_2 t_4, t_2^{-1}, t_2 t_1, t_2 t_3),
    \end{align*}
    for $\bG_m^4$.
    If $\sigma_T$ is given by $(\chi_1, \chi_2, \chi_3, \chi_4)$, then we have
    \begin{align*}
      \chi_1 = \chi_3 = \chi_4,
      \quad
      \chi_2 = \chi_2^{-1}\chi_1\chi_3\chi_4.
    \end{align*}
    In this case, $\sigma_T$ is also fixed by $w_0^{M_{\alpha_2}}$ and $\sigma_T$ is not regular.

    Assume that $I = \{\alpha_1, \alpha_3, \alpha_4\}$.
    The twisted Weyl element $w_0^{M_{I}} \rtimes \theta$ is regular.
    We have
    \begin{align*}
    s_{\alpha_i} \varpi_{\alpha_i}^{\vee} = - \varpi_{\alpha_i}^{\vee} + \varpi_{\alpha_2}^{\vee},
    \end{align*}
    for $i=1,3,4$.
    Thus, a regular element is either one of $s_{\alpha_i} \rtimes \theta$ for $i=1, 3, 4$ or $s_{\alpha_1}s_{\alpha_3}s_{\alpha_4} \rtimes \theta$.
    For $s_{\alpha_1} \rtimes \theta$, invariance requires
    \begin{align*}
    \chi_1 = \chi_3 = \chi_4, \chi_2 = \chi_1^2,
    \end{align*}
    so $\sigma_T$ is fixed by $s_{\alpha_1}$.
    The same condition holds for $s_{\alpha_1}s_{\alpha_3}s_{\alpha_4} \rtimes \theta$.
    In this case, $\sigma_T$ is not regular.
    This excludes $L\subsetneq M$, so $L=M$. The other assertions follow from the definition of an elliptic triplet, since $\theta$ represents the relevant class in $R^{\widetilde M}(\sigma_L)$.
\end{proof}

\subsection{Classification of regular elements for Weyl sets for \texorpdfstring{$\widetilde{\PGSO_8}$}{twisted PGSO8}}

We classify the regular elements in $W(G, M) \rtimes \theta$ for $\theta$-stable Levi subgroups $M$ of $G = \PGSO_8$ over any field $F$ of characteristic zero, using Lemma~\ref{lemma:Weyl-group-normalizer}.

\begin{lemma}
   For $I = \{\alpha_2\}$, the group $W(G, M_I)$ is generated by the classes of
   \[
      w'_1 = s_{\alpha_1 + \alpha_2 + \alpha_3}, \quad
      w'_2 = s_{\alpha_3 + \alpha_2 + \alpha_4}, \quad
      w'_3 = s_{\alpha_1 + \alpha_2 + \alpha_4}.
   \]
   The three roots $\alpha_1 + \alpha_2 + \alpha_3, \alpha_1 + \alpha_2 + \alpha_4, \alpha_2 + \alpha_3 + \alpha_4$ are mutually orthogonal.
   For $I = \{\alpha_1, \alpha_3, \alpha_4\}$, the group $W(G, M_I)$ is generated by the class of $s_{\alpha_1 + 2 \alpha_2 + \alpha_3 + \alpha_4}$.
\end{lemma}

\begin{proof}
   By Lemma \ref{lemma:Weyl-group-normalizer}, it suffices to compute the roots in $D_4$ which are orthogonal to $I$.
\end{proof}

\begin{corollary}\label{corollary:regular-elements-theta-stable-cases}
   In the setting of the lemma above, the regular elements in $W(G \rtimes \theta, M_I) = W(G, M_I) \rtimes \theta$ are as follows.
   \[
      \renewcommand{\arraystretch}{1.3}
      \begin{array}{@{}l@{\quad}l@{}}
         I & \text{Regular elements} \\ \hline
         \{\alpha_2\} & w'_1 \rtimes \theta, \quad w'_2 \rtimes \theta, \quad w'_3 \rtimes \theta, \quad w'_1 w'_2 w'_3 \rtimes \theta \\
         \{\alpha_1, \alpha_3, \alpha_4\} & s_{\alpha_1 + 2 \alpha_2 + \alpha_3 + \alpha_4} \rtimes \theta
      \end{array}
   \]
\end{corollary}

\begin{remark}\label{remark:conjugacy-of-regular-elements-dual-long-root}
  For $I = \{ \alpha_2 \}$, conjugation by $W(G, M_{I})$ gives
  \begin{align*}
   \begin{cases}
      \Ad(w_1')(w_1' \rtimes \theta) = w_2' \rtimes \theta, \\
      \Ad(w_2')(w_2' \rtimes \theta) = w_3' \rtimes \theta, \\
      \Ad(w_2')(w_1' \rtimes \theta) = w_1'w_2'w_3' \rtimes \theta, \\
   \end{cases}
  \end{align*}
  Thus all four regular elements are $W(G, M_I)$-conjugate.
\end{remark}

\begin{lemma}\label{lemma:regular-element-not-theta-stable-case}
  In the case of $I = \{\alpha_1, \alpha_2\}$, the group $W(G, M_{I})$ is generated by the element
  \begin{align*}
     w'_1 = w_0 w_0^{M_I},
  \end{align*}
  where $w_0$ is the longest element in the Weyl group of $G$ and $w_0^{M_I}$ is the longest element in the Weyl group of $M_I$.
  If we set $w'_2 = s_4 s_2 s_3 s_1 s_2 s_4$, then
  \begin{align*}
     w'_2 (\alpha_3) = \alpha_1, \\
     w'_2 (\alpha_2) = \alpha_2,
  \end{align*}
  Thus, the map $w'_2 \rtimes \theta$ preserves $\alpha_1$ and $\alpha_2$.

  The characteristic polynomial of $w'_2 \rtimes \theta$ on the space $\fa_{M_I}^{*}$ is given by
  \begin{align*}
     X^2 + X + 1.
  \end{align*}
  Similarly, the characteristic polynomial of $w'_1 w'_2 \rtimes \theta$ on $\fa_{M_I}^{*}$ is given by
  \begin{align*}
     X^2 - X + 1.
  \end{align*}
\end{lemma}

\begin{remark}
  Let $M'$ be the Levi subgroup of $\GSO_8$ corresponding to $I = \{\alpha_1, \alpha_2\}$.
  We have the isomorphism
  \begin{align*}
     \GL_3 \times \GL_1 \times \GL_1 \xrightarrow{\sim} M'_{\alpha_1, \alpha_2} \colon (g, a, b) \mapsto \diag(g, a, b, ab g^{\vee}).
  \end{align*}
  The center $\GL_1$ is embedded into $\GL_3 \times \GL_1 \times \GL_1$ by $t \mapsto (t 1_{\GL_3}, t, t)$, and the quotient $(\GL_3 \times \GL_1 \times \GL_1) / \GL_1$ is isomorphic to $M_I$.
  The action of $w'_2 \rtimes \theta$ on $M_I$ is given by
  \begin{align*}
    (g, a, b) \mapsto (g, \det(g)/ab, a).
  \end{align*}

  Similarly, the action of $w'_1 w'_2 \rtimes \theta$ on $M_I$ is given by
  \begin{align*}
    (g, a, b) \mapsto (g^{\vee}, ab/\det(g), a^{-1}).
  \end{align*}
\end{remark}

For $I = \emptyset$, we use the following fact.
\begin{lemma}
  We have
  \begin{align*}
    W(F_4) = W(D_4) \rtimes S_3.
  \end{align*}
\end{lemma}

Together with \cite{Carter1972-ConjugacyClassesWeylGroup}, this gives the following classification.

\begin{lemma}\label{lemma:classification-regular-conjugacy-classes-Weyl-elements}
    For $w \in W(D_4)$ with $u = w \rtimes \theta$ regular, the characteristic polynomial, Carter type \cite{Carter1972-ConjugacyClassesWeylGroup}, $W(D_4)$-conjugacy class size, and character group are given by
    \[
      \renewcommand{\arraystretch}{1.3}
      \begin{array}{@{}c@{\quad}l@{\quad}c@{\quad}c@{\quad}c@{}}
          & \text{Characteristic polynomial} & \text{Carter type} & \text{Class size} & X^*(\widehat{T}^{u}) \\ \hline
        (a) & x^4 - x^2 + 1 & F_4 & 48 & 0 \\
        (b) & (x^2 - x + 1)^2 & F_4(a_1) & 8 & 0 \\
        (c) & (x^2 + x + 1)^2 & A_2 + \widetilde{A_2} & 8 & (\mathbb{Z}/3\mathbb{Z})^{\oplus 2} \\
        (d) & (x + 1)^2(x^2 - x + 1) & C_3 + A_1 & 16 & (\mathbb{Z}/2\mathbb{Z})^{\oplus 2}
      \end{array}
    \]
    The characteristic polynomial determines the $W(D_4)$-conjugacy class of $u$, and
    \[
      X^*(\widehat{T}^{u}) \cong X^*(\widehat{T})/(1-u)X^*(\widehat{T}).
    \]
\end{lemma}

\begin{proof}
    The character group $X^*(\widehat{T})/(1 - w \rtimes \theta)X^*(\widehat{T})$ is a module over $\Z[x]/(x^2 + x + 1, x-1) \cong \bF_3$ or $\Z[x]/(x^3 + 1, x-1) \cong \bF_2$.
    The characteristic polynomial determines its order.
\end{proof}

\begin{remark}\label{remark:explicit-description-regular-Weyl-elements}
    We give an example of $w \rtimes \theta$ in each case above.
    \begin{enumerate}[label=(\alph*), align=left]
        \item
        $s_4 s_2 s_3 s_1 s_2 s_4 s_1 s_2 s_3 s_1 \rtimes \theta$
        \item
        $s_1s_3s_4s_2 \rtimes \theta$
        \item
        $s_3s_2s_4s_1s_2s_3s_1s_2 \rtimes \theta$.
        \item
        $w_0 \rtimes \theta$
    \end{enumerate}

    In the cases (a) and (b), the character $W_F \to \widehat{T}^{w \rtimes \theta}$ must be trivial and the group $W_0(\sigma) = W(G, T) \neq \{ 1 \}$.
\end{remark}

\subsection{Discrete series representations of Levi subgroups fixed by regular elements in the Weyl set}

The following classification uses Remark~\ref{remark:explicit-description-theta-invariant-Levi} and the $R$-group computation of \cite{Takanashi2026-HiragaIchinoIkedaConjectureFormalDegrees}*{Section 9.1}.

\begin{lemma}\label{lemma:weyl-fixed-representation-long-dual-root}
   Let $\sigma = \sigma' \otimes \chi_1 \otimes \chi_2 \otimes \chi_3$ be a discrete series representation of $M_{\alpha_2}(F) = (\GL_2(F) \times \GL_1(F)^3) / \GL_1(F)$, where $\sigma'$ is a discrete series representation of $\GL_2(F)$ and $\chi_i$ are characters of $\GL_1(F)$.

   The representation $\sigma$ is $w'_1 w'_2 w'_3 \rtimes \theta$-invariant if and only if $\chi_1 = \chi_2 = \chi_3$, $ \chi_1^2 = \mathbf{1}$ and $\sigma' \otimes \chi_1 \cong \sigma'$.
   In this case, we have $\omega_{\sigma'} = \chi_1$ and $\sigma' = \sigma'^{\vee}$.
   We have $W_0(\sigma) = \{1\}$ if and only if $\chi_1 \neq \mathbf{1}$.
\end{lemma}

\begin{lemma}\label{lemma:weyl-fixed-representation-short-dual-root}
    Let $\sigma = \sigma_1 \otimes \sigma_2 \otimes \sigma_3 \otimes \chi$ be a discrete series representation of
    $M_{\alpha_1, \alpha_3, \alpha_4}(F) = (\GL_2(F)^3 \times \GL_1(F)) / \GL_1^3(F)$, where $\sigma_i$ are discrete series representations of $\GL_2(F)$ and $\chi$ is a character of $\GL_1(F)$.

    Then, the representation $\sigma$ is $s_{\alpha_1 + 2 \alpha_2 + \alpha_3 + \alpha_4} \rtimes \theta$-invariant if and only if $\sigma_1 \cong \sigma_2 \cong \sigma_3 \cong \sigma_1^{\vee}$ and $\chi^2 = 1$.
    In this case, we have $\omega_{\sigma_1} = \chi$.
    We have $W_0(\sigma) = \{1\}$ if and only if there are a quadratic extension $K/F$ and a character $\eta$ of $K^\times$ such that $\chi=\omega_{K/F}$, $\sigma_1=\mathrm{AI}_{K/F}(\eta)$, $\eta_{\restriction_{F^\times}}=1$, and $\eta^3\neq1$.
\end{lemma}

\begin{proof}
  The invariance conditions follow from Remark~\ref{remark:explicit-description-theta-invariant-Levi}.
  If $\chi=1$, the adjoint action on the dual nilradical contains a trivial summand, so $W_0(\sigma)\neq1$.
  Otherwise, self-duality and $\omega_{\sigma_1}=\chi$ give $\sigma_1=\mathrm{AI}_{K/F}(\eta)$ with $\chi=\omega_{K/F}$ and $\eta_{\restriction_{F^\times}}=1$.
  Writing $\rho=\phi_{\sigma_1}$, the representation on the dual nilradical is
  \begin{align*}
    (\rho^{\otimes3}\otimes\chi^{-1})\oplus\chi
    =\mathrm{AI}_{K/F}(\eta^3)\oplus\rho^{\oplus3}\oplus\chi.
  \end{align*}
  It contains the trivial representation exactly when $\eta^3=1$, which gives the stated condition on $W_0(\sigma)$.
\end{proof}

\begin{remark}\label{remark:standard-parameters-of-theta-stable-cases}
   We continue the notation in Lemma~\ref{lemma:weyl-fixed-representation-long-dual-root} and Lemma~\ref{lemma:weyl-fixed-representation-short-dual-root}.
   In the first case, the $8$-dimensional standard parameter of the representation $i_{P_{\alpha_2}}^{G}(\sigma)$ is given by
   \begin{align*}
     \phi_{\sigma'} \oplus \chi_1 \oplus \mathbf{1} \oplus \mathbf{1} \oplus \chi_1^{-1} \oplus \phi_{\sigma'}^{\vee}
   \end{align*}
   with $\chi_1 = \chi_1^{-1}$ and $\sigma' = \sigma'^{\vee}$.

   Similarly, in the second case, put $\sigma' = \sigma_1$.
   The $8$-dimensional standard parameter of the representation $i_{P_{\alpha_1, \alpha_3, \alpha_4}}^{G}(\sigma)$ is given by
   \begin{align*}
      \phi_{\sigma'} \oplus \Ad(\phi_{\sigma'}) \oplus \mathbf{1} \oplus  \phi_{\sigma'}^{\vee}
   \end{align*}
   with $\sigma' = \sigma'^{\vee}$.
\end{remark}

\begin{lemma}\label{lemma:standard-parameters-of-theta-unstable-cases}
  A discrete series representation $\sigma = \sigma' \otimes \chi_1 \otimes \chi_2$ of $M_{\alpha_1, \alpha_2}(F) = (\GL_3(F) \times \GL_1(F) \times \GL_1(F))/\GL_1(F)$ is $w'_1w'_2 \rtimes \theta$-invariant if and only if $\chi_1 = \chi_2 = \mathbf{1}$ and $\sigma' \cong \sigma'^{\vee}$.
  Similarly, a discrete series representation $\sigma = \sigma' \otimes \chi_1 \otimes \chi_2$ of $M_{\alpha_1, \alpha_2}(F)$ is $w'_2 \rtimes \theta$-invariant if and only if $\chi_1 = \chi_2^{-1}$, $\chi_1^3 = 1$, and $\sigma' \otimes \chi_1 \cong \sigma'$.
  We also have $\omega_{\sigma'} = \mathbf{1}$.
  In both cases, we have $W_0(\sigma) = \{1\}$.
\end{lemma}

\begin{remark}\label{remark:standard-parameters-of-theta-unstable-cases}
  In the setting of the previous lemma, the $8$-dimensional standard parameter of the representation $i_{P_I}^{G}(\sigma)$ is given by
  \begin{align*}
    \chi_1^{-1} \oplus \phi_{\sigma'} \oplus \phi_{\sigma'}^{\vee} \oplus \chi_1.
  \end{align*}
\end{remark}

\begin{lemma}\label{lemma:cubic-character-invariant-discrete-series}
  Let $\sigma$ be a $w'_2 \rtimes \theta$-invariant discrete series representation as in Lemma~\ref{lemma:standard-parameters-of-theta-unstable-cases}, and assume that $\chi_1\neq\mathbf{1}$.
  Let $E/F$ be a cubic extension corresponding to the character $\chi_1$.
  Then, there exists a character $\eta$ of $E^{\times}$ such that $\sigma' = \mathrm{AI}_{E/F}(\eta/\eta^{a})$, where $a$ is a generator of $\Gal(E/F)$.
\end{lemma}

\begin{proof}
  Argue as in Lemma~\ref{lemma:quadratic-character-invariant-discrete-series}.
\end{proof}

The representatives in Remark~\ref{remark:explicit-description-regular-Weyl-elements}(c)--(d) give the following formulas.

\begin{lemma}\label{lemma:description-of-centralizer-of-regular-elements-2}
    In case $(c)$, any element of the group $\widehat{T}^{w \rtimes \theta}$ can be written as
    \begin{align*}
        t &= \alpha_1^{\vee}(\zeta) \alpha_2^{\vee}(\xi) \alpha_3^{\vee}(\zeta^{-1}) \alpha_4^{\vee}(\zeta^{-1}\xi^{-1}) \\
          &= (\alpha_1^{\vee} - \alpha_3^{\vee} - \alpha_4^{\vee})(\zeta)
             (\alpha_2^{\vee} - \alpha_4^{\vee})(\xi)
    \end{align*}
    for $\zeta, \xi \in \mu_3(\C)$.
    Writing $e_i^{\vee}$ for the $i$-th coordinate cocharacter of $\GL_8$, its image under $\Std_{\PGSO_8}$ is
    \begin{align*}
        \Std_{\PGSO_8}(t)
        &= (e_1^{\vee} - e_2^{\vee} + e_3^{\vee} - e_6^{\vee} + e_7^{\vee} - e_8^{\vee})(\zeta)
          \cdot (e_2^{\vee} + e_3^{\vee} - e_4^{\vee} + e_5^{\vee} - e_6^{\vee} - e_7^{\vee})(\xi) \\
        &= \diag(\zeta,\zeta^{-1}\xi,\zeta\xi,\xi^{-1},\xi,\zeta^{-1}\xi^{-1},\zeta\xi^{-1},\zeta^{-1}).
    \end{align*}
\end{lemma}

\begin{lemma}\label{lemma:description-of-centralizer-of-regular-elements-1}
    In case $(d)$, any element of the group $\widehat{T}^{w \rtimes \theta}$ can be written as
    \begin{align*}
        t = \alpha_2^{\vee}(\zeta) (\alpha_1^{\vee} + \alpha_3^{\vee} + \alpha_4^{\vee})(\xi)
    \end{align*}
    with $\zeta, \xi \in \mu_2(\C)$.
    Its image under $\Std_{\PGSO_8}$ is
    \begin{align*}
        \Std_{\PGSO_8}(t)
        &= (e_2^{\vee} - e_3^{\vee} + e_6^{\vee} - e_7^{\vee})(\zeta)
          \cdot (e_1^{\vee} - e_2^{\vee} + e_7^{\vee} - e_8^{\vee})(\xi) \\
        &= \diag(\xi,\zeta\xi,\zeta,1,1,\zeta,\zeta\xi,\xi).
    \end{align*}
\end{lemma}

We prove Lemma~\ref{lemma:description-of-centralizer-of-regular-elements-1}, the other case being similar.

\begin{proof}[Proof of Lemma \ref{lemma:description-of-centralizer-of-regular-elements-1}]
    We compute the character group of $\widehat{T}^{w \rtimes \theta}$.
    The images of characters $\varpi_1 + \varpi_3$, $\varpi_3 + \varpi_4$ are zero in the character group.
    The element $w \rtimes \theta$ acts by $-1$ on the images of $\varpi_2$ and $\varpi_1 + \varpi_3 + \varpi_4$.
    Thus, we must have $t_1 = t_3 = t_4 \in \mu_2$ and $t_2 \in \mu_2$.
\end{proof}

\begin{remark}\label{remark:torus-regular-invariant-rep}
   In both cases, if a representation $\sigma$ of $T$ has $L$-parameter given by a surjective homomorphism
   \begin{align*}
      W_F \to \widehat{T}^{w \rtimes \theta} \hookrightarrow \widehat{T},
   \end{align*}
   then the group $W_0(\sigma)$ is trivial.
\end{remark}

\section{Elliptic triplets and elliptic endoscopic data for \texorpdfstring{$\rG_2$}{G2} and \texorpdfstring{$\widetilde{\PGSO_8}$}{twisted PGSO8}}\label{section:elliptic-triplets-endoscopic-data}

\begin{proposition}\label{proposition:elliptic-triplets-elliptic-transfer-G2}
  Let $F$ be a local field of characteristic zero.
  Let $\tau = (M, \sigma, r)$ be an elliptic triplet for $\rG_2(F)$ with $M$ a proper Levi subgroup of $\rG_2$.
  Then, there exists a unique discrete series parameter $\phi^H(\tau)$ of $H = \SO_4$ such that $\iota_M^{\rG_2}\circ\phi^M(\sigma)=\iota_H^{\rG_2}\circ\phi^H(\tau)$.
\end{proposition}

\begin{proof}
  This follows from Lemmas~\ref{G2-elliptic-GL2-Levi-representations}, \ref{G2-elliptic-toral-representations} and Remark~\ref{remark:LIR-and-ellipticity}.
\end{proof}

\begin{proposition}\label{proposition:elliptic-triplets-elliptic-transfer-twisted-PGSO8}
  Let $\tau = (M, \sigma, \widetilde{r})$ be an elliptic triplet for $(G, \widetilde{G})$ with $M$ a proper Levi subgroup of $G$.
  Then, there exists a unique elliptic endoscopic group $H \neq \rG_2$ of $\widetilde{G}$ and a discrete series parameter $\phi^H(\tau)$ such that the $L$-parameter of $i^G_P(\sigma)$ is given by the composition of $\phi^H(\tau)$ and the embedding $\iota^{G}_{H} \colon \widehat{H} \hookrightarrow \widehat{G}$.
\end{proposition}

\begin{proof}
  The proof follows Proposition~\ref{proposition:elliptic-triplets-elliptic-transfer-G2}, using the preceding results and Lemma~\ref{lemma:classification-regular-conjugacy-classes-Weyl-elements}.
  For Lemma~\ref{lemma:standard-parameters-of-theta-unstable-cases} and case $(c)$ of Lemma~\ref{lemma:classification-regular-conjugacy-classes-Weyl-elements}, use the local version of \cite{Gan2025-TrialityAdjointLiftingGL3}*{Theorem 8.2}.
\end{proof}

\part{Local harmonic analysis}\label{part:local-harmonic-analysis}

\section{Preliminaries on endoscopic transfer}\label{section:endoscopic-transfer}

Let $F$ be a local field of characteristic zero.
Let $(G, \widetilde{G})$ consist of a reductive group over $F$ and its twisted space.
For each endoscopic datum $(G^{\fe}, \widetilde{G}^{\fe}, s, \xi)$ for $(G, \widetilde{G})$, we have the transfer map
\begin{align*}
  \Trans_{G^{\fe}}^{\widetilde{G}} \colon \cI(\widetilde{G}) \to \cS(G^{\fe})^{\Aut(\fe)}.
\end{align*}
For $f^{\widetilde{G}} \in \cI(\widetilde{G})$, we write $$f^{\widetilde{G}}_{G^{\fe}} = \Trans_{G^{\fe}}^{\widetilde{G}}(f^{\widetilde{G}}).$$
We also have the parabolic descent map
\begin{align*}
  \cI(\widetilde{G}) \to \cI(\widetilde{M})^{W(G, \widetilde{M})}: f^{\widetilde{G}} \mapsto f^{\widetilde{G}}_{\widetilde{M}}
\end{align*}
for any twisted Levi space $\widetilde{M}$ of $\widetilde{G}$.

\subsection{Parabolic descent commutes with endoscopic transfer}

\begin{lemma}\label{lemma:parabolic-descent-commutes-with-transfer}
  Let $(G, \widetilde{G} = G \rtimes \theta)$ be a setting for twisted endoscopy.
  Let $G^{\fe}$ be a relevant elliptic endoscopic datum for
  $(G,\widetilde G)$, let $M^{\fe}$ be a relevant Levi subgroup of
  $G^{\fe}$, and let $\widetilde M$ be the associated twisted Levi space
  of $\widetilde G$.  Give the Levi datum $(M^{\fe},\widetilde M)$ the
  auxiliary data and transfer-factor normalization obtained by restriction
  from those for $(G^{\fe},\widetilde G)$, and use compatible Haar measures.
  Let $f^{\widetilde{G}}$ be a test function on $\widetilde{G}$ and $f^{\widetilde{G}}_{G^{\fe}}$ be a transfer of $f^{\widetilde{G}}$ to $G^{\fe}$.
  Then, the parabolic descent of $f^{\widetilde{G}}$ to $\widetilde{M}$ transfers to the parabolic descent of $f^{\widetilde{G}}_{G^{\fe}}$ to $M^{\fe}$.
\end{lemma}

\begin{proof}
  Use the canonical descent maps on measure lines of
  \cite{MoeglinWaldspurger2016-StabilisationFormuleTracesTordue1}*{I.3.1, pp. 56--57}.
  The Levi datum and auxiliary data restrict by
  \cite{MoeglinWaldspurger2016-StabilisationFormuleTracesTordue1}*{I.3.4, pp. 62--63},
  and \cite{MoeglinWaldspurger2016-StabilisationFormuleTracesTordue1}*{I.4.11, condition (2), formula (4), pp. 95--96} gives in $\cS(M^{\fe})$
  \begin{align*}
    (f^{\widetilde{G}}_{G^{\fe}})_{M^{\fe}}
    =f^{\widetilde M}_{M^{\fe}},
  \end{align*}
  where $f^{\widetilde M} = (f^{\widetilde{G}})_{\widetilde M}$.
\end{proof}

\subsection{Endoscopic character relations for the Steinberg representations}

We use the following twisted Steinberg character formula. Compare
\cite{AtobeGanEtAl2024-LocalIntertwiningRelationsCoTemperedPackets}*{Proposition B.5.1}.

\begin{lemma}\label{lemma:twisted-character-formula-for-steinberg}
   Let $F$ be a non-archimedean local field of characteristic zero, and $G$ be a quasi-split connected reductive group over $F$.
   Let $\mathrm{pin}_{G}$ be an $F$-pinning of $G$.
   Let $\theta$ be a finite-order pinning-preserving automorphism of $G$ over $F$.
   Let $P_0=M_0N_0$ be the minimal parabolic subgroup determined by the pinning.
   Let $\widetilde{\St_{G}}$ be the Whittaker extension of the Steinberg representation $\St_{G}$ of $G(F)$ to $G \rtimes \theta$ with respect to $\mathrm{pin}_{G}$.
   Then, we have
   \begin{align*}
       \Theta_{\widetilde{\St_{G}}}(f)
       = \sum_{P,\theta(P)=P}(-1)^{\dim(A_{M_0}/A_M)^{\theta}}
       \Theta_{i_{P\rtimes\theta}^{G\rtimes\theta}
       (\widetilde{\delta_P^{-1/2}})}(f),
   \end{align*}
   for $f\in C_c^{\infty}(G(F)\rtimes\theta)$, where $P=MN$ runs over
   the $\theta$-stable standard parabolic subgroups and
   $\widetilde{\delta_P^{-1/2}}(\theta)=1$.
\end{lemma}

\begin{proof}
   Apply $\theta$ to the standard resolution of $\St_G$ by the
   representations $i_P^G(\delta_P^{-1/2})$ with their orientation lines.
   If $I$ is the set of relative simple roots of $M$, the coefficient
   of a $\theta$-stable term in the alternating trace is
   \begin{align*}
     (-1)^{|I|}\det(\theta\mid\R I)
     =(-1)^{|I/\langle\theta\rangle|}
     =(-1)^{\dim(A_{M_0}/A_M)^{\theta}}.
   \end{align*}
   The other terms have zero trace on $G(F)\rtimes\theta$.
   The action on the Steinberg quotient preserves the Whittaker
   functional inherited from the open Bruhat cell, hence is the
   Whittaker extension.
\end{proof}

Let $H$ be the principal endoscopic group of $G$ with respect to $\theta$ with the $L$-group ${}^LH = G^{\vee, \theta, \circ} \rtimes W_F$.
Also, we have the parabolic subgroup $P_H$ of $H$ corresponding to $P$, and the extension
\[
  \widetilde{\delta_{P}^{-\frac{1}{2}}}
\]
transfers to $\delta_{P_H}^{-\frac{1}{2}}$.
Lemma~\ref{lemma:parabolic-descent-commutes-with-transfer} gives the following endoscopic character relation for the Steinberg representation.

\begin{lemma}\label{lemma:ecr-for-steinberg}
  The extension $\widetilde{\St_{G}}$ of the Steinberg representation $\St_{G}$ of $G(F)$ to $G \rtimes \theta$ with respect to $\mathrm{pin}_{G}$ transfers to $\St_H$.
\end{lemma}

\section{K-spaces}

Let $F=\R$. We recall $K$-spaces from \cite{MoeglinWaldspurger2016-StabilisationFormuleTracesTordue1}*{I.1.11}.

\begin{definition}
    We call the following datum a $K$-datum over $\R$.
    \begin{itemize}
        \item A finite family $(G_p)_{p \in \Pi}$ of connected reductive groups over $\R$.
        \item A finite family $(\widetilde{G}_p)_{p \in \Pi}$ of twisted spaces over $\R$ under $G_p$.
        \item A finite family $(\phi_{p, q}, \widetilde{\phi}_{p, q}, \nabla_{p, q})_{p, q \in \Pi}$ consisting of the following data.
        \begin{itemize}
            \item Isomorphisms $\phi_{p, q}: G_q \to G_p$ defined over $\C$ and isomorphisms $\widetilde{\phi}_{p, q}: \widetilde{G}_q \to \widetilde{G}_p$ defined over $\C$ which are compatible with $\phi_{p, q}$.
            \item A cocycle $\nabla_{p, q}: \Gamma_{\R} \to G_{p, sc}$.
        \end{itemize}
        These data must satisfy the following conditions.
            \begin{enumerate}
                \item $\phi_{p, q} \circ \sigma(\phi_{p, q})^{-1} = \Int(\nabla_{p, q}(\sigma))$ and $\widetilde{\phi}_{p, q} \circ \sigma(\widetilde{\phi}_{p, q})^{-1} = \Int(\nabla_{p, q}(\sigma))$.
                \item $\phi_{p, q} \circ \phi_{q, r} = \phi_{p, r}$ and $\widetilde{\phi}_{p, q} \circ \widetilde{\phi}_{q, r} = \widetilde{\phi}_{p, r}$.
                \item 
                $\nabla_{p, r}(\sigma) = \phi_{p, q}(\nabla_{q, r}(\sigma)) \nabla_{p, q}(\sigma)$.
                \item $\widetilde{G}_p(\R) \neq \emptyset$ for all $p \in \Pi$.
                \item The family $(\nabla_{p, q})_{q \in \Pi}$ is mapped bijectively to the set 
                \begin{align*}
                    \pi(H^1(\Gamma_{\R}, G_{p, sc})) \cap H^1(\Gamma_{\R}, G_{p})^{\theta}. 
                \end{align*}
                Here, $\pi: H^1(\Gamma_{\R}, G_{p, sc}) \to H^1(\Gamma_{\R}, G_{p})$ is the natural map and $\theta$ is the automorphism of $G_p$ induced by $\widetilde{G}_p$.
            \end{enumerate}
    \end{itemize}    

    In this situation, we set 
    \begin{align*}
        KG = \bigsqcup_{p \in \Pi} G_p, \quad K\widetilde{G} = \bigsqcup_{p \in \Pi} \widetilde{G}_p,
    \end{align*}
    and we call $(KG, K\widetilde{G})$ a $K$-space over $\R$ associated to the $K$-datum $(G_p, \widetilde{G}_p)_{p \in \Pi}$.
\end{definition}

For a twisted reductive space $\widetilde{G}$ over $\R$, the associated $K$-space $(KG, K\widetilde{G})$ consists of its inner twists indexed by $$\pi(H^1(\Gamma_{\R}, G_{sc})) \cap H^1(\Gamma_{\R}, G)^{\theta},$$ with $\widetilde{G}$ as base point. This also defines $\cE(K\widetilde{G})$. See \cite{MoeglinWaldspurger2016-StabilisationFormuleTracesTordue1}*{I.1.11}.

Let $F$ be a number field and $S_\R$ its set of real places.
Global $K$-spaces are defined in \cite{MoeglinWaldspurger2016-StabilisationFormuleTracesTordue2}*{VI.1.16}, whose lemma gives the Hasse principle
\begin{align*}
    \pi(H^1(F, G_{sc})) \cap H^1(F, G)^{\theta} = \prod_{v \in S_{\R}} \bigl(\pi(H^1(F_v, G_{sc})) \cap H^1(F_v, G)^{\theta}\bigr).
\end{align*}

\section{Certain surjectivity results}

\begin{theorem}\label{theorem:cuspidal-surjective}
    The restriction of the transfer map 
    \begin{align*}
        \Trans^{K\widetilde{\PGSO_8}} 
        \colon
        \cI_{\cusp}(K\widetilde{\PGSO_8}(\R)) \to \cS_{\cusp}(\rG_2(\R)) 
        \oplus
        \cS_{\cusp}(\SO_4(\R))
    \end{align*}
    to the subspace of cuspidal functions is surjective.
\end{theorem}

\begin{proof}
  See \cite{MoeglinWaldspurger2016-StabilisationFormuleTracesTordue1}*{I.4.11, Proposition}.
\end{proof}

\begin{theorem}\label{theorem:PGSO8-to-G2-surjective}
   The transfer map
   \begin{align*}
      \Trans_{\rG_2}^{\widetilde{\PGSO_8}} \colon \cI(\widetilde{\PGSO_8}(\R)) \to \cS(\rG_2(\R))
   \end{align*}
   is surjective.
\end{theorem}

\begin{proof}
   By \cite{Gan2025-TrialityAdjointLiftingGL3}*{Proposition 4.7}, it suffices to show that the transfer map
   \begin{align*}
      \Trans_{\rG_2}^{\widetilde{\PGSO_8}} \colon \cI_{\cusp}(\widetilde{\PGSO_8}(\R)) \to \cS_{\cusp}(\rG_2(\R))
   \end{align*}
   is surjective.
   It is enough to realize a stable pseudo-coefficient of each discrete $L$-parameter for $\rG_2(\R)$ as a transfer.
   By the real spectral--geometric correspondence for elliptic characters and cuspidal functions (\cite{MoeglinWaldspurger2016-StabilisationFormuleTracesTordue1}*{IV.2.2}), this reduces to transferring the stable character of each discrete $L$-parameter for $\rG_2(\R)$ to a character of $\widetilde{\PGSO_8}(\R)$.
   This follows from the main theorem of \cite{KalethaMezo2026-RefinedLocalLanglandsConjectureDiscreteLParameters}.
\end{proof}

\begin{remark}
   The Kaletha--Mezo extension to $\widetilde{\PGSO_8}(\R)$ agrees with the canonical extension by uniqueness.
\end{remark}

\section{Spectral transfer versus geometric transfer}

Let $F$ be a local field of characteristic zero and let $\widetilde G$ be $\widetilde{\PGSO_8}$ or $\rG_2$.
We use the elliptic inner products on cuspidal functions and elliptic tempered distributions defined in \cite{MoeglinWaldspurger2016-StabilisationFormuleTracesTordue1}*{I.4.17} and \cite{MoeglinWaldspurger2018-FormuleTracesLocaleTordue}*{Section 7.3}.
The character inner product is independent of Haar measures.
For $F=\R$, these spaces are taken on the associated $K$-spaces, and the elliptic inner products are the sums over their components.

\begin{proposition}[\cite{MoeglinWaldspurger2016-StabilisationFormuleTracesTordue1}*{I.4.17}]\label{proposition:elliptic-inner-product-transfer}
  Let $f^{\widetilde G}\in\cI_{\cusp}(\widetilde G(F))$ and let $f^{\widetilde G}_H\in\cS_{\cusp}(H(F))$ be its transfer for each elliptic endoscopic group $H$.
  Then
  \begin{align*}
    (f^{\widetilde G},f^{\widetilde G})_{\mathrm{ell}}
    =\sum_H c(\widetilde G,H)
    (f^{\widetilde G}_H,f^{\widetilde G}_H)_{\mathrm{ell}},
  \end{align*}
  where
  \begin{align*}
    c(\widetilde{\PGSO_8},H)
    &=\begin{cases}
      1 & \text{if $H=\rG_2$ or $\SL_3$},\\
      1/2 & \text{if $H=\SO_4$},
    \end{cases}\\
    c(\rG_2,H)
    &=\begin{cases}
      1 & \text{if $H=\rG_2$},\\
      1/2 & \text{if $H=\SO_4$},\\
      1/3 & \text{if $H=\PGL_3$}.
    \end{cases}
  \end{align*}
  Let $I^{\widetilde G}$ be an elliptic tempered distribution and let $S^H$ be stable elliptic tempered distributions, invariant under automorphisms of the corresponding endoscopic data, such that
  \begin{align*}
    I^{\widetilde G}(f^{\widetilde G})
    =\sum_H S^H(f^{\widetilde G}_H)
  \end{align*}
  for every $f^{\widetilde G}\in\cI_{\cusp}(\widetilde G(F))$.
  Then
  \begin{align*}
    (I^{\widetilde G},I^{\widetilde G})_{\mathrm{ell}}
    =\sum_H c(\widetilde G,H)^{-1}
    (S^H,S^H)_{\mathrm{ell}}.
  \end{align*}
\end{proposition}

\begin{proof}
  The first identity and constants are given in \cite{MoeglinWaldspurger2016-StabilisationFormuleTracesTordue1}*{I.4.17, Proposition and p.~110}.
  Transfer is an isomorphism on the cuspidal spaces by \cite{MoeglinWaldspurger2016-StabilisationFormuleTracesTordue1}*{I.4.11}.
  The second identity follows by adjointness under the canonical antilinear isometries between elliptic tempered distributions and cuspidal functions supplied by pseudo-coefficients in \cite{MoeglinWaldspurger2018-FormuleTracesLocaleTordue}*{Sections 7.2--7.3} and \cite{MoeglinWaldspurger2016-StabilisationFormuleTracesTordue1}*{IV.2.2}.
\end{proof}

\begin{corollary}\label{corollary:discreteness-of-discrete-packet}
  Let $F$ be a local field of characteristic zero.
  Let $S^{\rG_2}$ be an elliptic stable linear form on $\rG_2(F)$ which is a linear combination of irreducible tempered characters of $\rG_2(F)$.
  Then, it is a linear combination of discrete series characters of $\rG_2(F)$.
\end{corollary}

\begin{proof}
  For $F=\R$, the assertion follows from the description of stable tempered packet characters as parabolic inductions in \cite{Shelstad1979-CharactersInnerFormsQuasiSplitGroup}*{Section 3, Lemma 3.1} and the stable elliptic decomposition in \cite{MoeglinWaldspurger2016-StabilisationFormuleTracesTordue1}*{IV.2.2}.
  Over $\C$ the elliptic tempered space is zero.
  Suppose that $F$ is non-archimedean.
  Proposition~\ref{proposition:elliptic-inner-product-transfer} makes $S^{\rG_2}$ orthogonal to transfers of stable discrete series characters of $\SO_4$, and ellipticity makes it orthogonal to distributions induced from proper Levi subgroups.
  Theorem~\ref{theorem:local-tempered-LIR-G2} then gives orthogonality to every non-discrete-series elliptic character.
\end{proof}

\section{Elliptic abstract spectral transfer}\label{section:elliptic-spectral-transfer}

\begin{theorem}[\cite{MoeglinWaldspurger2016-StabilisationFormuleTracesTordue2}*{XI.4, Th\'eor\`eme}]\label{theorem:elliptic-character-transfer}
   Let $F$ be a non-archimedean local field of characteristic zero.
   Let $\widetilde{G} = (G, \widetilde{G})$ be a twisted space over a connected reductive group.
   All distributions below are finite $\C$-linear combinations of tempered characters.
   \begin{enumerate}
      \item
      Every elliptic tempered distribution $I^{\widetilde{G}}$ of $\widetilde{G}(F)$ admits a unique family of $\Aut(\fe)$-invariant stable elliptic tempered distributions $\{S^{\widetilde{G}^{\fe}}\}$, indexed by equivalence classes of elliptic endoscopic data $\widetilde{G}^{\fe}$ of $\widetilde{G}$, such that
      \begin{align*}
         I^{\widetilde{G}}(f^{\widetilde{G}})
         =
         \sum_{\widetilde{G}^{\fe}}
         S^{\widetilde{G}^{\fe}}(f^{\widetilde{G}}_{\widetilde{G}^{\fe}}),
      \end{align*}
      for any function $f^{\widetilde{G}} \in \cI(\widetilde{G}(F))$.
      \item
      Conversely, every stable elliptic tempered distribution $S^{\widetilde{G}^{\fe}}$ of an elliptic endoscopic group $\widetilde{G}^{\fe}$ of $\widetilde{G}$ admits a unique elliptic tempered distribution $I^{\widetilde{G}}$ of $\widetilde{G}(F)$ such that
      \begin{align*}
         I^{\widetilde{G}}(f^{\widetilde{G}})
         =
         S^{\widetilde{G}^{\fe}}(f^{\widetilde{G}}_{\widetilde{G}^{\fe}}),
      \end{align*}
      for any function $f^{\widetilde{G}} \in \cI(\widetilde{G}(F))$.
   \end{enumerate}
\end{theorem}

\begin{proof}
   Existence follows from the cited theorem. Uniqueness follows from the cuspidal transfer isomorphism of \cite{MoeglinWaldspurger2016-StabilisationFormuleTracesTordue1}*{I.4.11} and \cite{MoeglinWaldspurger2016-StabilisationFormuleTracesTordue2}*{XI.4.4}.
\end{proof}

\part{Global representation theory}\label{part:global-representation-theory}

\section{Cuspidal automorphic representations on \texorpdfstring{$\PGL_3^{+}$}{PGL3+}}\label{section:cuspidal-PGL3-plus}

Let $F$ be a number field and $\A_F$ be the ring of adeles of $F$.
We set $\PGL_3^{+} = \PGL_3 \rtimes \Z/2\Z$, where $\Z/2\Z$ acts on $\PGL_3$ by the outer automorphism
\begin{align*}
  \iota \colon g \mapsto J {}^t g^{-1} J^{-1},
\end{align*}
where
\begin{align*}
  J = \begin{pmatrix}
    0 & 0 & 1 \\
    0 & -1 & 0 \\
    1 & 0 & 0
  \end{pmatrix}.
\end{align*}
We take the standard Borel pair $(B, T)$ of $\PGL_3$ and the pinning $\mathrm{pin}_{\PGL_3}$, which correspond to the standard basis of $F^3$.

\begin{lemma}[Clifford theory, cf. \cite{Kaletha2022-LocalLanglandsConjecturesDisconnectedGroups}*{Section 4.6}]
  Let $F$ be a local field of characteristic zero.
  \begin{enumerate}
    \item Let $\tau$ be an irreducible admissible representation of $\PGL_3(F)$. It extends to an irreducible admissible representation of $\PGL_3^{+}(F)$ if and only if $\tau$ is self-dual.
    In that case it has exactly two extensions.
    If $\tau$ is generic, let $\tau^{+}(+1)$ be the unique extension for which the action of $\iota$ preserves a Whittaker functional, and let $\tau^{+}(-1)$ be its twist by the non-trivial character of $\PGL_3^{+}(F)/\PGL_3(F)$.
    \item
    For any irreducible admissible representation $\tau^+$ of $\PGL_3^{+}(F)$, we have the following dichotomy.
    \begin{itemize}
      \item If the restriction of $\tau^+$ to $\PGL_3(F)$ is irreducible, then $\tau^+$ is one of the two extensions of a self-dual representation $\tau$ of $\PGL_3(F)$.
      \item If the restriction of $\tau^+$ to $\PGL_3(F)$ is reducible, then $\tau^+ \cong \Ind_{\PGL_3(F)}^{\PGL_3^{+}(F)} \tau$ for some non-self-dual representation $\tau$ of the group $\PGL_3(F)$. Let $\tau^{+}(+1)$ denote this representation of $\PGL_3^{+}(F)$.
    \end{itemize}
  \end{enumerate}
\end{lemma}

\begin{definition}
  The representation $\tau^{+}$ is said to be generic if its restriction contains a generic representation $\tau$ of $\PGL_3(F)$ and, when $\tau$ is self-dual, $\tau^{+} \cong \tau^{+}(+1)$.
\end{definition}

Strong multiplicity one for $\PGL_3$ gives the following classification.

\begin{proposition}\label{proposition:cuspidal-PGL3-plus}
  Let $\Pi^{+}$ be an irreducible cuspidal automorphic representation of $\PGL_3^+(\A_F)$.
  Let $S$ be a finite set of places of $F$ such that $\Pi_v^{+}$ is unramified for any place $v \notin S$.
  Then, we can find a cuspidal automorphic representation $\Pi$ of $\PGL_3(\A_F)$ that satisfies the following conditions.
  \begin{itemize}
    \item $\Pi_v$ is unramified for any place $v \notin S$ and $\Pi^{+}_v = \Pi_v^{+}(+1)$.
    \item $\Pi_v^{+} \cong \Ind_{\PGL_3(F_v)}^{\PGL_3^{+}(F_v)} \Pi_v$ for any place $v \in S$ such that $\Pi_v$ is not self-dual.
    \item $\Pi_v^{+} \cong \Pi_v^{+}(\epsilon_v)$ for some $\epsilon_v \in \{\pm 1\}$ for any place $v \in S$ such that $\Pi_v$ is self-dual.
    \item
    Furthermore, let $S_{sd}$ (resp. $S_{nsd}$) be the set of places $v \in S$ such that $\Pi_v$ is self-dual (resp. non-self-dual).

    Then, we have
    \begin{align*}
      \prod_{v \in S_{sd}} \epsilon_v = 1
    \end{align*}
    if $\Pi$ is self-dual.
    \end{itemize}
  We can form a colimit of such data in $S$ and consider $\Pi$ up to duality.

  Conversely, if we have an element in the colimit, we can uniquely construct an irreducible cuspidal automorphic representation $\Pi^{+}$ of $\PGL_3^{+}(\A_F)$ which is unramified outside $S$.
\end{proposition}

\begin{proof}
  Apply the proof of \cite{HarrisSoudryTaylor1993-AdicRepresentationsAssociatedModularFormsImaginary}*{Proposition 2}.
\end{proof}

\begin{definition}
  Let $F$ be a number field and $\Pi$ be a cuspidal automorphic representation of $\PGL_3(\A_F)$.
  For a finite set of places $S$ of $F$ such that $\Pi_v$ is unramified for any place $v \notin S$ and a family of signs $\epsilon = (\epsilon_v)_{v \in S} \in \prod_{v \in S}\{\pm 1\}$ such that $\prod_{v \in S_{sd}} \epsilon_v = 1$ if $\Pi$ is self-dual, let $\Pi^{+}(\epsilon)$ denote the corresponding representation of $\PGL_3^{+}(\A_F)$.

  We call the representation $\Pi^{+}(\epsilon)$ generic if $\epsilon_v = 1$ at every place $v \in S$ for which $\Pi_v$ is self-dual. Equivalently, each self-dual local extension is Whittaker-normalized.
\end{definition}

\begin{theorem}\label{theorem:theta-PGL3-to-G2}
  Let $F$ be a number field and let $\Pi$ be a cuspidal automorphic representation of $\PGL_3(\A_F)$.
  Let $\Pi^+$ be its generic extension to $\PGL_3^+(\A_F)$.
  Then, the exceptional theta space $\Theta(\Pi^+)$ is contained in the cuspidal spectrum of $\rG_2(\A_F)$, and its projection to the globally generic cuspidal spectrum is non-zero.
\end{theorem}

\begin{proposition}
  Let $F$ be a totally real number field and $\Sigma$ be a cuspidal automorphic representation of $\rG_2(\A_F)$.
  Suppose that $L^S(s,\Sigma,\Std_{\rG_2})$ has a pole at $s=1$ for all sufficiently large finite sets $S$ of places of $F$.
  Then, the exceptional theta lift of $\Sigma$ to $\PGL_3^+(\A_F)$ or some inner form $\mathrm{PD}^{\times}(\A_F)$ is non-zero.
\end{proposition}

\begin{proof}
  See Theorem~\ref{theorem:exceptional-theta-G2-to-PGL3}.
\end{proof}

\begin{remark}
  Let $F$ be a totally real number field.
  Let $\Pi$ be a cuspidal automorphic representation of $\PGL_3(\A_F)$.

  If $\Sigma$ is contained in the cuspidal spectrum of $\rG_2(\A_F)$ with the Satake parameter of the form $\iota^{\rG_2}_{\PGL_3} \circ c(\Pi)$, then $\Sigma$ is an exceptional theta lift of some cuspidal automorphic representation of $\PGL_3^{+}$ or an inner form of $\PGL_3$ over $F$.

  If $\Sigma$ lifts to $\mathrm{PD}^{\times}$ for a division algebra $D$ over $F$, its lift is $JL_{D}(\Pi)$ or $JL_{D}(\Pi^{\vee})$.
  Indeed, its standard Satake parameter is that of
  \begin{align*}
    \Pi  \boxplus \Pi^{\vee} \boxplus \mathbf{1}.
  \end{align*}
  In particular,
  \begin{align*}
    L^S(s,\Sigma,\Std_{\rG_2})
    =\zeta_F^S(s)L^S(s,\Pi)L^S(s,\Pi^{\vee})
  \end{align*}
  has a simple pole at $s=1$ for every sufficiently large $S$.
  Thus the hypothesis of Theorem~\ref{theorem:exceptional-theta-G2-to-PGL3} holds.
\end{remark}

\section{Certain triviality results}\label{section:quadratic-twists}

Let
\begin{align*}
  q\colon \GSO_8 \longrightarrow \PGSO_8
  \quad\text{and}\quad
  \lambda\colon \GSO_8 \longrightarrow \GL_1
\end{align*}
be the quotient by the scalar center and the similitude character,
respectively.  Let $k$ be a local field or a number field of
characteristic zero and let $\omega$ be a quadratic character of
$k^{\times}$ or $\A_k^{\times}/k^{\times}$, respectively.  Since
$\lambda(z1_8)=z^2$, the character $\omega\circ\lambda$ is trivial on
the scalar center and therefore descends to a character
$\chi_{\omega,1}$ of $\PGSO_8(k)$ or $\PGSO_8(\A_k)$, respectively.
Let $\theta$ be the triality automorphism of $\PGSO_8$ and set
\begin{align*}
  \chi_{\omega,2}=\chi_{\omega,1}^{\theta^{-1}},
  \quad
  \chi_{\omega,3}=\chi_{\omega,2}^{\theta^{-1}}.
\end{align*}
Here, for any automorphism $\alpha$, the superscript denotes pullback.
$\chi^{\alpha}=\chi\circ\alpha$, and similarly
$\pi^{\alpha}=\pi\circ\alpha$ for representations.  A calculation in
the $D_4$ root datum gives
\begin{align*}
  \chi_{\omega,1}\chi_{\omega,2}\chi_{\omega,3}=1.
\end{align*}
If $\omega$ is non-trivial, these are the three non-trivial elements of
a two-dimensional $\bF_2$-space on which $\theta$ acts transitively.  We
refer to twists by the $\chi_{\omega,i}$ as the triality quadratic
twists of representations of $\PGSO_8$.

\begin{lemma}\label{lemma:triality-quadratic-twist}
  Let $F$ be a local field.
  Let $\pi, \pi'$ be irreducible representations of $\PGSO_8(F)$ which are stable under the triality automorphism $\theta$.
  If
  \begin{align*}
    \pi' \cong \pi\otimes\chi_{\omega,i}
  \end{align*}
  for a quadratic character $\omega$ of $F^{\times}$ and some
  $i\in\{1,2,3\}$, then $\pi'\cong\pi$.
\end{lemma}

\begin{proof}
  We may assume that $\omega$ is non-trivial.  Let
  $\mathcal X_{\omega}=\{1,\chi_{\omega,1},\chi_{\omega,2},
  \chi_{\omega,3}\}$ and let $S$ be the subgroup of
  $\mathcal X_{\omega}$ consisting of the self-twists of $\pi$.
  Since $\pi$ is $\theta$-stable, the subgroup $S$ is $\theta$-stable.
  The $\theta$-stability of $\pi'$ gives
  \begin{align*}
    \pi\otimes\chi_{\omega,i}
    \cong
    \pi\otimes\chi_{\omega,i}^{\theta},
  \end{align*}
  so $S$ contains a non-trivial element.  Triality acts transitively on
  the three non-trivial elements of $\mathcal X_{\omega}$. Hence
  $S=\mathcal X_{\omega}$.  In particular,
  $\pi\otimes\chi_{\omega,i}\cong\pi$.
\end{proof}

\begin{lemma}\label{lemma:triality-quadratic-twist-global}
  Let $F$ be a number field.
  Let $\pi$ be a discrete automorphic representation of $\PGSO_8(\A_F)$ which is stable under the triality automorphism $\theta$.
  If
  \begin{align*}
    \pi\cong\pi\otimes\chi_{\omega,1}
  \end{align*}
  for a quadratic Hecke character $\omega$, then the $8$-dimensional
  standard lift satisfies
  \begin{align*}
    \operatorname{Lift}_{\PGSO_8}^{\GL_8}(\pi)
    \cong
    \operatorname{Lift}_{\PGSO_8}^{\GL_8}(\pi)\otimes(\omega\circ\det).
  \end{align*}
\end{lemma}

\begin{proof}
  The $\theta$-stability of $\pi$ gives
  \begin{align*}
    \pi
    \cong
    \pi^{\theta^{-1}}
    \cong
    (\pi\otimes\chi_{\omega,1})^{\theta^{-1}}
    \cong
    \pi\otimes\chi_{\omega,2}.
  \end{align*}
  At every unramified place, the $D_4$ root datum shows that the standard
  lift of $\pi_v\otimes\chi_{\omega_v,2}$ is
  $\operatorname{Lift}_{\PGSO_8}^{\GL_8}(\pi_v)\otimes(\omega_v\circ\det)$.  The asserted global
  isomorphism now follows from strong multiplicity one for isobaric
  automorphic representations of general linear groups.
\end{proof}

\begin{corollary}\label{corollary:no-quadratic-self-twist-PGSO8}
  We continue to use the notation in the previous lemma.
  Assume that
  \begin{align*}
    \operatorname{Lift}_{\PGSO_8}^{\GL_8}(\pi)=\mathbf{1}\boxplus\tau_7
  \end{align*}
  for a cuspidal automorphic representation $\tau_7$ of
  $\GL_7(\A_F)$.  Then no triality quadratic twist attached to a
  non-trivial quadratic Hecke character stabilizes $\pi$.
\end{corollary}

\begin{proof}
  It suffices, by triality, to consider $\chi_{\omega,1}$.  If this twist
  stabilizes $\pi$, Lemma~\ref{lemma:triality-quadratic-twist-global}
  gives
  \begin{align*}
    \mathbf{1}\boxplus\tau_7
    \cong
    \omega\boxplus(\tau_7\otimes\omega).
  \end{align*}
  The uniqueness of the isobaric decomposition, applied to the
  degree-one summands, gives $\omega=\mathbf{1}$.
\end{proof}

\begin{corollary}
  Under the preceding corollary's hypotheses, if the $A$-parameter of $\pi$ is generic, its discrete multiplicity is $m_{\disc}(\pi)=1$.
\end{corollary}

\section{A globalization result for generic representations for \texorpdfstring{$\rG_2(F)$}{G2(F)}}

\begin{lemma}\label{corollary:globalization-generic-supercuspidal-representation}
  Let $F$ be a number field, let $v_0,w$ be distinct finite places of $F$, and let $S$ be a finite set of finite places disjoint from $\{v_0,w\}$.
  Fix a global Whittaker character $\psi_N$ for $\rG_2$ which is unramified at finite places outside $S\cup\{v_0,w\}$.
  Let $\sigma$ be a generic supercuspidal representation of $\rG_2(F_{v_0})$.
  For $v\in S$, let $\pi_v$ be a generic discrete series representation of $\rG_2(F_v)$.

  Then there exists a globally generic cuspidal automorphic representation $\Sigma$ of $\rG_2(\A_F)$ such that $\Sigma_{v_0}\cong\sigma$, $\Sigma_v\cong\pi_v$ for $v\in S$, $\Sigma_w\cong\St_{\rG_2(F_w)}$, and $\Sigma_v$ is unramified at every finite place outside $S\cup\{v_0,w\}$.
\end{lemma}

\begin{proof}
  Put $G=\rG_2$ and $H=\PGSp_6$.
  By \cite{IchinoLapidMao2017-FormalDegreesSquareIntegrableRepresentationsOdd}*{Proposition A.4} and the Whittaker--Poincar\'e series argument of \cite{GanIchino2018-ShimuraWaldspurgerCorrespondenceMp2n}*{proof of Proposition A.2},
  $\bigotimes_{v\in S}\pi_v\otimes\St_{G(F_w)}$ is weakly contained in the globally generic cuspidal spectrum with component $\sigma$ at $v_0$ and unramified components at finite places outside $S\cup\{v_0,w\}$.
  Here we take hyperspecial test functions at those finite places and leave the archimedean factors unrestricted.
  Since $\St_{G(F_w)}$ is isolated in the unitary dual by \cite{GanSavin2023-LocalLanglandsConjectureG2}*{Proposition 11.3}, we may restrict to the summands with component $\St_{G(F_w)}$ at $w$.

  Their global theta lifts to $H$ are nonzero and globally generic by \cite{HarrisKhareThorne2023-LocalLanglandsParameterizationGenericSupercuspidalRepresentations}*{Theorem A.9}.
  The lower lift to $\PGL_3$ vanishes because $\St_{G(F_w)}$ has zero local lift, by \cite{GanSavin2023-HoweDualityDichotomyExceptionalThetaCorrespondences}*{Theorem 8.2}.
  The lower lift to $\SO_5$ vanishes by \cite{GinzburgRallisSoudry1997-TowerThetaCorrespondencesG2}*{Theorem 4.1(3)}, since genericity excludes the $\SL_3$- and $\SU_3$-periods by \cite{GrossSavin1998-MotivesGaloisGroupTypeG2Exceptional}*{Lemma 4.10}.
  Thus the tower property \cite{GinzburgRallisSoudry1997-TowerThetaCorrespondencesG2}*{Theorem 3.1} makes the lifts to $H$ cuspidal.

  To isolate the prescribed local components, we follow \cite{IchinoLapidMao2017-FormalDegreesSquareIntegrableRepresentationsOdd}*{proof of Lemma A.2}.
  Use determinant coordinates on the maximal Levi subgroups of $G$ and the basis $(2\alpha+\beta,\alpha+\beta)$ of $X^*(T)$ for a split maximal torus $T$, where $\alpha$ is short and $\beta$ is long.
  By \cite{GanSavin2023-HoweDualityDichotomyExceptionalThetaCorrespondences}*{Section 3.1 and Propositions 3.1--3.2}, the real part of the central character of the cuspidal support of a generic discrete series representation lies in $\frac12\Z$, or in $\frac12\Z^2$ for toral support, modulo the relative Weyl group.
  The same holds for discrete series representations of the proper Levi subgroups.

  The standard parameter compatibility in \cite{GanSavin2023-LocalLanglandsConjectureG2}*{Main Theorem (iii), (iv)} and \cite{IchinoLapidMao2017-FormalDegreesSquareIntegrableRepresentationsOdd}*{Proposition A.5} give a uniform bound $c<\frac12$ for the real Langlands twists in this global spectrum.
  In the toral coordinates above, the bound is
  \begin{align*}
    h(s_1,s_2):=\max\{\abs{s_1},\abs{s_2},\abs{s_1+s_2}\}\leq c,
  \end{align*}
  since the standard seven-dimensional representation has weights $0,\pm s_1,\pm s_2,\pm(s_1+s_2)$.
  This norm is Weyl invariant, and the determinant twists of $M_s$ and $M_l$ give $(s,s)$ and $(s,0)$, respectively.
  As in the proof of \cite{IchinoLapidMao2017-FormalDegreesSquareIntegrableRepresentationsOdd}*{Lemma A.2}, continuity of cuspidal support makes the limiting real Langlands twist a difference of half-integral vectors.
  The bound $c<\frac12$ forces this difference to be zero.
  The limiting induction is therefore tempered, and discreteness forces its Levi subgroup to be $G$.
  Thus each $\pi_v$ is isolated in the indicated spectrum, and $\bigotimes_{v\in S}\pi_v$ must occur.
\end{proof}

\begin{proposition}\label{proposition:invariance-of-global-generic-theta-lift-to-PGSO_8}
  Let $F$ be a number field.
  Let $\Sigma$ be a globally generic cuspidal automorphic representation of $\rG_2(\A_F)$.
  Let $\tau$ be a globally generic irreducible constituent of the global exceptional theta lift of $\Sigma$ to $\PGSp_6$, and let $\Pi$ be a globally generic irreducible constituent of its global similitude theta lift to $\PGSO_8$.
  Assume that there exist finite places $v_1, v_2$ of $F$ such that $\theta_{\rG_2}^{\PGSp_6}(\Sigma_{v_1})$ is supercuspidal and $\Pi_{v_2}$ is isomorphic to the Steinberg representation of $\PGSO_8(F_{v_2})$.

  Then $\Pi$ is cuspidal and stable under $\Out(\PGSO_8)(F)$.
  We also have $m_{\cusp}(\Pi) = 1$.
\end{proposition}

\begin{proof}
  The existence of $\tau$ follows from \cite{HarrisKhareThorne2023-LocalLanglandsParameterizationGenericSupercuspidalRepresentations}*{Appendix}, and its supercuspidal component at $v_1$ makes it cuspidal.
  The Whittaker coefficient calculation in \cite{GinzburgRallisSoudry1997-PeriodsPolesSymplecticOrthogonalThetaLifts}*{p.110, after Proposition 2.7}, with $n=3$, shows that its classical theta lift to $\SO_8$ is nonzero and globally generic. The same holds for the similitude lift $\Pi$.
  At $v_2$, local theta correspondence gives $\tau_{v_2}\cong\St_{\PGSp_6}$. This representation has zero theta lift to the split $\PGO_6$, so the tower property makes $\Pi$ cuspidal, as in \cite{GanSavin2023-LocalLanglandsConjectureG2}*{proof of Lemma 12.5}.
  The representation $\Pi$ is invariant under $\O_8(F) \setminus \SO_8(F)$.

  The cuspidal representation $\Pi^{\theta}$ has the same Satake parameter as $\Pi$ almost everywhere, hence the same Arthur lift.
  Inflate $\Pi$ and $\Pi^{\theta}$ along
  $q\colon\GSO_8\to\PGSO_8$.  Applying
  \cite{Xu2025-GlobalLPacketsQuasisplitGSp2nGO2n}*{Theorem 1.1} with
  trivial central character, we obtain
  \begin{align*}
    \Pi^{\theta}
    \cong
    \Pi\otimes\chi_{\omega,1}
  \end{align*}
  for a quadratic Hecke character $\omega$ of
  $\A_F^{\times}/F^{\times}$.  Here $\chi_{\omega,1}$ is the character
  of $\PGSO_8(\A_F)$ obtained by descending $\omega\circ\lambda$ from
  $\GSO_8(\A_F)$, as defined above.

  Assume that $\omega$ is non-trivial.
  Applying $\theta^{-1}$ to the displayed isomorphism gives
  \begin{align*}
    \Pi
    \cong
    \Pi^{\theta^{-1}}\otimes\chi_{\omega,2}.
  \end{align*}
  Compatibility of the two theta lifts at unramified places and strong multiplicity one give
  \begin{align*}
    \operatorname{Lift}_{\PGSO_8}^{\GL_8}(\Pi)=\mathbf{1}\boxplus\operatorname{Lift}_{\PGSp_6}^{\GL_7}(\tau).
  \end{align*}
  The standard lift of $\tau$ is cuspidal since its component at $v_2$ is $\St_{\GL_7}$.
  As $\Pi$ and $\Pi^{\theta^{-1}}$ have the same standard lift, the unramified calculation for $\chi_{\omega,2}$ gives
  \begin{align*}
    \mathbf{1}\boxplus\operatorname{Lift}_{\PGSp_6}^{\GL_7}(\tau)
    \cong \omega\boxplus(\operatorname{Lift}_{\PGSp_6}^{\GL_7}(\tau)\otimes\omega).
  \end{align*}
  By uniqueness of the isobaric decomposition, comparison of the
  degree-one summands gives $\omega=\mathbf{1}$, a contradiction.
  Hence $\Pi^{\theta} = \Pi$. Since $\theta$ and $\O_8(F) \setminus \SO_8(F)$ generate $\Out(\PGSO_8)(F)$, the representation $\Pi$ is stable under this group.

  Corollary~\ref{corollary:no-quadratic-self-twist-PGSO8} excludes non-trivial triality quadratic self-twists of $\Pi$, so \cite{Xu2025-GlobalLPacketsQuasisplitGSp2nGO2n}*{Theorem 1.1} gives its cuspidal multiplicity one.
\end{proof}

\begin{corollary}\label{corollary:summary-global-invariance-of-generic-theta-lift-to-PGSO_8}
  We continue to use the notation in Proposition~\ref{proposition:invariance-of-global-generic-theta-lift-to-PGSO_8}.
  Let $\phi$ be the global $L$-parameter of $\Pi$.
  In this situation, the global $L$-packet $\Pi_{\phi}$ consists of automorphic representations of $\PGSO_8(\A_F)$ which are stable under the action of $\Out(\PGSO_8)(F)$.
  Also, each local $L$-packet $\Pi_{\phi_v}$ is a distinguished Xu $L$-packet $\Pi^X_{\iota_{\rG_2}^{\PGSO_8} \phi^{\rG_2}_v}$ for some generic parameter $\phi^{\rG_2}_v$ of $\rG_2$.
  For any finite place $v$ of $F$, the parameter $\phi^{\rG_2}_v$ is the $L$-parameter of $\Sigma_v$.

  The twisted discrete distribution is therefore
  \begin{align*}
    I_{\disc, \phi}^{\widetilde{\PGSO_8}} = \sum_{\pi \in \Pi_{\phi}} \tr \widetilde{ \pi},
  \end{align*}
  where $\widetilde{\pi}$ is the extension of $\pi$ to $\widetilde{\PGSO_8}(\A_F)$ defined by $\theta \in \widetilde{\PGSO_8}(F)$.
  These representations are the restricted tensor products of the local extensions of $\pi_v$ to $\widetilde{\PGSO_8}(F_v)$ defined by $\theta \in \widetilde{\PGSO_8}(F_v)$.
\end{corollary}

\section{Lifting from \texorpdfstring{$\SO_4$}{SO4} to \texorpdfstring{$\rG_2$}{G2}}\label{section:lifting-SO4-G2}

An orthogonal pair of long and short roots gives the embedding
\begin{align*}
   \SO_4 = \SL_{2, l} \times \SL_{2, s} / \Delta \mu_2 \to \rG_2.
\end{align*}
The composition with $\Std_{\rG_2}$ is given by $\Std_{\SO_4} \oplus \Ad_{\SL_{2, s}}$.
Let $\wedge^2_{+}=\Ad_{\SL_{2,s}}$ as a representation of $\SO_4$.
\begin{lemma}
  The map
  \begin{align*}
    \Psi(\SO_4)
    &\hookrightarrow \Psi(\PGL_4) \times \Psi(\PGL_3) \colon \psi \mapsto (\Std_{\SO_4}(\psi), \wedge^2_{+}(\psi))
  \end{align*}
  from the refined $A$-parameters of Appendix~\ref{section:refinement-SO4} is injective, and its image is contained
  in the set of pairs of elliptic orthogonal $A$-parameters of
  $\PGL_4$ and $\PGL_3$.
\end{lemma}

\begin{proof}
  If $\psi$ is represented by $(\Pi_1,\Pi_2)$, then
  \begin{align*}
    \wedge^2\Std_{\SO_4}(\psi)
    =\Ad(\Pi_1)\boxplus\Ad(\Pi_2),\quad
    \wedge^2_+(\psi)=\Ad(\Pi_1).
  \end{align*}
  Thus the two transfers determine both adjoint parameters.
  For cuspidal $\Pi_1,\Pi_2$, injectivity follows from
  Theorem~\ref{theorem:SO4-simultaneous-twist}.
  The non-generic cases follow directly from their explicit descriptions in Appendix~\ref{section:refinement-SO4}.
  The same descriptions show that both transfers are elliptic and
  orthogonal.
\end{proof}

Let $\Psi_{\mathrm{orth}, \mathrm{ell}}(\PGL_4)$ (resp. $\Psi_{\mathrm{orth}, \mathrm{ell}}(\PGL_3)$) be the set of elliptic orthogonal $A$-parameters of $\PGL_4$ (resp. $\PGL_3$).

The lemma reduces the fibers of
\begin{align*}
  \Psi(\SO_4) \to \Psi(\PGL_7) \colon \psi \mapsto \Std_{\SO_4}(\psi) \boxplus \wedge^2_{+}(\psi)
\end{align*}
to those of
\begin{align}\label{equation:lifting-PGL4-PGL3-PGL7}
  \Psi_{\mathrm{orth}, \mathrm{ell}}(\PGL_4) \times \Psi_{\mathrm{orth}, \mathrm{ell}}(\PGL_3)
  \to
  \Psi(\PGL_7)
  \colon (\phi_4, \phi_3) \mapsto \phi_4 \boxplus \phi_3.
\end{align}

\begin{remark}
   The fiber over $\Std_{\SO_4} \psi^{\SO_4} \boxplus \wedge^2_{+} \psi^{\SO_4}$ is a singleton if $\Std_{\SO_4} \psi^{\SO_4}$ is simple, that is, of the form
   \begin{align*}
      \phi_d \boxtimes S_e^A
   \end{align*}
   for a cuspidal automorphic representation $\phi_d$ of $\GL_d(\A_F)$ and an integer $e \geq 1$ such that $de = 4$.

   If $\psi^{\SO_4}$ is non-generic, so is $\Std_{\SO_4} \psi^{\SO_4} \boxplus \wedge^2_{+} \psi^{\SO_4}$.
\end{remark}

\subsection{The computation of lifting from \texorpdfstring{$\SO_4$}{SO4} to \texorpdfstring{$\rG_2$}{G2}}

\begin{lemma}\label{lemma:SO4-dihedral-case}
    Let $\pi_i = \mathrm{AI}_{K/F}(\chi_i)$ be cuspidal representations of $\GL_2(\A_F)$ with $\omega_{\pi_1} = \omega_{\pi_2}$, where $K/F$ is a common quadratic extension and $\chi_i$ are characters of $K^{\times}\backslash\A_K^{\times}$.
    Let $\sigma$ be the nontrivial element of $\Gal(K/F)$.
    Then the following hold.
    \begin{enumerate}
      \item
      We have
      $(\chi_1/\chi_2) \restriction_{F^{\times} \backslash \A_F^{\times}} = 1$
      and
      $(\chi_1/\chi_2)^{\sigma} = (\chi_1/\chi_2)^{-1}$.
      Moreover,
      \begin{align*}
        \pi_1 \boxtimes \pi_2^{\vee} = \mathrm{AI}_{K/F}(\chi_1/\chi_2) \boxplus \mathrm{AI}_{K/F}(\chi_1/\chi_2^{\sigma})
      \end{align*}
      and
      \begin{align*}
        \Ad(\pi_1) = \mathrm{AI}_{K/F}(\chi_1/\chi_1^{\sigma}) \boxplus \omega_{K/F}.
      \end{align*}
      \item
      The isobaric sum
      \begin{align*}
        \pi_1 \boxtimes \pi_2^{\vee} \boxplus \Ad(\pi_1) =  \mathrm{AI}_{K/F}(\chi_1/\chi_2) \boxplus \mathrm{AI}_{K/F}(\chi_1/\chi_2^{\sigma}) \boxplus \mathrm{AI}_{K/F}(\chi_1/\chi_1^{\sigma}) \boxplus \omega_{K/F}
      \end{align*}
      is a multiplicity-free isobaric sum if and only if $\mathrm{AI}_{K/F}(\chi_2) \neq \mathrm{AI}_{K/F}(\chi_1)$ and $\mathrm{AI}_{K/F}(\chi_2) \neq \mathrm{AI}_{K/F}(\chi_1^2/\chi_1^{\sigma})$.
      \item
      We have
      \begin{align*}
        \mathrm{AI}_{K/F}((\chi_1^{\sigma}/\chi_2)^{-1}) = \mathrm{AI}_{K/F}(\chi_1/\chi_2^{\sigma}).
      \end{align*}
      \item
      We set $\alpha_1 =  \chi_1/\chi_2$, $\alpha_2 = (\chi_1^{\sigma}/\chi_2)^{-1}$, and $\alpha_3 = \chi_1^{\sigma}/\chi_1$.
      Then, we have $\alpha_i \restriction_{F^{\times} \backslash \A_F^{\times}} = 1$ for $i = 1, 2, 3$ and $\alpha_1 \alpha_2 \alpha_3 = 1$.
      \item
      Conversely, if $\alpha_1, \alpha_2, \alpha_3$ are characters of $K^{\times} \backslash \A_K^{\times}$ such that $\alpha_i \restriction_{F^{\times} \backslash \A_F^{\times}} = 1$ for $i = 1, 2, 3$ and $\alpha_1 \alpha_2 \alpha_3 = 1$, then there exist characters $\chi_1, \chi_2$ of $K^{\times} \backslash \A_K^{\times}$ such that $\alpha_1 =  \chi_1/\chi_2$, $\alpha_2 = (\chi_1^{\sigma}/\chi_2)^{-1}$ and $\alpha_3 = \chi_1^{\sigma}/\chi_1$.
    \end{enumerate}
\end{lemma}

\subsection{Fibers in the dihedral case}

\begin{lemma}\label{lemma:SO4-dihedral-case-fiber}
  Retain the notation of Lemma~\ref{lemma:SO4-dihedral-case}, and suppose that $\mathrm{AI}_{K/F}(\chi_2) \neq \mathrm{AI}_{K/F}(\chi_1)$ and $\mathrm{AI}_{K/F}(\chi_2) \neq \mathrm{AI}_{K/F}(\chi_1^2/\chi_1^{\sigma})$.
  \begin{enumerate}
    \item
    There are only three possibilities.
      \begin{enumerate}
        \item
        $\alpha_i$ are not quadratic and $\mathrm{AI}_{K/F}(\alpha_i)$ are cuspidal for $i = 1, 2, 3$.
        \item
        Only one of $\alpha_i$ is quadratic, say $\alpha_1$, and $\mathrm{AI}_{K/F}(\alpha_i)$ is cuspidal for $i = 2, 3$.
        \item
        $\alpha_i$ are all quadratic and $\mathrm{AI}_{K/F}(\alpha_i)$ are all sums of two distinct quadratic characters for $i = 1, 2, 3$.
      \end{enumerate}
    \item
    In Case $(a)$, the fiber of the image of $(\pi_1 \boxtimes \pi_2^{\vee}, \Ad(\pi_1))$ by the map~\ref{equation:lifting-PGL4-PGL3-PGL7} above consists of 3 elements given by
    \begin{align*}
      (\mathrm{AI}_{K/F}(\alpha_i) \boxplus \mathrm{AI}_{K/F}(\alpha_j), \mathrm{AI}_{K/F}(\alpha_k) \boxplus \omega_{K/F})
    \end{align*}
    for $\{i, j, k\} = \{1, 2, 3\}$.
    \item
    In Case $(b)$, the fiber of the image of $(\pi_1 \boxtimes \pi_2^{\vee}, \Ad(\pi_1))$ by the map~\ref{equation:lifting-PGL4-PGL3-PGL7} above consists of 3 elements given similarly to those in Case $(a)$.
    \item
    In Case $(c)$, there exists a $3$-dimensional subspace $V$ in the space $\Hom(F^{\times}\backslash\A_F^{\times}, \mu_2(\C))$ such that the parameter
    \begin{align*}
      \pi_1 \boxtimes \pi_2^{\vee} \boxplus \Ad(\pi_1) \boxplus \mathbf{1}
    \end{align*}
    is the isobaric sum of the elements in $V$.
    The fiber of the image of $(\pi_1 \boxtimes \pi_2^{\vee}, \Ad(\pi_1))$ by the map~\ref{equation:lifting-PGL4-PGL3-PGL7} above is in bijection with the set of $2$-dimensional subspaces of $V$ and consists of 7 elements.
    The bijection is given by
    \begin{align*}
       W \mapsto (\boxplus_{\alpha \in V \setminus W} \alpha, \boxplus_{\beta \in W \setminus \{\mathbf{1}\}} \beta).
    \end{align*}
  \end{enumerate}
\end{lemma}

\begin{proof}
  \begin{enumerate}
    \item
    This follows from $\alpha_1 \alpha_2 \alpha_3 = 1$.
    \item
    This follows from Lemma~\ref{lemma:SO4-dihedral-case} (5).
    \item
    Argue as in the preceding case.
    \item
    Choose quadratic characters $\beta_1,\beta_2$ of $F^{\times}\backslash\A_F^{\times}$ with $\alpha_i = \beta_i \circ \mathrm{Nm}_{K/F}$ for $i=1,2$. Then
    \begin{align*}
      \pi_1 \boxtimes \pi_2^{\vee} =  (\beta_1 \boxplus \beta_1 \otimes \omega_{K/F}) \boxplus (\beta_2 \boxplus \beta_2 \otimes \omega_{K/F}),
    \end{align*}
    and
    \begin{align*}
      \Ad(\pi_1) = \beta_1\beta_2 \boxplus \beta_1\beta_2 \otimes \omega_{K/F} \boxplus \omega_{K/F}.
    \end{align*}
    Thus $(\pi_1 \boxtimes \pi_2^{\vee}, \Ad(\pi_1))$ has the stated form for $W$ generated by $\beta_1\beta_2$ and $\omega_{K/F}$.
    If the characters in $W$ are given by
    \begin{align*}
      \Ad(\pi) \boxplus \mathbf{1},
    \end{align*}
    and $V/W$ is generated by $\gamma$, the corresponding pair is
    \begin{align*}
      (\pi \boxtimes (\pi^{\vee} \otimes \gamma), \Ad(\pi)).
    \end{align*}
  \end{enumerate}
\end{proof}

\begin{lemma}\label{lemma:SO4-not-multiplicity-free-case}
  Let us continue to use the notation in Lemma~\ref{lemma:SO4-dihedral-case}.
  \begin{enumerate}
    \item Assume that $\mathrm{AI}_{K/F}(\chi_2) = \mathrm{AI}_{K/F}(\chi_1)$.
    Then, the fiber of the image of $(\pi_1 \boxtimes \pi_2^{\vee} , \Ad(\pi_1))$ by the map~\ref{equation:lifting-PGL4-PGL3-PGL7} above is a singleton.
    \item Assume that $\mathrm{AI}_{K/F}(\chi_2) = \mathrm{AI}_{K/F}(\chi_1^2/\chi_1^{\sigma})$.
    Then, we must have
    \begin{align*}
      (\chi_1/\chi_1^{\sigma})^3 \neq \mathbf{1},
    \end{align*}
    and the fiber of the image of $(\pi_1 \boxtimes \pi_2^{\vee}, \Ad(\pi_1))$ by the map~\ref{equation:lifting-PGL4-PGL3-PGL7} above is a singleton.
  \end{enumerate}
\end{lemma}

\begin{proof}
  \begin{enumerate}
    \item
    If we have $\chi_2 = \chi_1$, then
    \begin{align*}
      \pi_1 \boxtimes \pi_2^{\vee} 
      \boxplus \Ad(\pi_1) &= \\ \mathrm{AI}_{K/F}(\chi_1/\chi_1^{\sigma}) 
      \boxplus &\mathrm{AI}_{K/F}(\chi_1/\chi_1^{\sigma}) 
      \boxplus \omega_{K/F} 
      \boxplus \omega_{K/F} \boxplus \mathbf{1}.
    \end{align*}
    If $(\phi_4, \phi_3)$ is a pair in the fiber, then $\phi_4, \phi_3$ must be multiplicity-free and they share a common isobaric summand.
    Thus $\phi_4$ must be equal to
    \begin{align*}
      \mathrm{AI}_{K/F}(\chi_1/\chi_1^{\sigma}) 
      \boxplus \omega_{K/F} 
      \boxplus \mathbf{1}.
    \end{align*}
    This implies that $\phi_3$ must be equal to
    \[
      \mathrm{AI}_{K/F}(\chi_1/\chi_1^{\sigma}) 
      \boxplus \omega_{K/F},
    \]
    proving this case. The case $\chi_2 = \chi_1^{\sigma}$ is similar.
    \item
    If we have $\chi_2 = \chi_1^2/\chi_1^{\sigma}$, then
    \begin{align*}
      \pi_1 \boxtimes \pi_2^{\vee} \boxplus \Ad(\pi_1) 
      =  
      \mathrm{AI}_{K/F}(\chi_1/\chi_1^{\sigma}) 
      \boxplus \mathrm{AI}_{K/F}(\chi_1/\chi_1^{\sigma}) 
      \boxplus \mathrm{AI}_{K/F}((\chi_1/\chi_1^{\sigma})^2) 
      \boxplus \omega_{K/F}.
    \end{align*}
    As in (1), this forces $\phi_4 = \mathrm{AI}_{K/F}(\chi_1/\chi_1^{\sigma}) \boxplus \mathrm{AI}_{K/F}((\chi_1/\chi_1^{\sigma})^2)$ and $\phi_3 = \mathrm{AI}_{K/F}(\chi_1/\chi_1^{\sigma}) \boxplus \omega_{K/F}$.
    The case $\chi_2 = (\chi_1^2/\chi_1^{\sigma})^{\sigma}$ is similar.
  \end{enumerate}
\end{proof}

The following results are the local analogues of Lemmas~\ref{lemma:SO4-dihedral-case-fiber} and~\ref{lemma:SO4-not-multiplicity-free-case}.

\begin{lemma}\label{lemma:SO4-dihedral-case-local}
  Let $F$ be a local field of characteristic zero and let $\pi_i = \mathrm{AI}_{K/F}(\chi_i)$ be irreducible discrete series representations of $\GL_2(F)$ with $\omega_{\pi_1} \omega_{\pi_2} = 1$, where $K/F$ is a common quadratic extension and $\chi_i$ are characters of $K^{\times}$.
  Let $\sigma$ be the nontrivial element of $\Gal(K/F)$.
  In the equal-central-character convention of Section~\ref{section:refined-LLC-SO4}, the pair used here corresponds to $(\pi_1,\pi_2^{\vee})$, and the standard representation of its $L$-parameter for $\SO_4$ is $\phi_{\pi_1}\otimes\phi_{\pi_2}$.

  \begin{enumerate}
  \item
  If $\pi_1 = \pi_2^{\vee}$, then we have
  \begin{align*}
    \phi_{\pi_1} \otimes \phi_{\pi_2} \oplus \Ad \phi_{\pi_1}
    = \mathrm{AI}_{K/F}(\chi_1/\chi_1^{\sigma})
    \oplus \mathrm{AI}_{K/F}(\chi_1/\chi_1^{\sigma})
    \oplus \omega_{K/F}
    \oplus \omega_{K/F} \oplus \mathbf{1}.
  \end{align*}
  The fiber of the map
  \begin{align*}
    \Phi_{\disc}(\SO_4) \to \Phi(\GL_7)
    \colon
    \phi
    \mapsto
    \mathrm{Std}_{\SO_4} \circ \phi
    \oplus \wedge^2_+ \circ \phi
  \end{align*}
  is a singleton $\phi_{(\pi_1, \pi_2^{\vee})}$ which is given by the class of $(\pi_1, \pi_2^{\vee})$ in the equal-central-character convention.

  If $\pi_1$ runs through the representations above, then the parameter $\phi_{\pi_1} \otimes \phi_{\pi_2} \oplus \Ad(\phi_{\pi_1})$ runs through the sums of the form
  \begin{align*}
     \phi_{\tau}
     \oplus \phi_{\tau}
     \oplus \omega_{K/F}
     \oplus \omega_{K/F}
     \oplus \mathbf{1},
  \end{align*}
  for $\tau$ dihedral with respect to $K/F$ and self-dual.

  \item
  If $\chi_2 = \chi_1^{\sigma}/\chi_1^2$, then we have
  \begin{align*}
    \phi_{\pi_1} \otimes \phi_{\pi_2} \oplus \Ad \phi_{\pi_1}
    =
    \mathrm{AI}_{K/F}(\chi_1^{\sigma}/\chi_1)
    \oplus \mathrm{AI}_{K/F}(\chi_1^{\sigma}/\chi_1)
    \oplus \mathrm{AI}_{K/F}(\chi_1^2/(\chi_1^{\sigma})^2)
    \oplus \omega_{K/F}.
  \end{align*}

  If $\pi_1, \pi_2$ run through the representations above, then the parameter $\phi_{\pi_1} \otimes \phi_{\pi_2} \oplus \Ad(\phi_{\pi_1})$ runs through the sums of the form
  \begin{align*}
     \phi_{\tau} \oplus \phi_{\tau} \oplus \Ad(\phi_{\tau})
  \end{align*}
  for $\tau = \mathrm{AI}_{K/F}(\chi)$ dihedral with respect to $K/F$ and self-dual with $\chi^3 \neq 1$.
  \end{enumerate}
\end{lemma}

\begin{proof}
  The displayed decompositions follow from the tensor-product formula for quadratic induction. For the fiber assertion in (1), write $\xi=\chi_1/\chi_1^\sigma$.
  Discreteness of an $L$-parameter for $\SO_4$ implies that both its standard representation and its $\wedge^2_+$ representation are multiplicity-free.
  If $\mathrm{AI}_{K/F}(\xi)$ is irreducible, the repeated constituents in $\mathbf1\oplus\omega_{K/F}^{\oplus2}\oplus\mathrm{AI}_{K/F}(\xi)^{\oplus2}$ therefore force the standard representation to be $\mathbf1\oplus\omega_{K/F}\oplus\mathrm{AI}_{K/F}(\xi)$ and the $\wedge^2_+$ representation to be $\omega_{K/F}\oplus\mathrm{AI}_{K/F}(\xi)$.
  If $\mathrm{AI}_{K/F}(\xi)=a\oplus b$ is reducible, then $a,b$ are distinct nontrivial quadratic characters with $ab=\omega_{K/F}$. The same argument forces the standard representation to be $\mathbf1\oplus a\oplus b\oplus ab$ and the $\wedge^2_+$ representation to be $a\oplus b\oplus ab$.
  In both cases the standard representation contains a trivial line. Reflection in that line centralizes the parameter and has determinant $-1$, so its $\O_4$-conjugacy class is a single $\SO_4$-conjugacy class.
\end{proof}

\section{On the theta lifting from \texorpdfstring{$\PGL_3^{+}$}{PGL3+} to \texorpdfstring{$\rG_2$}{G2}}\label{section:theta-PGL3-to-G2}

\subsection{Notation for the theta lifts}\label{theta-common-notation}

The following conventions apply to this section, Section~\ref{section:theta-G2-to-A2}, and Appendix~\ref{section:theta-A2-real-local}.
Let $F$ be a number field. For an algebraic group $H$ over $F$, put $[H]=H(F)\backslash H(\A_F)$ and write $\cA(H)$ for the smooth automorphic forms with the usual finiteness conditions.
Fix a nontrivial unitary character $\psi:F\backslash\A_F\longrightarrow\C^\times$ and compatible Haar measures. At archimedean places use smooth moderate-growth globalizations. All local Fourier functionals are continuous.
For a unipotent subgroup $V$ and a character $\chi$ of $[V]$, put
\[
 f_{V,\chi}(g)=\int_{[V]}f(vg)\overline{\chi(v)}\rd v.
\]

We recall the notation of \cite{GanSavin2003-RealGlobalLiftsPGL3G2}*{Section 4}.
Let $\cH = G_{E_6}$ be the split adjoint group of type $E_6$, with extended Dynkin diagram
\[
\begin{tikzcd}
	&& {-\widetilde{\beta}} \\
	&& {\beta_2} \\
	{\beta_1} & {\beta_3} & {\beta_4} & {\beta_5} & {\beta_6}
	\arrow[no head, from=1-3, to=2-3]
	\arrow[no head, from=2-3, to=3-3]
	\arrow[no head, from=3-2, to=3-1]
	\arrow[no head, from=3-3, to=3-2]
	\arrow[no head, from=3-4, to=3-3]
	\arrow[no head, from=3-5, to=3-4]
\end{tikzcd}.
\]
The dual pair $\PGL_3 \times G_2$ is obtained by embedding $G_2$ as the triality-fixed subgroup of the standard $D_4$ Levi subgroup $\cH_M$ of $\cH$.
The relative roots for $\cH_M$ are the restrictions of $\{\pm\beta_1,\pm\beta_6,\pm\widetilde{\beta}\}$, giving three $8$-dimensional root subgroups that afford the $8$-dimensional representations of $\cH_M^{\der}\cong\Spin_8$.
The centralizer of $G_2$ is consequently $8$-dimensional and isomorphic to $\PGL_3$.
By symmetry, the outer automorphism of $E_6$ preserves both subgroups, acting trivially on $G_2$ and by the non-trivial pinning-preserving outer automorphism on $\PGL_3$.

The root system of the maximal torus $M \subset G_2$ in $\cH$ is the root system of type $G_2$.
The long root space is $1$-dimensional, and the short root space is $9$-dimensional and has the natural structure of the matrix algebra $J = M_{3}(F)$.
We also use its associated Jordan algebra $M_3(F)^+$, with product $x\circ y=(xy+yx)/2$, trace form $(x,y)\mapsto\tr(xy)$, norm $N_J(x)=\det x$, and adjugate $x^\#$. Put $x\times y=(x+y)^\#-x^\#-y^\#$, over $F$ and its completions.
For each such root $\delta$, we set
\begin{align*}
    \mathcal{N}_{\delta}
    =
    \begin{cases}
        \bG_{a} &\quad \text{$\delta$ is long}, \\
        J &\quad \text{$\delta$ is short}.
    \end{cases}
\end{align*}
The group $\mathcal{N}_{\delta} \cap G_2$ is the root group $N_{\delta}$ in $G_2$.

Let $\alpha$ (resp. $\beta$) be the short  simple root (resp. the long simple root) of $G_2$.
Corresponding to the maximal parabolic subgroups in $G_2$, we have the following two parabolic subgroups of $\cH$.
\begin{itemize}
    \item
    $\cP_1 = \cM_1 \cN_1$ containing $\cN_{\alpha}$.
    \item
    $\cP_2 = \cM_2 \cN_2$ containing $\cN_{\beta}$.
\end{itemize}
Both groups contain $\PGL_3$.

The hom functor $\cV_2 = \Hom(\cN_2,\bG_a)$ identifies with the abelianization of $\Lie(\overline{\cN_2})$.
\begin{align*}
    \cV_2
    =
    F_{-\beta} \oplus J_{-\alpha-\beta} \oplus J_{-2\alpha-\beta} \oplus F_{-3\alpha-\beta}.
\end{align*}
Use coordinates $(a,x,y,d)$ with $a,d\in\bG_a$ and $x,y\in J$.

The differential of the $\cN_{\alpha}$-action on $\cV_2$ is given, for $z\in\Lie(\cN_{\alpha})\cong J$, by
\begin{align*}
    z \cdot (a, x, y, d) = (\tr(xz), y \times z, dz, 0).
\end{align*}
Also, the inclusion $N_{2} \hookrightarrow \cN_2$ induces the projection $\cV_2 \twoheadrightarrow V_2$, which can be written as
\begin{align*}
    (a, x, y, d) \mapsto (a, \tr(x), \tr(y), d).
\end{align*}
The center $\cZ$ of $\cN_2$ equals the center of $N_2$.

\subsection{The proof of Theorem \ref{theorem:theta-PGL3-to-G2}}\label{theta-PGL3-to-G2-proof}

Let $\Omega^+$ be the extension of the global minimal representation to
$\cH^+(\A_F)$, where
\begin{align*}
  \cH^{+}=\cH\rtimes\Z/2\Z.
\end{align*}
The residual construction of \cite{GinzburgRallisSoudry1997-AutomorphicThetaSimplyLacedGroups}*{Theorem 4.3}, with the extension in \cite{GanSavin2022-TwistedCompositionAlgebrasArthurPacketsTrialitySpin8}*{Section 14.3}, gives an automorphic realization
\begin{align*}
  \vartheta \colon \Omega^+ \longrightarrow \cA(\cH)
\end{align*}
which is equivariant under $\cH^+(F)\cdot\cH(\A_F)$.
Let $E_{ij}$ denote the matrix units and take
\begin{align*}
  x_0=E_{12}-E_{23},\quad x_0^{\#}=-E_{13},\quad
  X=(1,x_0,x_0^{\#},0)\in\cV_2.
\end{align*}
The Jordan involution corresponding to $\iota$ fixes $x_0$ and $x_0^{\#}$, so the outer involution fixes $X$.
By \cite{GanSavin2003-RealGlobalLiftsPGL3G2}*{Proposition 5.2 and Section 7, p.2718}, the functional
$L_X(\phi)=\vartheta(\phi)_X(1)$ is nonzero and invariant under the adelic derived subgroup of the stabilizer of $X$ in $\cM_2$.
The local uniqueness results of \cite{Gan2008-SiegelWeilAutomorphicCharactersSnitz}*{Sections 3.1 and 3.5} therefore give $L_X=\bigotimes_v\ell_{X,v}$, as in \cite{Gan2008-SiegelWeilAutomorphicCharactersSnitz}*{Section 3.7}.
Since the local involutions preserve the lines $\C\ell_{X,v}$, we may normalize them to fix $\ell_{X,v}$.
The normalizing signs have product $1$, since $L_X$ is invariant under the rational involution.
Thus $\vartheta$ remains $\cH^+(F)\cdot\cH(\A_F)$-equivariant, and $L_X$ is invariant under the adelic component group.
At each non-archimedean place, this agrees with \cite{GanSavin2022-TwistedCompositionAlgebrasArthurPacketsTrialitySpin8}*{Section 8.2} and \cite{GanSavin2023-HoweDualityDichotomyExceptionalThetaCorrespondences}*{Lemma 8.1}.

Let $\Pi$ be a cuspidal automorphic representation of $\PGL_3(\A_F)$ and write $\Pi^+=\Pi^{+}((+1)_v)$ for the corresponding cuspidal automorphic representation of $\PGL_3^{+}(\A_F)$ with signs $(+1)_v$.
We use the convention of \cite{GanSavin2003-RealGlobalLiftsPGL3G2}*{Section 7}, which replaces $\Pi^{+}$ by ${\Pi^{+}}^{\vee}$ without affecting the result.
For $\phi\in\Omega^+$ and $f^+\in\Pi^+$, define
\begin{align*}
    \Theta(\phi, f^+)(g)
    =
    \int_{\PGL_3^{+}(F) \backslash \PGL_3^{+}(\A_F)}
    \vartheta(h\cdot\phi)(g)f^+(h)\,dh,
\end{align*}
for $g \in G_2(\A_F)$.

Let $U$ be the $\Z/2\Z$-stable maximal unipotent subgroup of $\PGL_3$ and $N$ the maximal unipotent subgroup of $G_2$.
Normalize the generic character $\psi_N$ as in \cite{GanSavin2003-RealGlobalLiftsPGL3G2}*{p.2717, line 18}.
With the above choice of $x_0$ and the root-group coordinates of \cite{GanSavin2003-RealGlobalLiftsPGL3G2}*{Lemma 6.1}, the corresponding generic character is $\psi_U(u)=\psi(u_{12}+u_{23})$ for $u\in U(\A_F)$.
It is invariant under the pinned involution $\iota$.
Set
\begin{align*}
  U^+=U\rtimes\Z/2\Z,\quad
  \psi_U^+(u\eta)=\psi_U(u),\quad
  C_F=\Delta\{\pm1\}\backslash\prod_v\{\pm1\},
\end{align*}
where $C_F$ is the quotient of the adelic component group by the diagonal
rational component, with probability Haar measure. Choose a measurable
section $\eta\mapsto\dot\eta$ in the adelic component group.
With compatible quotient measures, define
\begin{align*}
  W_{f^+}^+(h)
  &=\int_{U^+(F)\backslash U^+(\A_F)}
    f^+(xh)\overline{\psi_U^+(x)}\,dx \\
  &=\int_{C_F}f^+_{U,\psi_U}(\dot\eta h)\,d\eta.
\end{align*}
The quotient in the first integral is compact, and a change of variables gives
\begin{align*}
  W_{f^+}^+(a h)=\psi_U^+(a)W_{f^+}^+(h)
  \quad (a\in U^+(\A_F)).
\end{align*}

We first verify that this coefficient is non-zero on $\Pi^+$.
Choose a factorizable cusp form $f=\otimes_v f_v\in\Pi$ whose ordinary
Whittaker coefficient satisfies $W_f(1)\neq0$.
At every self-dual place, let $A_v$ be the Whittaker-normalized involution
on $\Pi_v$, and replace $f_v$ by $(f_v+A_vf_v)/2$.
This preserves its Whittaker value. Only finitely many of these projections
change the vector, since the spherical vector is fixed outside a finite set.
The construction in Proposition~\ref{proposition:cuspidal-PGL3-plus},
as in \cite{HarrisSoudryTaylor1993-AdicRepresentationsAssociatedModularFormsImaginary}*{Proposition 2},
then gives the cusp form
\begin{align*}
  F^+(g\eta)=f(g)+f(\iota(g))
  \quad
  (g\in\PGL_3(\A_F),\ \eta\in(\Z/2\Z)(\A_F))
\end{align*}
in $\Pi^+$.
Indeed, the sum is invariant under the rational involution, and the
local vectors belong to the positive extension at each self-dual place
and to the induced representation at each non-self-dual place.
Since $\iota$ preserves $U$ and $\psi_U$, we obtain
\begin{align*}
  W_{F^+}^+(1)=W_f(1)+W_{f\circ\iota}(1)=2W_f(1)\neq0.
\end{align*}
Local Whittaker uniqueness and Clifford theory give
\begin{align*}
  \dim\Hom_{U^+(F_v)}(\Pi_v^+,\psi_{U,v}^+)=1.
\end{align*}
At a non-self-dual place, this is the invariant line in the
two-dimensional space of ordinary Whittaker functionals.
Thus the non-zero global functional $f^+\mapsto W_{f^+}^+(1)$ factors
into these local functionals.

The computation of
\cite{GanSavin2003-RealGlobalLiftsPGL3G2}*{Section 7} works up to
\cite{GanSavin2003-RealGlobalLiftsPGL3G2}*{p.2718, line 7}, with the
additional integration over $C_F$.
Namely, we have
\begin{align*}
    \Theta(\phi, f^+)_{N, \psi_N}(1)
    =
    \int_{U(\A_F)\backslash \PGL_3(\A_F)} \int_{C_F}
    f^+_{U, \psi_U}(h\dot\eta)
    K_{\phi}(h\dot\eta)\,d\eta\,dh,
\end{align*}
where
\begin{align*}
  K_{\phi}(h\dot\eta)
  =
  \int_{N_{\alpha}(\A_F)}
  \psi(n)\vartheta(h\dot\eta\cdot\phi)_X(n)\,dn.
\end{align*}

The normalization of the rank-one Fourier functional on $\Omega^+$,
and the fact that the component group commutes with $N_\alpha$, give
\begin{align*}
  K_{\phi}(\dot\eta h)=K_{\phi}(h).
\end{align*}
For each $\eta$, make the change of variables
$h=\dot\eta h'\dot\eta^{-1}$ in the unfolded integral.
Conjugation by $\dot\eta$ preserves $U$, $\psi_U$, and the Haar measure.
Consequently,
\begin{align*}
  \Theta(\phi, f^+)_{N,\psi_N}(1)
  &=\int_{U(\A_F)\backslash\PGL_3(\A_F)}\int_{C_F}
    f^+_{U,\psi_U}(\dot\eta h)K_{\phi}(h)\,d\eta\,dh \\
  &=\int_{U(\A_F)\backslash\PGL_3(\A_F)}
    W_{f^+}^+(h)K_{\phi}(h)\,dh.
\end{align*}
Choose factorizable $f^+$ and $\phi$ such that
$W_{f^+}^+(1)\vartheta(\phi)_X(1)\neq0$.
Apply \cite{GanSavin2003-RealGlobalLiftsPGL3G2}*{Section 7, pp.2718--2719} to $W_{f^+}^+$.
The unramified support calculation reduces the integral to a finite set of places containing the archimedean ones. Over this product of local fields, a suitable Schwartz convolution of $\phi$ localizes the integral near $x_0$ in the locally closed rank-two orbit and makes it nonzero.
Cuspidality follows from \cite{GanSavin2003-RealGlobalLiftsPGL3G2}*{Theorem 5.1(i)}.

\section{Theta lifting from \texorpdfstring{$\rG_2$}{G2} to \texorpdfstring{$\PGL_3^{+}$}{PGL3+} and \texorpdfstring{$\mathrm{PD}^{\times}$}{PDtimes}}\label{section:theta-G2-to-A2}

Retain the conventions of Section~\ref{theta-common-notation}, assume $F$ totally real, and put $G=\rG_2$, the split group over $F$.

For a central simple algebra $D$ of degree $3$, let $D^+$ be its associated Jordan algebra. Put $G'_J=\Aut(J)=\PGL_3^+=\PGL_3\rtimes\Z/2\Z$ for $J=M_3(F)^+$ and $G'_J=\Aut(J)^\circ=\mathrm{PD}^\times$ for $J=D^+$ with $D$ division.
Let $\Omega_J$ be the global minimal representation used for the exceptional dual pair $G\times\Aut(J)$ in the corresponding group of type $E_6$, with automorphic realization $\vartheta_J$. In the split case, choose $\Omega_J=\Omega^+$ and $\vartheta_J=\vartheta$ with the normalization of Section~\ref{theta-PGL3-to-G2-proof}. For $f\in\cA_{\cusp}(G)$, $\phi\in\Omega_J$, and $h\in G'_J(\A_F)$, set
\begin{equation}\label{theta-A2:eq:backward-lift}
 \theta_J(\phi,f)(h)=\int_{[G]}\vartheta_J(h\phi)(g)\overline{f(g)}\rd g.
\end{equation}
The integral converges by rapid decrease of cusp forms and moderate growth of the minimal automorphic forms. Let $\Theta_J(\sigma)$ be the span of these functions as $f$ varies in $\sigma$ and $\phi$ varies in $\Omega_J$.

\begin{theorem}\label{theorem:exceptional-theta-G2-to-PGL3}
Let $\sigma\subset\cA_{\cusp}(G)$ be an irreducible cuspidal automorphic representation. Suppose that $L^S(s,\sigma,\Std_G)$ has a pole at $s=1$ for all sufficiently large finite sets $S$ of places of $F$. Then $\Theta_J(\sigma)\ne0$ for $J=M_3(F)^+$ or for $J=D^+$ with $D$ a central division algebra of degree $3$.
\end{theorem}

Following \cite{GanSavin2022-TwistedCompositionAlgebrasArthurPacketsTrialitySpin8}, we compare Segal's Eisenstein residues with theta lifts of the trivial representation of a torus normalizer.

\subsection{The integral representation and its residue}\label{theta-A2:sec:residue}

In this section, $N$ denotes the Heisenberg radical of $G$. Its nondegenerate characters, up to conjugation by its Levi subgroup, are parametrized by \'{e}tale cubic algebras $E$ over $F$. Fix a corresponding character $\psi_E$ of $[N]$.
We say that $f$ has a nonzero $E$-th Fourier coefficient if $f_{N,\psi_E}\ne0$. Write $G_E=\Spin_8^E$ for the quasi-split group attached to $E$, and let $P_E=M_EN_E$ be its Heisenberg parabolic. Thus
\[
 M_E\cong\{g\in\Res_{E/F}\GL_2:\det(g)\in\mathbf G_m\},
 \quad M_E^{\der}\cong\Res_{E/F}\SL_2.
\]
The common determinant defines a character $\det:M_E\to\mathbf G_m$, and $\delta_{P_E}=|\det|^5$. Put $I_E(s)=I_{P_E(\A_F)}^{G_E(\A_F)}(\abs{\det}^s)$, with normalized induction, so that sections transform by $|\det|^{s+5/2}$. Let $\cE(s,f_s)$ denote Segal's normalized Eisenstein series, with the unramified normalization of \cite{Segal2017-NewWayIntegralsStandardLFunctionG2}*{Section~2}. Write
\[
 \cR_E=\operatorname{span}_{\C}\left\{\Res_{s=1/2}\cE(s,f_s):f_s\text{ a standard section}\right\}\subset\cA(G_E).
\]

\begin{proposition}\label{theorem:integral-representation-standard-G2}
Suppose that the $E$-th Fourier coefficient of $\sigma$ is nonzero. For a sufficiently large finite set $S$ and factorizable data, one has
\begin{equation}\label{theta-A2:eq:segal-integral}
 \int_{[G]}\cE(s,f_s)(g)\overline{f(g)}\rd g
 =L^S(s+\tfrac12,\sigma^\vee,\Std_G)\,d_S(s,f_s,f).
\end{equation}
The data at $S$ can be chosen so that $d_S(s,f_s,f)$ is holomorphic and nonzero at $s=1/2$.
\end{proposition}
\begin{proof}
This is \cite{Segal2017-NewWayIntegralsStandardLFunctionG2}*{Theorem 3.1}. Since $\Std_G$ is self-dual, the partial $L$-function in \eqref{theta-A2:eq:segal-integral} equals $L^S(s+\tfrac12,\sigma,\Std_G)$.
\end{proof}

By \cite{Gan2005-MultiplicityFormulaCubicUnipotentArthurPackets}*{Theorem 3.1}, fix an \'{e}tale cubic algebra $E$ for which $\sigma$ has a nonzero $E$-th Fourier coefficient. The Eisenstein series in \eqref{theta-A2:eq:segal-integral} has at most a simple pole at $1/2$. See \cite{Segal2019-DegenerateResidualSpectrumQuasiSplitFormsSpin8} and the normalization in \cite{Segal2017-NewWayIntegralsStandardLFunctionG2}. Choosing $d_S(1/2,f_s,f)\ne0$, we obtain a simple pole of the $L$-function and
\begin{equation}\label{theta-A2:eq:residue-period}
 \int_{[G]} R(g)\overline{f(g)}\rd g\ne0,
 \quad R=\Res_{s=1/2}\cE(s,f_s)\in\cR_E.
\end{equation}
After removal of the pole, the Eisenstein family has uniform moderate growth. The rapid decrease of $f$ therefore allows us to take the residue under the integral.

\begin{proposition}[Segal]\label{theta-A2:prop:residual-data}
The space $\cR_E$ is semisimple. If $E=F\times K$, where $K/F$ is quadratic \'{e}tale, then
\begin{equation}\label{theta-A2:eq:nonfield-residue}
 \cR_E\cong\pi_E:=\bigotimes_v'\pi_{1,v}.
\end{equation}
Here $\pi_{1,v}$ is the spherical Langlands quotient of $I_{E_v}(1/2)$.
\end{proposition}
\begin{proof}
The nonfield assertion follows from \cite{Segal2019-DegenerateResidualSpectrumQuasiSplitFormsSpin8}*{Proposition 6.4 and (6.2)}, which realizes the residue at the unitary point of an Eisenstein series induced from a Levi subgroup of type $A_2$. Semisimplicity, including the field case, is proved in \cite{Segal2019-DegenerateResidualSpectrumQuasiSplitFormsSpin8}*{Sections 5--6}. The normalizing factor in \cite{Segal2017-NewWayIntegralsStandardLFunctionG2}*{Section~2.8} satisfies $j_E^S(\tfrac12)=\zeta_F^S(3)\zeta_E^S(2)\ne0$ and is holomorphic there, also for $E=F^3$. Here $\zeta_E^S$ is the product of the partial zeta functions of the field factors of $E$. This normalization therefore does not change the residue space.
\end{proof}

\subsection{Twisted composition algebras and local theta lifts}\label{theta-A2:sec:local}

A rank-$2$ $E$-twisted composition algebra is a free $E$-module $C$ of rank $2$, equipped with a nondegenerate $E$-valued quadratic form $Q$ and a quadratic map $\beta_C:C\to C$ such that
\[
 \beta_C(ev)=e^\#\beta_C(v),\quad Q(\beta_C(v))=Q(v)^\#,
 \quad N_C(v):=b_Q(v,\beta_C(v))\in F.
\]
Here $b_Q(v,w)=Q(v+w)-Q(v)-Q(w)$, and $e^\#$ is the quadratic adjoint of $E$. An isomorphism preserves $Q$ and $\beta_C$. See \cite{GanSavin2022-TwistedCompositionAlgebrasArthurPacketsTrialitySpin8}*{Section 4.1}. We consider the algebras obtained from embeddings $E\hookrightarrow D^+$, where $D$ is a central simple algebra of degree $3$.

An embedding $E\hookrightarrow J$ determines a rank-$2$ $E$-twisted composition algebra $C$, a group $H_C=\Aut_E(C)$, and the dual pair
\[
 H_C\times G_E\longrightarrow\widetilde G_J,
\]
where $\widetilde G_J$ is of type $E_6$, with its component group when required. Put $T_C=H_C^\circ$. We have the see-saw inclusions $G\subset G_E$ and $H_C\subset\Aut(J)$. We write $\theta_{C_v}(\mathbf1)$ for the local theta lift of the trivial representation of $H_C(F_v)$.

We use the following local statements, as in \cite{GanSavin2022-TwistedCompositionAlgebrasArthurPacketsTrialitySpin8}*{Sections 12--13 and (16.7)}, with the real Fourier assertion proved in Appendix~\ref{section:theta-A2-real-local}.
\begin{enumerate}[label=(\roman*)]
\item The nonzero local lifts $\theta_{C_v}(\mathbf1)$ are irreducible and give the relevant constituents of the maximal semisimple quotient of $I_{E_v}(1/2)$.
\item For a nondegenerate $E_v$-twisted cube $\Sigma_v$, with character $\psi_{\Sigma_v}$ of $N_E(F_v)$, one has
\begin{equation}\label{theta-A2:eq:local-model}
 \Hom_{N_E(F_v)}\bigl(\theta_{C_v}(\mathbf1),\psi_{\Sigma_v}\bigr)
 \cong
 \begin{cases}
 \C,&C_{\Sigma_v}\cong C_v,\\
 0,&C_{\Sigma_v}\not\cong C_v.
 \end{cases}
\end{equation}
In the present case of inner forms of type $A_2$ the quadratic algebra denoted $K$ in \cite{GanSavin2022-TwistedCompositionAlgebrasArthurPacketsTrialitySpin8} is split. Thus the stabilizer twist $\mu_K$ in \cite{GanSavin2022-TwistedCompositionAlgebrasArthurPacketsTrialitySpin8}*{(16.7)} is trivial.
\end{enumerate}
The nonarchimedean assertion in (ii) is \cite{GanSavin2022-TwistedCompositionAlgebrasArthurPacketsTrialitySpin8}*{Proposition 12.3}. At a real place, the Harish-Chandra-module identification is given by \cite{GanLokeEtAl2025-FamilySpinEightDualPairsRealGroups}*{Theorem 3}, and the continuous Fourier-model formula is proved in Proposition~\ref{theta-A2:prop:real-fourier-model}. Every central simple algebra of degree $3$ over $\R$ is split, so this covers all real components used here.

If $E$ is not a field, its embedding into $M_3(F)^+$ is unique up to $\PGL_3(F)$-conjugacy. Denote the corresponding twisted composition algebra by $C_0$. Then
\begin{equation}\label{theta-A2:eq:pi-theta-local}
 \pi_{1,v}\cong\theta_{C_{0,v}}(\mathbf1),
 \quad T_{C_0}\cong\Res_{E/F}\mathbf G_m/\mathbf G_m.
\end{equation}
No cubic division algebra can contain a nonfield \'{e}tale cubic algebra, since its nontrivial idempotents would give zero divisors. By the Brauer--Hasse--Noether theorem, a central simple algebra of degree $3$ which is split at every place is split over $F$. Skolem--Noether then gives conjugacy of its embeddings of $E$.

\subsection{Nondegenerate Fourier coefficients and multiplicity one}\label{theta-A2:sec:detection}

Suppose that $E=F^3$ or $E=F\times K$, where $K/F$ is a quadratic field. Put $\fg=\Lie(G_E)$, $\fm_E=\Lie(M_E)$, and $\fn_E=\Lie(N_E)$.

We use the Bourbaki numbering of $D_4$, with $\alpha_2$ the central vertex, and use $\beta$ for the highest root.
\begin{equation}\label{theta-A2:eq:D4-roots}
 \begin{gathered}
 \alpha_1=e_1-e_2,\quad\alpha_2=e_2-e_3,\quad
 \alpha_3=e_3-e_4,\quad\alpha_4=e_3+e_4,\\
 \beta=e_1+e_2.
 \end{gathered}
\end{equation}
Choose compatible root vectors $e_\alpha$. Thus $Z_E=Z(N_E)=U_\beta$, and, over $\overline F$, $N_E/Z_E$ is the tensor product of the standard two-dimensional representations of the three $\SL_2$-factors corresponding to $\alpha_1,\alpha_3,\alpha_4$. If $E=F\times K$, choose the labeling so that Galois fixes $\alpha_1$ and interchanges $\alpha_3,\alpha_4$.

For a Harish-Chandra module $V$, let $\Ann V$ denote its annihilator in the universal enveloping algebra $\cU(\fg_\C)$. We write $\AV(\Ann V)\subset\fg_\C^*$ for the zero set of its associated graded ideal under the standard filtration. We identify $\fg_\C^*$ with $\fg_\C$ by an invariant bilinear form. Let $\cO_{A_2}$ be the nilpotent orbit of type $A_2$ in $\mathfrak{so}_8(\C)$, with Jordan partition $(3,3,1,1)$.

\begin{proposition}\label{theta-A2:prop:detection}
Suppose that $E=F^3$ or $F\times K$, with $K/F$ a quadratic field. Let $\Pi\ne\mathbf1$ be an irreducible automorphic subrepresentation of $G_E(\A_F)$. Suppose that, at a real place $w$,
\begin{equation}\label{theta-A2:eq:A2-annihilator}
 \overline{\cO_{A_2}}\subset\AV(\Ann\Pi_w).
\end{equation}
Then $\Pi$ has a nonzero Fourier coefficient along $N_E$ associated to a nondegenerate $E$-twisted cube.
\end{proposition}

We use Fourier--Jacobi coefficients and the oscillator operators of \cite{GanSavin2005-MinimalRepresentationsDefinitionsProperties}.

\subsubsection{Fourier--Jacobi coefficients and a raising calculation}

Assume, for a contradiction, that all nondegenerate cube coefficients of $\Pi$ vanish. For $f\in\Pi$, restrict its central Fourier coefficient $f_{Z_E,\psi}$ to $P_E^{\semi}=M_E^{\der}N_E$, and let $W$ be the closure of the resulting Jacobi functions in the smooth topology of uniform convergence, with derivatives, on compact sets. The metaplectic cover of $\Sp(N_E/Z_E)$ splits over $M_E^{\der}$. On each absolute $\SL_2$-factor its defining representation is a sum of four copies of the standard representation, so its metaplectic index is even. We use the splitting trivial on rational points.

Ikeda's Fourier--Jacobi decomposition, in the closed-span formulation recalled in \cite{Gan2005-MultiplicityFormulaCubicUnipotentArthurPackets}*{Section 3.7 and Remarks 3.5, 3.8}, shows that $W$ is the closed span of functions
\begin{equation}\label{theta-A2:eq:Jacobi-tensor}
 (nm)\longmapsto\theta_\varphi(nm)\,a(m),
 \quad m\in M_E^{\der}(\A_F),
\end{equation}
where $a$ is a smooth automorphic function on $M_E^{\der}$. Here $\theta_\varphi$ belongs to the oscillator representation with central character $\psi$. See also \cite{Ikeda1994-TheoryJacobiFormsFourierJacobiCoefficientsEisensteinSeries}*{Proposition 1.3}. All Fourier functionals used next are continuous, since their unipotent adelic quotients are compact.

\begin{lemma}\label{theta-A2:lem:two-factor}
For every auxiliary function $a$ in \eqref{theta-A2:eq:Jacobi-tensor}, the following hold.
\begin{enumerate}[label=(\roman*)]
\item If $E=F^3$, its simultaneous nontrivial Whittaker coefficient on any two distinct $\SL_2$-factors vanishes, after any right translation.
\item If $E=F\times K$, its nontrivial Whittaker coefficient along the upper unipotent of $\SL_2(\A_K)$ vanishes, after any right translation.
\end{enumerate}
\end{lemma}
\begin{proof}
It suffices to consider the pair $\alpha_3,\alpha_4$. Rational triality gives the other pairs in the split case. In the nonsplit case a nontrivial character of $K\backslash\A_K$ has two nonzero conjugate absolute coordinates, so exactly this pair occurs.

In $N_E/Z_E$ take the Lagrangian spanned by the root spaces
\[
 L=\langle e_1+e_3,\ e_2+e_3,\ e_1-e_4,\ e_1+e_4\rangle.
\]
We also write $L$ for its corresponding unipotent subgroup. This subspace is Galois stable, and is stable under $U_{\alpha_3}U_{\alpha_4}$. Integration of the theta function along $[L]$ is evaluation of its Schwartz function at $0$. Therefore a nonzero simultaneous Whittaker coefficient of $a$, together with $\varphi(0)\ne0$, gives a nonzero Fourier coefficient on
\[
 U=L\,U_\beta U_{\alpha_3}U_{\alpha_4},
\]
with character nontrivial precisely on the three indicated root groups. By continuity, this functional is nonzero on an actual vector of $\Pi$.

To obtain a neutral coefficient, put
\[
 x=c\,e_{-\beta}+b_3e_{-\alpha_3}+b_4e_{-\alpha_4},\quad
 h=\beta^\vee+\alpha_3^\vee+\alpha_4^\vee=(1,1,2,0).
\]
Here $c,b_3,b_4$ are nonzero, and in the nonsplit case $b_3,b_4$ are conjugate. These elements extend to a rational $\mathfrak{sl}_2$ triple of type $3A_1$. For its grading, $\dim\fg_h(\ge2)=5$ and $\dim\fg_h(1)=6$. The Lie algebra of $U$ contains $\fg_h(\ge2)$, and its intersection with $\fg_h(1)$ is the isotropic plane
\[
 \ell_2=\langle e_1-e_4,e_1+e_4\rangle.
\]
Choose an $F$-rational Lagrangian $\ell_3\supset\ell_2$ for the symplectic form $(X,Y)\mapsto\langle x,[X,Y]\rangle$, and put $N_x=\exp(\fg_h(\ge2)+\ell_3)$. Fourier expansion along $N_x/U\cong\mathbf G_a$ gives a nonzero extension of the original character. Every such extension is conjugate to the neutral character by an element of $\exp(\ell_2^\perp)(F)$. This follows from the nondegenerate symplectic pairing on $\ell_2^\perp/\ell_2$. Thus a neutral $3A_1$ coefficient is nonzero.

The $\SL_2$-factor corresponding to $\alpha_1$ centralizes the triple. For its standard two-dimensional representation $V_2$, the root calculation is
\begin{equation}\label{theta-A2:eq:raising-dimensions}
 \fg_h(1)\cong3V_2,\quad
 \dim\fg(0,2)=1,\quad \dim\fg(2,2)=0.
\end{equation}
Its weights on $\fg$ have absolute value at most $2$. These are the hypotheses of \cite{JiangLiuSavin2016-RaisingNilpotentOrbitsWaveFrontSets}*{Section~6, Corollary 6.6}, with $m=3$ and the linear cover. That corollary raises the coefficient by a nonzero element $d\,e_{-\alpha_1}$, $d\in F^\times$.

The resulting $\mathfrak{sl}_2$-triple has semisimple element
\[
 h'=\beta^\vee+\alpha_1^\vee+\alpha_3^\vee+\alpha_4^\vee=(2,0,2,0).
\]
The $F$-rational reflection $s_{\alpha_2}$ sends $h'$ to $2\beta^\vee=(2,2,0,0)$. Its neutral unipotent is therefore $N_E$. In cube coordinates the resulting character has the form $(0,e,0,b)$ with $e\in E^\times$ and $b\in F^\times$. Its discriminant is $\pm4bN_{E/F}(e)\ne0$. The sign depends only on the chosen cube pairing. This contradicts the assumed vanishing of all nondegenerate coefficients.
\end{proof}

\subsubsection{Differential identities}

In the split case, Lemma~\ref{theta-A2:lem:two-factor} and Fourier expansion on each product $[U_{\alpha_i}]\times[U_{\alpha_j}]$ give vanishing mixed differences for distinct $\SL_2$-factors. Conjugating by rational Weyl elements and using upper and lower unipotents, which generate each $\SL_2(\A_F)$, we obtain
\[
 a(g_1,g_2,g_3)=a_1(g_1)+a_2(g_2)+a_3(g_3),
\]
where constants can be absorbed into any summand. Equivalently,
\begin{equation}\label{theta-A2:eq:mixed-derivatives}
 R_XR_Ya=0\quad\text{for }X\in\mathfrak{sl}_{2,i},\ Y\in\mathfrak{sl}_{2,j},\ i\ne j.
\end{equation}
Indeed, subtracting the value at the identity in each factor reduces this to the mixed-difference identity. In the nonsplit case, absence of all nontrivial $K$-Whittaker coefficients implies invariance under the upper unipotent of $\SL_2(\A_K)$. Rational Weyl conjugation gives invariance under the lower unipotent, hence under $\SL_2(\A_K)$. Thus
\begin{equation}\label{theta-A2:eq:single-derivatives}
 R_Xa=0\quad\text{for }X\in\Lie(\Res_{K/F}\SL_2).
\end{equation}

We use the oscillator-corrected operators of \cite{GanSavin2005-MinimalRepresentationsDefinitionsProperties}*{Section~2.6}. Let $z=e_\beta$. For $X\in\fm_E^{\der}(\C)$ there is a quadratic element $n_X\in\cU_2(\fn_E)$ such that
\begin{equation}\label{theta-A2:eq:oscillator-operator}
 p_X=zX-\tfrac12n_X
\end{equation}
annihilates the oscillator representation. Equivalently, $\tfrac12n_X$ is the symmetrized quadratic moment map for the action on $\fn_E/\C z$, multiplied by $z$ in the oscillator realization. On \eqref{theta-A2:eq:Jacobi-tensor},
\[
 p_X(\theta_\varphi a)=d\psi(z)\,\theta_\varphi R_Xa.
\]
Hence the products $p_Xp_Y$ for distinct $\SL_2$-factors annihilate $W$ in the split case, and the operators $p_X$ for the $K$-factor annihilate $W$ in the nonsplit case. Passing to the closure is valid because differential operators are continuous in the chosen topology.

We explain why these identities hold on the real component of $\Pi$. The central Fourier map induces a $P_E(\A_F)$-equivariant injection
\begin{equation}\label{theta-A2:eq:induced-FJ}
 \widetilde T:\Pi\hookrightarrow
 \Ind_{P_E^{\semi}(\A_F)}^{P_E(\A_F)}W.
\end{equation}
This is the injectivity argument of \cite{GanSavin2005-MinimalRepresentationsDefinitionsProperties}*{Lemma 5.6}.

The families of operators above are stable under the adjoint action of $P_E$. In fact $p_X$ commutes with $\cU(\fn_E)$, and for a Levi element $m$,
\[
 \Ad(m)p_X=\lambda(m)p_{\Ad(m)X},
\]
where $\lambda$ is its action on $z$. The Levi preserves each $\mathfrak{sl}_2$-factor, with the stated Galois descent. Evaluating a differentiated induced function at an arbitrary $p\in P_E(\A_F)$ therefore reduces to one of the same identities on $W$. The real-place argument at the end of \cite{GanSavin2005-MinimalRepresentationsDefinitionsProperties}*{Section~5} now shows that these operators belong to $\Ann\Pi_w$.

\subsubsection{Principal symbols}

Identify $\fg_\C$ with its dual using an invariant bilinear form, and put
\begin{equation}\label{theta-A2:eq:A2-point}
 \xi=e_{-\beta}+e_{-\alpha_1}+e_{-\alpha_3}+e_{-\alpha_4}.
\end{equation}
Then $\xi\in\cO_{A_2}$. For example, in the eight-dimensional orthogonal representation its Jordan partition is $(3,3,1,1)$. An explicit check is obtained in the ordered basis $e_1,\ldots,e_4,f_1,\ldots,f_4$, with the form pairing $e_i$ with $f_i$.
\[
 \xi=E_{61}-E_{52}+E_{21}-E_{56}+E_{43}-E_{78}+E_{83}-E_{74},
\]
\[
 \xi^2=-2E_{51}-2E_{73},\quad \xi^3=0,
 \quad \rank\xi=4,\quad\rank\xi^2=2.
\]
Since $\xi$ has zero component dual to $\fg_{\beta^\vee}(1)$, the quadratic moment term in the principal symbol of $p_X$ vanishes at $\xi$. The $zX$ term is nonzero for $X=e_{\alpha_i}$, for $i\in\{1,3,4\}$. Thus
\[
 \operatorname{symb}(p_{e_{\alpha_i}})(\xi)\ne0,
 \quad
 \operatorname{symb}(p_{e_{\alpha_i}}p_{e_{\alpha_j}})(\xi)\ne0.
\]
In the split case use a product with $i\ne j$. In the nonsplit case use a single operator from either complex $\mathfrak{sl}_2$-factor of $\Lie(\Res_{K/F}\SL_2)$. Each is in $\Ann\Pi_w$, so its symbol must vanish on $\AV(\Ann\Pi_w)$. This contradicts \eqref{theta-A2:eq:A2-annihilator}, and proves Proposition~\ref{theta-A2:prop:detection}.

\subsubsection{Multiplicity one}

\begin{proposition}\label{theta-A2:prop:multiplicity}
For $E=F^3$ or $F\times K$, we have
\[
 \dim\Hom_{G_E(\A_F)}(\pi_E,\cA(G_E))=1.
\]

\end{proposition}
\begin{proof}
We follow the Fourier-functional argument of \cite{GanSavin2022-TwistedCompositionAlgebrasArthurPacketsTrialitySpin8}*{Section 16.11}. For any nonzero automorphic embedding of $\pi_E$, Proposition~\ref{theta-A2:prop:detection} and Lemma~\ref{theta-A2:lem:residual-AV} give a nondegenerate cube coefficient. By \eqref{theta-A2:eq:local-model}, its twisted composition algebra $C$ is isomorphic to $C_0$ at every place. The local-to-global observation following \eqref{theta-A2:eq:pi-theta-local} implies $C\cong C_0$ globally. Rational Levi conjugation therefore reduces every such coefficient to one fixed character $\psi_{C_0}$.

Composition with this Fourier functional gives an injection
\[
 \Hom_{G_E(\A_F)}(\pi_E,\cA(G_E))
 \hookrightarrow\Hom_{N_E(\A_F)}(\pi_E,\psi_{C_0}).
\]
The target is one-dimensional by \eqref{theta-A2:eq:local-model} and the restricted tensor product with unramified normalization. Segal's residue \eqref{theta-A2:eq:nonfield-residue} gives a nonzero element of the source.
\end{proof}

\subsection{Regularization of the isotropic torus integral}\label{theta-A2:sec:regularization}

Throughout this subsection suppose that $E$ is not a field, put $J=M_3(F)^+$ and $C=C_0$, and write $H=H_C$ and $T=H^\circ$. The ambient connected group of type $E_6$ is split.

\subsubsection{A regularizing operator}

We use the three cocharacters $\lambda_i$ of \cite{GanSavin2022-TwistedCompositionAlgebrasArthurPacketsTrialitySpin8}*{Section~12.5}, with
\[
 \lambda_1(t)\lambda_2(t)\lambda_3(t)=1.
\]
Each is a minuscule cocharacter of the ambient adjoint $E_6$, and its centralizer is a Levi subgroup of type $D_5$. Its unipotent radical is abelian of dimension $16$. For $E=F^3$, use all three cocharacters. For $E=F\times K$, use the Galois-fixed cocharacter. It spans the $F$-split part of $T$. Both opposite parabolics are defined over $F$.

Choose a finite place $u$ splitting $E$, outside all places where the data are ramified, and let $q_u$ be the residue cardinality. Write $A_i=\Omega_{J,u}(\lambda_i(\varpi_u))$. By \cite{MagaardSavin1997-ExceptionalThetaCorrespondencesI}*{Theorem 1.1(2)}, the unnormalized Jacquet module along a $D_5$ parabolic is the sum of a twist of the minimal representation of its Levi subgroup and a one-dimensional representation. The two central exponents, measured by $\lambda_i$, are $2,4$ in one direction and $-2,-4$ in the opposite direction. Indeed, $\lambda_i(t)$ acts with weight one on the $16$-dimensional radical, and the determinant exponents in \cite{MagaardSavin1997-ExceptionalThetaCorrespondencesI}*{Theorem 1.1(2)} are $2/16$ and $4/16$. See also \cite{GanSavin2022-TwistedCompositionAlgebrasArthurPacketsTrialitySpin8}*{Section 12.5}.

Define the finite sum of translations
\begin{equation}\label{theta-A2:eq:regularizing-operator}
 D=\prod_i
 \bigl(A_i+A_i^{-1}-q_u^2-q_u^{-2}\bigr)
 \bigl(A_i+A_i^{-1}-q_u^4-q_u^{-4}\bigr),
\end{equation}
using the cocharacters specified above, and let
\[
 c_D=D(\mathbf1)=\prod_i(2-q_u^2-q_u^{-2})(2-q_u^4-q_u^{-4})\ne0.
\]
For a central covering group, take integral multiples of the cocharacters and scale the exponents accordingly. The scalar remains nonzero. By inversion invariance, $D$ commutes with $H(\A_F)$, and it also commutes with $G_E(\A_F)$.

Let $Q_i=L_iU_i$ be one of these $D_5$ parabolics. The constant term of $\vartheta_J(D\phi)$ is zero at every $gt$ with $g\in G_E(\A_F)$ and $t\in T(\A_F)$. Indeed, the constant-term functional factors through $(\Omega_{J,u})_{U_i(F_u)}$, and \eqref{theta-A2:eq:regularizing-operator} annihilates this Jacquet module. Evaluation at $gt$ is compatible with this assertion because $gt$ commutes with $\lambda_i$ and belongs to $L_i(\A_F)$. The same holds for the opposite parabolic.

\subsubsection{Decay in the torus direction}

\begin{lemma}\label{theta-A2:lem:torus-decay}
For every $N>0$, there are a continuous seminorm $p_N$ and an exponent $B_N$ such that
\begin{equation}\label{theta-A2:eq:torus-decay}
 |\vartheta_J(tD\phi)(g)|\le p_N(\phi)\,\|g\|^{B_N}e^{-N\|H_T(t)\|},
 \quad g\in G_E(\A_F),\ t\in[T].
\end{equation}
Here $H_T$ is the logarithmic adelic norm map on the split part of $T$. We choose representatives for the norm-one quotient in a compact set and take $\|g\|$ to be an adelic height. The same assertion holds after differentiation in $g$.
\end{lemma}
\begin{proof}
We follow the proof of \cite{GanSavin2020-ExceptionalSiegelWeilFormulaPolesSpin}*{Proposition 6.1}. By compactness of the norm-one quotient, it suffices to consider the split norm directions of $T$. In split rank two, these have six chambers bounded by the rays $\pm\lambda_i$. On a chamber with simple restricted-root values $a,b\ge0$, choose the $D_5$ radical with weights $a,a+b$ when $a\ge b$, and that with weights $b,a+b$ otherwise. In split rank one, use the appropriate one of the two opposite rational parabolics. In each chamber the weights are bounded below by a positive multiple of $\|H_T(t)\|$, and the corresponding constant terms vanish by the preceding calculation. The estimates of \cite{GanSavin2020-ExceptionalSiegelWeilFormulaPolesSpin}*{Section 6.1}, applied to right translates by $g$ and their derivatives, give the assertion, including continuity of the seminorm.
\end{proof}

\begin{definition}\label{theta-A2:def:regularized-theta}
For $\phi\in\Omega_J$, define
\begin{equation}\label{theta-A2:eq:regularized-theta}
 \theta_C^{\reg}(\phi,\mathbf1)(g)
 =c_D^{-1}\int_{[H]}\vartheta_J(hD\phi)(g)\rd h.
\end{equation}
Lemma~\ref{theta-A2:lem:torus-decay} proves absolute convergence, smoothness, and moderate growth of \eqref{theta-A2:eq:regularized-theta}.
The integral is independent of the regularizing operator. Indeed, for two such operators $D,D'$, integrate $\vartheta_J(hDD'\phi)(g)$ over $[H]$. The operators commute, and right invariance of the measure gives $c_{D'}\int_{[H]}\vartheta_J(hD\phi)(g)\rd h=c_D\int_{[H]}\vartheta_J(hD'\phi)(g)\rd h$.
Let $\Theta_C^{\reg}(\mathbf1)$ denote the image of this map.
\end{definition}

\subsubsection{Non-vanishing of the regularized lift}

\begin{proposition}\label{theta-A2:prop:regularized-theta}
The map \eqref{theta-A2:eq:regularized-theta} is nonzero, and its image is an automorphic realization of $\pi_E$.
\end{proposition}
\begin{proof}
We use the nondegenerate Fourier-coefficient calculation of \cite{GanSavin2022-TwistedCompositionAlgebrasArthurPacketsTrialitySpin8}*{Proposition 14.5}. Let $\psi_C$ be the cube character of $N_E$ attached to $C$. In the minimal orbit for the Heisenberg parabolic of the ambient $E_6$, the fiber over this character is a closed $H$-torsor. Choose an $F$-point $\widetilde\psi_C$ of this fiber. If $N_J$ denotes the ambient Heisenberg radical, put
\[
 W_\phi(h)=\vartheta_J(h\phi)_{N_J,\widetilde\psi_C}(1).
\]
Outside a sufficiently large finite set $S$, the normalized spherical vector restricts to the characteristic function of $H(\mathcal O_v)$ on the torsor. At the remaining places choose data with $W_\phi\not\equiv0$. Convolution by a Schwartz function on $N_J(F_S)$ multiplies $W_\phi$ by a Schwartz Fourier-transform cutoff. Since the torsor is closed, a compactly supported cutoff near a point where $W_\phi\ne0$ gives data satisfying
\begin{equation}\label{theta-A2:eq:localized-torsor}
 W_\phi\in C_c^\infty(H(\A_F)),\quad
 \int_{H(\A_F)}W_\phi(h)\rd h\ne0.
\end{equation}

Choose these data before applying $D$. Since $D$ is a finite sum of translations in $H(F_u)$, the function $W_{D\phi}$ is again compactly supported. The minimal-orbit Fourier expansion and the torsor identity of \cite{GanSavin2022-TwistedCompositionAlgebrasArthurPacketsTrialitySpin8}*{Proposition 14.5} now unfold absolutely to give
\begin{align}
 \theta_C^{\reg}(\phi,\mathbf1)_{N_E,\psi_C}(1)
 &=c_D^{-1}\int_{H(\A_F)}W_{D\phi}(h)\rd h\notag\\
 &=\int_{H(\A_F)}W_\phi(h)\rd h\ne0.\label{theta-A2:eq:regularized-Fourier}
\end{align}
The second equality follows by translation of the compactly supported integrand, with total scalar $c_D$.
The image is an automorphic realization of $\pi_E$ by Proposition~\ref{theta-A2:prop:multiplicity}.
\end{proof}

\subsubsection{The see-saw identity}

\begin{proposition}\label{theta-A2:prop:paired-seesaw}
For every $f\in\cA_{\cusp}(G)$ and $\phi\in\Omega_J$, the identity
\begin{equation}\label{theta-A2:eq:paired-seesaw}
 \int_{[G]}\theta_C^{\reg}(\phi,\mathbf1)(g)\overline{f(g)}\rd g
 =c_D^{-1}\int_{[H]}\theta_J(D\phi,f)(h)\rd h
\end{equation}
holds with absolutely convergent integrals. In particular, if $\theta_J(\phi',f)=0$ for all $\phi'\in\Omega_J$, then every period on the left is zero. If $T_C$ is anisotropic, the same statement holds for the ordinary theta lift with $D=1$.
\end{proposition}
\begin{proof}
Choose a decay exponent in \eqref{theta-A2:eq:torus-decay} for which its right-hand side is integrable over $[H]$. The corresponding power of $\|g\|$ is integrable against $|f(g)|$ on a Siegel set, by rapid decrease of cusp forms. Therefore
\[
 \int_{[G]}\int_{[H]}
 |\vartheta_J(hD\phi)(g)f(g)|\rd h\rd g<\infty.
\]
Fubini's theorem and the two commuting inclusions in the see-saw give \eqref{theta-A2:eq:paired-seesaw}. For an anisotropic torus, uniform moderate growth on the compact quotient $[H]$ gives the same absolute convergence. For the vanishing statement, take $\phi'=D\phi$.
\end{proof}

\subsection{Comparison of realizations and non-vanishing}\label{theta-A2:sec:conclusion}

\begin{proposition}\label{theorem:exceptional-siegel-weil}
Every element of $\cR_E$ is a finite sum of theta lifts of the trivial representation of $H_C(\A_F)$. If $E$ is not a field, these lifts are regularized as in Section~\ref{theta-A2:sec:regularization}, and the only $C$ needed is $C_0$. If $E$ is a field, the ordinary theta lifts suffice.
\end{proposition}
\begin{proof}
First suppose that $E$ is not a field. By Proposition~\ref{theta-A2:prop:regularized-theta}, the nonzero space $\Theta_{C_0}^{\reg}(\mathbf1)$ is an automorphic realization of the irreducible representation $\pi_E$. By Proposition~\ref{theta-A2:prop:multiplicity}, this representation has multiplicity one in $\cA(G_E)$. Thus \eqref{theta-A2:eq:nonfield-residue} gives equality of subspaces
\[
 \cR_E=\Theta_{C_0}^{\reg}(\mathbf1)\subset\cA(G_E).
\]

If $E$ is a field, all tori $T_C$ in question are anisotropic. The ordinary global theta lifts are nonzero by \cite{GanSavin2022-TwistedCompositionAlgebrasArthurPacketsTrialitySpin8}*{Proposition 14.5 and Corollary 14.6}. Segal's square-integrable residue is the sum of the representations $\pi_{\mathcal S}$ described in \cite{Segal2019-DegenerateResidualSpectrumQuasiSplitFormsSpin8}*{Theorem 5.4(3), (5.20)}. Here $\mathcal S$ is a finite set of places where $E_v$ is a field, and $|\mathcal S|\ne1$. For these constituents, the coherence assertion of \cite{GanSavin2022-TwistedCompositionAlgebrasArthurPacketsTrialitySpin8}*{Lemma 15.5(1)(ii)} and the count in \cite{GanSavin2022-TwistedCompositionAlgebrasArthurPacketsTrialitySpin8}*{Section~16.10} apply. By \cite{GanSavin2022-TwistedCompositionAlgebrasArthurPacketsTrialitySpin8}*{Theorems 16.6 and 16.8}, the sum of theta lifts over the coherent $C$ fills the corresponding isotypic subspace of $\cA_2(G_E)$. It therefore contains Segal's residual subspace.

\end{proof}

\begin{proof}[Proof of Theorem~\ref{theorem:exceptional-theta-G2-to-PGL3}]
Choose $E$, $f$, and $R$ as in \eqref{theta-A2:eq:residue-period}. By Proposition~\ref{theorem:exceptional-siegel-weil}, write
\[
 R=\sum_{j=1}^m\theta_{C_j}^{\reg}(\phi_j,\mathbf1),
\]
where the superscript is omitted when the torus is anisotropic. At least one summand has a nonzero pairing with $f$. If all theta lifts of $f$ to its corresponding $\Aut(J_j)$ vanished, the see-saw identity of Proposition~\ref{theta-A2:prop:paired-seesaw} would make that pairing zero. Hence one such theta lift is nonzero.

Suppose that $J_j=D^+$ with $D$ division. Then $E$ is a field and $T_C\cong\Res_{E/F}\mathbf G_m/\mathbf G_m$ is anisotropic. Also $H_C(F)=T_C(F)$, whereas $H_C(\A_F)/T_C(\A_F)$ is compact. With compatible measures,
\[
 \theta_C(\phi,\mathbf1)=
 \int_{H_C(\A_F)/T_C(\A_F)}
       \theta_{T_C}(h\phi,\mathbf1)\rd\bar h.
\]
A nonzero pairing of the left-hand side with $f$ implies a nonzero pairing for some $h\phi$ on the right. Since $[T_C]$ is compact, the connected see-saw identity gives
\[
 \int_{[T_C]}\theta_J(h\phi,f)(t)\rd t
 =\int_{[G]}\theta_{T_C}(h\phi,\mathbf1)(g)\overline{f(g)}\rd g\ne0.
\]
Thus the lift to the connected group $\mathrm{PD}^\times(\A_F)$ is nonzero, as asserted.
\end{proof}

\part{Global harmonic analysis}\label{part:global-harmonic-analysis}

\section{Stabilization of the trace formula}

\begin{theorem}\label{theorem:stabilization-twisted-PGSO8}
  We have the following stabilization of the trace formula for $\widetilde{\PGSO_8}$ and $\rG_2$ for each Hecke--Satake family $c$ for $\widetilde{\PGSO_8}$ and $\rG_2$.
  \begin{align*}
    I^{\widetilde{\PGSO_8}}_{\disc, c}(f^{\widetilde{\PGSO_8}})
    =
    S^{\rG_2}_{\disc, c}(f^{\widetilde{\PGSO_8}}_{\rG_2})
    +
    \frac{1}{2} S^{\SO_4}_{\disc, c}(f^{\widetilde{\PGSO_8}}_{\SO_4})
    +
    S^{\SL_3}_{\disc, c}(f^{\widetilde{\PGSO_8}}_{\SL_3}),
  \end{align*}
  and
  \begin{align*}
    I^{\rG_2}_{\disc, c}(f^{\rG_2})
    =
    S^{\rG_2}_{\disc, c}(f^{\rG_2})
    +
    \frac{1}{2} S^{\SO_4}_{\disc, c}(f^{\rG_2}_{\SO_4})
    +
    \frac{1}{3} S^{\PGL_3}_{\disc, c}(f^{\rG_2}_{\PGL_3}).
  \end{align*}
\end{theorem}

\begin{remark}\label{remark:killing-SL3-PGL3-contribution}
  The contribution $S^{\SL_3}_{\disc, c}$ comes from a trivial representation of $\SL_3(\A_F)$ or a cuspidal representation of $\SL_3(\A_F)$.
  The latter comes as a restriction of a cuspidal representation of $\GL_3(\A_F)$ and the lifting to $\PGSO_8(\A_F)$ is obtained from the adjoint lifting.
  Thus, if we have $S^{\SL_3}_{\disc, c} \neq 0$, then the lift of the Hecke--Satake family $c$ to $\SO_8$ via the standard representation $\Spin_8 \to \GL_8$ is associated to the $A$-parameter of the form
  \begin{align*}
    S_3^A \boxplus S_5^A
  \end{align*}
  for $\GL_8$, or an $A$-parameter $\psi$ for $\GL_8$ such that the standard $L$-function
  \begin{align*}
    L^S(s, \psi, \Std)
  \end{align*}
  is holomorphic at $s = 1$.
  We call such a Hecke--Satake family $c$ of $\SL_3$-type.

  Similarly, the contribution $S^{\PGL_3}_{\disc, c}$ is not zero only if the Hecke--Satake family $c$ is associated to the $A$-parameter of the form
  \begin{align*}
    \tau \oplus \tau^{\vee} \oplus \mathbf{1} \oplus \mathbf{1}
  \end{align*}
  for a discrete automorphic representation $\tau$ of $\GL_3(\A_F)$.
  We call such a Hecke--Satake family $c$ of $\PGL_3$-type.
\end{remark}

\section{Discrete part of the trace formula for \texorpdfstring{$\rG_2$}{G2}}

Let $F$ be a number field.
Let $G$ be the split semisimple group of type $G_2$ over $F$.
We fix a Borel pair $(B_0, T_0)$ of $G$ and let $\alpha$ and $\beta$ be the short root and the long root, respectively.

The standard Levi subgroups $T_0$, $\GL_{2, s}$, $\GL_{2, l}$ and $G$ represent the Levi conjugacy classes over $F$.
Let $M$ be a standard Levi subgroup of $G$ and $P$ be the standard parabolic subgroup of $G$ with Levi component $M$.
We set $W^M_0 = W(M, T_0)$.
For $t\geq0$, the contribution of $M$ to Arthur's invariant trace formula $I^{G}_{\disc, t}$ is
\begin{align*}
  \frac{1}{\abs{W(G, M)}}
  \sum_{w \in W(M)_{\reg}} \abs{\det(w-1)_{\restriction{\fa_M^{G}}}}^{-1} \tr(M_{P | wP,t}(0) \circ l(w) \circ I_{P, t}(f^G)),
\end{align*}
where
\begin{itemize}
  \item
  $I_{P,t}$ is the constituent of the global induced representation
  \[
    \cI_P(L^2_{\disc}(M(F) \backslash M(\A_F)^1))
  \]
  defined by $\pi$ with infinitesimal character $\mu_{\pi}$ such that $\lVert\operatorname{Im}(\mu_{\pi})\rVert = t$.
  \item
  $M_{P | wP,t}(0)$ is the global intertwining operator defined by the meromorphic continuation of the usual intertwining integral.
\end{itemize}
We write $M_{P | wP,t}(0) \circ l(w)$ as $M_P(w)$ for short.

For $M = G$, the contribution is $\tr(R_{\disc, t}(f^G))$ from the regular representation of $G(\A_F)^1$.

\begin{remark}\label{remark:spectral-contribution-different-parabolic}
  Replacing $P$ by a parabolic subgroup $P'$ with the same Levi component $M$ preserves its contribution. See \cite{FinisLapidMuller2011-SpectralSideArthurSTraceFormula}*{(5.8)--(5.9)}.
  The contribution is therefore also unchanged when $(M, \sigma, w)$ is conjugated by $W(M)$, or by $W(M, M')$ for another Levi subgroup $M'$.
\end{remark}

We give the contribution $I_{\disc, t, M}^G$ from each $M$ explicitly based on the description of the regular elements in $W(M)$ in Lemmas~\ref{lemma:regular-weyl-G2-non-toric-Levi} and~\ref{lemma:regular-weyl-G2-toric-Levi}.

\begin{lemma}\label{lemma:discrete-part-G2}
  Let $F$ be a number field.
  We set $G = \rG_2$ and let $M$ be a standard Levi subgroup of $G$.
  Let $P$ be the standard parabolic subgroup of $G$ with Levi component $M$.

  \begin{enumerate}
    \item
    We assume that $M = T_0$.
    Then, we have
    \begin{align*}
      W(T_0)_{\reg} = \{s_{\alpha} s_{\beta}, (s_{\alpha} s_{\beta})^2, s_{\beta} s_{\alpha}, (s_{\beta} s_{\alpha})^2, (s_{\alpha} s_{\beta})^3 \}.
    \end{align*}
    The value of $\abs{\det(w-1)_{\restriction{\fa_M^{G}}}}$ depends on the order of the rotation.
    \begin{enumerate}
      \item
      For $w = s_{\alpha} s_{\beta}$ and $s_{\beta} s_{\alpha}$, we have $\abs{\det(w-1)_{\restriction{\fa_M^{G}}}} = 1$.
      An automorphic character $\sigma$ of $T_0(\A_F)^1$ satisfies $w \sigma \cong \sigma$ if and only if $\sigma$ is trivial.

      In this case,
      \begin{align*}
        \Std_{\rG_2} \circ \iota^G_M c(\sigma) = \mathbf{1}^{\boxplus 7}.
      \end{align*}
      \item
      For $w = (s_{\alpha} s_{\beta})^2$ and $w = (s_{\beta} s_{\alpha})^2$, we have $\abs{\det(w-1)_{\restriction{\fa_M^{G}}}}= 3$.
      For any automorphic character $\sigma$ of $T_0(\A_F)^1$ such that $w \sigma \cong \sigma$, we can find a cubic character $\chi$ of $\A_F^{\times}$ such that
      \begin{align*}
        \sigma = \chi \circ \beta.
      \end{align*}

      In this case,
      \begin{align*}
        \Std_{\rG_2} \circ \iota^G_M c(\sigma) = \mathbf{1} \boxplus \chi^{\boxplus 3} \boxplus (\chi^{-1})^{\boxplus 3}.
      \end{align*}

      \item
      For $w = (s_{\alpha} s_{\beta})^3$, we have $\abs{\det(w-1)_{\restriction{\fa_M^{G}}}} = 4$.
      Any automorphic character $\sigma$ of $T_0(\A_F)^1$ such that
      $w\sigma\cong\sigma$ is of the form
      \begin{align*}
        \sigma = (\omega_1\circ\alpha)(\omega_2\circ\beta),
      \end{align*}
      where $\omega_1,\omega_2$ are quadratic Hecke characters,
      possibly trivial.  They determine a biquadratic extension
      precisely when they are independent.  We have
      \begin{align*}
        \Std_{\rG_2} \circ \iota^G_M c(\sigma) = \mathbf{1} \boxplus \omega_1^{\boxplus 2} \boxplus \omega_2^{\boxplus 2} \boxplus (\omega_1\omega_2)^{\boxplus 2}.
      \end{align*}
    \end{enumerate}

    \item
    We assume that $M = \GL_{2, s}$.
    Then, the unique element in $W(M)_{\reg}$ is represented by
    \begin{align*}
      w = w^G_0 w_0^M = s_{\beta} s_{\alpha} s_{\beta} s_{\alpha} s_{\beta}
    \end{align*}
    where $w^G_0$ is the longest element in $W(G, T_0)$ and $w_0^M$ is the longest element in $W(M, T_0)$.
    For any discrete automorphic representation $\sigma$ of $M(\A_F)^1$ such that $w \sigma \cong \sigma$, we have $\sigma = \sigma^{\vee}$.

    In this case,
    \begin{align*}
      \Std_{\rG_2} \circ \iota^G_M c(\sigma) = \sigma \boxplus \sigma^{\vee} \boxplus \omega_{\sigma} \boxplus \omega_{\sigma}^{-1} \boxplus \mathbf{1}.
    \end{align*}

    \item
    We assume that $M = \GL_{2, l}$.
    Then, the unique element in $W(M)_{\reg}$ is represented by
    \begin{align*}
      w = w^G_0 w_0^M = s_{\alpha} s_{\beta} s_{\alpha} s_{\beta} s_{\alpha}
    \end{align*}
    where $w^G_0$ is the longest element in $W(G, T_0)$ and $w_0^M$ is the longest element in $W(M, T_0)$.
    For any discrete automorphic representation $\sigma$ of $M(\A_F)^1$ such that $w \sigma \cong \sigma$, we have $\sigma = \sigma^{\vee}$.

    In this case,
    \begin{align*}
      \Std_{\rG_2} \circ \iota^G_M c(\sigma) = \sigma \boxplus \sigma^{\vee} \boxplus \Ad(\sigma).
    \end{align*}
  \end{enumerate}
\end{lemma}

\section{Discrete part of the trace formula for \texorpdfstring{$\widetilde{\PGSO_8}$}{widetilde{PGSO8}}}\label{section:discrete-trace-formula-PGSO8}
Let $F$ be a number field.
Let $G = \PGSO_8$ be the split semisimple group of type $\mathrm{D}_4$ over $F$.
Let $\widetilde{G} = \widetilde{\PGSO_8} = \PGSO_8 \rtimes \theta$ be a twisted space defined by the triality automorphism $\theta$ of $\PGSO_8$.
We fix a Borel pair $(B_0, T_0)$ of $\PGSO_8$ and let $\Delta_0 = \{\alpha_1, \alpha_2, \alpha_3, \alpha_4\}$ be the set of simple roots such that $\alpha_2$ is the central root.

Proper Levi subgroups $M$ with $\Norm_{\widetilde{G}}(M)(F) = \emptyset$ contribute zero.
By Lemma~\ref{lemma:Levi-conjugacy-class-D_4}, the mutually nonconjugate Levi subgroups $M_S$ for $S=\emptyset$, $\{\alpha_2\}$, $\{\alpha_1,\alpha_2\}$, $\{\alpha_1,\alpha_3,\alpha_4\}$, $\Delta_0$ suffice.

The contribution associated with $M_S$ is
\begin{align*}
  \frac{1}{\abs{W(G, M_S)}}
  \sum_{\widetilde{w} \in W_{\widetilde{G}}(M_S)_{\reg}} \abs{\det(\widetilde{w}-1)_{\restriction{\fa_{M_S}^{G}}}}^{-1} \tr(M_{P_S | \widetilde{w}P_S}(\widetilde{w})I_{P_S, t, \widetilde{w}}(f^{\widetilde{G}})),
\end{align*}
where
\begin{itemize}
  \item
  For $\gamma \in \widetilde{G}(\A_F)$ and $f$ in the $t$-part of
  $\cI_{P_S}(L^2_{\disc}(M_S(F) \backslash M_S(\A_F)^1))$, we define
  \begin{align*}
    I_{P_S, t, \widetilde{w}}(\gamma)f(g) = f(\widetilde{w}^{-1} g \gamma).
  \end{align*}
  \item
  $M_{P_S | \widetilde{w}P_S}(\widetilde{w})$ is the global intertwining operator defined by the meromorphic continuation of the usual intertwining integral.
\end{itemize}

The regular elements in $W_{\widetilde{G}}(M_S)$ and the scalar $\abs{\det(\widetilde{w}-1)_{\restriction{\fa_{M_S}^{G}}}}^{-1}$ can be computed explicitly by using Corollary~\ref{corollary:regular-elements-theta-stable-cases} and Lemmas~\ref{lemma:regular-element-not-theta-stable-case},~\ref{lemma:classification-regular-conjugacy-classes-Weyl-elements},~\ref{lemma:description-of-centralizer-of-regular-elements-1}, and~\ref{lemma:description-of-centralizer-of-regular-elements-2}.
The calculation parallels the $\rG_2$ case and is omitted.

\part{Proof of the endoscopic character relations}\label{part:ECR}

\section{Tempered local intertwining relations for \texorpdfstring{$\rG_2$}{G2}}\label{section:local-LIR-G2}

We now formulate the essential case of local intertwining relations for $\rG_2$.
In this section, we assume that $F\cong\dot F_v$ for a place $v$ of a number field $\dot F$ and that $\psi_F$ is the local component of a global additive character $\psi$ of $\A_{\dot F}/\dot F$. Since $\rG_2$ is of adjoint type, the following theorem is independent of this choice of $\psi_F$.

\begin{theorem}\label{theorem:local-tempered-LIR-G2}
  Let $F$ be a local field of characteristic zero.
  We set $G = \rG_2$ and let $P$ be a proper standard parabolic subgroup of $G$ with the standard Levi component $M$.
  Let $w$ be a non-trivial element in the relative Weyl group $W(M)$.
  Let $\sigma$ be a discrete series representation of $M$ such that $w \sigma \cong \sigma$, the $R$-group $R^{G}(\sigma)$ is non-trivial and $W_0(M)_{\sigma}$ is trivial.
  Then, we have
  \begin{align*}
    \tr(R_P(w, \sigma, \psi_F)I_P(\sigma)(f^G))
    =
    S^{\SO_4}_{\phi^{\SO_4}(\sigma)}(f^{G}_{\SO_4}).
  \end{align*}
\end{theorem}

\begin{lemma}\label{lemma:globalization-unitary-character}
  Let $\dot E/\dot F$ be a quadratic extension of number fields, let $w$ be a place of $\dot F$, and let $S$ be a finite set of places disjoint from $\{w\}$ at which $\dot E/\dot F$ is nonsplit.
  For unitary characters $\chi_u$ of $\dot E_u^\times/\dot F_u^\times$, $u\in S$, there exists a unitary Hecke character $\dot\chi$ of $\dot E$ trivial on $\A_{\dot F}^\times$, with $\dot\chi_u=\chi_u$ for $u\in S$ and unramified at every finite $u\notin S\cup\{w\}$.
\end{lemma}

\begin{proof}
  Put $T=\operatorname{Res}_{\dot E/\dot F}\mathbf G_m/\mathbf G_m$.
  The product of $T(\dot F_u)$ for $u\in S$ and the maximal compact subgroups at the remaining places other than $w$ is compact.
  Its image in $T(\dot F)\backslash T(\A_{\dot F})$ is injective, since $T(\dot F)\to T(\dot F_w)$ is injective.
  Extend the prescribed character, trivial on the latter compact factors, from this closed subgroup by Pontryagin duality. Compare \cite{Takanashi2025-NoteSauvageotDensityPrinciple}*{Corollary 8.3 and its proof}.
\end{proof}

\begin{proof}[Proof of Theorem~\ref{theorem:local-tempered-LIR-G2}]
  For $F=\R$, the assertion follows from real tempered endoscopy with Whittaker normalization. See \cite{Shelstad2008-TemperedEndoscopyRealGroupsIII}*{Theorem 11.5, Corollary 11.6} and Remark~\ref{remark:Whittaker-normalization-Levi}.
  There are no triples satisfying the hypotheses over $\C$.
  We henceforth assume that $F$ is non-archimedean and first treat $M=M_s$ and $M=M_l$ together.
  By Lemma~\ref{G2-elliptic-GL2-Levi-representations}, write
  \begin{align*}
    \sigma=\mathrm{AI}_{E/F}(\chi),\quad
    \chi_{\restriction_{F^\times}}=1,
  \end{align*}
  where $E/F$ is quadratic.
  For $M=M_l$, we also have $\chi^3\neq1$.

  Choose a totally real field $\dot F$ with at least three real places and a totally imaginary quadratic extension $\dot E/\dot F$ with $(\dot F_v,\dot E_v)\cong(F,E)$.
  This is possible by approximation and Krasner's lemma.
  We apply Lemma~\ref{lemma:globalization-unitary-character}.

  Fix distinct real places $u_1,u_2,w$.
  Apply the lemma with $S$ containing $v$, every real place except $w$, and every finite place where $\dot E/\dot F$ ramifies.
  Prescribe $\dot\chi_v=\chi$, a nonzero weight at each real place in $S$, and $\dot\chi_u=1$ at every finite $u\in S\setminus\{v\}$.
  Then
  \begin{align*}
    \dot\sigma=\mathrm{AI}_{\dot E/\dot F}(\dot\chi),\quad
    \omega=\omega_{\dot E/\dot F},\quad
    \rho_j=\mathrm{AI}_{\dot E/\dot F}(\dot\chi^j)
  \end{align*}
  are self-dual, and $\rho_j$ is cuspidal for $j=1,2,3$ by the prescribed real weights.
  Every finite component $\dot\sigma_u$ with $u\neq v$ is a principal series. At nonsplit places the inducing character is trivial, and at split places it is unramified.

  Put $\widetilde G=\widetilde{\PGSO_8}$ and $H=\SO_4$.
  The global data for the two cases are
  \begin{alignat*}{2}
    M=M_s:&\quad M'=M_{\alpha_2},\quad
      \Sigma=\dot\sigma\boxtimes\omega\boxtimes\omega\boxtimes\omega,
      &\quad\Std_H\circ\dot\phi^H&=1\boxplus\omega\boxplus\rho_1,\\
    M=M_l:&\quad M'=M_{\alpha_1,\alpha_3,\alpha_4},\quad
      \Sigma=\dot\sigma\boxtimes\dot\sigma\boxtimes\dot\sigma\boxtimes\omega,
      &\quad\Std_H\circ\dot\phi^H&=\rho_1\boxplus\rho_2.
  \end{alignat*}
  The central-character conditions make $\Sigma$ descend to $M'$.
  In both cases choose the refinement with
  $\wedge^2_+\circ\dot\phi^H=\omega\boxplus\rho_1$.
  These are discrete generic global $A$-parameters for $H$ with $|\cS_{\dot\phi^H}|=2$, and their components at $u_1,u_2$ are discrete.
  The standard seven-dimensional parameters are, respectively,
  \begin{align*}
    1\boxplus\omega^{\boxplus2}\boxplus\rho_1^{\boxplus2},
    \quad
    \omega\boxplus\rho_1^{\boxplus2}\boxplus\rho_2.
  \end{align*}
  The standard eight-dimensional parameter is obtained by adding $1$.
  Strong multiplicity one for general linear groups therefore excludes the other proper-Levi contributions and the $\PGL_3$ endoscopic term.
  The repeated cuspidal summands exclude a discrete global $A$-parameter for $\PGSO_8$.
  The twisted $\SL_3$ term is also zero. Its adjoint standard $L$-function is regular at $1$, whereas the displayed eight-dimensional parameter contains $1$.

  Let $c$ be the resulting Hecke--Satake family, and denote the global normalized intertwining distributions for $(M,\dot\sigma)$ and $(M',\Sigma)$ by $I^{\rG_2}_{M,\dot\sigma}$ and $I^{\widetilde{\PGSO_8}}_{M',\Sigma}$.
  Their Weyl elements are
  \begin{align*}
    w_M=w_0^{\rG_2}w_0^M,\quad
    \widetilde w_{M'}=w_0^{\PGSO_8}w_0^{M'}\rtimes\theta.
  \end{align*}
  For either $\rG_2$ Levi, $|W(G,M)|=2$ and $|\det(w_M-1)|=2$, giving coefficient $1/4$.
  On the twisted side, $M_{\alpha_2}$ has relative Weyl group of order $8$. Its four Weyl-conjugate inducing data each contribute one regular term with determinant of absolute value $2$, giving coefficient $4/(8\cdot2)=1/4$.
  For $M_{\alpha_1,\alpha_3,\alpha_4}$ the relative Weyl group has order $2$ and the unique regular term has determinant of absolute value $2$, again giving $1/4$. See Corollary~\ref{corollary:regular-elements-theta-stable-cases} and Remark~\ref{remark:spectral-contribution-different-parabolic}.
  The representations on the dual nilradicals on the two sides are
  \begin{alignat*}{2}
    M=M_s:&\quad \rho_1^{\boxplus2}\boxplus\omega,
      &\quad&\rho_1^{\boxplus4}\boxplus\omega^{\boxplus3},\\
    M=M_l:&\quad \rho_3\boxplus\rho_1\boxplus\omega,
      &\quad&\rho_3\boxplus\rho_1^{\boxplus3}\boxplus\omega.
  \end{alignat*}
  Every summand is self-dual, and its completed $L$-function is finite and nonzero at $1$.
  Thus the global normalizing factors equal $1$ by the functional equations, as in \cite{Arthur2013-EndoscopicClassificationRepresentations}*{(4.2.10), (4.6.8)--(4.6.9)}.
  We evaluate the meromorphic factorization at the self-intertwining point.
  The inducing automorphisms factor into the local Whittaker extensions by \cite{Arthur2013-EndoscopicClassificationRepresentations}*{Lemma 4.2.3}, applied to the general linear factors.
  Hence all local operators have the normalization of Remark~\ref{remark:Whittaker-normalization-Levi}.

  Write
  $S^{H}_{\dot\phi^H}(f^H)=\prod_uS^{H}_{\dot\phi^H_u}(f^H_u)$.
  The stable multiplicity formula gives $S^H_{\disc,c(\dot\phi^H)}=\frac12S^{H}_{\dot\phi^H}$.
  The discrete expansion and stabilization in \cite{MoeglinWaldspurger2016-StabilisationFormuleTracesTordue2}*{X.5.1, X.8.1} consequently give
  \begin{align}
    R^{\rG_2}_{\disc,c}(f)+\tfrac14I^{\rG_2}_{M,\dot\sigma}(f)
      &=S^{\rG_2}_{\disc,c}(f)+\tfrac14S^{H}_{\dot\phi^H}(f_H),
      \label{equation:LIR-global-G2}\\
    \tfrac14I^{\widetilde{\PGSO_8}}_{M',\Sigma}(\widetilde f)
      &=S^{\rG_2}_{\disc,c}(\widetilde f_{\rG_2})
        +\tfrac14S^{H}_{\dot\phi^H}(\widetilde f_H).
      \label{equation:LIR-global-D4}
  \end{align}
  These identities also hold on the associated $K$-spaces.
  At the auxiliary real places, the ordinary relation follows from real tempered endoscopy as above. The twisted relation is the real case of Theorem~\ref{theorem:local-intertwining-relation-PGSO8-SO4}, proved below from the discrete expansion and stabilization just used.
  Since $\dot\phi^H_{u_i}$ is discrete, simultaneous cuspidal transfer on the real $K$-space allows independent $H$ and $\rG_2$ transfers by Theorem~\ref{theorem:cuspidal-surjective}.
  Applying this at $u_1$ and $u_2$ in \eqref{equation:LIR-global-D4}, the separation argument of \cite{Arthur2013-EndoscopicClassificationRepresentations}*{Lemma 5.4.2 and its proof} yields
  \begin{align}\label{equation:LIR-global-twisted-separated}
    S^{\rG_2}_{\disc,c}=0,\quad
    I^{\widetilde{\PGSO_8}}_{M',\Sigma}(\widetilde f)=S^{H}_{\dot\phi^H}(\widetilde f_H).
  \end{align}
  In \eqref{equation:LIR-global-G2}, the known real relation at $u_1$ expresses the right-hand difference as a linear combination of nonidentity elliptic $R$-group characters.
  Since $R^{\rG_2}_{\disc,c}$ has nonnegative spectral multiplicities, \cite{KalethaMinguezEtAl2014-EndoscopicClassificationRepresentationsInnerFormsUnitary}*{Lemma 3.2.2} gives
  \begin{align*}
    R^{\rG_2}_{\disc,c}=0,\quad I^{\rG_2}_{M,\dot\sigma}(f)=S^{H}_{\dot\phi^H}(f_H).
  \end{align*}
  At each finite $u\neq v$, the principal-series inducing data reduce the local comparison to general linear groups and tori by Lemma~\ref{lemma:parabolic-descent-commutes-with-transfer} and induction in stages.
  The real comparisons were fixed above, including parabolic descent if $\dot\chi_w=1$.
  Factorization and a choice of nonzero local factors away from $v$ now give both local identities at $v$.

  Finally, let $M=T$ and let $\omega_1,\omega_2$ be the independent quadratic characters defining $\sigma$ in Lemma~\ref{G2-elliptic-toral-representations}.
  Put $\omega_3=\omega_1\omega_2$, let $E/F$ correspond to $\omega_3$, and set
  \begin{align*}
    \chi=\omega_1\circ N_{E/F},\quad
    \mathrm{AI}_{E/F}(\chi)=\omega_1\oplus\omega_2.
  \end{align*}
  Then $\chi_{\restriction_{F^\times}}=1$.
  Apply the short-root construction with local component
  $\dot\sigma_v=I_{B_{\GL_2}}^{\GL_2}(\omega_1\boxtimes\omega_2)$.
  This principal series is irreducible, while $\dot\sigma$ remains cuspidal by its prescribed real components.
  The local $L$-parameter for $H$ has standard representation
  $1\oplus\omega_1\oplus\omega_2\oplus\omega_3$.
  Under induction in stages, the intertwining operator for representations of $\GL_2$ combines with the intertwining operator defined by
  \begin{align*}
    (w_0^{\rG_2}w_0^{M_s})w_0^{M_s}&=w_0^{\rG_2},\\
    (w_0^{\PGSO_8}s_{\alpha_2}\rtimes\theta)s_{\alpha_2}
      &=w_0^{\PGSO_8}\rtimes\theta.
  \end{align*}
  Remark~\ref{remark:Whittaker-normalization-Levi} identifies the normalized operators in these equalities.
  Thus the same global comparison proves the toral relation, and also its twisted counterpart associated with $H$.
\end{proof}

\section{Local intertwining relation between \texorpdfstring{$\widetilde{\PGSO_8}$}{twisted PGSO8} and \texorpdfstring{$\SO_4$}{SO4}}\label{section:local-LIR-PGSO8-SO4}

Let $F$ be a local field of characteristic zero.  Fix the
pinning of $G=\PGSO_8$ from Section~\ref{section:general-notation} and the
pinning-preserving triality automorphism $\theta$, and put
$\widetilde G=G\rtimes\theta$.  Fix a non-trivial additive character
$\psi_F$ and the corresponding Whittaker datum.
Let $\widehat\theta$ be the dual pinned automorphism of
$\widehat G=\Spin_8(\C)$, identify
$(\widehat G^{\widehat\theta})^\circ$ with $\rG_2(\C)$, and let $\alpha$
and $\beta$ be its short and long simple roots.  Choose $s_H$ in its
pinned torus by
\begin{align*}
  \alpha(s_H)=1,
  \quad
  \beta(s_H)=-1.
\end{align*}
Then
\begin{align*}
  \widehat H
  =\Cent_{\widehat G}(s_H\widehat\theta^{-1})^\circ
  \cong\SO_4(\C).
\end{align*}
Let
\begin{align*}
  \mathfrak e_H
  =
  (H=\SO_4,\mathcal H_H,
  s_H\rtimes\widehat\theta^{-1},\xi_H)
\end{align*}
be the resulting elliptic twisted endoscopic datum.

\begin{theorem}\label{theorem:local-intertwining-relation-PGSO8-SO4}
  Let $F$ be a local field of characteristic zero.
  Let $(M, \sigma, \widetilde{r})$ be an elliptic tempered triple of $\widetilde{\PGSO_8}$ such that $M$ is the standard Levi subgroup of $\PGSO_8$ associated with one of the sets $\emptyset$, $\{\alpha_2\}$, or $\{\alpha_1, \alpha_3, \alpha_4\}$, and let $P$ be the standard parabolic subgroup with Levi component $M$.
  Assume that the associated elliptic endoscopic datum is $\mathfrak e_H$. For $M=T$, this is the class represented by $w_0^{\PGSO_8}\rtimes\theta$ in Remark~\ref{remark:explicit-description-regular-Weyl-elements}.
  Then, we have
  \begin{align*}
    \tr(R_{P}(\widetilde{r}, \sigma, \psi_F)I_{P}(\sigma)(f^{\widetilde{\PGSO_8}}))
    =
    S^{\SO_4}_{\phi^{\SO_4}(\sigma)}(f^{\widetilde{\PGSO_8}}_{\SO_4}),
  \end{align*}
  for $f^{\widetilde{\PGSO_8}} \in \cI(\widetilde{\PGSO_8}(F))$.
\end{theorem}

\begin{proof}
  Over $\C$ there are no triples satisfying the hypotheses. 

  We assume that $F=\R$.
  Since $\PGSO_8$ is adjoint, we may use the standard real additive
  character.
  By Lemmas~\ref{lemma:weyl-fixed-representation-long-dual-root}
  and~\ref{lemma:weyl-fixed-representation-short-dual-root}, after Weyl
  conjugacy the inducing representation belongs to one of the two families
  \begin{align*}
    \sigma'\boxtimes\epsilon\boxtimes\epsilon\boxtimes\epsilon,
    \quad
    \sigma'\boxtimes\sigma'\boxtimes\sigma'\boxtimes\epsilon,
  \end{align*}
  respectively, where
  \begin{align*}
    \sigma'=\mathrm{AI}_{\C/\R}(\chi_n),\quad
    \chi_n(z)=(z/\bar z)^n,\quad
    n\in\Z\setminus\{0\},\quad \epsilon=\operatorname{sgn}.
  \end{align*}
  The Weyl element is
  $w_0^{\PGSO_8}w_0^M\rtimes\theta$.
  There is no toral triple over $\R$ satisfying the hypotheses of the theorem. The characters of the roots
  $\alpha_1$, $\alpha_2$, and $\alpha_1+\alpha_2$
  in Lemma~\ref{lemma:description-of-centralizer-of-regular-elements-1}
  are $\zeta$, $\xi$, and $\zeta\xi$, so one is trivial and $W_0\ne1$.

  Choose an imaginary quadratic field $K/\Q$ and a split finite place $w$.
  Lemma~\ref{lemma:globalization-unitary-character} gives a unitary Hecke
  character $\chi_0$ of $K$, trivial on $\A_\Q^\times$, with infinity type
  $\chi_n$, trivial at every nonsplit finite place and unramified away
  from $w$ at split finite places.
  For $d=2,3$, respectively, take
  \[
    \dot F=\Q(\sqrt5),\quad
    \dot F=\Q(\zeta_7+\zeta_7^{-1}),\qquad
    \dot E=K\dot F,\quad
    \dot\chi=\chi_0\circ N_{\dot E/K}.
  \]
  Here $\zeta_7$ is a primitive seventh root of unity.
  Since $\dot F\cap K=\Q$, this gives a quadratic extension
  $\dot E/\dot F$ and a character trivial on $\A_{\dot F}^\times$.
  At every real place $u$ we have $\dot\chi_u=\chi_n$.
  At every nonsplit finite place it is trivial. Thus all finite components
  of $\dot\sigma=\mathrm{AI}_{\dot E/\dot F}(\dot\chi)$ are principal
  series. Set $\omega=\omega_{\dot E/\dot F}$ and
  $\rho_j=\mathrm{AI}_{\dot E/\dot F}(\dot\chi^j)$.
  The nonzero infinity type makes $\rho_1,\rho_2,\rho_3$ cuspidal.
  Use the corresponding $\Sigma$ and $\dot\phi^H$ in the two families
  preceding \eqref{equation:LIR-global-D4}.
  Taking $\psi_{\dot F}=\psi_\Q\circ\operatorname{Tr}_{\dot F/\Q}$
  makes the real additive characters identical in both globalizations.

  At a finite place $u$, the inducing torus parameter factors through a
  single character. Its centralizer in $\Spin_8(\C)$ is connected by
  \cite{Steinberg1968-EndomorphismsLinearAlgebraicGroups}*{Theorem 8.1}.
  Hence the induced representation of $\PGSO_8(\dot F_u)$ is irreducible
  by \cite{Keys1987-LIndistinguishabilityRGroupsQuasisplit}*{Theorem 2.4
  and Proposition 2.6}.
  Its normalized extension is the one fixing the Whittaker functional.
  Write the pinned torus in coordinates for which $\theta$ permutes the
  first three factors. The norm map is
  \begin{align*}
    (t_1,t_3,t_4,t_2)\longmapsto(t_1t_3t_4,t_2).
  \end{align*}
  Its restriction to the fixed torus is $(x,y)\mapsto(x^3,y)$.
  The compatible centralizer measure contributes $\abs{3}_{\dot F_u}^{-1}$,
  while the discriminant and transfer normalization each contribute
  $\abs{3}_{\dot F_u}^{1/2}$.
  The transfer factor on the related split tori is $1$ by
  \cite{KottwitzShelstad1999-FoundationsTwistedEndoscopy}*{Section 5.3}.
  Thus the toral character identity has coefficient $1$.
  Lemma~\ref{lemma:parabolic-descent-commutes-with-transfer} and the
  Whittaker normalization give the local comparison at $u$.

  Write $I^{\widetilde{\PGSO_8}}_{P,\widetilde r}(\sigma)$ for the real intertwining distribution on the left-hand side of the theorem. Compare it with the transfer of $S^{\SO_4}_{\phi^{\SO_4}(\sigma)}$.
  Their common $D_4$ infinitesimal character $\nu$ is represented by
  $(n,n,0,0)$ or $(2n,n,n,0)$, respectively, and is singular.
  The associated real $K$-space has only the split and compact forms. See \cite{KnusMerkurjevRostTignol1998-BookInvolutions}*{Section 31,
  p.\ 437, following (31.41)}.
  Extend $I^{\widetilde{\PGSO_8}}_{P,\widetilde r}(\sigma)$ by zero on the compact form, which has no
  proper real Levi. The transfer of $S^{\SO_4}_{\phi^{\SO_4}(\sigma)}$ also vanishes there.
  Compatibility with infinitesimal characters and tempered spectral
  transfer give its singular infinitesimal character $\nu$, whereas every
  compact irreducible has regular infinitesimal character. See
  \cite{MoeglinWaldspurger2016-StabilisationFormuleTracesTordue1}*{IV.2.1
  and IV.3.2}.

  Apply the calculation preceding \eqref{equation:LIR-global-D4} on the global
  $K$-space, with Hecke family $c$ and infinitesimal-character tuple
  $(\nu,\ldots,\nu)$ fixed.
  The discrete expansion and stabilization of
  \cite{MoeglinWaldspurger2016-StabilisationFormuleTracesTordue2}*{X.5.1, X.8.1}
  apply with these choices.
  The global inner components with a compact real place contribute zero
  to this infinitesimal-character component.
  The Weyl coefficients and global normalizing factor are those computed
  before \eqref{equation:LIR-global-D4}.
  Fixing matching nonzero factors at the finite places gives, for $d=2,3$,
  \begin{align}\label{equation:LIR-real-tensor-comparison}
    \prod_{u\mid\infty}I^{\widetilde{\PGSO_8}}_{P,\widetilde r}(\sigma)(f_u)
    =\prod_{u\mid\infty}S^{\SO_4}_{\phi^{\SO_4}(\sigma)}((f_u)_H)
      +4S_d^{\rG_2}\left(\bigotimes_{u\mid\infty}(f_u)_{\rG_2}\right),
  \end{align}
  where $S_d^{\rG_2}$ denotes the resulting stable global distribution with the
  finite components fixed.

  Simultaneous cuspidal transfer on the real $K$-space,
  \cite{MoeglinWaldspurger2016-StabilisationFormuleTracesTordue1}*{I.4.11,
  Proposition (ii)}, supplies $f_0$ with
  \[
    (f_0)_{\rG_2}=0,\qquad S^{\SO_4}_{\phi^{\SO_4}(\sigma)}((f_0)_H)=1.
  \]
  The automorphism average of the stable discrete packet character is a
  nonzero sum of discrete-series characters, so it is nonzero on the
  invariant stable cuspidal space.
  Put $c_\R=I^{\widetilde{\PGSO_8}}_{P,\widetilde r}(\sigma)(f_0)$.
  Taking $(f_0,f_0)$ and $(f_0,f)$ in the degree-two identity gives
  $c_\R^2=1$ and
  $c_\R I^{\widetilde{\PGSO_8}}_{P,\widetilde r}(\sigma)(f)=S^{\SO_4}_{\phi^{\SO_4}(\sigma)}(f_H)$.
  The degree-three identity at $(f_0,f_0,f_0)$ gives $c_\R^3=1$.
  Hence $c_\R=1$, proving the desired identity for every $f$.

  Finally, suppose that $F$ is non-archimedean.
  For $M=M_{\alpha_2}$ and $M=M_{\alpha_1,\alpha_3,\alpha_4}$, use the
  corresponding data in the proof of
  Theorem~\ref{theorem:local-tempered-LIR-G2}.
  Equation~\eqref{equation:LIR-global-twisted-separated} and the local
  comparisons away from the distinguished place give the asserted
  identity. For $M=T$, use the short-root globalization with
  principal-series local component and the final induction-in-stages
  calculation in that proof.
\end{proof}

\section{The endoscopic character relations for \texorpdfstring{$L$-parameters}{L-parameters} of \texorpdfstring{$\PGSp_6$-type}{PGSp6-type}}\label{section:ECR-PGSp6}

Let $\phi^{\rG_2} \in \Phi(\rG_2(F))$ be a discrete series parameter of $\rG_2(F)$.

If $F$ is non-archimedean, we have constructed a Xu $L$-packet $\Pi^{X}_{\iota_{\rG_2}^{\PGSO_8} \phi^{\rG_2}}$ whose elements are canonically extended to $\widetilde{\PGSO_8}(F)$.

For archimedean $F$, let $\Pi_{\iota_{\rG_2}^{\PGSO_8} \phi^{\rG_2}}$ be Langlands's $L$-packet. Equivariance under outer automorphisms likewise gives its canonical extension to $\widetilde{\PGSO_8}(F)$.

In either case write $\Pi_{\iota_{\rG_2}^{\PGSO_8} \phi^{\rG_2}}$ for this packet and $\widetilde{\pi}$ for the canonical extension of each member $\pi$ to $\widetilde{\PGSO_8}(F)$.

\begin{definition}
  We set
  \begin{align*}
    I_{\iota_{\rG_2}^{\PGSO_8} \phi^{\rG_2}}^{\widetilde{\PGSO_8}}(f^{\widetilde{\PGSO_8}})
    =
    \sum_{\pi \in \Pi_{\iota_{\rG_2}^{\PGSO_8} \phi^{\rG_2}}} \tr(\widetilde{\pi})(f^{\widetilde{\PGSO_8}}),
  \end{align*}
  for $f^{\widetilde{\PGSO_8}} \in \cI(\widetilde{\PGSO_8}(F))$.
\end{definition}

\begin{proposition}\label{proposition:positive-integrality}
  The distribution $I_{\iota_{\rG_2}^{\PGSO_8} \phi^{\rG_2}}^{\widetilde{\PGSO_8}}$ transfers to a stable distribution $S_{\phi^{\rG_2}}^{\rG_2}$ on $\rG_2(F)$ which is a non-negative integral linear combination of discrete series characters of $\rG_2(F)$.
\end{proposition}

\begin{remark}
  For $F = \R$, the main result of \cite{KalethaMezo2026-RefinedLocalLanglandsConjectureDiscreteLParameters} identifies $S_{\phi^{\rG_2}}^{\rG_2}$ with the stable character of $\phi^{\rG_2}$.
  Henceforth assume that $F$ is non-archimedean.
\end{remark}

\subsection{Proof of Proposition~\ref{proposition:positive-integrality} for \texorpdfstring{$\phi^{\rG_2} \in \Phi(\rG_2)_{\PGSp_6}$}{parameters of PGSp6-type}}

  We assume that $\phi^{\rG_2}$ is a discrete series parameter of $\rG_2$ such that $\Std_{\rG_2} \circ \phi^{\rG_2}$ does not contain the trivial representation of $L_F$.
  \begin{proposition}\label{proposition:ECR-PGSp_6}
    The distribution $I_{\iota_{\rG_2}^{\PGSO_8} \phi^{\rG_2}}^{\widetilde{\PGSO_8}}$ transfers to a stable distribution
    \begin{align*}
      \sum_{\sigma \in \Pi_{\phi^{\rG_2}}} \tr{\sigma}.
    \end{align*}
    Here $\Pi_{\phi^{\rG_2}}$ is the $L$-packet of $\rG_2(F)$ associated with $\phi^{\rG_2}$.
  \end{proposition}

  \begin{lemma}\label{lemma:coefficient-exhaustion-PGSp6}
    Let $\phi\in\Phi_{\disc}(\rG_2(F))_{\PGSp_6}$. Suppose the coefficients of a distribution
    $I=\sum_{\tau\in\Pi_{\disc}(\rG_2(F))}m(\tau)\tr\tau$ are nonnegative integers,
    positive on $\Pi_\phi$, and satisfy $\sum_\tau m(\tau)^2=\abs{\cS_\phi}$.
    Then $I=\sum_{\sigma\in\Pi_\phi}\tr\sigma$.
  \end{lemma}

  \begin{proof}
    Each member of $\Pi_\phi$ contributes at least one to the sum of squares. By Theorem~\ref{theorem:properties-exceptional-theta}, there are $\abs{\cS_\phi}$ such members, so each contributes exactly one and all other coefficients vanish.
  \end{proof}

  \begin{proof}[Proof of Proposition~\ref{proposition:ECR-PGSp_6} for supercuspidal parameters]
  We assume that $\Pi_{\phi^{\rG_2}}$ contains a generic supercuspidal representation $\sigma$ such that $\theta_{\rG_2}^{\PGSp_6}(\sigma)$ is supercuspidal.
  First, we take a totally real number field $\dot{F}$ and a finite place $v_0$ such that $\dot{F}_{v_0} = F$.
  Choose a finite set of finite places $S$ disjoint from $v_0$ with $\abs{S}\geq2$, and fix $w\in S$.
  Applying Lemma~\ref{corollary:globalization-generic-supercuspidal-representation} with auxiliary place $w$ and Steinberg data at $S\setminus\{w\}$, we obtain a globally generic cuspidal automorphic representation $\Sigma$ of $\rG_2(\A_{\dot{F}})$ such that $\Sigma_v$ is unramified for any finite place $v \not \in S \cup \{v_0\}$, $\Sigma_{v_0} = \sigma$ and $\Sigma_S = \otimes_{v \in S} \St_{\rG_2(\dot{F}_v)}$.
  By Corollary~\ref{corollary:summary-global-invariance-of-generic-theta-lift-to-PGSO_8}, the lift $\Pi$ to $\PGSO_8(\A_{\dot{F}})$ exists.
  Let $\phi$ be a global $L$-parameter corresponding to $\Pi$.
  Then we have
  \begin{align*}
    I_{\disc, \iota^{\PGSO_8}_{\rG_2} c(\Sigma)}^{\widetilde{\PGSO_8}} = \sum_{\pi \in \Pi_\phi} \tr \widetilde{\pi}.
  \end{align*}
  The assumption that $\Sigma_S = \otimes_{v \in S} \St_{\rG_2(\dot{F}_v)}$ implies that the global $S$-group of $\Pi$ is trivial.
  Thus, we obtain
  \begin{align*}
    \sum_{\pi \in \Pi_\phi} \tr \widetilde{\pi}
    =
    {\prod_{v}}' I_{\iota_{\rG_2}^{\PGSO_8} \phi^{\rG_2}_v}^{\widetilde{\PGSO_8}},
  \end{align*}
  for some $\phi_v^{\rG_2} \in \Phi(\rG_2(\dot{F}_v))$.
  For finite places $v \notin S \cup \{v_0\}$, we take a spherical function $f_v^{\widetilde{\PGSO_8}}$ and for $v \in S$, we take a pseudo-coefficient $f_v^{\widetilde{\PGSO_8}}$ of the Steinberg representation of $\PGSO_8(\dot{F}_v)$ which is equal to the twisted Euler--Poincar\'e function of $\widetilde{\PGSO_8}(\dot{F}_v)$.
  Then, the transfer to $\rG_2(\dot{F}_v)$ is a spherical function for $v \notin S \cup \{v_0\}$ and a pseudo-coefficient of the Steinberg representation of $\rG_2(\dot{F}_v)$ for $v \in S$.
  Thus, by the simple stable trace formula, we have
  \begin{align*}
    I_{\disc, \iota^{\PGSO_8}_{\rG_2} c(\Sigma)}^{\widetilde{\PGSO_8}}(f^{\widetilde{\PGSO_8}})
    =
    \sum_{\Sigma', \Sigma'_S \cong \otimes_{v \in S} \St_{\rG_2(\dot{F}_v)}, c(\Sigma') = c(\Sigma)} m_{\cusp}(\Sigma') \tr \Sigma'(f^{\widetilde{\PGSO_8}}_{\rG_2}).
  \end{align*}
  The distribution $I_{\iota_{\rG_2}^{\PGSO_8} \phi^{\rG_2}_{v_0}}^{\widetilde{\PGSO_8}}$ is elliptic, and the transfer map
  \begin{align*}
    \cI_{\cusp}(\widetilde{\PGSO_8}(\dot{F}_{v_0})) &\to
     \cS_{\cusp}(\rG_2(\dot{F}_{v_0})) \oplus \cS_{\cusp}(\SO_4(\dot{F}_{v_0})) \oplus \cS_{\cusp}(\SL_3(\dot{F}_{v_0})),
  \end{align*}
  is surjective.
  Thus $I_{\iota_{\rG_2}^{\PGSO_8} \phi^{\rG_2}_{v_0}}^{\widetilde{\PGSO_8}}$ factors through its $\rG_2$-component, giving a stable elliptic distribution on $\rG_2(\dot{F}_{v_0})$.
  At infinite places, $I_{\iota_{\rG_2}^{\PGSO_8} \phi^{\rG_2}_v}^{\widetilde{\PGSO_8}}$ transfers to a multiplicity-free sum of irreducible characters of $\rG_2(\dot{F}_v)$.
  At finite places $v \not \in S \cup \{v_0\}$, take $f_v^{\widetilde{\PGSO_8}}$ to be the spherical Hecke unit of $\widetilde{\PGSO_8}(\dot{F}_v)$. The fundamental lemma transfers it to the spherical Hecke unit of $\rG_2(\dot{F}_v)$.
  Thus at $v_0$, $I_{\iota_{\rG_2}^{\PGSO_8} \phi^{\rG_2}_{v_0}}^{\widetilde{\PGSO_8}}$ transfers to a stable non-negative integral combination of irreducible characters of $\rG_2(\dot{F}_{v_0})$.
  The elliptic inner product formula~\ref{proposition:elliptic-inner-product-transfer} implies that, if we have
  \begin{align*}
    S^{\rG_2}_{\phi^{\rG_2}} = \sum_{\sigma' \in \Pi_{\disc}(\rG_2(F))} m(\sigma') \tr{\sigma'},
  \end{align*}
  then
  \begin{align*}
    \sum_{\sigma'} m(\sigma')^2 = \abs{\cS_{\iota_{\rG_2}^{\PGSO_8} \circ \phi^{\rG_2}}} = \abs{\cS_{\phi^{\rG_2}}}.
  \end{align*}
  We have $m(\sigma) \geq 1$ for the generic member $\sigma$ of $\Pi_{\phi^{\rG_2}}$ and claim that $m(\sigma') \geq 1$ for every $\sigma' \in \Pi_{\phi^{\rG_2}}$.

  Let $\pi'$ be the lift of $\sigma'$ to $\PGSO_8(F)$.
  Then, $\pi'$ is a member of $\Pi_{\iota_{\rG_2}^{\PGSO_8} \phi^{\rG_2}}$.
  Since the global component group is trivial, the multiplicity formula gives a cuspidal member $\Pi' \in \Pi_{\phi}$ such that $\Pi'_{v_0} = \pi'$ and $\Pi'_u = \Pi_u$ for every $u \neq v_0$.
  We apply the local-component replacement argument of \cite{GanSavin2023-LocalLanglandsConjectureG2}*{Section 8.4, proof of Lemma 8.3}.
  Let $\tau$ be the cuspidal theta lift of $\Sigma$ to $\PGSp_6$.
  As in the proof of Proposition~\ref{proposition:invariance-of-global-generic-theta-lift-to-PGSO_8}, its Steinberg component makes $\operatorname{Lift}_{\PGSp_6}^{\GL_7}(\tau)$ cuspidal, and $\operatorname{Lift}_{\PGSO_8}^{\GL_8}(\Pi')=\mathbf1\boxplus\operatorname{Lift}_{\PGSp_6}^{\GL_7}(\tau)$.
  Thus, for a sufficiently large finite set $T$ of places,
  \begin{align*}
    L^T(s,\Pi',\Std_{\PGSO_8})
    =\zeta_{\dot F}^T(s)L^T(s,\operatorname{Lift}_{\PGSp_6}^{\GL_7}(\tau))
  \end{align*}
  has a simple pole at $s=1$.
  The classical theta non-vanishing theorem \cite{GanTakeda2011-RegularizedSiegelWeilNonvanishingOrthogonal}*{Theorem 7.7(a)}, with $m=8$ and $r=3$, therefore gives a nonzero theta lift of $\Pi'$ to $\PGSp_6$.
  Here one applies the theorem to the restriction to $\SO_8$ and passes to similitudes.
  At any $u\in S$, the local Steinberg representation has first occurrence at $\Sp_6$ and has no theta lift to $\Sp_4$. See \cite{AtobeGan2017-LocalThetaCorrespondenceTemperedRepresentationsLanglands}*{Section 4}.
  The tower property consequently makes the nonzero global lift cuspidal. See \cite{GanTakeda2011-RegularizedSiegelWeilNonvanishingOrthogonal}*{proof of Theorem 7.7}.
  Denote this cuspidal representation of $\PGSp_6(\A_{\dot F})$ by $\tau'$.
  Local Howe duality gives $\tau'_{v_0}=\theta_{\rG_2}^{\PGSp_6}(\sigma')$ and $\tau'_u=\tau_u$ for $u\neq v_0$.
  As in \cite{GanSavin2023-LocalLanglandsConjectureG2}*{Section 8.4}, we have
  \begin{align*}
    L^T(s,\tau',\mathrm{Spin})
    =L^T(s,\tau,\mathrm{Spin})
    =\zeta_{\dot F}^T(s)L^T(s,\operatorname{Lift}_{\PGSp_6}^{\GL_7}(\tau)),
  \end{align*}
  so this Spin $L$-function also has a pole at $s=1$.
  We may now apply \cite{GanSavin2020-ExceptionalSiegelWeilFormulaPolesSpin}*{Theorem 10.1} to $\tau'$ to find an inner form $G'$ of $\rG_2$ with a nonzero cuspidal backward theta lift.
  At $v_0$, the group $G'$ is split and the local component of this lift is $\sigma'$ by Howe duality.
  As $\abs{S} \geq 2$, we can apply the simple trace formula to compare the trace formula for $\rG_2$ and $G'$.
  We have
  \begin{align*}
    I_{\disc, c(\Sigma)}^{\rG_2}(f^{\rG_2})
    =
    I_{\disc, c(\Sigma)}^{G'}(f^{G'}),
  \end{align*}
  if we take the Euler--Poincar\'e function at places in $S$.
  The cuspidal sum above therefore contains a representation $\Sigma'$ with $\Sigma'_{v_0}=\sigma'$. Hence $m(\sigma')\geq1$ throughout the packet, and Lemma~\ref{lemma:coefficient-exhaustion-PGSp6} gives the required character sum.
  \end{proof}

  \begin{proof}[Proof of Proposition~\ref{proposition:ECR-PGSp_6} in the general case]
    We take a totally real number field $\dot{F}$ and finite places $v_0, v_1$ such that $\dot{F}_{v_0} = F$.
    We keep the local factors at every real place fixed in the comparison, as in the supercuspidal case.
    We fix a generic supercuspidal representation $\sigma$ of $\rG_2(\dot{F}_{v_1})$ such that $\theta_{\rG_2}^{\PGSp_6}(\sigma)$ is supercuspidal.
    Let $\sigma_0$ be the generic discrete series representation of $\rG_2(F)$ corresponding to $\phi^{\rG_2}$.
    Choose a finite set of finite places $S$ disjoint from $\{v_0,v_1\}$ with $\abs{S}\geq2$, and fix $w\in S$.
    Applying Lemma~\ref{corollary:globalization-generic-supercuspidal-representation} with supercuspidal component $\sigma$ at $v_1$ and auxiliary place $w$, we obtain a globally generic cuspidal automorphic representation $\Sigma$ of $\rG_2(\A_{\dot{F}})$ such that $\Sigma_v$ is unramified for any finite place $v \not \in S \cup \{v_0, v_1\}$, $\Sigma_{v_0} = \sigma_0$, $\Sigma_{v_1} = \sigma$ and $\Sigma_S = \otimes_{v \in S} \St_{\rG_2(\dot{F}_v)}$.

    At $v_1$, the supercuspidal case gives the packet sum with coefficient $1$.
    This multiplicity-free factor isolates $v_0$, where the same argument shows that $I_{\iota_{\rG_2}^{\PGSO_8} \phi^{\rG_2}_{v_0}}^{\widetilde{\PGSO_8}}$ transfers to a stable non-negative integral combination of irreducible characters of $\rG_2(\dot{F}_{v_0})$.
    We write the distribution as
    \begin{align*}
      S^{\rG_2}_{\phi^{\rG_2}} = \sum_{\sigma' \in \Pi_{\disc}(\rG_2(F))} m(\sigma') \tr{\sigma'}.
    \end{align*}
    The same argument gives $m(\sigma') \geq 1$ for every $\sigma' \in \Pi_{\phi^{\rG_2}}$.
    The elliptic inner-product formula gives, as in the supercuspidal case,
    \begin{align*}
      \sum_{\sigma'\in\Pi_{\disc}(\rG_2(F))}m(\sigma')^2
      =\abs{\cS_{\phi^{\rG_2}}}.
    \end{align*}
    Lemma~\ref{lemma:coefficient-exhaustion-PGSp6} therefore gives
    $m(\sigma')=1$ on $\Pi_{\phi^{\rG_2}}$ and $m(\sigma')=0$ outside
    this packet.
  \end{proof}

  It remains to treat a general discrete series parameter $\phi^{\rG_2}$ of $\rG_2$.
  \begin{proof}[Proof of Proposition~\ref{proposition:positive-integrality}]
    The same argument, applying Lemma~\ref{corollary:globalization-generic-supercuspidal-representation} to the generic member of $\Pi_{\phi^{\rG_2}}$ at $v_0$, shows that $I_{\iota_{\rG_2}^{\PGSO_8} \phi^{\rG_2}_{v_0}}^{\widetilde{\PGSO_8}}$ transfers to a stable non-negative integral combination of irreducible characters of $\rG_2(\dot{F}_{v_0})$.
  \end{proof}

\section{The endoscopic character relations for \texorpdfstring{$L$-parameters}{L-parameters} of \texorpdfstring{$\PGL_3$-type}{PGL3-type}}\label{section:ECR-PGL3}

\begin{theorem}[\cite{Takanashi2025-NoteSauvageotDensityPrinciple}]\label{theorem:globalization-PGL3}
  Let $F$ be a totally real number field.
  Let $S$ be a finite set of finite places of $F$ and for each $v \in S$, let $\tau_v$ be an irreducible discrete series representation of $\PGL_3(F_v)$.
  There is a cuspidal automorphic representation $\dot{\tau}$ of $\PGL_3(\A_F)$ that satisfies the following conditions.
  \begin{itemize}
    \item $\dot{\tau}_v = \tau_v$ for any $v \in S$.
    \item $\dot{\tau}_v$ is unramified for any finite place $v \notin S$.
    \item $\dot{\tau}_v$ is a spherical principal series representation with sufficiently regular infinitesimal character at every infinite place $v$, so that the lifts of $\dot{\tau}_v$ to $\rG_2(F_v)$ and $\PGSO_8(F_v)$ are irreducible spherical principal series representations.
    \item $\dot{\tau} \neq \dot{\tau}^{\vee}$.
  \end{itemize}
\end{theorem}

\begin{lemma}\label{lemma:description-induction-from-A2-Levi}
  Let $F$ be a non-archimedean local field of characteristic zero.
  We use the following notation.
  \begin{itemize}
    \item
    We fix a Borel pair $(B, T)$ of $\PGSO_8$ and let $\alpha_1, \alpha_2, \alpha_3, \alpha_4$ be the simple roots corresponding to $T$ following Bourbaki's numbering.
    We also fix a pinning of $\PGSO_8$ and let $\Aut_{\mathrm{pin}}(\PGSO_8)$ be the group of automorphisms of $\PGSO_8$ preserving the pinning.
    We set $G^{+} = \PGSO_8 \rtimes \Aut_{\mathrm{pin}}(\PGSO_8)$.
    \item
    Let $\tau$ be a discrete series representation of $\PGL_3(F)$.
    We see it as a representation of the standard Levi subgroup $M = (\GL_3(F) \times \GL_1(F) \times \GL_1(F))/\GL_1(F)$ of $\PGSO_8(F)$ with trivial central character.
    Let $P$ be the standard parabolic subgroup of $\PGSO_8$ with Levi subgroup $M$.
    Let $I_P(\tau)$ be the normalized parabolic induction of $\tau$ to $\PGSO_8(F)$.
  \end{itemize}

  Then the following assertions hold.

  \begin{enumerate}
    \item
    We set $I = \{\alpha_1, \alpha_2\}$.
    Let $\mathrm{Fix}_{W(G^{+}, T)}(I)$ be the subgroup of $W(G^{+}, T)$ consisting of elements that fix $I$.
    The natural map $\mathrm{Fix}_{W(G^{+}, T)}(I) \to \Out(\PGSO_8)(F)$ is an isomorphism.
    \item
    For each $w^{+} \in \mathrm{Fix}_{W(G^{+}, T)}(I)$, we have the normalized intertwining operator $R_{P}(w^{+}, \tau, \psi_F): I_P(\tau) \to I_P(\tau)$.
    This extends the representation $I_P(\tau)$ to a representation of $G^{+}(F)$ and restricts to the canonical extension $\widetilde{I_P(\tau)}$ to $\widetilde{\PGSO_8}(F)$.
    \item
    The $R$-group $R^{\PGSO_8(F)}_{\tau}$ is isomorphic to $\Z/2\Z$ if $\tau$ is self-dual and to $1$ otherwise.
    Each irreducible component is preserved by the action of the group $\mathrm{Fix}_{W(G^{+}, T)}(I) \cong \Out(\PGSO_8)(F)$.
  \end{enumerate}
\end{lemma}

\begin{proof}
  It suffices to show the first assertion, but this follows from
  \begin{align*}
    \mathrm{Fix}_{W(\PGSO_8, T)}(I) = \{1\}.
  \end{align*}
\end{proof}

\begin{lemma}\label{lemma:elliptic-norm-PGL3}
  Let $F$ be a non-archimedean local field of characteristic zero.
  Let $\phi^{\PGL_3}$ be a discrete series parameter of $\PGL_3(F)$.
  The squared elliptic norm of $I_{\iota_{\PGL_3}^{\PGSO_8} \circ \phi^{\PGL_3}}^{\widetilde{\PGSO_8}}$ is $3$ if $\phi^{\PGL_3}$ is not self-dual and $6$ otherwise.
\end{lemma}

\begin{proof}
  Set $G=\PGSO_8$ and $M=M_{\{\alpha_1,\alpha_2\}}$, and let $\tau$ be the discrete series representation of $\PGL_3(F)$ corresponding to $\phi^{\PGL_3}$.
  Regard $\tau$ as a representation of $M(F)$ and let $I_P(\tau)$ be its normalized parabolic induction, as in Lemma~\ref{lemma:description-induction-from-A2-Levi}.
  In the notation of Lemma~\ref{lemma:regular-element-not-theta-stable-case}, put
  \begin{align*}
    \widetilde u=w'_2\rtimes\theta,\quad r=w'_1=w_0w_0^M.
  \end{align*}
  Write $I^{\widetilde{\PGSO_8}}_{P,\widetilde w}(\tau)$ for the twisted character of $I_P(\tau)$ defined by the normalized intertwining operator $R_P(\widetilde w,\tau,\psi_F)$. These operators give the canonical extension $\widetilde{I_P(\tau)}$.
  Compatibility of the packet with parabolic induction, including the canonical extensions of its constituents, gives
  \begin{align*}
    I^{\widetilde{\PGSO_8}}_{P,\widetilde u}(\tau):=\tr\widetilde{I_P(\tau)}
    =I_{\iota_{\PGL_3}^{\PGSO_8}\circ\phi^{\PGL_3}}^{\widetilde{\PGSO_8}}.
  \end{align*}

  By Lemma~\ref{lemma:standard-parameters-of-theta-unstable-cases}, $W_0(\tau)=\{1\}$.
  Hence the $R$-group is
  \begin{align*}
    R^G(\tau)=W(G,M)_\tau
    =\begin{cases}
      \{1\},&\tau\not\cong\tau^\vee,\\
      \{1,r\},&\tau\cong\tau^\vee.
    \end{cases}
  \end{align*}
  In either case, $\widetilde u$ centralizes $R^G(\tau)$.
  Since $G$ is adjoint, $\fa_{\widetilde G}=0$.
  The characteristic polynomial of $\widetilde u$ on the two-dimensional space $\fa_M$ is $X^2+X+1$ by Lemma~\ref{lemma:regular-element-not-theta-stable-case}. Thus
  \begin{align*}
    \abs{\det(1-\widetilde u\mid\fa_M)}=1+1+1=3.
  \end{align*}
  The elliptic triple $(M,\tau,\widetilde u)$ therefore has squared norm
  \begin{align*}
    (I^{\widetilde{\PGSO_8}}_{P,\widetilde u}(\tau),I^{\widetilde{\PGSO_8}}_{P,\widetilde u}(\tau))_{\mathrm{ell}}
    &=\abs{\operatorname{Cent}_{R^G(\tau)}(\widetilde u)}
      \abs{\det(1-\widetilde u\mid\fa_M)}\\
    &=\begin{cases}
      1\cdot3=3,&\tau\not\cong\tau^\vee,\\
      2\cdot3=6,&\tau\cong\tau^\vee,
    \end{cases}
  \end{align*}
  by \cite{MoeglinWaldspurger2018-FormuleTracesLocaleTordue}*{7.3, Th\'eor\`eme}.

  More explicitly, in the self-dual case, write $I_P(\tau)=\pi_+\oplus\pi_-$, where the normalized self-intertwining operator attached to $r$ acts by $+1$ and $-1$, respectively.
  Then
  \begin{align*}
    I^{\widetilde{\PGSO_8}}_{P,\widetilde u}(\tau)&=\tr\widetilde\pi_++\tr\widetilde\pi_-,\\
    I^{\widetilde{\PGSO_8}}_{P,r\widetilde u}(\tau)&=\tr\widetilde\pi_+-\tr\widetilde\pi_-.
  \end{align*}
  The characteristic polynomial of $r\widetilde u$ on $\fa_M$ is $X^2-X+1$, so $\abs{\det(1-r\widetilde u\mid\fa_M)}=1$.
  The same theorem gives
  \begin{align*}
    (I^{\widetilde{\PGSO_8}}_{P,r\widetilde u}(\tau),I^{\widetilde{\PGSO_8}}_{P,r\widetilde u}(\tau))_{\mathrm{ell}}=2,\quad
    (I^{\widetilde{\PGSO_8}}_{P,\widetilde u}(\tau),I^{\widetilde{\PGSO_8}}_{P,r\widetilde u}(\tau))_{\mathrm{ell}}=0,
  \end{align*}
  where the second equality follows because $\widetilde u$ and $r\widetilde u$ represent distinct conjugacy classes under $R^G(\tau)$.
  Consequently, the Gram matrix of the two constituent characters is
  \begin{align*}
    \bigl((\tr\widetilde\pi_\epsilon,\tr\widetilde\pi_{\epsilon'})_{\mathrm{ell}}\bigr)_{\epsilon,\epsilon'\in\{+,-\}}
    =\begin{pmatrix}2&1\\1&2\end{pmatrix},
  \end{align*}
  and the squared norm of their sum is $2+2+2\cdot1=6$.
\end{proof}

\begin{corollary}\label{corollary:coefficient-square-sum-PGL3}
  The distribution $S_{\iota_{\PGL_3}^{\rG_2} \circ \phi^{\PGL_3}}^{\rG_2}$ is a non-negative integral linear combination of irreducible discrete series characters of the form
  \begin{align*}
    \sum_{\sigma' \in \Pi_{\disc}(\rG_2(F))} m(\sigma') \tr{\sigma'},
  \end{align*}
  with
  \begin{align*}
    \begin{cases}
      \sum_{\sigma' \in \Pi_{\disc}(\rG_2(F))} m(\sigma')^2 = 3 & \text{if $\phi^{\PGL_3}$ is not self-dual}, \\
      \sum_{\sigma' \in \Pi_{\disc}(\rG_2(F))} m(\sigma')^2 = 6 & \text{if $\phi^{\PGL_3}$ is self-dual}.
    \end{cases}
  \end{align*}
\end{corollary}

\begin{proof}
  The non-negativity and integrality follow from Proposition~\ref{proposition:positive-integrality}.
  Since $c(\widetilde{\PGSO_8},\rG_2)=1$, Proposition~\ref{proposition:elliptic-inner-product-transfer} identifies the squared elliptic norm of this stable distribution with that of $I_{\iota_{\PGL_3}^{\PGSO_8}\circ\phi^{\PGL_3}}^{\widetilde{\PGSO_8}}$.
  The discrete series characters of $\rG_2(F)$ are orthonormal for the elliptic inner product, so this squared norm is $\sum_{\sigma'}m(\sigma')^2$.
  The result now follows from Lemma~\ref{lemma:elliptic-norm-PGL3}.
\end{proof}

\begin{proposition}
    For each elliptic endoscopic group $H$ of $\widetilde{\PGSO_8}$ (resp. $\rG_2$), every stable linear combination $S^H$ of irreducible elliptic tempered characters of $H(F)$ transfers to a linear combination of elliptic irreducible tempered characters of $\widetilde{\PGSO_8}(F)$ (resp. $\rG_2(F)$).
\end{proposition}

\begin{remark}
  For each discrete series representation $\tau$ of $\PGL_3(F)$, let $I_{\tau}^{\rG_2}$ be the transfer to $\rG_2(F)$, i.e.
  \begin{align*}
    I_{\tau}^{\rG_2}(f^{\rG_2}) = \tr\tau (f^{\rG_2}_{\PGL_3}),
  \end{align*}
  for $f^{\rG_2} \in \cI(\rG_2(F))$.
  Then, $I_{\tau}^{\rG_2}$ is orthogonal to any transfer of a stable discrete series character of $\SO_4(F)$ and thus orthogonal to any non-discrete-series elliptic character of $\rG_2(F)$.
  We have
  \begin{align*}
    (I_{\tau}^{\rG_2}, I_{\tau}^{\rG_2})_{\mathrm{ell}}
    =c(\rG_2,\PGL_3)^{-1}(\tr\tau,\tr\tau)_{\mathrm{ell}}
    =(1/3)^{-1}\cdot1=3,
  \end{align*}
  by Proposition~\ref{proposition:elliptic-inner-product-transfer} and the elliptic orthonormality of discrete series characters of $\PGL_3(F)$.
\end{remark}

We will describe the transfer from $\PGL_3(F)$ to $\rG_2(F)$.

\begin{theorem}\label{theorem:ECR-PGL3}
  Let $F$ be a non-archimedean local field of characteristic zero, and let $\phi^{\PGL_3}$ be a discrete series parameter of $\PGL_3(F)$.
  We fix an endoscopic datum $\mathfrak{e} = (\PGL_3, s, \iota^{\rG_2}_{\PGL_3})$.
  Let $\tau$ be the discrete series representation of $\PGL_3(F)$ corresponding to $\phi^{\PGL_3}$.
  \begin{enumerate}
    \item
    If $\tau$ is not self-dual, then the distribution $S_{\iota_{\PGL_3}^{\rG_2} \circ \phi^{\PGL_3}}^{\rG_2}$ is the sum of three discrete series characters of $\rG_2(F)$.
    The discrete series representations are parameterized by the three characters of $\cS_{\iota_{\PGL_3}^{\rG_2} \circ \phi^{\PGL_3}} = \mu_3(\C)$, and for each $\rho \in \Irr_{\C}(\cS_{\iota_{\PGL_3}^{\rG_2} \circ \phi^{\PGL_3}})$, let $\sigma(\rho)$ be the corresponding discrete series representation of $\rG_2(F)$.
    Then, we have
    \begin{align*}
      I_{\tau}^{\rG_2}(f^{\rG_2}) &= \sum_{\rho \in \Irr_{\C}(\cS_{\iota_{\PGL_3}^{\rG_2} \circ \phi^{\PGL_3}})} \rho(s) \tr{\sigma(\rho)}(f^{\rG_2}), \\
      I_{\tau^{\vee}}^{\rG_2}(f^{\rG_2}) &= \sum_{\rho \in \Irr_{\C}(\cS_{\iota_{\PGL_3}^{\rG_2} \circ \phi^{\PGL_3}})} \rho(s^{-1}) \tr{\sigma(\rho)}(f^{\rG_2}),
    \end{align*}
    for $f^{\rG_2} \in \cI(\rG_2(F))$.
    \item
    If $\tau$ is self-dual, then the distribution $S_{\iota_{\PGL_3}^{\rG_2} \circ \phi^{\PGL_3}}^{\rG_2}$ is the dimension-weighted sum of the three discrete series characters of $\rG_2(F)$.
    The discrete series representations are parameterized by the three irreducible representations of $\cS_{\iota_{\PGL_3}^{\rG_2} \circ \phi^{\PGL_3}} = S_3$, and for each $\rho \in \Irr_{\C}(\cS_{\iota_{\PGL_3}^{\rG_2} \circ \phi^{\PGL_3}})$, let $\sigma(\rho)$ be the corresponding discrete series representation of $\rG_2(F)$.
    Then, we have
    \begin{align*}
      I_{\tau}^{\rG_2}(f^{\rG_2}) &= \sum_{\rho \in \Irr_{\C}(\cS_{\iota_{\PGL_3}^{\rG_2} \circ \phi^{\PGL_3}})} \tr\rho(s) \tr{\sigma(\rho)}(f^{\rG_2}), \\
      I_{\tau^{\vee}}^{\rG_2}(f^{\rG_2}) &= \sum_{\rho \in \Irr_{\C}(\cS_{\iota_{\PGL_3}^{\rG_2} \circ \phi^{\PGL_3}})} \tr\rho(s^{-1}) \tr{\sigma(\rho)}(f^{\rG_2}),
    \end{align*}
    for $f^{\rG_2} \in \cI(\rG_2(F))$.
    \item
    The representations $\sigma(\rho)$ are obtained as theta lifts of $\tau$ from $\PGL_3^{+}(F)$ or of $JL(\tau)$ from an inner form $\mathrm{PD}^{\times}$ of $\PGL_3(F)$.
  \end{enumerate}
\end{theorem}

\begin{proof}
  We combine the globalization argument of Proposition~\ref{proposition:ECR-PGSp_6} with exceptional theta lifting from $\PGL_3^{+}$ and inner forms of $\PGL_3$ to $\rG_2$, and its converse.

  Let $\dot{F}$ be a totally real number field.
  Let $S_{F}$ be a finite set of finite places such that for any $v \in S_{F}$, we have $\dot{F}_v = F$.
  Let $S_{\infty}$ denote the set of infinite places of $\dot{F}$.
  Then, by Theorem~\ref{theorem:globalization-PGL3}, we can find a cuspidal automorphic representation $\dot{\tau}$ of $\PGL_3(\A_{\dot{F}})$ such that $\dot{\tau}_v = \tau$ for any $v \in S_{F}$, $\dot{\tau}_v$ is unramified for any finite place $v \notin S_{F}$ and $\dot{\tau}_w$ is a spherical principal series representation with sufficiently regular infinitesimal character at every infinite place $w$.

  For $c = \iota_{\PGL_3}^{\PGSO_8} \circ c(\dot{\tau})$, compute the proper-Levi contributions as in Proposition~\ref{proposition:ECR-PGSp_6} to obtain
  \begin{align*}
    I_{\disc, c}^{\widetilde{\PGSO_8}}(f^{\widetilde{\PGSO_8}})
    =
    \frac{1}{3}
    {\prod_{v}}^{'} \tr(\widetilde{I_P(\dot{\tau}_v)})(f_v^{\widetilde{\PGSO_8}}),
  \end{align*}
  by Lemma~\ref{lemma:description-induction-from-A2-Levi} and the decomposition of global intertwining operators into local ones.
  By the stabilization of the trace formula, this is equal to
  \begin{align*}
     \sum_{H} \iota(\widetilde{\PGSO_8}, H) S^H_{\disc, c}(f^{\widetilde{\PGSO_8}}_{H}),
  \end{align*}
  with
  \begin{align*}
    S^{\SO_4}_{\disc, c} = 0, \quad S^{\SL_3}_{\disc, c} = 0.
  \end{align*}
  Thus, we obtain
  \begin{align*}
    \frac{1}{3}
    {\prod_{v}}^{'}
    S_{\iota_{\PGL_3}^{\rG_2} \circ \phi^{\PGL_3}_v}^{\rG_2}
    ((f^{\widetilde{\PGSO_8}}_{\rG_2})_v)
    =
    S^{\rG_2}_{\disc, c}(f^{\widetilde{\PGSO_8}}_{\rG_2}).
  \end{align*}
  Surjectivity of $\cI(\widetilde{\PGSO_8}) \to \cS(\rG_2)$ gives
  \begin{align*}
    S^{\rG_2}_{\disc, c}(f^{\rG_2})
    =
    \frac{1}{3}
    {\prod_{v}}^{'}
    S_{\iota_{\PGL_3}^{\rG_2} \circ \phi^{\PGL_3}_v}^{\rG_2}
    (f_v^{\rG_2}),
  \end{align*}
  for any $f^{\rG_2}_v \in \cS(\rG_2(\dot{F}_v))$.

  Substituting into the stable trace formula for $\rG_2$ and multiplying by $3$, we obtain
  \begin{align*}
    3 I_{\disc, c}^{\rG_2}(f^{\rG_2})
    =
    {\prod_{v}}^{'}
    S_{\iota_{\PGL_3}^{\rG_2} \circ \phi^{\PGL_3}_v}^{\rG_2}
    (f_v^{\rG_2})
    +
    {\prod_{v}}^{'}
    I_{\dot{\tau}_v}^{\rG_2}(f_v^{\rG_2})
    +
    {\prod_{v}}^{'}
    I_{\dot{\tau}_v^{\vee}}^{\rG_2}(f_v^{\rG_2}).
  \end{align*}

  First, we observe that the distribution at $w \in S_{\infty}$ is given by the characters of tempered (spherical) principal series representations of $\rG_2(\dot{F}_w)$.
  Let $\Sigma_w$ denote the representation of $\rG_2(\dot{F}_w)$ corresponding to the parameter $\iota_{\PGL_3}^{\rG_2} \circ \phi^{\PGL_3}_w$.
  For each finite place $v\notin S_F$, let $\Sigma_v$ be the spherical representation of $\rG_2(\dot{F}_v)$ with parameter $\iota_{\PGL_3}^{\rG_2}\circ\phi_v^{\PGL_3}$.
  Comparison of the standard $L$-parameters shows that the proper Levi subgroups contribute $0$.
  Since the left-hand side is tempered at the real places, we have
  \begin{align*}
    I_{\disc, c}^{\rG_2}(f^{\rG_2})
    =
    \sum_{\substack{\Sigma' \in \Pi_{\cusp}(\rG_2(\A_{\dot{F}}))\\
      c(\Sigma')=\iota_{\PGL_3}^{\rG_2}\circ c(\dot\tau)}}
      m_{\cusp}(\Sigma') \tr \Sigma' (f^{\rG_2}).
  \end{align*}
  Comparing the spherical Hecke characters at finite places outside $S_F$ and the irreducible characters at infinity, we obtain
  \begin{align*}
    3 \sum_{\substack{
      \Sigma' \in \Pi_{\cusp}(\rG_2(\A_{\dot{F}}))\\
      c(\Sigma')=\iota_{\PGL_3}^{\rG_2}\circ c(\dot\tau)\\
      \Sigma'_v\cong\Sigma_v\ (v\notin S_F)}}
      &m_{\cusp}(\Sigma') \prod_{v \in S_F} \tr \Sigma'_v (f_v^{\rG_2}) \\
    &=
    \prod_{v \in S_F}
    S_{\iota_{\PGL_3}^{\rG_2} \circ \phi^{\PGL_3}_v}^{\rG_2} (f_v^{\rG_2})
    +
    \prod_{v \in S_F}
    I_{\dot{\tau}_v}^{\rG_2}(f_v^{\rG_2})
    +
    \prod_{v \in S_F}
    I_{\dot{\tau}_v^{\vee}}^{\rG_2}(f_v^{\rG_2}).
  \end{align*}

  Put $\phi=\iota_{\PGL_3}^{\rG_2}\circ\phi^{\PGL_3}$.
  For each distribution $I \in \{S_\phi^{\rG_2}, I^{\rG_2}_{\tau}, I^{\rG_2}_{\tau^\vee}\}$ and each discrete series representation $\sigma' \in \Pi_{\disc}(\rG_2(F))$, let $c_I(\sigma')$ denote the coefficient of $\tr \sigma'$ in $I$.
  We first check that these coefficients vanish outside the Gan--Savin packet $\Pi_\phi$.
  Every cuspidal realization $\Sigma'$ on the fixed Satake component on the left hand side has a nonzero backward theta lift to $\PGL_3^+$ or an inner form $\mathrm{PD}^{\times}$ by Theorem~\ref{theorem:exceptional-theta-G2-to-PGL3}.
  Cuspidal-support decomposition and uniqueness of isobaric decomposition exclude noncuspidal split sources on this Satake component. Strong multiplicity one then identifies every source constituent with $\dot\tau$ or $\dot\tau^\vee$, up to extension to $\PGL_3^+$ or Jacquet--Langlands transfer.
  Local compatibility of the exceptional theta correspondence with parameters (Theorem~\ref{theorem:properties-exceptional-theta} and Lemma~\ref{lemma:injectivity-G_2-to-GL7}) therefore implies that $\Sigma'_v\in\Pi_\phi$ for every $v\in S_F$.

  For $\sigma'\notin\Pi_\phi$, the coefficient of $\boxtimes_{v\in S_F}\sigma'$ on the cuspidal side is zero. Write $a_1,a_2,a_3$ for the coefficients of $\tr\sigma'$ in $S_\phi^{\rG_2},I^{\rG_2}_{\tau},I^{\rG_2}_{\tau^\vee}$, respectively. Linear independence of characters gives $\sum_{i=1}^3a_i^n=0$ for $n=\abs{S_F}$. Taking $n=1,2,3$ in the globalization, Newton's identities give $\prod_{i=1}^3(X-a_i)=X^3$. Thus
  \begin{align}\label{equation:packet-support-PGL3}
    c_{S_\phi^{\rG_2}}(\sigma')=c_{I^{\rG_2}_{\tau}}(\sigma')=c_{I^{\rG_2}_{\tau^\vee}}(\sigma')=0
    \quad(\sigma'\notin\Pi_\phi).
  \end{align}

  Let $\dot\tau^+$ be the generic extension of $\dot\tau$ and put $\cV=\Theta(\dot\tau^+)$.
  By Theorem~\ref{theorem:theta-PGL3-to-G2}, $\cV$ is cuspidal and has a nonzero globally generic projection.
  Since the cuspidal spectrum is semisimple, local Howe duality at finite places and spherical Howe duality at real places imply that $\cV$ is irreducible. See \cite{GanSavin2023-HoweDualityDichotomyExceptionalThetaCorrespondences}*{Theorem 8.5(ii)} and \cite{LokeSavin2019-DualitySphericalRepresentationsExceptionalThetaCorrespondences}*{Corollaries 4.5 and 4.8}.
  Thus $\cV$ coincides with its globally generic projection, and its local component at each $v\in S_F$ is the generic member of $\Pi_\phi$ by \cite{GanSavin2023-HoweDualityDichotomyExceptionalThetaCorrespondences}*{Theorem 8.5(iii)}.

  We claim that $\cV$ occurs with multiplicity one in the cuspidal spectrum.
  By local Howe duality and the source classification above, its only possible source representation on $\PGL_3^+$ is $\dot\tau^+$, where $\dot\tau$ and $\dot\tau^\vee$ give the same source. This source occurs with multiplicity one by strong multiplicity one for $\PGL_3$ and the classification of cuspidal representations of $\PGL_3^+$.
  A division-algebra source is impossible, since local dichotomy excludes the generic target at a nonsplit finite place and a degree-three algebra is split at every real place.
  If an additional cuspidal copy were orthogonal to $\cV$, its backward lift to $\dot\tau^+$ would be nonzero by Theorem~\ref{theorem:exceptional-theta-G2-to-PGL3}, but zero by orthogonality to $\cV$. The integrals converge absolutely by rapid decrease of both cusp forms and moderate growth of the theta kernel. This proves the claim, as in \cite{GanSavin2022-TwistedCompositionAlgebrasArthurPacketsTrialitySpin8}*{proof of Corollary 14.6(ii)}.

  Comparing coefficients of the generic member $\sigma_{\gen}$ of $\Pi_\phi$ gives
  \begin{align*}
    \sum_{I \in \{S_\phi^{\rG_2}, I^{\rG_2}_{\tau}, I^{\rG_2}_{\tau^\vee}\}} c_I(\sigma_{\gen})^{\abs{S_F}} = 3.
  \end{align*}

  Varying $\dot{F}$ makes $\abs{S_F}$ arbitrary, so $c_{S_\phi^{\rG_2}}(\sigma_{\gen}) = 1$ and $c_{I^{\rG_2}_{\tau}}(\sigma_{\gen}) = c_{I^{\rG_2}_{\tau^\vee}}(\sigma_{\gen}) = 1$.
  The injection into the irreducible representations of the component group in Theorem~\ref{theorem:properties-exceptional-theta} bounds the packet size by three.
  By \eqref{equation:packet-support-PGL3} and Corollary~\ref{corollary:coefficient-square-sum-PGL3}, the squares of the remaining non-negative integral coefficients sum to $2$ or $5$.
  Neither number is a single integer square, so there are exactly two further packet members and both have nonzero coefficients.
  Let $\sigma_1, \sigma_2$ be the other two representations in the packet $\Pi_\phi$.
  In the non-self-dual case both remaining coefficients are $1$.
  By the packet description in \cite{GanSavin2023-LocalLanglandsConjectureG2}*{proof of Proposition 3.2}, we may assume that $\sigma_1 = \theta_{\mathrm{PD}^{\times}}^{\rG_2}(JL_{D}(\tau))$ and $\sigma_2 = \theta_{\mathrm{PD}^{\times}}^{\rG_2}(JL_{D}(\tau^{\vee}))$, where $JL_{D}(\tau)$ is the Jacquet--Langlands transfer of $\tau$ to an inner form $\mathrm{PD}^{\times}$ of $\PGL_3(F)$.

  By the Brauer--Hasse--Noether theorem, the global division algebra $\dot{D}$ unramified outside $S_F$ with $\dot{D}_v = D$ for all $v \in S_F$ does not exist if $\abs{S_F} \not\equiv 0 \pmod{3}$.
  This implies that
  \begin{align*}
    \sum_{I \in \{S_\phi^{\rG_2}, I^{\rG_2}_{\tau}, I^{\rG_2}_{\tau^\vee}\}} c_I(\sigma_1)^{\abs{S_F}} = 0, \\
    \sum_{I \in \{S_\phi^{\rG_2}, I^{\rG_2}_{\tau}, I^{\rG_2}_{\tau^\vee}\}} c_I(\sigma_2)^{\abs{S_F}} = 0,
  \end{align*}
  unless $\abs{S_F} \equiv 0 \pmod{3}$.
  Since $c_{S_\phi^{\rG_2}}(\sigma_1) = c_{S_\phi^{\rG_2}}(\sigma_2) = 1$, we have
  \begin{align*}
    \{ c_{I^{\rG_2}_{\tau}}(\sigma_i), c_{I^{\rG_2}_{\tau^\vee}}(\sigma_i) \} = \{ \omega, \omega^{-1}\},
  \end{align*}
  for $i = 1, 2$, where $\omega = \frac{-1 + \sqrt{3}i}{2}$ is a primitive third root of unity.
  Elliptic orthogonality of $I^{\rG_2}_{\tau}$ and $I^{\rG_2}_{\tau^\vee}$ gives, after switching $\sigma_1$ and $\sigma_2$ if necessary, $c_{I^{\rG_2}_{\tau}}(\sigma_1) = \omega$ and $c_{I^{\rG_2}_{\tau}}(\sigma_2) = \omega^{-1}$.

  Now suppose that $\tau$ is self-dual. As above, $c_{S_\phi^{\rG_2}}(\sigma_{\gen}) = 1$ and $c_{I^{\rG_2}_{\tau}}(\sigma_{\gen}) = c_{I^{\rG_2}_{\tau^\vee}}(\sigma_{\gen}) = 1$.
  Take $\sigma_1=\theta_{\mathrm{PD}^{\times}}^{\rG_2}(JL_D(\tau))$ and let $\sigma_2$ be the remaining nongeneric member of $\Pi_\phi$.
  The representation $\sigma_1$ is nonzero and nongeneric by \cite{GanSavin2023-HoweDualityDichotomyExceptionalThetaCorrespondences}*{Proposition 7.1 and Theorem 7.2}. In residue characteristic $3$, self-duality and discreteness force $\tau$ to be Steinberg, and the assertion follows from $JL_D(\tau)=\mathbf1$. See \cite{GanSavin2023-HoweDualityDichotomyExceptionalThetaCorrespondences}*{Theorem 8.2(iii) and Proposition 7.1(i)}.
  Put $m_i=c_{S_\phi^{\rG_2}}(\sigma_i)$ and $c_i=c_{I^{\rG_2}_{\tau}}(\sigma_i)$ for $i=1,2$. Since $\tau$ is self-dual, $I^{\rG_2}_{\tau}=I^{\rG_2}_{\tau^\vee}$.
  By \eqref{equation:packet-support-PGL3} and Corollary~\ref{corollary:coefficient-square-sum-PGL3}, we have
  \begin{align*}
    m_1^2+m_2^2=6-c_{S_\phi^{\rG_2}}(\sigma_{\gen})^2=6-1=5.
  \end{align*}
  Since $m_1,m_2$ are non-negative integers, $(m_1,m_2)$ is either $(2,1)$ or $(1,2)$.
  Set $n=\abs{S_F}$. For $n=1$, a division source would be nonsplit only at the place in $S_F$, since $\dot\tau_v$ is unramified at all other finite places. The Brauer--Hasse--Noether theorem excludes such an algebra, and the local correspondence excludes a split source. Comparing the coefficient of $\sigma_1$ therefore gives $m_1+2c_1=0$.
  For $n=2$, the coefficient of $\sigma_1\boxtimes\sigma_1$ is an actual cuspidal multiplicity, so
  \begin{align*}
    \frac{m_1^2+2c_1^2}{3}=\frac{m_1^2}{2}\in\Z_{\geq0}.
  \end{align*}
  Thus $m_1=2$, $m_2=1$, and $c_1=-1$. Elliptic orthogonality of $S_\phi^{\rG_2}$ and $I^{\rG_2}_{\tau}$ gives $1+2c_1+c_2=0$, hence $c_2=1$.
\end{proof}

\begin{remark}\label{remark:completion-LLC}
  In the proof of Theorem~\ref{theorem:ECR-PGL3}, we have shown that the three discrete series representations in the $L$-packet $\Pi_{\iota_{\PGL_3}^{\rG_2} \circ \phi^{\PGL_3}}(\rG_2(F))$ are obtained as theta lifts from $\PGL_3^{+}(F)$ or an inner form $\mathrm{PD}^{\times}$ of $\PGL_3(F)$.
  In particular, every irreducible representation of $\mathrm{PD}^{\times}$ has nonzero theta lift, as conjectured in \cite{Savin1999-SupercuspidalRepresentationsG2}*{Conjecture 4.1(1)}, including in residue characteristic $3$. This case was left open in \cite{GanSavin2023-HoweDualityDichotomyExceptionalThetaCorrespondences}*{p. 3, after Theorem 1.2}.
  Thus the injection in \cite{GanSavin2023-LocalLanglandsConjectureG2}*{Proposition 3.2} is a bijection also in residue characteristic $3$.
\end{remark}

\section{The endoscopic character relations between \texorpdfstring{$\rG_2$}{G2} and \texorpdfstring{$\SO_4$}{SO4}}\label{section:ECR-G2-SO4}

We fix an endoscopic datum $\mathfrak{e} = (G^{\mathfrak{e}} = \SO_4, s^{\mathfrak{e}}, \iota^{\rG_2}_{\SO_4})$ for $\rG_2$.
For example, we can take $s^{\mathfrak{e}}$ to be the unique element in $\widehat{T}_{\rG_2}$ such that $\alpha(s^{\mathfrak{e}}) = 1$ and $\beta(s^{\mathfrak{e}}) = -1$.
For a self-dual discrete series representation $\tau$ of $\PGL_3(F)$, let $\phi^{\SO_4}(\tau)$ denote the parameter for $\SO_4$ with standard parameter $\phi_\tau\oplus\mathbf1$.

\begin{definition}
  Let $F$ be a local field of characteristic $0$.
  Let $\phi^{\SO_4}$ be a discrete series $L$-parameter of $\SO_4(F)$.
  We assume that $\phi^{\SO_4}$ is not of the form $\phi^{\SO_4}(\tau)$ for any self-dual discrete series representation $\tau$ of $\PGL_3(F)$.
  We set $\phi^{\rG_2} = \iota^{\rG_2}_{\SO_4} \circ \phi^{\SO_4}$.
  We call an $L$-parameter for $\rG_2(F)$ of this type an $L$-parameter of $\SO_4$-type.
  Let $\Phi(\rG_2(F))_{\SO_4}$ denote the set of $L$-parameters for $\rG_2$ of $\SO_4$-type.
\end{definition}

\begin{lemma}
  Let $F$ be a local field of characteristic $0$.
  Let $\phi^{\SO_4}$ be a discrete series $L$-parameter of $\SO_4(F)$ whose standard $L$-parameter does not contain the trivial representation of $L_F$.

  Then, the $S$-group $S_{\phi^{\rG_2}} = \cS_{\phi^{\rG_2}}$ is isomorphic to $\cS_{\iota^{\PGSp_6}_{\rG_2} \circ \phi^{\rG_2}}$.
  Thus, it is an abelian $2$-group.
  It is also isomorphic to the group $S_{\phi^{\SO_4}} = \Cent_{\SO_4}(\phi^{\SO_4})$.
\end{lemma}

\begin{proof}
  The first assertion follows from \cite{GrossSavin1998-MotivesGaloisGroupTypeG2Exceptional}*{Proposition 1.10}.
  Since $s^{\mathfrak e}\in S_{\phi^{\SO_4}}\subset S_{\phi^{\rG_2}}$ and $S_{\phi^{\rG_2}}$ is abelian, we have $S_{\phi^{\rG_2}}\subset\Cent_{\rG_2}(s^{\mathfrak e})=\SO_4$.
  Hence $S_{\phi^{\rG_2}}=S_{\phi^{\SO_4}}$.
\end{proof}

\begin{remark}\label{remark:isomorphic-S-group-PGSO8-for-parameters-of-SO4-type}
  The restrictions of the two standard representations are
  \begin{align*}
    \Std_{\rG_2}\circ\phi^{\rG_2}&=(\Std_{\SO_4}\oplus\wedge^2_+)\circ\phi^{\SO_4},\\
    \Std_{\PGSO_8}\circ\iota^{\PGSO_8}_{\rG_2}\circ\phi^{\rG_2}&=\mathbf1\oplus(\Std_{\rG_2}\circ\phi^{\rG_2}).
  \end{align*}
  Neither summand in the first line has fixed vectors. Indeed, $\Std_{\SO_4}\circ\phi^{\SO_4}$ has none by hypothesis, and $\wedge^2_+$ is a summand of the adjoint representation of $\SO_4$, whose invariants vanish by discreteness. Thus the centralizer in $\Spin_8(\C)$ preserves the unique fixed line in the second representation and acts on it by a sign.
  Multiplication by a central element mapping to $-1$ in $\SO_8(\C)$ makes this action trivial. The pointwise stabilizer of this line is $\Spin_7(\C)$, so
  \begin{align*}
    \Cent_{\Spin_8}(\iota^{\PGSO_8}_{\rG_2}\phi^{\rG_2})
    =Z(\Spin_8)\Cent_{\Spin_7}(\iota^{\PGSp_6}_{\rG_2}\phi^{\rG_2}).
  \end{align*}
  Since $Z(\Spin_8)\cap\Spin_7=Z(\Spin_7)$, quotienting by the centers and taking connected components gives
  \begin{align*}
    \cS_{\phi^{\rG_2}}
    \cong\cS_{\iota^{\PGSp_6}_{\rG_2}\circ\phi^{\rG_2}}
    \cong\cS_{\iota^{\PGSO_8}_{\rG_2}\circ\phi^{\rG_2}}.
  \end{align*}

  For non-archimedean $F$, similitude theta lifting from $\PGSp_6(F)$ to $\PGSO_8(F)$ gives a bijection between the distinguished Xu packets $\Pi^{X}_{\iota^{\PGSp_6}_{\rG_2} \circ \phi^{\rG_2}}$ and $\Pi^{X}_{\iota^{\PGSO_8}_{\rG_2} \circ \phi^{\rG_2}}$ by \cite{Takanashi2026-HiragaIchinoIkedaConjectureFormalDegrees}*{Proposition 3.58 and its proof}. Exceptional theta lifting from $\rG_2(F)$ to $\PGSp_6(F)$ identifies these with $\Pi_{\phi^{\rG_2}}$.
\end{remark}

We globalize each involution by two successive $\GL_2$ globalizations with the same central character, as in the argument following Lemma~\ref{lemma:globalization-unitary-character}.

\begin{lemma}[cf. \cite{Arthur2013-EndoscopicClassificationRepresentations}*{Proposition 6.3.1}]\label{lemma:globalization-SO4}
  Let $F$ be a non-archimedean local field of characteristic zero, and let $\phi^{\SO_4}$ be a discrete parameter whose image $\phi=\iota_{\SO_4}^{\rG_2}\circ\phi^{\SO_4}$ is discrete.
  Write $s\in S_\phi$ for the distinguished involution of this endoscopic datum.
  There exist a totally real number field $\dot F$, a finite place $v$ with $\dot F_v\cong F$, and unitary cuspidal automorphic representations $\Pi_1,\Pi_2$ of $\GL_2(\A_{\dot F})$ with the same central character, with the following properties.
  \begin{enumerate}
    \item
    The refined generic global $A$-parameter $\dot\phi^{\SO_4}$ for $\SO_4$ represented by $(\Pi_1,\Pi_2)$ satisfies $\dot\phi^{\SO_4}_v=\phi^{\SO_4}$ and
    \begin{alignat*}{2}
      \Std_{\SO_4}\circ\dot\phi^{\SO_4}&=\Pi_1\boxtimes\Pi_2^\vee, &
      \quad\wedge^2_+\circ\dot\phi^{\SO_4}&=\Ad(\Pi_1).
    \end{alignat*}
    Both representations on the right are cuspidal, of degrees four and three respectively.
    \item
    Put $\dot\phi=\iota_{\SO_4}^{\rG_2}\circ\dot\phi^{\SO_4}$.
    Then
    \begin{align*}
      S_{\dot\phi}=\{1,\dot s\}\cong\Z/2\Z,
      \quad S_{\dot\phi}\longrightarrow S_\phi:\dot s\longmapsto s
    \end{align*}
    identifies $S_{\dot\phi}$ with $\langle s\rangle$.
    \item
    At every real place $w$, the parameter $\dot\phi_w$ is discrete and regular.
    Restriction of the packet pairing gives a surjection
    \begin{align*}
      \Pi_{\dot\phi_w}(\rG_2(\R))\longrightarrow S_{\dot\phi}^{\cD}.
    \end{align*}
    \item
    At every finite place $w\neq v$, at least one of the two local parameters $\phi_{\Pi_{1,w}},\phi_{\Pi_{2,w}}$ factors through the diagonal torus of $\GL_2(\C)$.
    Consequently $\dot\phi^{\SO_4}_w$ factors through a proper Levi subgroup of $\SO_4$, and $\dot\phi_w$ factors through a proper Levi subgroup of $\rG_2$.
  \end{enumerate}
  There also exist such data with (3) replaced by the condition that both $\Pi_{i,w}$ are spherical tempered principal series at every real place $w$, and that $\dot\phi_w$ is regular with nonzero purely imaginary infinitesimal character.
\end{lemma}

\begin{proof}
  Choose a lift of the refined parameter $\phi^{\SO_4}$ to a pair of two-dimensional parameters $(\rho_1,\rho_2)$ with common determinant $\lambda_v$, so that
  \begin{align*}
    \Std_{\SO_4}\circ\phi^{\SO_4}=\rho_1\otimes\rho_2^\vee,
    \quad \wedge^2_+\circ\phi^{\SO_4}=\Ad(\rho_1).
  \end{align*}
  We use the ordered pair, as in Section~\ref{section:refined-LLC-SO4}, to retain the refinement.
  The lift can be taken unitary. Discreteness of $\phi^{\SO_4}$ implies that both $\rho_i$ are discrete parameters for $\GL_2(F)$.
  Let $\pi_i$ be the corresponding unitary discrete series representations.

  Choose $\dot F$ totally real with $\dot F_v\cong F$, and fix a real place $w_0$.
  By \cite{Takanashi2025-NoteSauvageotDensityPrinciple}*{Corollary 8.3}, applied to $\mathbf G_m$, extend $\lambda_v$ to a unitary Hecke character $\dot\lambda$ of $\dot F$.
  Choose a finite place $u\neq v$ where $\dot\lambda_u$ is unramified, and an unramified unitary character $\mu_u$ with $\mu_u^2=\dot\lambda_u$.
  Apply the fixed-central-character globalization theorem
  \cite{Takanashi2025-NoteSauvageotDensityPrinciple}*{Theorem 8.5 and its proof}
  to obtain $\Pi_1$ with central character $\dot\lambda$, with
  \begin{align*}
    \Pi_{1,v}=\pi_1,\quad
    \Pi_{1,u}=\St_{\GL_2}\otimes(\mu_u\circ\det),
  \end{align*}
  and with discrete series components at all real places.
  At $w_0$ we use the sufficiently regular discrete series supplied by that theorem.

  Let $R$ be the set of finite places $w\neq v$ where $\Pi_{1,w}$ is ramified.
  Apply the same theorem a second time, with central character $\dot\lambda$, to obtain $\Pi_2$ with $\Pi_{2,v}=\pi_2$.
  For each $w\in R$ prescribe a relatively compact family of irreducible unitary principal series with central character $\dot\lambda_w$, of positive Plancherel measure and with boundary of measure zero.

  At every real place, choose the discrete series weights for $\Pi_2$ sufficiently far from those of $\Pi_1$.
  More precisely, up to the common central unitary twist, write their parameters as
  \begin{align*}
    I_{m_w},\ I_{n_w},\quad
    I_k=\Ind_{W_\C}^{W_\R}(z/|z|)^k,
    \quad n_w>3m_w>0.
  \end{align*}
  Equality of the central characters requires $m_w\equiv n_w\pmod2$.
  We prescribe these weights at $w\neq w_0$ and take the weight at $w_0$ sufficiently large in the second application of the theorem.
  For a finite $w\neq v$ in $R$, the second parameter factors through a torus. For $w\notin R$, the first representation is unramified and its parameter factors through a torus.

  At $u$ the Steinberg parameter of $\Pi_{1,u}$ has nontrivial Deligne $\SL_2$ part. Both the local parameter of an automorphic induction of a character and the principal series parameter of $\Pi_{2,u}$ have trivial Deligne $\SL_2$ part. Since twisting preserves this part, $\Pi_1$ is neither dihedral nor a twist of $\Pi_2$.
  The cuspidality criterion for the Rankin--Selberg transfer
  \cite{Ramakrishnan2000-ModularityRankinSelbergSeriesMultiplicityOne}
  shows that $\Pi_1\boxtimes\Pi_2^\vee$ is cuspidal.
  Also $\Ad(\Pi_1)$ is cuspidal because $\Pi_1$ is not dihedral.
  Thus the standard seven-dimensional parameter is the multiplicity-free orthogonal sum
  \begin{align}\label{equation:successive-GL2-globalization}
    \Std_{\rG_2}\circ\dot\phi
    =(\Pi_1\boxtimes\Pi_2^\vee)\boxplus\Ad(\Pi_1).
  \end{align}
  Its centralizer in $\SO_7(\C)$ is $\{1,\operatorname{diag}(-I_4,I_3)\}$.
  The nonidentity element is the distinguished $\SO_4$ involution in $\rG_2(\C)$, so $S_{\dot\phi}$ is exactly this group and localizes to $\langle s\rangle$.

  At a real place the standard parameter is
  \begin{align*}
    \operatorname{sgn}\oplus I_{m_w+n_w}\oplus I_{n_w-m_w}\oplus I_{2m_w}.
  \end{align*}
  The inequalities on the weights make the $L$-parameter for $\rG_2(\R)$ discrete and regular, with component group $(\Z/2\Z)^2$.
  Its packet on split $\rG_2(\R)$ consists of three discrete series and hence gives three distinct characters of this group.
  Restriction to $S_{\dot\phi}$ attains both its characters, since each fiber of the restriction map on the full dual has only two elements.
  Compare
  \cite{Arthur2013-EndoscopicClassificationRepresentations}*{Lemma 6.1.2 and the proof of Proposition 6.3.1, pp. 327--328}.

  Finally, write $\rho_{i,w}=\phi_{\Pi_{i,w}}$.
  Suppose at a finite place that $\rho_{1,w}=a\oplus b$ and put $r=\rho_{2,w}\otimes a^{-1}$.
  Equality of determinants gives $\det r=b/a$, and
  \begin{align*}
    \Std_{\rG_2}\circ\dot\phi_w
    &=r\oplus r^\vee\oplus\det r\oplus(\det r)^{-1}\oplus1.
  \end{align*}
  This is the parameter obtained from the short-root $\GL_2$ Levi.
  If instead $\rho_{2,w}=a\oplus b$, put $r=\rho_{1,w}\otimes a^{-1}$. Then
  \begin{align*}
    \Std_{\rG_2}\circ\dot\phi_w
    &=r\oplus r^\vee\oplus\Ad(r),
  \end{align*}
  which comes from the long-root $\GL_2$ Levi.
  In each case the standard representation of the $L$-parameter for $\SO_4$ is $r\oplus r^\vee$, and the corresponding $\GL_2$ Levi maps onto the indicated Levi of $\rG_2$.

  For the spherical alternative, choose $\dot\lambda$ with spherical real components and apply the fixed-central-character version of the spherical globalization argument in \cite{Takanashi2025-NoteSauvageotDensityPrinciple}*{Theorem 8.7 and Remark 8.8} twice, with the same finite local conditions.
  After a common unitary norm twist, \cite{Eikemeier2022-AsymptoticsHeckeOperatorsQuasisplitSimpleGroups}*{Theorem 1.6} fixes the real central character.
  At a fixed finite level, averaging over the finite ray class quotient projects onto $\dot\lambda$ and gives the required fixed-central-character Plancherel limit.
  Up to the common central twist, choose the real parameters in the form
  \begin{align*}
    |\cdot|^{ia_w}\oplus|\cdot|^{-ia_w},
    \quad |\cdot|^{ib_w}\oplus|\cdot|^{-ib_w},
    \quad b_w>3a_w>0.
  \end{align*}
  Once the first globalization is fixed, the second Weyl-law limit allows us to impose these inequalities.
\end{proof}

\begin{remark}\label{remark:successive-globalization-local-transfer}
  At every finite $w\neq v$ in this construction, $\dot s$ lies in the connected center of the indicated $\rG_2$ Levi.
  Its image in $\cS_{\dot\phi_w}$ is therefore trivial.
  The transfer of the stable character of $\dot\phi^{\SO_4}_w$ is determined by compatibility of parabolic induction with transfer, Lemma~\ref{lemma:parabolic-descent-commutes-with-transfer}. The induced endoscopic datum on the Levi is the identity datum.
  For tempered inducing data, this agrees with the previously established local intertwining relations.
  In the spherical alternative, the same statements hold at every real place, since both parameters factor through a torus.
\end{remark}

\begin{lemma}\label{lemma:SO4-global-PGSO8-packet}
  We continue to use the notation in the proof of Lemma~\ref{lemma:globalization-SO4}, with either choice of real components.
  Let $c = \iota_{\SO_4}^{\PGSO_8} \circ c(\dot{\phi}^{\SO_4})$ be the Satake parameter of $\PGSO_8$ obtained from $\dot{\phi}^{\SO_4}$.
  Then, we have
  \begin{align*}
    I_{\disc, c}^{\widetilde{\PGSO_8}}
    =
    \sum_{\pi \in \Pi_{\iota_{\SO_4}^{\PGSO_8} \circ \dot{\phi}^{\SO_4}}, \rho_\pi = \mathbf{1}} \tr \widetilde{\pi},
  \end{align*}
  where $\rho_{\pi}$ denotes the character of $\cS_{\iota_{\SO_4}^{\PGSO_8} \circ \dot{\phi}^{\SO_4}}$ corresponding to $\pi$.
\end{lemma}

\begin{proof}
  We first prove that $I_{\disc, c}^{\widetilde{\PGSO_8}}$ is nonzero by choosing a factorizable test function.
  Write $H=\SO_4$ and $\phi_w^H=\dot\phi_w^{\SO_4}$.
  At the distinguished finite place $v$, simultaneous cuspidal transfer is surjective onto the stable cuspidal spaces of the elliptic endoscopic groups. See \cite{MoeglinWaldspurger2016-StabilisationFormuleTracesTordue1}*{I.4.11}.
  Choose $f_v^{\widetilde{\PGSO_8}}$ with
  \begin{align*}
    (f_v^{\widetilde{\PGSO_8}})_{\rG_2}
    =(f_v^{\widetilde{\PGSO_8}})_{\SL_3}=0, \quad S^H_{\phi_v^H}((f_v^{\widetilde{\PGSO_8}})_H) =1.
  \end{align*}
  For the discrete-series choice at a real place $w$, the standard eight-dimensional parameter is
  \begin{align*}
    1\oplus\operatorname{sgn}\oplus I_{m_w+n_w}
      \oplus I_{n_w-m_w}\oplus I_{2m_w}.
  \end{align*}
  Its orthogonal irreducible summands are distinct by the inequalities in Lemma~\ref{lemma:globalization-SO4}, so the lifted parameter for $\PGSO_8(\R)$ is discrete.
  Theorem 6.1.1 of \cite{KalethaMezo2026-RefinedLocalLanglandsConjectureDiscreteLParameters} expresses the transfer
  defined by $I^{\widetilde{\PGSO_8}}_{\dot\phi_w^{\SO_4}}(f_w)=S^H_{\phi_w^H}((f_w)_H)$
  as an elliptic combination of twisted discrete-series characters.
  Its Whittaker-normalized generic coefficient is $1$, so $I^{\widetilde{\PGSO_8}}_{\dot\phi_w^{\SO_4}}$ is nonzero. The Whittaker extension of the generic member agrees with its canonical extension.
  The real pseudo-coefficient correspondence in \cite{MoeglinWaldspurger2016-StabilisationFormuleTracesTordue1}*{IV.2.2, p.~444} supplies a cuspidal function with $I^{\widetilde{\PGSO_8}}_{\dot\phi_w^{\SO_4}}(f_w^{\widetilde{\PGSO_8}})=\|I^{\widetilde{\PGSO_8}}_{\dot\phi_w^{\SO_4}}\|_{\mathrm{ell}}^2>0$.
  Rescale it to make this value $1$.
  For the spherical choice, parabolic descent from the torus gives a nonzero transferred character, and we again choose a test function on which it has value $1$.

  At every other finite place, compatibility of parabolic induction with transfer and the established local intertwining relations give a nonzero transferred character, by the Levi factorization in Lemma~\ref{lemma:globalization-SO4}.
  Choose a test function on which it has value $1$, using the spherical unit at almost every place.
  The standard transfer $\Pi_1\boxtimes\Pi_2^\vee$ of the refined generic global $A$-parameter for $H$ is cuspidal, so its Arthur component group modulo the center is trivial.
  Consequently its stable discrete contribution is the product of its local stable packet characters, by the refined classification in Theorem~\ref{theorem:refined-Arthur-classification-SO4}.
  The degree-four and degree-three cuspidal summands of the standard lift identify the chosen refinement at the Satake family $c$.
  For the chosen test function, stabilization therefore gives
  \begin{align*}
    I_{\disc,c}^{\widetilde{\PGSO_8}}(f^{\widetilde{\PGSO_8}})
    =\frac12
      S^H_{\disc,c}((f^{\widetilde{\PGSO_8}})_H)
    =\frac12\neq0.
  \end{align*}

  Let $\dot{\phi'}$ be the $L$-parameter of $\SO_8$ obtained from $\dot{\phi}^{\SO_4}$.
  Then, the standard $L$-parameter of $\dot{\phi'}$ is given by $\Std_{\SO_4} \circ \dot{\phi}^{\SO_4} \boxplus \wedge^2_{+} \circ \dot{\phi}^{\SO_4} \oplus \mathbf{1}$ and there exists a triality-invariant discrete automorphic representation $\pi$ of $\PGSO_8(\A_{\dot{F}})$ with the $L$-parameter $\dot{\phi'}$.
  By \cite{Xu2025-GlobalLPacketsQuasisplitGSp2nGO2n}*{Theorem 1.1}, every contribution belongs to a twist of the global $L$-packet containing $\pi$ by $\chi_{\omega,1}$, for a quadratic Hecke character $\omega$.
  If such a twisted packet contributes at $c$, then at almost all places $w$ its spherical member has the same Satake parameter as $\pi_w$. Thus
  \begin{align*}
    \pi_w\otimes\chi_{\omega_w,1}\cong\pi_w.
  \end{align*}
  Since $\pi_w$ is triality-invariant, it is also fixed by $\chi_{\omega_w,2}$.
  The unramified calculation in the proof of Lemma~\ref{lemma:triality-quadratic-twist-global} and strong multiplicity one give
  \begin{align*}
    \operatorname{Lift}_{\PGSO_8}^{\GL_8}(\pi)
    \cong
    \operatorname{Lift}_{\PGSO_8}^{\GL_8}(\pi)\otimes(\omega\circ\det).
  \end{align*}
  The isobaric summands of $\operatorname{Lift}_{\PGSO_8}^{\GL_8}(\pi)$ have degrees $4$, $3$, and $1$, with the degree-one summand equal to $\mathbf{1}$. Comparing this summand gives $\omega=\mathbf{1}$.
  The same argument excludes non-trivial quadratic self-twists of every member of this packet.
  Moreover, $m_{\dot{\phi'}}=1$ in Xu's formula since the standard parameter has an odd-dimensional summand. See \cite{Xu2025-GlobalLPacketsQuasisplitGSp2nGO2n}*{Lemma 2.2(4)}.
  Hence \cite{Xu2025-GlobalLPacketsQuasisplitGSp2nGO2n}*{Theorem 1.1} leaves only this global $L$-packet, with multiplicity one for the members with trivial global pairing.
  Their canonical extensions give the asserted twisted expansion by Corollary~\ref{corollary:canonical-extension-global}.
\end{proof}

\begin{theorem}\label{theorem:ECR-G2-SO4}
  Let $F$ be a non-archimedean local field of characteristic zero.
  Let $\phi^{\rG_2}$ be a discrete series $L$-parameter for $\rG_2(F)$ of $\SO_4$-type.

  Then, the following endoscopic character relation holds:
  \begin{align*}
    \sum_{\sigma \in \Pi_{\phi^{\rG_2}}} \tr\rho_{\sigma}(s) \tr\sigma(f^{\rG_2})
    =
    S^{\SO_4}_{\phi^{\SO_4}}(f^{\rG_2}_{\SO_4}),
  \end{align*}
  where $s$ is the image in $\cS_{\phi^{\rG_2}} = S_{\phi^{\rG_2}}$ of the distinguished element $s^{\mathfrak{e}}$ defining the fixed endoscopic datum.
\end{theorem}

\begin{proof}
  Put $\phi=\phi^{\rG_2}$, fix the distinguished involution $s$ of the given endoscopic datum, and apply Lemma~\ref{lemma:globalization-SO4}.
  Write $c=c(\dot\phi)$. The global group $S_{\dot\phi}=\{1,\dot s\}$ maps isomorphically to $\langle s\rangle\subset S_\phi$.
  Put
  \begin{alignat*}{2}
    S^{\rG_2}_{\dot\phi_u}&=\sum_{\sigma_u\in\Pi_{\dot\phi_u}}\tr\sigma_u, &
    \quad I^{\rG_2}_{\dot\phi_u^{\SO_4}}&=\Trans_{\SO_4}^{\rG_2}
      \bigl(S^{\SO_4}_{\dot\phi^{\SO_4}_u}\bigr).
  \end{alignat*}
  At $u=v$ this stable packet sum is the transfer given by Proposition~\ref{proposition:ECR-PGSp_6}.
  At finite $u\neq v$ the corresponding identities follow by parabolic descent, as explained in Remark~\ref{remark:successive-globalization-local-transfer}. 
  At real places they follow from the real endoscopic character relations.
  In particular, for $u\neq v$,
  \begin{align*}
    I^{\rG_2}_{\dot\phi_u^{\SO_4}}=\sum_{\sigma_u\in\Pi_{\dot\phi_u}}
      \rho_{\sigma_u}(\dot s_u)\tr\sigma_u.
  \end{align*}

  The preceding lemma and the twisted stabilization, together with
  Theorem~\ref{theorem:local-intertwining-relation-PGSO8-SO4}, give
  \begin{align*}
    S^{\rG_2}_{\disc,c}=\frac12{\prod_u}'S^{\rG_2}_{\dot\phi_u}.
  \end{align*}
  There is exactly one nonidentity element of $S_{\dot\phi}$, hence exactly one $\SO_4$ term for this global parameter.
  The invariant trace formula for $\rG_2$ therefore reads
  \begin{align}\label{equation:SO4-cyclic-global-separation}
    R^{\rG_2}_{\disc,c}
    =\frac12{\prod_u}'S^{\rG_2}_{\dot\phi_u}+\frac12{\prod_u}'I^{\rG_2}_{\dot\phi_u^{\SO_4}}.
  \end{align}
  There is no proper-Levi contribution at this Satake family, since its standard parameter is the cuspidal sum of degrees four and three in \eqref{equation:successive-GL2-globalization}.

  Expand the elliptic transfer at $v$ as
  \begin{align*}
    I^{\rG_2}_{\phi^{\SO_4}}=\sum_{\sigma\in\Pi_{\disc}(\rG_2(F))}
       c(s,\sigma)\tr\sigma,
    \quad
    b(\sigma)=\begin{cases}1&\sigma\in\Pi_\phi,\\0&\sigma\notin\Pi_\phi.\end{cases}
  \end{align*}
  The separation argument of
  \cite{Arthur2013-EndoscopicClassificationRepresentations}*{Proposition 6.6.5}
  can now be applied to $S_{\dot\phi}$ alone.
  Choose packet members away from $v$ whose product character on $S_{\dot\phi}$ is $\eta\in\{1,-1\}$, and isolate their characters in \eqref{equation:SO4-cyclic-global-separation}.
  Both choices of $\eta$ are possible by Lemma~\ref{lemma:globalization-SO4}(3).
  Nonnegative integral multiplicities on the left give
  \begin{align}\label{equation:SO4-cyclic-Fourier-coefficients}
    \frac{b(\sigma)+\eta c(s,\sigma)}2\in\Z_{\geq0}
    \quad(\eta=1,-1).
  \end{align}
  Adding these two integers shows that $c(s,\sigma)=0$ outside $\Pi_\phi$, and that $c(s,\sigma)\in\{1,-1\}$ on $\Pi_\phi$.

  To identify the signs, apply the spherical alternative of Lemma~\ref{lemma:globalization-SO4}, keeping the local parameter at $v$ fixed, and use the same notation for these new global data.
  At every $u\neq v$, the element $\dot s_u$ lies in the connected center of a Levi subgroup, so $I^{\rG_2}_{\dot\phi_u^{\SO_4}}=S^{\rG_2}_{\dot\phi_u}$ by Remark~\ref{remark:successive-globalization-local-transfer}.
  Lemma~\ref{lemma:SO4-global-PGSO8-packet} and stabilization therefore give
  \begin{align*}
    R^{\rG_2}_{\disc,c}
    =\frac12(S^{\rG_2}_{\phi}+I^{\rG_2}_{\phi^{\SO_4}}){\prod_{u\neq v}}'S^{\rG_2}_{\dot\phi_u}.
  \end{align*}

  Fix $\sigma\in\Pi_\phi$ with $\rho_\sigma(s)=1$.
  The local packet at $v$ in Lemma~\ref{lemma:SO4-global-PGSO8-packet} is the distinguished Xu packet.
  Indeed, its triality-invariant member has a tempered standard parameter containing $\mathbf1$. Triality invariance gives both gamma-factor vanishings in Theorem~\ref{theorem:characterization-distinguished-L-packets}, and Lemma~\ref{lemma:injectivity-G_2-to-GL7} identifies the resulting parameter with $\phi$.
  By Remark~\ref{remark:isomorphic-S-group-PGSO8-for-parameters-of-SO4-type}, the lift
  \begin{align*}
    \pi_v=\theta_{\PGSp_6}^{\PGSO_8}\circ\theta_{\rG_2}^{\PGSp_6}(\sigma)
  \end{align*}
  belongs to this packet. Its pairing restricts trivially to $S_{\dot\phi}$ by \cite{GanSavin2023-LocalLanglandsConjectureG2}*{Section 8.3} and \cite{AtobeGan2017-LocalThetaCorrespondenceTemperedRepresentationsLanglands}*{Theorem 4.3(2)}.
  Choose generic members at the other finite places and spherical members at the real places.
  The global pairing is trivial, so the multiplicity formula in Lemma~\ref{lemma:SO4-global-PGSO8-packet} supplies a discrete automorphic representation $\Pi$ of $\PGSO_8(\A_{\dot F})$ with these components.
  It is cuspidal since its real components are tempered. See \cite{Eikemeier2022-AsymptoticsHeckeOperatorsQuasisplitSimpleGroups}*{Remark 1.4}.

  For a sufficiently large finite set $S$ of places, its partial standard $L$-function is
  \begin{align*}
    L^S(s,\Pi,\Std)
    =\zeta_{\dot F}^S(s)
      L^S(s,\Pi_1\boxtimes\Pi_2^\vee)
      L^S(s,\Ad\Pi_1),
  \end{align*}
  and has a simple pole at $s=1$.
  By \cite{GanTakeda2011-RegularizedSiegelWeilNonvanishingOrthogonal}*{Corollary 7.9(a)}, with $m=8$ and $r=3$, the backward classical theta lift to $\PGSp_6$ is nonzero. The passage to similitudes is as in \cite{BakicGanSavin2023-SimilitudeExceptionalThetaCorrespondences}*{Section 8}.
  At a real place the infinitesimal character of $\Pi_w$ has purely imaginary coordinates.
  A lift from $\Sp_4(\R)$ would have coordinates of the form $(x_1,x_2,1,0)$, up to the Weyl group, by \cite{Przebinda1996-DualityCorrespondenceInfinitesimalCharacters}*{Theorems 1.13 and 1.19}.
  Thus the lift to $\Sp_4$ vanishes, and the tower property makes the lift to $\PGSp_6$ cuspidal.
  Denote it by $\tau$. Its local component at $v$ is $\theta_{\rG_2}^{\PGSp_6}(\sigma)$ by Howe duality.

  The Satake family $c$ is triality-invariant. The unramified classical theta correspondence therefore gives
  \begin{align*}
    L^S(s,\tau,\mathrm{Spin})=L^S(s,\Pi,\Std).
  \end{align*}
  Compare \cite{GanSavin2023-LocalLanglandsConjectureG2}*{Section 12.3}.
  Its pole at one gives a nonzero cuspidal backward theta lift to an inner form of $\rG_2$ by \cite{GanSavin2020-ExceptionalSiegelWeilFormulaPolesSpin}*{Theorem 10.1}.
  At every real place, the infinitesimal character of $\tau_w$ is regular and purely imaginary by the same classical correspondence.
  The exceptional correspondence of infinitesimal characters in \cite{Li1999-CorrespondencesInfinitesimalCharactersReductiveDualPairs}*{Lemma 2.1, Theorem 8.1 and Table 1} excludes compact $\rG_2(\R)$, whose irreducible representations have real infinitesimal characters.
  Hence the global inner form is split by the Hasse principle.
  At $v$ the backward lift has component $\sigma$ by Howe duality.
  Its positive multiplicity in $R^{\rG_2}_{\disc,c}$ gives
  \begin{align*}
    \frac{1+c(s,\sigma)}2>0,
  \end{align*}
  so $c(s,\sigma)=1$ whenever $\rho_\sigma(s)=1$.

  Finally, $S_\phi$ is an elementary abelian $2$-group and $\Pi_\phi$ is indexed by all of its characters, by Theorem~\ref{theorem:properties-exceptional-theta}.
  Since $s\neq1$, exactly half of them take the value $1$ at $s$.
  Orthogonality of the $\SO_4$ transfer to the stable packet sum, by Proposition~\ref{proposition:elliptic-inner-product-transfer}, gives
  \begin{align*}
    \sum_{\sigma\in\Pi_\phi}c(s,\sigma)=0.
  \end{align*}
  Since every remaining coefficient is $1$ or $-1$, they must all be $-1$.
  Thus $c(s,\sigma)=\rho_\sigma(s)$ for every $\sigma\in\Pi_\phi$, as required.
\end{proof}

We finally consider the endoscopic character relations between $\rG_2$ and $\SO_4$ for $S_3$-packets of $\rG_2(F)$.
Let $\tau$ be a self-dual discrete series representation of $\PGL_3(F)$.
The $L$-packet $\Pi_{\iota_{\PGL_3}^{\rG_2} \circ \phi_{\tau}}(\rG_2(F))$ consists of the three discrete series representations $\sigma(\mathbf{1}), \sigma(\Std), \sigma(\epsilon)$ of $\rG_2(F)$.
The member $\sigma(\mathbf{1})$ is generic. The two nongeneric members $\sigma(\Std), \sigma(\epsilon)$ correspond to the standard and sign representations of $S_3$, respectively.
The parameter $\phi^{\SO_4}(\tau)$ has standard parameter $\phi_{\tau}\oplus\mathbf1$.

\begin{remark}
  The set of elements of order $2$ in the group $\cS_{\iota^{\rG_2}_{\PGL_3} \circ \phi_{\tau}} \cong S_3$ forms one conjugacy class under the action of $\mu_3 \subset \cS_{\iota^{\rG_2}_{\PGL_3} \circ \phi_{\tau}}$.
  Thus, the resulting endoscopic datum and $L$-parameter for $s \in \cS_{\iota^{\rG_2}_{\PGL_3} \circ \phi_{\tau}}$ of order $2$ are all conjugate.
\end{remark}

\begin{theorem}\label{theorem:ECR-S3}
  Let $F$ be a non-archimedean local field of characteristic zero.
  Let $\tau$ be a discrete series representation of $\PGL_3(F)$ which is self-dual.
  Let $s \in \cS_{\iota^{\rG_2}_{\PGL_3} \circ \phi_{\tau}}$ be a non-trivial element of order $2$.
  The element $s$ defines an endoscopic group $H \cong \SO_4$.
  Let $\phi^{H}(\tau)$ be the $L$-parameter of $H$ given by $\phi_{\tau} \oplus \mathbf{1}$ where $\phi_{\tau}$ is the $L$-parameter of $\tau$.
  
  Then, the following endoscopic character relation holds:
  \begin{align*}
    \sum_{\sigma \in \Pi_{\iota_{\PGL_3}^{\rG_2} \circ \phi_{\tau}}} \tr\rho_{\sigma}(s) \tr\sigma(f^{\rG_2})
    =
    S^{H}_{\phi^{H}(\tau)}(f^{\rG_2}_{H}).
  \end{align*}
\end{theorem}

\begin{proof}
  Put $\phi=\iota_{\PGL_3}^{\rG_2}\circ\phi_\tau$ and apply Lemma~\ref{lemma:globalization-SO4} to $\phi^H(\tau)$, with discrete series at the real places.
  The global parameter $\dot\phi$ has $S_{\dot\phi}=\{1,\dot s\}$.
  As in the proof of Theorem~\ref{theorem:ECR-G2-SO4}, the local Xu packet at $v$ is the distinguished packet by Theorem~\ref{theorem:characterization-distinguished-L-packets} and Lemma~\ref{lemma:injectivity-G_2-to-GL7}.

  The global involution $\dot s$ acts by opposite signs on the two copies of $\mathbf1$ and on the two copies of $\phi_\tau$ in the standard parameter for $\PGSO_8$.
  Thus it maps to the nonidentity element $r$ of the local $R$-group. As in Lemma~\ref{lemma:elliptic-norm-PGL3}, let $\pi_+$ and $\pi_-$ be the constituents of $I_P(\tau)$ on which $R_P(r,\tau,\psi_F)$ acts by $+1$ and $-1$, respectively. The two local terms in Lemma~\ref{lemma:SO4-global-PGSO8-packet} are
  \begin{align*}
    I^{\widetilde{\PGSO_8}}_{P,\widetilde u}(\tau)
      &=\tr\widetilde\pi_++\tr\widetilde\pi_-,&
    I^{\widetilde{\PGSO_8}}_{P,r\widetilde u}(\tau)
      &=\tr\widetilde\pi_+-\tr\widetilde\pi_-.
  \end{align*}
  The first is the transfer from $\rG_2$ by Theorem~\ref{theorem:ECR-PGL3}.
  Put $I^{\widetilde{\PGSO_8}}_{\phi^H(\tau)}=\Trans_H^{\widetilde{\PGSO_8}}(S^H_{\phi^H(\tau)})$. By twisted stabilization and the local identities away from $v$, the difference $I^{\widetilde{\PGSO_8}}_{P,r\widetilde u}(\tau)-I^{\widetilde{\PGSO_8}}_{\phi^H(\tau)}$ is a transfer from $\rG_2$, hence orthogonal to the transfer from $H$ by Proposition~\ref{proposition:elliptic-inner-product-transfer}. The latter has squared elliptic norm $2$ since the $H$-packet is a singleton and $c(\widetilde{\PGSO_8},H)=1/2$. Lemma~\ref{lemma:elliptic-norm-PGL3} gives the same squared norm for the signed character, so the difference has norm zero. Thus
  \[
     \tr\widetilde\pi_+-\tr\widetilde\pi_-=I^{\widetilde{\PGSO_8}}_{\phi^H(\tau)}.
  \]

  The comparison in Theorem~\ref{theorem:ECR-G2-SO4} now gives \eqref{equation:SO4-cyclic-global-separation}, with
  \begin{align*}
    S^{\rG_2}_{\phi}&=\sum_{\sigma\in\Pi_\phi}\dim\rho_\sigma\tr\sigma,\\
    I^{\rG_2}_{\phi^H(\tau)}&:=\Trans_H^{\rG_2}(S^H_{\phi^H(\tau)}).
  \end{align*}
  Write $I^{\rG_2}_{\phi^H(\tau)}=\sum_\sigma c_\sigma\tr\sigma$ and put $b(\sigma)=\dim\rho_\sigma$ on $\Pi_\phi$ and $b(\sigma)=0$ otherwise.
  Separating the two characters of $S_{\dot\phi}$ at the real places as in \eqref{equation:SO4-cyclic-Fourier-coefficients} gives
  \begin{align*}
    \frac{b(\sigma)+\eta c_\sigma}{2}\in\Z_{\geq0}
    \quad(\eta=1,-1).
  \end{align*}
  Consequently $c_\sigma=0$ for $\sigma\notin\Pi_\phi$.

  Elliptic orthogonality of $I^{\rG_2}_{\phi^{H}(\tau)}$ to both $S^{\rG_2}_{\iota_{\PGL_3}^{\rG_2} \circ \phi_{\tau}}$ and $I^{\rG_2}_{\tau}$ gives
  \begin{align*}
    c_{\sigma(\mathbf{1})} + 2 c_{\sigma(\Std)} + c_{\sigma(\epsilon)} = 0, \\
    c_{\sigma(\mathbf{1})} - c_{\sigma(\Std)} + c_{\sigma(\epsilon)} = 0.
  \end{align*}
  Hence $c_{\sigma(\mathbf{1})} + c_{\sigma(\epsilon)} = 0$ and $c_{\sigma(\Std)} = 0$.
  Thus for some constant $c$ we have $I^{\rG_2}_{\phi^{H}(\tau)} = c \cdot \sum_{\sigma \in \Pi_{\iota_{\PGL_3}^{\rG_2} \circ \phi_{\tau}}} \tr\rho_{\sigma}(s) \tr\sigma$.
  Comparing the regular nilpotent terms in the local character expansions gives $c = 1$ (cf. \cite{AtobeGanEtAl2024-LocalIntertwiningRelationsCoTemperedPackets}*{Section D.4}).
\end{proof}

\section{Summary of local consequences}\label{section:local-summary}

\begin{theorem}\label{theorem:summary}
  Let $F$ be a local field of characteristic zero.
  Let $\phi^{\rG_2}$ be a generic $L$-parameter for $\rG_2(F)$.
  Let $\Pi_{\phi^{\rG_2}}$ denote the $L$-packet for $\rG_2(F)$ associated to $\phi^{\rG_2}$.
  Let $\Pi_{\iota \circ \phi^{\rG_2}}$ denote the distinguished $L$-packet for $\PGSO_8(F)$ associated to $\iota \circ \phi^{\rG_2}$.

  \begin{enumerate}
  \item
  Let $\cS_{\phi^{\rG_2}} = \pi_0(\Cent_{\rG_2^{\vee}}(\phi^{\rG_2}))$ be the $S$-group associated to $\phi^{\rG_2}$.
  There is an injection $\sigma \mapsto \rho_{\sigma}$ from $\Pi_{\phi^{\rG_2}}$ to $\Irr(\cS_{\phi^{\rG_2}})$, bijective for non-archimedean $F$.
  For $s \in \cS_{\phi^{\rG_2}}$ and a corresponding endoscopic datum $\mathfrak{e} = (G^{\mathfrak{e}}, s^{\fe}, \iota^{\rG_2}_{G^{\mathfrak{e}}})$ for $\rG_2$, it satisfies
  \begin{align*}
    \sum_{\sigma \in \Pi_{\phi^{\rG_2}}} \tr\rho_{\sigma}(s) \tr\sigma(f^{\rG_2})
    =
    S^{G^{\mathfrak{e}}}_{\phi^{\mathfrak{e}}}(f^{\rG_2}_{G^{\mathfrak{e}}}).
  \end{align*}

  \item
  Every member of the distinguished packet $\Pi_{\iota \circ \phi^{\rG_2}}$ has a canonical extension to $\widetilde{\PGSO_8}(F)$ satisfying Kaletha's twisted endoscopic character relations \cite{Kaletha2022-LocalLanglandsConjecturesDisconnectedGroups}.
  Namely, for any $s \in \widetilde{\cS_{\iota \circ \phi^{\rG_2}}}$ and a corresponding endoscopic datum $\mathfrak{e} = (G^{\mathfrak{e}}, s^{\fe}, \iota^{\PGSO_8}_{G^{\mathfrak{e}}})$, we have
  \begin{align*}
    \sum_{\pi \in \Pi_{\iota \circ \phi^{\rG_2}}} \tr\rho_{\pi}(s) \tr\widetilde{\pi}(f^{\widetilde{\PGSO_8}})
    =
    S^{G^{\mathfrak{e}}}_{\phi^{\mathfrak{e}}}(f^{\widetilde{\PGSO_8}}_{G^{\mathfrak{e}}}).
  \end{align*}
  \end{enumerate}
\end{theorem}

\part{Special cases of the Arthur conjecture for \texorpdfstring{$\rG_2$}{G2}}\label{part:applications}

\section{Global endoscopic lifts from \texorpdfstring{$\PGL_3$}{PGL3}}

Let $F$ be a number field.
Let $\rho$ be a cuspidal automorphic representation of $\PGL_3(\A_F)$.
We define the global $S$-group $\cS_{\iota \circ \rho}$ as
\begin{align*}
    \cS_{\iota \circ \rho}
    =
    \begin{cases}
      \mu_3 = Z(\SL_3) & \text{if $\rho$ is not self-dual}, \\
      S_3 = \mu_3 \rtimes \Z/2\Z = Z(\SL_3) \rtimes \Z/2\Z & \text{if $\rho$ is self-dual}.
    \end{cases}
\end{align*}
For each place $v$, write $s_v$ for the image of $s\in\cS_{\iota\circ\rho}$ in the local component group $\cS_{\iota\circ\rho_v}$.

\begin{theorem}\label{theorem:Arthur-multiplicity-PGL3}
  Let $\Pi_{\iota \circ \rho}$ denote the global $L$-packet for $\rG_2(\A_F)$ associated to $\rho$.

  The discrete multiplicity formula is
  \begin{align*}
    m_{\disc}(\sigma)
    =
    \frac{1}{\abs{\cS_{\iota \circ \rho}}}
    \sum_{s \in \cS_{\iota \circ \rho}}
    \tr\rho_{\sigma}(s),
  \end{align*}
  where $\sigma=\otimes'_v\sigma_v$ is an irreducible representation in the $L$-packet $\Pi_{\iota \circ \rho}$. Here $\rho_{\sigma}$ is the representation of $\cS_{\iota \circ \rho}$ obtained by pulling back each local enhancement $\rho_{\sigma_v}$ along $\cS_{\iota \circ \rho}\to\cS_{\iota \circ \rho_v}$ and taking their tensor product.
\end{theorem}

\begin{proof}
  Write $c=c(\iota\circ\rho)$.
  For $s\in \cS_{\iota\circ\rho}$ and a factorizable test function $f=\otimes'_v f_v$ on $\rG_2(\A_F)$, set
  \begin{align*}
    I^{\rG_2}_{\iota\circ\rho,s}(f)=
    {\prod_v}'\left(
      \sum_{\sigma_v\in\Pi_{\iota\circ\rho_v}}
      \tr\rho_{\sigma_v}(s_v)\tr\sigma_v(f_v)
    \right),
  \end{align*}
  where $s_v$ is the image of $s$ at $v$. At $s=1$ this gives $I^{\rG_2}_{\iota\circ\rho,1}=S^{\rG_2}_{\iota\circ\rho}$, the product of the local stable packet characters.

  Consider first the twisted trace formula for $\widetilde{\PGSO_8}$.
  Its standard Satake parameter is
  $1\boxplus1\boxplus\rho\boxplus\rho^\vee$.
  The same cuspidal-support comparison as in Section~\ref{section:ECR-PGL3}, using uniqueness of isobaric decomposition, leaves the $A_2$-Levi
  $M=(\GL_3\times\GL_1\times\GL_1)/\GL_1$ with inducing data $\rho$ and $\rho^\vee$.
  Write $\widetilde u$ and $r$ for the Weyl elements used in Lemma~\ref{lemma:elliptic-norm-PGL3}.
  The relative Weyl group has order $2$, and
  \begin{align*}
    \abs{\det(1-\widetilde u\mid\fa_M)}=3,
    \quad
    \abs{\det(1-r\widetilde u\mid\fa_M)}=1.
  \end{align*}
  The local gamma-factors in the global normalization are those for $\rho$ and $\rho^\vee$. Here $\wedge^2\rho=\rho^\vee$ because the central character is trivial.
  The functional equations give $1$ for each global product.
  Thus the global intertwining operators factor into the local operators with the Whittaker normalization used in Theorem~\ref{theorem:summary}.

  For each regular twisted Weyl element $\widetilde w$, let $I^{\widetilde{\PGSO_8}}_{P,\widetilde w}(\rho)$ be the product of the local twisted intertwining distributions in Lemma~\ref{lemma:elliptic-norm-PGL3}, where $P$ is the standard parabolic with Levi $M$.
  If $\rho\not\cong\rho^\vee$, the inducing orbit has two members, each contributing $(1/2)(1/3)$, and hence
  \begin{align*}
    I^{\widetilde{\PGSO_8}}_{\disc,c}=\frac13 I^{\widetilde{\PGSO_8}}_{P,\widetilde u}(\rho).
  \end{align*}
  If $\rho\cong\rho^\vee$, there is one inducing representation and two regular twisted Weyl elements, giving
  \begin{align*}
    I^{\widetilde{\PGSO_8}}_{\disc,c}
      =\frac16 I^{\widetilde{\PGSO_8}}_{P,\widetilde u}(\rho)+\frac12 I^{\widetilde{\PGSO_8}}_{P,r\widetilde u}(\rho).
  \end{align*}
  The local character relations identify $I^{\widetilde{\PGSO_8}}_{P,\widetilde u}(\rho)$ with the pullback of $S^{\rG_2}_{\iota\circ\rho}$.
  In the self-dual case, $I^{\widetilde{\PGSO_8}}_{P,r\widetilde u}(\rho)$ is the pullback of the stable character product for the refined generic global $A$-parameter for $\SO_4$ whose standard representation is $1\boxplus\rho$.
  This parameter has trivial Arthur component group modulo the center, so its stable discrete contribution has coefficient $1$, by Theorem~\ref{theorem:refined-Arthur-classification-SO4}.
  In the non-self-dual case there is no $\SO_4$ contribution.
  The $\SL_3$ contribution is zero in both cases. The standard $L$-function of our eight-dimensional parameter has a double pole at $1$, whereas that of a cuspidal adjoint lift is holomorphic there. The remaining exceptional $\SL_3$-type parameter in Remark~\ref{remark:killing-SL3-PGL3-contribution} also has different isobaric constituents.
  Subtracting the $\SO_4$ term in the twisted stabilization and using surjectivity of transfer therefore gives
  \begin{align}\label{equation:stable-coefficient-global-PGL3}
    S^{\rG_2}_{\disc,c}=\frac1{\abs{\cS_{\iota\circ\rho}}}S^{\rG_2}_{\iota\circ\rho}.
  \end{align}

  We next apply the invariant trace formula for $\rG_2$.
  Its proper-Levi terms vanish at $c$.
  Indeed, its standard parameter is $1\boxplus\rho\boxplus\rho^\vee$, whereas a parameter from either proper $\GL_2$-Levi has standard constituents of dimensions at most two, apart from a possible single adjoint constituent of dimension three.
  It cannot have the two cuspidal degree-three constituents of our parameter, counted with multiplicity.
  Thus $I^{\rG_2}_{\disc,c}=R^{\rG_2}_{\disc,c}$.

  Let $a\in \cS_{\iota\circ\rho}$ have order three.
  If $\cS_{\iota\circ\rho}=\mu_3$, the two distinct $\PGL_3$ representations $\rho$ and $\rho^\vee$ contribute, and stabilization together with \eqref{equation:stable-coefficient-global-PGL3} gives
  \begin{align*}
    R^{\rG_2}_{\disc,c}=\frac13(S^{\rG_2}_{\iota\circ\rho}+I^{\rG_2}_{\iota\circ\rho,a}+I^{\rG_2}_{\iota\circ\rho,a^{-1}}).
  \end{align*}
  If $\cS_{\iota\circ\rho}=S_3$, choose a transposition $t\in \cS_{\iota\circ\rho}$.
  The self-dual representation $\rho$ contributes once on the $\PGL_3$ side, and the refined generic global $A$-parameter for $\SO_4$ contributes once on the $\SO_4$ side.
  The local character relations therefore give
  \begin{align*}
    R^{\rG_2}_{\disc,c}
      =\frac16S^{\rG_2}_{\iota\circ\rho}+\frac13I^{\rG_2}_{\iota\circ\rho,a}+\frac12I^{\rG_2}_{\iota\circ\rho,t}.
  \end{align*}
  The coefficients are the conjugacy-class sizes $1,2,3$ divided by $6$.
  In either case we have
  \begin{align*}
    R^{\rG_2}_{\disc,c}=\frac1{\abs{\cS_{\iota\circ\rho}}}\sum_{s\in \cS_{\iota\circ\rho}}I^{\rG_2}_{\iota\circ\rho,s}.
  \end{align*}
  Expand the products and compare irreducible characters. The coefficient of $\tr\sigma$ is
  \begin{align*}
    \frac1{\abs{\cS_{\iota\circ\rho}}}\sum_{s\in \cS_{\iota\circ\rho}}
       {\prod_v}'\tr\rho_{\sigma_v}(s_v)
    =\frac1{\abs{\cS_{\iota\circ\rho}}}\sum_{s\in \cS_{\iota\circ\rho}}\tr\rho_\sigma(s)
    =\dim\rho_\sigma^{\cS_{\iota\circ\rho}}.
  \end{align*}
\end{proof}

\section{Global endoscopic lifts from \texorpdfstring{$\SO_4$}{SO4}}

Let $\psi^{\SO_4}\in\Psi(\SO_4)$ be represented by a pair $(\Pi_1,\Pi_2)$ of cuspidal automorphic representations of $\GL_2(\A_F)$ with $\omega_{\Pi_1}=\omega_{\Pi_2}$.
Let $\psi$ be a formal symbol with Hecke--Satake family $c(\psi)=\iota_{\SO_4}^{\rG_2}\circ c(\psi^{\SO_4})$, and put $\psi_v=\iota_{\SO_4}^{\rG_2}\circ\psi_v^{\SO_4}$ at every place $v$.
Write
\begin{align*}
  \Pi_1\boxtimes\Pi_2^\vee=\boxplus_{i\in I}\tau_i,
\end{align*}
where the $\tau_i$ are distinct orthogonal cuspidal representations of $\GL_{d_i}(\A_F)$ with $\sum_{i\in I}d_i=4$.
Assume that either $\Pi_1\cong\Pi_2$ and $\rho=\Ad(\Pi_1)$ is cuspidal, or $(\Pi_1\boxtimes\Pi_2^\vee)\boxplus\Ad(\Pi_1)$ is multiplicity-free.

In the first case, put $\cS_\psi=\cS_{\iota\circ\rho}=S_3$, with the localization maps defined above.
In the second case, define
\begin{align*}
  \cS_\psi
    =\left\{(\epsilon_i)\in\{\pm1\}^I:
       \prod_{i\in I}\epsilon_i^{d_i}=1\right\}.
\end{align*}
This is the central extension
\begin{align*}
  1\longrightarrow\langle(-1)_{i\in I}\rangle
  \longrightarrow\cS_\psi
  \longrightarrow\cS_{\psi^{\SO_4}}
  \longrightarrow1
\end{align*}
of the Arthur $S$-group for $\SO_4$.
For each $v$, use an orthogonal identification
\begin{align*}
  \Std_{\SO_4}\circ\psi_v^{\SO_4}
    \cong\bigoplus_{i\in I}\phi_{\tau_{i,v}}
\end{align*}
compatible with the refinement of $\psi_v^{\SO_4}$.
The transformation acting by $\epsilon_i$ on the $i$th summand has determinant one. Its image under $\iota_{\SO_4}^{\rG_2}$ centralizes $\psi_v$ and defines a map $\cS_\psi\to\cS_{\psi_v}$, well defined up to conjugacy.

\begin{theorem}\label{theorem:Arthur-multiplicity-SO4}
  For $\Pi_\psi=\bigotimes_v'\Pi_{\psi_v}$, we have
  \begin{align*}
    \cA^2(\rG_2)_\psi
      &=\sum_{\sigma\in\Pi_\psi}m_{\mathrm A}(\sigma)\sigma,\\
    m_{\mathrm A}(\sigma)
      &=\frac1{\abs{\cS_{\psi}}}\sum_{s\in \cS_{\psi}}\tr\rho_\sigma(s)
       =\dim\rho_\sigma^{\cS_{\psi}},
  \end{align*}
  where $\rho_\sigma$ is the tensor product of the local enhancements pulled back along $\cS_{\psi}\to\cS_{\psi_v}$.
\end{theorem}

\begin{proof}
  The $\PGL_3$-type case is already covered by Theorem~\ref{theorem:Arthur-multiplicity-PGL3}.
  Otherwise put $\Phi=\mathbf1\boxplus(\Pi_1\boxtimes\Pi_2^\vee)\boxplus\Ad(\Pi_1)$, a discrete generic parameter for $\SO_8$.
  The fiber product $\GSO_4\times_{\GL_1}\GSO_4$ is taken with respect to the similitude characters.
  Consider its endoscopic transfer to $\GSO_8$ with source represented by
  \begin{align*}
    (\Pi_1^\vee,\Pi_1^\vee)\boxtimes(\Pi_1,\Pi_2).
  \end{align*}
  The diagonal scalar center acts trivially, so the transfer descends to $\PGSO_8$.
  The parameter $\mathbf1\boxplus\Ad(\Pi_1)$ has no GL-type simple constituent in the sense of \cite{Xu2025-GlobalLPacketsQuasisplitGSp2nGO2n}*{Definition 4.4 and Remark 4.5}.
  Thus these local endoscopic transfers form a global Xu packet by \cite{Xu2025-GlobalLPacketsQuasisplitGSp2nGO2n}*{Theorem 1.2}.
  Its generic member $\pi$ has trivial global pairing and occurs in the discrete spectrum by \cite{Xu2025-GlobalLPacketsQuasisplitGSp2nGO2n}*{Theorem 1.1}.
  Let $s_1$ be a pinned outer automorphism exchanging the standard representation and a half-spin representation.
  By \cite{ChenevierGan2025-TrialityFunctoriality}*{(6.11) and Proposition 6.13}, the standard weak lifts satisfy
  \begin{align*}
    \operatorname{Lift}_{\PGSO_8}^{\GL_8}(\pi)=\operatorname{Lift}_{\PGSO_8}^{\GL_8}(s_1\pi)
    =(\Pi_1^\vee\boxtimes\Pi_1)\boxplus(\Pi_1^\vee\boxtimes\Pi_2)=\Phi.
  \end{align*}
  The global classification for $\SO_8$ therefore identifies their standard local parameters with $\Phi_v$ at every place $v$, by \cite{Xu2025-GlobalLPacketsQuasisplitGSp2nGO2n}*{(2.18)}.
  Both contain $\mathbf1$, so Theorem~\ref{theorem:characterization-distinguished-L-packets} identifies the local packet with the distinguished packet at tempered non-archimedean places.
  At archimedean places, use the usual local Langlands correspondence.
  The non-tempered case follows by normalized induction, since the trivial constituents remain in the tempered inducing data, as in \cite{Xu2018-LPacketsQuasisplitGSp2nGO2n}*{Proposition 3.11}.

  Let $Y$ be the group of quadratic Hecke characters composed with the similitude character $\lambda$, and let $\alpha\colon\cS_\Phi\to Y$ be Xu's determinant map.
  For an automorphic member $\pi$ of this packet, put $Y(\pi)=\{\chi\in Y:\pi\otimes\chi\cong\pi\}$.
  If $\omega\circ\lambda\in Y(\pi)$, Lemma~\ref{lemma:triality-quadratic-twist-global} gives $\Phi\otimes\omega\cong\Phi$.
  For $\omega\ne1$, the parameter $\Phi$ therefore contains the two one-dimensional constituents $\mathbf1,\omega$.
  The sign transformation equal to $-1$ on these two constituents and $1$ on the others maps to $\omega\circ\lambda$ under $\alpha$, by \cite{Xu2018-LPacketsQuasisplitGSp2nGO2n}*{Section 6.2 and Lemma 6.6(2)}.
  Hence $Y(\pi)=\alpha(\cS_\Phi)$.

  The determinant-one condition defining $\cS_\psi$ identifies it with $\Ker\alpha$.
  Xu's multiplicity formula therefore gives
  \begin{align*}
    I_{\disc,c(\psi)}^{\widetilde{\PGSO_8}}(f)
    =\frac1{\abs{\cS_{\psi}}}\sum_{s\in \cS_{\psi}}{\prod_v}'
      \left(\sum_{\pi_v\in\Pi_{\iota\circ\psi_v}}
        \langle s_v,\pi_v\rangle\tr\widetilde\pi_v(f_v)\right).
  \end{align*}

  By Theorem~\ref{theorem:summary}, the term $s=1$ is the transfer of $S^{\rG_2}_{\psi}$, defined as in the preceding proof.
  For each $s\ne1$, the corresponding refined generic parameter $\psi_s$ for $\SO_4$ has $\cS_{\psi_s}\cong \cS_{\psi}/\langle s\rangle$, by the fibers in Section~\ref{section:lifting-SO4-G2}.
  Its stable contribution has coefficient $\frac12\abs{\cS_{\psi_s}}^{-1}=\abs{\cS_{\psi}}^{-1}$ by Theorem~\ref{theorem:stable-multiplicity-SO4}.
  Subtracting these terms in the twisted stabilization gives the coefficient $\abs{\cS_{\psi}}^{-1}$ for $S^{\rG_2}_{\psi}$.
  The proper-Levi terms for $\rG_2$ vanish by the multiplicity-free condition, and the remaining endoscopic terms vanish by Remark~\ref{remark:killing-SL3-PGL3-contribution}.
  The ordinary stabilization and the local character relations now give $\abs{\cS_{\psi}}^{-1}\sum_{s\in \cS_{\psi}}I^{\rG_2}_{\psi,s}$.
  Comparison of irreducible characters proves the assertion.
\end{proof}

\section{Global stable \texorpdfstring{$L$}{L}-packets}

Let $F$ be a number field and let $\Pi$ be a cuspidal automorphic representation of $\PGSO_8(\A_F)$ with triality-invariant Hecke--Satake family. Suppose its standard lift has isobaric decomposition
\[
  \operatorname{Lift}_{\PGSO_8}^{\GL_8}(\Pi)=\Pi_7\boxplus\mathbf1,
\]
where $\Pi_7$ is a cuspidal representation of $\GL_7(\A_F)$ of orthogonal type. Thus $L^S(s,\Pi_7,\Sym^2)$ has a pole at $s=1$ for $S$ containing the archimedean and ramified places.

We first check the two hypotheses of \cite{HundleyLiu2023AutomorphicDescenttoG2}*{Theorem 1.0.2}. At almost every $v$, triality invariance gives the same value $x_v$ for the three outer fundamental characters of $\Spin_8(\C)$ on $c(\Pi_v)$. Write $y_v$ for the value of the remaining fundamental character. On $\rG_2(\C)$ these four characters restrict to three copies of $1+\chi_7$ and one copy of $\chi_{14}+2\chi_7$, where $\chi_7$ and $\chi_{14}$ are the two fundamental characters of $\rG_2$. Fundamental characters give coordinates on the semisimple conjugacy classes of either simply connected group. Thus the $\rG_2$-class $c_v$ with $\chi_7(c_v)=x_v-1$ and $\chi_{14}(c_v)=y_v-2(x_v-1)$ maps to $c(\Pi_v)$. Since the eight-dimensional standard representation restricts to $\mathbf1\oplus\Std_{\rG_2}$, its seven-dimensional summand gives $c(\Pi_{7,v})=\Std_{\rG_2}(c_v)$.

On $\rG_2(\C)$ one has $\wedge^3\Std_{\rG_2}\cong\Sym^2\Std_{\rG_2}\oplus\Std_{\rG_2}$. Comparing the unramified factors as in \cite{HundleyLiu2023AutomorphicDescenttoG2}*{Lemma 4.2.15(2)}, and enlarging $S$ if necessary, gives
\[
  L^S(s,\Pi_7,\wedge^3)=L^S(s,\Pi_7,\Sym^2)L^S(s,\Pi_7).
\]
Since $L^S(1,\Pi_7)\ne0$, the exterior cube has a pole at $1$. The descent theorem now gives a globally generic cuspidal representation $\Sigma_{\gen}$ of $\rG_2(\A_F)$ with Hecke--Satake family $c=(c_v)_{v\notin S}$.

Let $s_1$ be the pinned involution exchanging the standard representation of $\Spin_8$ with a half-spin representation. Since $c(\Pi_v)$ comes from $\rG_2$ almost everywhere, $\Pi$ and $s_1\Pi$ have the same standard lift by strong multiplicity one. By \cite{Xu2025-GlobalLPacketsQuasisplitGSp2nGO2n}*{(2.18)}, their standard parameters at every $v$ are therefore both $\mathbf1\oplus\phi_{\Pi_{7,v}}$.

At a tempered finite place, the trivial summand makes both gamma factors in Theorem~\ref{theorem:characterization-distinguished-L-packets} vanish. Thus $\Pi_v$ lies in the distinguished packet for a parameter $\phi_v$ of $\rG_2(F_v)$, and Theorem~\ref{theorem:invariance-lifting-non-archimedean} gives triality invariance of its members.

At a spherical finite place, apply the same comparison of standard lifts to all three outer nodes. The three outer characters of the Satake parameter agree, so the fundamental-character calculation above gives its $\rG_2$-class $c_v$. The spherical member belongs to the distinguished packet for this unramified parameter by compatibility of spherical theta lifting with Satake parameters (Theorem~\ref{theorem:spherical-theta-parameters}).

At an archimedean place the local correspondence is equivariant under pinned automorphisms. The standard and half-spin representations of the parameter each have a fixed vector of nonzero norm since their fixed spaces are nondegenerate by complete reducibility. The stabilizer of the first vector is $\Spin_7(\C)$, and the stabilizer in $\Spin_7(\C)$ of the second, a nonisotropic spinor, is $\rG_2(\C)$. This gives the local parameter $\phi_v$ there.

\begin{theorem}\label{theorem:Arthur-multiplicity-stable-G2}
  At each ramified non-tempered finite place, assume that the local Xu packet containing $\Pi_v$ is the distinguished packet for a generic $\rG_2$-parameter $\phi_v$. For the family $\psi=(\phi_v)_v$, let $\Pi_\psi={\bigotimes_v}'\Pi_{\phi_v}$. The discrete spectrum at $c$ is
  \[
    L^2_{\disc}\bigl(\rG_2(F)\backslash\rG_2(\A_F)\bigr)_c
    =\bigoplus_{\Sigma\in\Pi_\psi}m_{\disc}(\Sigma)\,\Sigma,
    \qquad m_{\disc}(\Sigma)=\prod_v\dim\rho_{\Sigma_v}.
  \]
\end{theorem}

\begin{proof}
  The two summands of $\mathbf1\boxplus\Pi_7$ are distinct and odd-dimensional. The determinant-one sign changes are therefore precisely the diagonal signs, so the global Arthur component group modulo the center is trivial. A quadratic twist contributing at $c(\Pi)$ would stabilize the standard lift, by the unramified triality calculation in Lemma~\ref{lemma:SO4-global-PGSO8-packet}. Comparison of the unique degree-one constituents forces the twisting character to be trivial. Thus Xu's multiplicity formula \cite{Xu2025-GlobalLPacketsQuasisplitGSp2nGO2n}*{Theorem 1.1 and Lemma 2.2(4)} gives each member of this global packet multiplicity one on $\PGSO_8$. Corollary~\ref{corollary:canonical-extension-global} identifies the global twisted characters with products of the canonical local extensions.

  At $c$, cuspidality of $\Pi_7$ excludes the proper-Levi, $\PGL_3$-, and $\SO_4$-terms by Lemma~\ref{lemma:discrete-part-G2}, Remark~\ref{remark:killing-SL3-PGL3-contribution}, and Section~\ref{section:lifting-SO4-G2}. The standard $L$-function on the twisted side is $\zeta^S(s)L^S(s,\Pi_7)$, whose pole at $1$ also excludes the $\SL_3$-term by the same remark. For matching factorizable test functions $\widetilde f$ and $f$, stabilization and the local character identities of Theorem~\ref{theorem:summary} now give
  \begin{align*}
    R^{\rG_2}_{\disc,c}(f)
    &=I^{\widetilde{\PGSO_8}}_{\disc,c(\Pi)}(\widetilde f)
      ={\prod_v}' S^{\rG_2}_{\phi_v}(f_v)\\
    &={\prod_v}'\left(\sum_{\sigma_v\in\Pi_{\phi_v}}
         \dim\rho_{\sigma_v}\,\tr\sigma_v(f_v)\right).
  \end{align*}
  Almost every $\Sigma_v$ is spherical and has trivial enhancement. Comparing coefficients of irreducible characters gives the multiplicity formula.
\end{proof}

\part{Appendices}\label{part:appendices}

\appendix

\section{A refinement of Arthur's results on \texorpdfstring{$\SO_4$}{SO4}}\label{section:refinement-SO4}

\subsection{\texorpdfstring{$\GSO_4$}{GSO4} and \texorpdfstring{$\SO_4$}{SO4}}

For a split four-dimensional quadratic space $(V,q)$ over a field $F$, put $\GSO_4 = \GSO(V,q)$ and $\SO_4 = \SO(V,q)$.
The identities $\GL_2 = \GSp_2$ and $\SL_2 = \Sp_2$ give the accidental isomorphism
\begin{align*}
   \alpha \colon \GL_2 \times \GL_2 / \nabla \GL_1 \xrightarrow{\sim} \GSO_4 \colon (g_1, g_2) \mapsto g_1 \otimes g_2,
\end{align*}
which also induces the isomorphism
\begin{align*}
   \{ (g_1, g_2) \in \GL_2 \times \GL_2 \mid \det g_1 \det g_2 = 1\} / \nabla \GL_1 \xrightarrow{\sim} \SO_4.
\end{align*}
Let $\beta$ denote their inverse maps.

\begin{remark}
   Interchanging the two $\GL_2$ factors defines a non-trivial outer automorphism $\theta$ of $\SO_4$ extending to $\GSO_4$.
\end{remark}

For a $p$-adic field $F$, the dual isomorphism is
\begin{align*}
   \widehat{\alpha} \colon \widehat{\GSO_4} = \GSpin_4 \xrightarrow{\sim} \{ (g_1, g_2) \in \GL_2 \times \GL_2 \mid \det g_1 = \det g_2\}
\end{align*}
and we obtain the inverse of the isomorphism
\begin{align*}
   \{ (g_1, g_2) \in \GL_2(\C) \times \GL_2(\C) \mid \det g_1 = \det g_2\}/\Delta \GL_1  \xrightarrow{\sim} \SO_4(\C)
   \colon (g_1, g_2) \mapsto g_1 \otimes g_2^{\vee}.
\end{align*}

\begin{remark}
   Let $V_1,V_2$ be the two standard two-dimensional representations of the group of pairs with equal determinants above. We have an isomorphism of representations
   \begin{align*}
      V_1 \otimes V_2^{\vee} \simeq V_1^{\vee} \otimes V_2
   \end{align*}
   since $V_i^{\vee}\simeq V_i\otimes(\det V_i)^{-1}$ and $\det V_1=\det V_2$.
\end{remark}

\begin{definition}
   Let $\wedge^2_{+}$ denote the composition of the adjoint representation of $\PGL_2$ and
   \begin{align*}
      \SO_4 \xrightarrow{\sim} \SL_2 \times \SL_2 / \Delta \mu_2 \xrightarrow{\pr_1} \SL_2/\mu_2.
   \end{align*}
   Similarly, let $\wedge^2_{-}$ denote the composition of the adjoint representation of $\PGL_2$ and
   \begin{align*}
      \SO_4 \xrightarrow{\sim} \SL_2 \times \SL_2 / \Delta \mu_2 \xrightarrow{\pr_2} \SL_2/\mu_2.
   \end{align*}

   The maps above
   \begin{align*}
      \SO_4 \twoheadrightarrow \SO_3 = \PGL_2
   \end{align*}
   are surjective with kernels isomorphic to $\SL_2$, hence surjective on $F$-rational points for every field $F$ of characteristic $0$.
\end{definition}

\subsection{A refined local Langlands correspondence for \texorpdfstring{$\SO_4$}{SO4}}\label{section:refined-LLC-SO4}

Let $F$ be a $p$-adic field.
For any irreducible admissible representation $\pi$ of a reductive group $G$ over $F$, let $\omega_{\pi}$ denote the central character of $\pi$.
Then, we have a bijection
\begin{align*}
   \alpha^{*} \colon
   \Irr(\GSO_4(F)) \xrightarrow{\sim}
   \{ (\pi_1, \pi_2) \in \Irr(\GL_2(F)) \times \Irr(\GL_2(F)) \mid \omega_{\pi_1} = \omega_{\pi_2} \}.
\end{align*}
On the dual side, we have the corresponding bijection
\begin{align*}
   \widehat{\alpha}^{*} \colon
   \{ (\phi_1, \phi_2) \in \Phi(\GL_2(F)) \times \Phi(\GL_2(F)) \mid \det \phi_1 = \det \phi_2\}
   \xrightarrow{\sim} \Phi(\GSO_4(F))
\end{align*}
and
\begin{align*}
   \widehat{\alpha}^{*} \colon
   \{ (\phi_1, \phi_2) \in \Phi(\GL_2(F)) \times \Phi(\GL_2(F)) \mid \det \phi_1 = \det \phi_2\}
   / \sim_{\Delta \GL_1}
   \xrightarrow{\sim} \Phi(\SO_4(F)),
\end{align*}
where $\sim_{\Delta \GL_1}$ is the equivalence relation defined by twists by characters of $F^{\times} \cong W_F^{\mathrm{ab}}$.
Clifford theory gives the following lemma.
\begin{lemma}
   Let $\mu$ be the similitude character of $\GSO_4(F)$ and let $\sigma_1', \sigma_2'$ be irreducible admissible representations of $\GSO_4(F)$ whose restrictions to $\SO_4(F)$ contain a common irreducible admissible representation $\sigma$.
   Then $\sigma_1' \cong \sigma_2' \otimes \chi \circ \mu$ for some character $\chi$ of $F^{\times}$.
\end{lemma}

\begin{remark}
   The twist by $\chi \circ \mu$ corresponds to the twist $(\pi_1, \pi_2) \mapsto (\pi_1 \otimes \chi \circ \det, \pi_2 \otimes \chi \circ \det)$ via the bijection $\alpha^*$.
\end{remark}

\begin{definition}
   Define
      \begin{align*}
         \cL_{\SO_4}^{\mathrm{ref}} \colon \Irr(\SO_4(F)) \to \Phi(\SO_4(F)) \colon \sigma \mapsto \phi_{\sigma}
      \end{align*}
   by choosing an irreducible admissible representation $\sigma'$ of $\GSO_4(F)$ whose restriction contains $\sigma$ and projecting its $L$-parameter $\phi'$ under $\widehat{\GSO_4} \twoheadrightarrow \widehat{\SO_4}$.
   Such a $\sigma'$ exists, and the preceding lemma makes the image independent of its choice.
\end{definition}

It is compatible with Arthur's local Langlands correspondence.

\begin{theorem}[cf.\ \cite{GeeTaibi2019-ArthurSMultiplicityFormulaGSp4Restriction}*{Theorem 8.3.3}]\label{theorem:compatibility-with-Arthur-local}
   Let $\overline{\Irr}(\SO_4)$ (resp. $\overline{\Phi}(\SO_4)$) be the set of equivalence classes of irreducible admissible representations of $\SO_4(F)$ (resp. $L$-parameters for  $\SO_4(F)$) under the action of the outer automorphism group $\Out(\SO_4) \cong \Z/2\Z$.
   The following diagram commutes.
   \begin{align*}
      \begin{CD}
         \Irr(\SO_4(F)) @>{\cL_{\SO_4}^{\mathrm{ref}}}>> \Phi(\SO_4(F)) \\
         @VVV @VVV \\
         \overline{\Irr}(\SO_4(F)) @>{\cL^{A}_{\SO_4}}>> \overline{\Phi}(\SO_4(F)).
      \end{CD}
   \end{align*}
   Here $\cL^A_{\SO_4}$ is Arthur's local Langlands correspondence.
\end{theorem}

\begin{proof}
   For tempered parameters this is the cited theorem.  In general,
   restriction from $\GSO_4$ to $\SO_4$ commutes with normalized
   parabolic induction and preserves the positive Langlands data.
   Applying the tempered case to the inducing representations and
   taking Langlands quotients proves the assertion. See also
   \cite{GeeTaibi2019-ArthurSMultiplicityFormulaGSp4Restriction}*{Section 8.1, proof of Theorem 8.1.2}.
\end{proof}

\begin{remark}
   Over archimedean local fields, the same construction is compatible with the Langlands--Shelstad correspondence as in Theorem~\ref{theorem:compatibility-with-Arthur-local}.
   For tempered representations, use the Harish-Chandra character formula and compatibility with parabolic induction, then take Langlands quotients for the general case.
\end{remark}

The construction also gives the following lemma.
\begin{lemma}
   Let $\sigma_1, \sigma_2$ be irreducible admissible representations of $\SO_4(F)$.
   They have a common $L$-parameter if and only if there exists an irreducible admissible representation $\sigma'$ of $\GSO_4(F)$ such that both $\sigma_1$ and $\sigma_2$ are contained in $\sigma'|_{\SO_4(F)}$.
\end{lemma}

\begin{proposition}
   Let $F$ be a local field.
   Let $(V, q)$ be a non-degenerate quadratic space of dimension $4$ over $F$.
   Let $\sigma'$ be an irreducible admissible representation of $\GSO(V, q)$.
   Then, the restriction $\sigma'|_{\SO(V, q)}$ is multiplicity free.
\end{proposition}

\begin{proof}
   For non-archimedean $F$, this is \cite{GeeTaibi2019-ArthurSMultiplicityFormulaGSp4Restriction}*{Proposition 8.2.1}.
   For archimedean $F$, the similitude character embeds $\GSO(V,q)/(F^\times\SO(V,q))$ into $F^\times/(F^\times)^2$, which is cyclic.
   Now apply \cite{AdlerPrasad2019-MultiplicityUponRestrictionDerivedSubgroup}*{Section 6, p.~9}.
\end{proof}

\begin{remark}
   If $\phi^{\SO_4}$ is represented by $(\phi_1,\phi_2)$, then the lifts
   $\wedge^2_{\pm}\colon\Phi(\SO_4(F))\to\Phi(\GL_3(F))$ are given by
   \begin{align*}
      \wedge^2_+\phi^{\SO_4}&=\Ad_{\GL_2}\circ\phi_1,\\
      \wedge^2_-\phi^{\SO_4}&=\Ad_{\GL_2}\circ\phi_2.
   \end{align*}
   The local Langlands correspondence defines the corresponding lifts
   $\Irr(\SO_4(F))\to\Irr(\GL_3(F))$, denoted by $\sigma\mapsto\wedge^2_{\pm}\sigma$.
\end{remark}

\begin{proposition}\label{proposition:refined-packet-characterization}
   Let $\phi$ be an $L$-parameter for $\SO_4(F)$.
   Let $\overline{\Pi}_{\phi}$ be Arthur's $L$-packet associated to $\phi$ and let $\Pi_{\phi}$ be the refined packet associated to $\phi$.
   Then, the following are equivalent for an irreducible admissible representation $\sigma$ of $\SO_4(F)$.
   \begin{enumerate}
      \item $\sigma \in \Pi_{\phi}$.
      \item $\sigma \in \overline{\Pi}_{\phi}$ and $\wedge^2_{+}\phi = \wedge^2_{+} \sigma$.
   \end{enumerate}
\end{proposition}

\begin{proof}
   The implication $(1) \Rightarrow (2)$ is clear. For the converse, Theorem~\ref{theorem:compatibility-with-Arthur-local} gives $\sigma \in \Pi_{\phi}$ or $\sigma^{\theta} \in \Pi_{\phi}$.
   Let $(\pi_1, \pi_2)$ be the pair corresponding to $\sigma$, so $\wedge^2_{+} \sigma = \Ad(\pi_1)$ and $\wedge^2_{+} \sigma^{\theta} = \Ad(\pi_2)$.
   In the second case, $\Ad(\pi_2) = \wedge^2_{+} \phi = \Ad(\pi_1)$, hence $\pi_1 \cong \pi_2 \otimes \chi \circ \det$ for a quadratic character $\chi$ of $F^{\times}$. Therefore
   \begin{align*}
      \Std_{\SO_4} \circ \phi = \Ad \circ \phi_{\pi_1} \otimes \chi \oplus \chi.
   \end{align*}
   Every element in Arthur's packet is therefore $\theta$-invariant, so $\sigma^{\theta} = \sigma$.
\end{proof}

\subsection{A refined Arthur classification for \texorpdfstring{$\SO_4$}{SO4}}

Let $F$ be a number field.
For any reductive group $G$ over $F$, let $\Pi_{2}(G(\A_F))$ (resp. $\Pi_{\cusp}(G(\A_F))$) denote the set of equivalence classes of discrete (resp. cuspidal) automorphic representations of $G(\A_F)$.
For an automorphic representation $\pi$ of $G(\A_F)$, let $c(\pi)$ denote the Hecke--Satake parameter of $\pi$.
Similarly to the local case, we have a bijection
\begin{align*}
   \alpha^{*} \colon \Pi_{?}(\GSO_4(\A_F)) \xrightarrow{\sim}
   \{ (\pi_1, \pi_2) \in \Pi_{?}(\GL_2(\A_F)) \times \Pi_{?}(\GL_2(\A_F)) \mid \omega_{\pi_1} = \omega_{\pi_2} \}
\end{align*}
for $? \in \{2, \cusp\}$.

\begin{lemma}[\cite{GeeTaibi2019-ArthurSMultiplicityFormulaGSp4Restriction}*{Lemma 5.1.1}]
   The restriction from $\GSO_4(\A_F)$ to $\SO_4(\A_F)$ of discrete automorphic representations is compatible with the formation of Satake parameters.
\end{lemma}

\begin{theorem}[\cite{GeeTaibi2019-ArthurSMultiplicityFormulaGSp4Restriction}*{Theorem 5.1.2}]
   Any discrete (resp. cuspidal) automorphic representation of $\SO_4(\A_F)$ is contained in a restriction of a discrete (resp. cuspidal) automorphic representation of $\GSO_4(\A_F)$.
\end{theorem}

Theorem~\ref{theorem:compatibility-with-Arthur-local} gives the following compatibility.
\begin{theorem}
   Let $\Sigma$ be a discrete (resp. cuspidal) automorphic representation of $\SO_4(\A_F)$.
   Let $\Sigma'$ be a discrete (resp. cuspidal) automorphic representation of $\GSO_4(\A_F)$ such that $\Sigma$ is contained in $\Sigma'|_{\SO_4(\A_F)}$.
   Let $\alpha^* \Sigma' = (\Pi_1, \Pi_2)$ be the pair of discrete (resp. cuspidal) automorphic representations of $\GL_2(\A_F)$ corresponding to $\Sigma'$.
   Then, the $A$-parameter of $\Sigma$ is given by the Rankin--Selberg product $\Pi_1 \times \Pi_2^{\vee} = \Pi_1^{\vee} \times \Pi_2$ of the $L$-parameters of $\Pi_1$ and $\Pi_2$.
\end{theorem}

\begin{theorem}[cf.\ \cite{LapidMao2015-ConjectureWhittakerFourierCoefficientsCuspForms}*{Section 6.3}]\label{theorem:SO4-simultaneous-twist}
   Let $\Pi^a = (\pi^a_1, \pi^a_2),\Pi^b = (\pi^b_1, \pi^b_2)$ be two pairs of cuspidal automorphic representations of $\GL_2(\A_F)$ such that $\omega_{\pi^a_1} = \omega_{\pi^a_2}$ and $\omega_{\pi^b_1} = \omega_{\pi^b_2}$.
   Let $\Sigma'_a$ (resp. $\Sigma'_b$) denote the pullbacks $\beta^*(\Pi^a)$ (resp. $\beta^*(\Pi^b)$).
   If we have $\Ad(\pi^a_i) = \Ad(\pi^b_i)$ for $i = 1, 2$ and $\pi^a_1 \times {\pi^a_2}^{\vee} = \pi^b_1 \times {\pi^b_2}^{\vee}$, then there exists a character $\chi$ of $\A_F^{\times}/F^{\times}$ such that $\pi^a_i = \pi^b_i \otimes \chi \circ \det$.

   Furthermore, let $\Sigma$ be a discrete automorphic representation of $\SO_4(\A_F)$.
   Then, we have
   \begin{align*}
      m_{\disc}(\Sigma) = 1.
   \end{align*}
\end{theorem}

\begin{proof}
   By \cite{Ramakrishnan2000-ModularityRankinSelbergSeriesMultiplicityOne}*{Theorem 4.1.2}, write $\pi_i^a=\pi_i^b\otimes\chi_i\circ\det$ with Hecke characters $\chi_i$.
   Equality of the central characters within each pair makes $\chi_1\chi_2^{-1}$ quadratic, and equality of the tensor transfers makes it a self-twist of $\pi_1^b\times{\pi_2^b}^{\vee}$.
   The comparison in \cite{LapidMao2015-ConjectureWhittakerFourierCoefficientsCuspForms}*{Section 6.3, pp.~487--488} then shows that the two twists can be chosen equal. The same section proves cuspidal multiplicity one for $\SO_4$.

   For a residual representation $\Sigma$, choose a discrete lift to $\GSO_4$ as above. The lift is not cuspidal, since restriction preserves cuspidality. Thus at least one of its $\GL_2$ factors is a character $\eta\circ\det$.
   After interchanging the factors and twisting both by $\eta^{-1}$, the pair is $(\mathbf{1},\tau)$ with $\omega_\tau=1$. Thus $\Sigma$ factors through the second projection $\SO_4\to\PGL_2$, whose kernel is $\SL_2$.
   This projection is surjective on rational and adelic points. Since $\SL_2(F)\backslash\SL_2(\A_F)$ has finite volume, the $\SL_2(\A_F)$-invariant subspace of the discrete automorphic spectrum identifies with the discrete spectrum of $\PGL_2$.
   Multiplicity one for $\GL_2$ therefore gives $m_{\disc}(\Sigma)=1$. For the classification and multiplicity one used here, see \cite{GeeTaibi2019-ArthurSMultiplicityFormulaGSp4Restriction}*{Remarks 2.5.1(1) and 4.1.7(1)}.
\end{proof}

\begin{remark}
   We continue to use the notation in the theorem above.
   If $\Sigma$ occurs in $\Sigma'_a|_{\SO_4(\A_F)}$, then
   \begin{align*}
      c(\Ad(\pi^a_1))&=\wedge^2_+c(\Sigma),\\
      c(\Ad(\pi^a_2))&=\wedge^2_-c(\Sigma).
   \end{align*}
\end{remark}

\begin{corollary}
   Let $\Sigma_1, \Sigma_2$ be discrete automorphic representations of $\SO_4(\A_F)$.
   Then, they have a common Satake parameter if and only if there exists a discrete automorphic representation $\Sigma'$ of $\GSO_4(\A_F)$ such that both $\Sigma_1$ and $\Sigma_2$ are contained in $\Sigma'|_{\SO_4(\A_F)}$.
\end{corollary}

\begin{definition}
   A refined $A$-packet for $\SO_4$ consists of the discrete automorphic representations of $\SO_4(\A_F)$ contained in the restriction of a fixed discrete automorphic representation of $\GSO_4(\A_F)$.
   Put
   \begin{align*}
      \Psi(\SO_4) =
       \{ (\Pi_1, \Pi_2) \in \Pi_{2}(\GL_2(\A_F)) \times \Pi_{2}(\GL_2(\A_F)) \mid \omega_{\Pi_1} = \omega_{\Pi_2}\}/\sim_{\Delta \GL_1},
   \end{align*}
   where $\sim_{\Delta \GL_1}$ denotes simultaneous twisting by a character of $\A_F^{\times}/F^{\times}$.
   An element of $\Psi(\SO_4)$ is called a refined $A$-parameter for $\SO_4$.

   For $\psi \in \Psi(\SO_4)$, the refined local $A$-packet $\Pi_{\psi_v}(\SO_4)$ consists of the irreducible admissible constituents of the restriction to $\SO_4(F_v)$ of the local component at $v$ of a discrete automorphic representation of $\GSO_4(\A_F)$ corresponding to $\psi$.
   Let $\Pi_{\psi}(\SO_4) = \otimes_v ' \Pi_{\psi_v}(\SO_4)$.
\end{definition}

\begin{remark}
   For a refined $A$-parameter $\psi$ for $\SO_4$, write its standard formal parameter as
   $\boxplus_{i\in I}(\tau_i\boxtimes S_{d_i}^A)$, with distinct orthogonal simple summands, and put
   \begin{align*}
      S_{\psi}
      &=\left\{(\epsilon_i)\in\{\pm1\}^I:
         \prod_{i\in I}\epsilon_i^{d_i\deg\tau_i}=1\right\},\\
      \cS_{\psi}
      &=S_{\psi}/\langle(-1)_{i\in I}\rangle.
   \end{align*}
   At each place $v$, define $S_{\psi_v}=\Cent(\Image(\psi_v),\widehat{\SO}_4)$ and
   $\cS_{\psi_v}=\pi_0(S_{\psi_v}/Z(\widehat{\SO}_4)^{\Gamma_{F_v}})$.
   Localization sends $(\epsilon_i)$ to the transformation acting by $\epsilon_i$ on the local parameter of $\tau_i\boxtimes S_{d_i}^A$.
   Let $\epsilon_{\psi}$ be Arthur's character of
   $\cS_{\psi}$.
   Let $\Delta \colon \cS_{\psi} \to \prod_v \cS_{\psi_v}$ denote the natural diagonal map.
\end{remark}

\begin{remark}
   Since the map
   \begin{align*}
      \SO_4 \xrightarrow{\pr_1} \SL_2/\mu_2 = \PGL_2
   \end{align*}
   is surjective on $F$-rational points for every field $F$ of characteristic $0$, local $A$-packets can be described using the local $L$-packets of $\PGL_2(F)$.
   For example, let $\phi$ be a local generic $L$-parameter for $\PGL_2(F)$.
   The local $A$-packet of the local $A$-parameter
   \begin{align*}
      L_F \times \SL_2(\C) \xrightarrow{\phi \times S_2^A} \SL_2(\C) \times \SL_2(\C) \twoheadrightarrow \SO_4(\C)
   \end{align*}
   consists of the pullback to $\SO_4(F)$ of the unique irreducible admissible representation of $\PGL_2(F)$ corresponding to $\phi$.
\end{remark}

\begin{lemma}
   Let $\psi$ be a refined $A$-parameter for $\SO_4$.
   \begin{enumerate}
      \item
      Suppose that $\psi$ is represented by a pair
      $(\chi_1\boxtimes S_2^A,\chi_2\boxtimes S_2^A)$, and put
      $\delta=\chi_1\chi_2^{-1}$.  The equality of the central
      characters gives $\delta^2=1$, and
      \begin{align*}
         \Pi_1\times\Pi_2^{\vee}
         =
         (\delta\boxtimes S_3^A)
         \boxplus
         (\delta\boxtimes S_1^A).
      \end{align*}
      We also have $\cS_{\psi} = 1$ and $\epsilon_{\psi} = \mathbf{1}$.
      \item
      If $\psi$ is represented by a pair $(\Pi_1, \Pi_2)$ where $\Pi_1$ is the standard representation $S_2^A$ of Arthur's $\SL_2$ and $\Pi_2$ is a cuspidal automorphic representation of $\PGL_2(\A_F)$, then we have $\cS_{\psi} = 1$ and $\epsilon_{\psi} = \mathbf{1}$.
      \item
      Assume that $\psi$ is represented by a pair $(\Pi_1, \Pi_2)$ of cuspidal automorphic representations of $\GL_2(\A_F)$ with $\omega_{\Pi_1} = \omega_{\Pi_2}$.
      Then, $\Pi_1 \times \Pi_2^{\vee}$ is non-cuspidal precisely when one of the representations is non-dihedral and $\Pi_1 \cong \Pi_2 \otimes \chi \circ \det$ for an automorphic character $\chi$ of $\A_F^{\times}$, or both representations are dihedral and induced from the same quadratic extension of $F$.
      If the Rankin--Selberg product is cuspidal, or if the first
      non-cuspidal case occurs, then $\cS_{\psi}=1$.  In the common
      dihedral case, we have
      $\cS_{\psi}=\Z/2\Z$ or $(\Z/2\Z)^2$, according to its isobaric
      decomposition.
      In any case, we have $\epsilon_{\psi} = \mathbf{1}$.
   \end{enumerate}
\end{lemma}

\begin{proof}
   First suppose that $\psi$ is non-generic.  In the
   residual--residual case, the signs on the $3$- and $1$-dimensional
   summands must agree because the centralizer lies in $\SO_4(\C)$.
   Hence
   \begin{align*}
      S_{\psi}=Z(\widehat{\SO}_4)
      \quad\text{and}\quad
      \cS_{\psi}=1.
   \end{align*}
   In the residual--cuspidal case, the standard parameter is the
   irreducible four-dimensional orthogonal representation
   $\Pi_2^{\vee}\boxtimes S_2^A$, so the same conclusion follows from
   Schur's lemma.  Thus $\epsilon_{\psi}=\mathbf{1}$ in both
   non-generic cases.

   Now suppose that $\psi$ is generic.  If the Rankin--Selberg product
   is cuspidal, then $\cS_{\psi}=1$.  In the non-dihedral
   twist-equivalent case, equality of the central characters implies
   $\chi^2=1$, and the standard parameter has the form
   \begin{align*}
      \chi
      \boxplus
      (\chi\otimes\Ad(\Pi_2)).
   \end{align*}
   The determinant-one condition again forces the signs on the
   $1$- and $3$-dimensional summands to agree, so
   $\cS_{\psi}=1$.  In the remaining common dihedral case, the asserted
   component groups follow from
   Lemma~\ref{lemma:computation-rankin-selberg-induced}.

   Finally, for a generic parameter, the simple self-dual summands have
   the same orthogonal parity.  Their pairwise global Rankin--Selberg
   root numbers are therefore $+1$ by
   \cite{Arthur2013-EndoscopicClassificationRepresentations}*{Theorem
   1.5.3}.  Arthur's definition of $\epsilon_{\psi}$ now gives
   $\epsilon_{\psi}=\mathbf{1}$ also in the common dihedral case.
\end{proof}

\begin{remark}
   Suppose that a refined $A$-parameter is represented by $(\Pi_1,S_2^A)$,
   where $\Pi_1$ is a cuspidal automorphic representation of
   $\PGL_2(\A_F)$.  Its packet consists of the pullback of $\Pi_1$
   under the projection
   \begin{align*}
      \SO_4(\A_F) \xrightarrow{\pr_1} \PGL_2(\A_F).
   \end{align*}
   This representation is irreducible since the projection is
   surjective.  For $(S_2^A,\Pi_2)$, the analogous statement holds with
   the second projection $\pr_2$.
\end{remark}

We obtain the refined Arthur classification.
\begin{theorem}\label{theorem:refined-Arthur-classification-SO4}
   Let $\psi$ be a refined $A$-parameter for $\SO_4$.
   Then, we have
   \begin{align*}
      L^2_{\disc, c(\psi)}(\SO_4(F) \backslash \SO_4(\A_F)) \cong \bigoplus_{\Sigma \in \Pi_{\psi}(\SO_4)} m_{\disc}(\Sigma) \Sigma,
   \end{align*}
   with
   \begin{align*}
      m_{\disc}(\Sigma) = \dim\Hom_{\cS_{\psi}}(\rho_{\Sigma} \circ \Delta, \mathbf{1}) \leq 1.
   \end{align*}
   Here $\rho_{\Sigma}=\bigotimes_v\rho_{\Sigma_v}$ is the tensor product of the local enhancements.
\end{theorem}

\begin{proof}
   For non-tempered $\psi$, the preceding lemma gives $\cS_{\psi}=1$, and the explicit descriptions of the two non-generic types give singleton refined local packets. The restriction construction and Theorem~\ref{theorem:SO4-simultaneous-twist} then imply the assertion.
   We assume that $\psi$ is represented by a pair $(\Pi_1, \Pi_2)$ of cuspidal automorphic representations of $\GL_2(\A_F)$ with $\omega_{\Pi_1} = \omega_{\Pi_2}$.
   Then, the theorem of Arthur \cite{Arthur2013-EndoscopicClassificationRepresentations} implies that if $\Sigma$ contributes to the left-hand side, then we have $\rho_{\Sigma} \circ \Delta = \mathbf{1}$ and $\Sigma \in \overline{\Pi}_{\psi}$.
   Lift $\Sigma$ to a discrete automorphic representation of $\GSO_4$
   corresponding to a pair $(\Pi'_1,\Pi'_2)$.
   Compatibility with Satake parameters gives
   $c(\Ad(\Pi'_1))=c(\Ad(\Pi_1))$.
   Strong multiplicity one for the isobaric transfers to $\GL_3$
   implies $\Ad(\Pi'_1)=\Ad(\Pi_1)$.
   Local compatibility of the adjoint lift and restriction therefore
   gives $\wedge^2_{+}\psi_v=\wedge^2_{+}\Sigma_v$ at every place $v$.
   Proposition~\ref{proposition:refined-packet-characterization} gives $\Sigma_v \in \Pi_{\psi_v}$ at every place, hence $\Sigma \in \Pi_{\psi}$.

   Conversely, suppose $\Sigma \in \Pi_{\psi}$ and $\rho_{\Sigma}\circ\Delta=\mathbf{1}$. By the same proposition, each $\Sigma_v$ belongs to the corresponding Arthur packet. Arthur's multiplicity formula then gives the stated discrete multiplicity.
\end{proof}

Outer-automorphism invariance has the following characterization.

\begin{lemma}
   Let $\Sigma$ be a discrete automorphic representation of $\SO_4(\A_F)$.
   Let $\Sigma'$ be a discrete automorphic representation of $\GSO_4(\A_F)$ such that $\Sigma$ is contained in $\Sigma'|_{\SO_4(\A_F)}$.
   Let $\Pi = (\pi_1, \pi_2)$ be the pair of discrete automorphic representations of $\GL_2(\A_F)$ corresponding to $\Sigma'$.
   Let $c(\Sigma)$ denote the Hecke--Satake parameter of $\Sigma$.
   Let $\theta$ be an element of $\O_4(F) \setminus \SO_4(F)$.
   Then, the following are equivalent.
   \begin{enumerate}
      \item
      We have
      $\Sigma = \Sigma^{\theta}$.
      \item
      We have
      $c(\Sigma) = c(\Sigma^{\theta})$.
      \item
      We have
      \begin{align*}
         \Ad(\pi_1) = \Ad(\pi_2).
      \end{align*}
      \item
      There exists an automorphic quadratic character $\chi$ of $\A_F^{\times}$ such that $\pi_1 \cong \pi_2 \otimes \chi \circ \det$.
   \end{enumerate}
\end{lemma}

\begin{proof}
   The implications $(1) \Rightarrow (2) \Rightarrow (3)$ are clear.
   The implication $(3) \Rightarrow (4)$ is a result of Ramakrishnan \cite{Ramakrishnan2000-ModularityRankinSelbergSeriesMultiplicityOne}.
   The implication $(4) \Rightarrow (1)$ is a result of Arthur \cite{Arthur2013-EndoscopicClassificationRepresentations}.
\end{proof}

\subsection{Endoscopic character relations for \texorpdfstring{$\SO_4$}{SO4}}

\begin{lemma}
   Let $\pi_1$ and $\pi_2$ be cuspidal automorphic representations of $\GL_2(\A_F)$.
   Let $\Pi$ be the Rankin--Selberg product $\pi_1 \times \pi_2$.
   \begin{enumerate}
      \item
      If $\pi_1$ is not dihedral, then $\Pi$ is not cuspidal if and only if there exists an automorphic character $\chi$ of $\A_F^{\times}$ such that $\pi_1 \cong \pi_2 \otimes \chi \circ \det$.
      \item
      If $\pi_1$ is dihedral, then $\Pi$ is not cuspidal if and only if there exists a quadratic extension $E/F$ such that $\pi_1$ and $\pi_2$ are both induced from Hecke characters of $E^{\times} \backslash \A_E^{\times}$.
   \end{enumerate}
\end{lemma}

\begin{proof}
   We have
   \begin{align*}
      L(s, \Pi \times \Pi^{\vee}) = \zeta(s) L(s, \Ad(\pi_1)) L(s, \Ad(\pi_2)) L(s, \Ad(\pi_1) \times \Ad(\pi_2)).
   \end{align*}
   Since $\Pi$ is cuspidal if and only if $L(s, \Pi \times \Pi^{\vee})$ has a simple pole at $s = 1$ and the $L$-functions $L(s, \Ad(\pi_i))$ do not have poles at $s = 1$, the $L$-function $L(s, \Ad(\pi_1) \times \Ad(\pi_2))$ has a pole at $s = 1$ if and only if $\Pi$ is not cuspidal.

   If $\pi_1$ is not dihedral, then $\Ad(\pi_1)$ is cuspidal and the $L$-function $L(s, \Ad(\pi_1) \times \Ad(\pi_2))$ has a pole at $s = 1$ if and only if $\Ad(\pi_1) \cong \Ad(\pi_2)$.
   This is equivalent to the existence of an automorphic character $\chi$ of $\A_F^{\times}$ such that $\pi_1 \cong \pi_2 \otimes \chi \circ \det$ by a result of Ramakrishnan \cite{Ramakrishnan2000-ModularityRankinSelbergSeriesMultiplicityOne}.

   Assume that $\pi_1$ is dihedral and induced from a Hecke character $\chi_1$ of a quadratic extension $E/F$.
   Then, we have
   \begin{align*}
      \Ad(\pi_1) = \mathrm{AI}_{E/F}(\chi_1/\chi_1^{\tau}) \boxplus \omega_{E/F},
   \end{align*}
   where $\tau$ is the non-trivial element of $\Gal(E/F)$.
   The first summand is cuspidal if and only if $\chi_1/\chi_1^{\tau}$ is not quadratic.
   Then, the $L$-function $L(s, \Ad(\pi_1) \times \Ad(\pi_2))$ has a pole at $s = 1$ only if $\pi_2$ is also dihedral.
   From now on, we assume that $\pi_2$ is also dihedral and induced from a Hecke character $\chi_2$ of a quadratic extension $E'/F$.
   Then, we have
   \begin{align*}
      \Ad(\pi_2) = \mathrm{AI}_{E'/F}(\chi_2/\chi_2^{\tau'}) \boxplus \omega_{E'/F},
   \end{align*}
   where $\tau'$ is the non-trivial element of $\Gal(E'/F)$.
   If the character $\chi_1/\chi_1^{\tau}$ is not quadratic and the representations $\Ad(\pi_1)$ and $\Ad(\pi_2)$ share a common cuspidal summand coming from the automorphic induction of a Hecke character, then, by considering the central characters, we have $\omega_{E/F} \cong \omega_{E'/F}$.
   Thus, we obtain $E = E'$.
   Otherwise, the two representations share a common quadratic character as an isobaric summand.
   If the quadratic character is given as $\omega_{E''/F}$, then, by considering $\pi_i \times \pi_i^{\vee} \otimes \omega_{E''/F}$, we have $\pi_i \cong \pi_i \otimes \omega_{E''/F}$.
   Thus, the two representations $\pi_1$ and $\pi_2$ are both induced from Hecke characters of $E''^{\times} \backslash \A_{E''}^{\times}$.
\end{proof}

\begin{lemma}\label{lemma:computation-rankin-selberg-induced}
   Let $E/F$ be a quadratic extension of number fields.
   Let $\pi_1$ and $\pi_2$ be two cuspidal automorphic representations of $\GL_2(\A_F)$ which are both induced from Hecke characters $\chi_1$ and $\chi_2$ of $E^{\times} \backslash \A_E^{\times}$.
   \begin{enumerate}
      \item
      We have
      \begin{align*}
         \pi_1 \times \pi_2^{\vee} = \mathrm{AI}_{E/F}(\chi_1/\chi_2) \boxplus \mathrm{AI}_{E/F}(\chi_1/\chi_2^{\tau}).
      \end{align*}
      Furthermore, the first and the second summands do not share a common isobaric summand.
      \item
      If we also assume that $\omega_{\pi_1} = \omega_{\pi_2}$, then we have
      \begin{align*}
         \chi_1|_{\A_F^{\times}} = \chi_2|_{\A_F^{\times}}.
      \end{align*}
      In this case, we have
      \begin{align*}
         (\chi_1/\chi_2)^{\tau} = (\chi_1/\chi_2)^{-1},
      \end{align*}
      and
      \begin{align*}
         (\chi_1/\chi_2^{\tau})^{\tau} = (\chi_1/\chi_2^{\tau})^{-1}.
      \end{align*}
      Thus, the summand $\mathrm{AI}_{E/F}(\chi_1/\chi_2)$ (resp. $\mathrm{AI}_{E/F}(\chi_1/\chi_2^{\tau})$) is a sum of quadratic characters of $\A_F^{\times}$ if and only if $\chi_1/\chi_2$ (resp. $\chi_1/\chi_2^{\tau}$) is quadratic.
   \end{enumerate}
\end{lemma}

\begin{proof}
   Only the second claim in $(1)$ needs a proof.
   The two summands $\mathrm{AI}_{E/F}(\chi_1/\chi_2)$ and $\mathrm{AI}_{E/F}(\chi_1/\chi_2^{\tau})$ share a common isobaric summand if and only if we have
   \begin{align*}
      \chi_1/\chi_2 = \chi_1/\chi_2^{\tau},
   \end{align*}
   or
   \begin{align*}
      \chi_1/\chi_2 = (\chi_1/\chi_2^{\tau})^{\tau} = \chi_1^{\tau}/\chi_2.
   \end{align*}
   These identities give $\chi_2 = \chi_2^{\tau}$ or $\chi_1 = \chi_1^{\tau}$, contradicting the cuspidality of $\pi_2$ or $\pi_1$, respectively.
\end{proof}

\begin{corollary}\label{corollary:component-group-quadratic}
   We continue to use the notation in the lemma above and assume that $\omega_{\pi_1} = \omega_{\pi_2}$.
   Then, the component group $\cS_{\psi}$ of the refined $A$-parameter $\psi$ corresponding to the pair $(\pi_1, \pi_2)$ is given as follows.
   \begin{enumerate}
      \item
      If both $\chi_1/\chi_2$ and $\chi_1/\chi_2^{\tau}$ are not quadratic, then we have $\cS_{\psi} = \Z/2\Z$.
      \item
      If exactly one of $\chi_1/\chi_2$ and $\chi_1/\chi_2^{\tau}$ is quadratic, then we have $\cS_{\psi} = \Z/2\Z$.
      \item
      If both $\chi_1/\chi_2$ and $\chi_1/\chi_2^{\tau}$ are quadratic, then we have $\cS_{\psi} = (\Z/2\Z)^2$.
   \end{enumerate}
\end{corollary}

\begin{remark}
   In Cases $(1), (2), (3)$ in the corollary above, the $A$-parameter $\psi$ factors through an endoscopic group $\SO_{E/F, 2} \times \SO_{E/F, 2}$.
\end{remark}

\begin{remark}
   Corollary~\ref{corollary:component-group-quadratic} also describes $\cS_{\psi}$ for local $A$-parameters $\psi$ for $\SO_4$.
\end{remark}

Let $F$ be a local field of characteristic $0$ and $E/F$ a quadratic extension.
To describe the component group explicitly, fix a four-dimensional split non-degenerate quadratic space $(V,q)$ over $\C$ with basis $\{e_1,e_2,e_3,e_4\}$ in which $q$ has matrix
\begin{align*}
   \begin{pmatrix}
      0 & 0 & 0 & 1 \\
      0 & 0 & 1 & 0 \\
      0 & 1 & 0 & 0 \\
      1 & 0 & 0 & 0
   \end{pmatrix}.
\end{align*}
We fix an embedding $\mathrm{S}(\O_2 \times \O_2) \hookrightarrow \SO_4$ by
\begin{align*}
   (
   \begin{pmatrix}
      a & b \\
      c & d
   \end{pmatrix},
   \begin{pmatrix}
      a' & b' \\
      c' & d'
   \end{pmatrix}
   )
   \mapsto
   \begin{pmatrix}
      a & 0 & 0 & b \\
      0 & a' & b' & 0 \\
      0 & c' & d' & 0 \\
      c & 0 & 0 & d
   \end{pmatrix}.
\end{align*}

We continue to use the notation in Corollary \ref{corollary:component-group-quadratic}.
In Cases $(1), (2)$, the group $S_{\psi}$ is generated by the image of the element
\begin{align*}
   s = (I_2,-I_2) \in \mathrm{S}(\O_2 \times \O_2)
\end{align*}
and the center of $\SO_4(\C)$.

For Case $(3)$, choose a basis of $V$ in which $q$ has Gram matrix $I_4$, and put
\begin{align*}
   s = s_1 = \diag(1,1,-1,-1).
\end{align*}
Let $\eta_1, \eta_2$ be quadratic characters of $F^{\times}$ such that $\chi_1/\chi_2 = \eta_1 \circ \mathrm{Nm}_{E/F}$ and $\chi_1/\chi_2^{\tau} = \eta_2 \circ \mathrm{Nm}_{E/F}$.
Then, the parameter $\psi$ is given by
\begin{align*}
   \eta_1  \oplus \eta_1 \omega_{E/F} \oplus \eta_2 \oplus \eta_2 \omega_{E/F}.
\end{align*}
The other endoscopic data are given by the semisimple elements
\begin{align*}
   s_2 = \diag(1,-1,1,-1),\quad
   s_3 = \diag(1,-1,-1,1).
\end{align*}
Let $E_2$ (resp. $E_3$) be the quadratic extension of $F$ corresponding to the quadratic character $\eta_1 \eta_2$ (resp. $\eta_1 \eta_2 \omega_{E/F}$).
Then, the endoscopic group $H_{s_2}$ (resp. $H_{s_3}$) corresponding to $s_2$ (resp. $s_3$) is given by $\SO_{E_2/F, 2} \times \SO_{E_2/F, 2}$ (resp. $\SO_{E_3/F, 2} \times \SO_{E_3/F, 2}$).
The character $\chi_{s_2}$ (resp. $\chi_{s_3}$) of the endoscopic group $H_{s_2}$ (resp. $H_{s_3}$) is given by $(\eta_1 \circ \mathrm{Nm}_{E_2/F}, \eta_1 \omega_{E/F} \circ \mathrm{Nm}_{E_2/F})$ (resp. $(\eta_1 \circ \mathrm{Nm}_{E_3 /F}, \eta_2 \omega_{E/F} \circ \mathrm{Nm}_{E_3/F})$).

\begin{theorem}\label{theorem:ECR}
   Let $\phi$ be a local $L$-parameter for $\SO_4$ which is given by a pair $(\pi_1, \pi_2)$ of irreducible discrete series representations of $\GL_2(F)$ with $\omega_{\pi_1} = \omega_{\pi_2}$.
   Assume that $\pi_1$ and $\pi_2$ are both induced from characters $\chi_1$ and $\chi_2$ of $E^{\times}$ for a quadratic extension $E/F$.
   Then, for any nontrivial $s \in \cS_{\phi}$, we have
   \begin{align*}
      \sum_{\sigma \in \Pi_{\phi}(\SO_4)} \langle s, \sigma \rangle \tr(\sigma(f^{\SO_4})) =
      \chi_s(f^{\SO_4}_{H_s}),
   \end{align*}
   where $H_s$ is the endoscopic group corresponding to $s$, $\chi_s$ is the character of $H_s$ corresponding to $s$, and $f^{\SO_4}$ is a test function on $\SO_4(F)$.
\end{theorem}

\begin{proof}
   Apply the proof of \cite{Arthur2013-EndoscopicClassificationRepresentations}*{Proposition 6.6.5} using Theorem~\ref{theorem:refined-Arthur-classification-SO4}.
\end{proof}

\subsection{Stable multiplicity formula for \texorpdfstring{$A$}{A}-packets of \texorpdfstring{$\SO_4$}{SO4}}

\begin{theorem}\label{theorem:stable-multiplicity-SO4}
   Let $\psi$ be a refined $A$-parameter for $\SO_4$.
   Then, we have the stable multiplicity formula
   \begin{align*}
      S_{\disc, c(\psi)}^{\SO_4}(f^{\SO_4}) = \frac{1}{\abs{\cS_{\psi}}} \sum_{\Sigma \in \Pi_{\psi}(\SO_4)} \tr(\Sigma(f^{\SO_4}))
   \end{align*}
   for any test function $f^{\SO_4} \in C_c^{\infty}(\SO_4(\A_F))$.
\end{theorem}

\begin{proof}
   This follows from Theorems~\ref{theorem:refined-Arthur-classification-SO4} and~\ref{theorem:ECR}, and the stabilization of the trace formula for $\SO_4$ given as
   \begin{align*}
      I_{\disc, c(\psi)}^{\SO_4}(f^{\SO_4}) = S_{\disc, c(\psi)}^{\SO_4}(f^{\SO_4}) + \frac{1}{4} \sum_{E/F} I_{\disc, c(\psi)}^{\SO(2, E/F)^2}(f^{\SO_4}_{\SO(2, E/F)^2}).
   \end{align*}
\end{proof}

\section{On real exceptional theta correspondences}\label{section:theta-A2-real-local}

Retain the notation of Section~\ref{section:theta-G2-to-A2}.

\subsection{Fourier functionals on the spherical theta lift}\label{theta-A2:sec:real-models}

Throughout this subsection the ground field is $\R$, $E=\R^3$ or $\R\times\C$, and $J=M_3(\R)^+$. Normalize $\psi(x)=e^{ix}$ and put $C_0=E^\perp\subset J$, $H=\Aut_E(C_0)$. Let $\Omega^\infty$ be the smooth minimal representation of split $E_6$, and put $\pi_1=\theta_{C_0}(\mathbf1)$, with its Casselman--Wallach topology. Its irreducible spherical Harish-Chandra module is identified in \cite{GanLokeEtAl2025-FamilySpinEightDualPairsRealGroups}*{Theorem 3}.

\subsubsection{Preliminaries}\label{theta-A2:sec:real-preliminaries}

\emph{Cubic algebras and cubes.}
Write $\tr_E=\tr_{E/\R}$ for the trace and $N_{E/\R}$ for the norm of $E$. The quadratic adjoint $e\mapsto e^\#$ is the polynomial map given on $E^\times$ by $e^\#=N_{E/\R}(e)e^{-1}$. Put $e\times f=(e+f)^\#-e^\#-f^\#$. Thus
\[
 (e_1,e_2,e_3)^\#=(e_2e_3,e_3e_1,e_1e_2),\quad
 (r,z)^\#=(|z|^2,r\overline z)
\]
for $E=\R^3$ and $E=\R\times\C$, respectively.

Put $Z_E=[N_E,N_E]$. The root coordinates identify $V_E=N_E/Z_E$ with $\R\oplus E\oplus E\oplus\R$. Its elements $\Sigma=(a,e,f,b)$ are called \emph{$E$-twisted cubes}. Over $\C$, the action of $M_E$ is the tensor product of three standard two-dimensional representations, twisted by the inverse of their common determinant. Its quartic relative invariant is
\[
 \Delta(\Sigma)=(ab-\tr_E(ef))^2
      +4aN_{E/\R}(f)+4bN_{E/\R}(e)-4\tr_E(e^\#f^\#).
\]
A cube is \emph{nondegenerate} if $\Delta(\Sigma)\ne0$. We also use cubes as character coordinates. With the root-coordinate pairing between $V_E$ and its dual,
\[
 \psi_\Sigma(u)=\psi\bigl(\langle\Sigma,u\bmod Z_E\rangle\bigr),
 \quad u\in N_E(\R).
\]
In particular, $\psi_\Sigma$ is trivial on $Z_E$. A \emph{reduced cube} is one of the form $(1,0,f,b)$. Then $\Delta(\Sigma)=b^2+4N_{E/\R}(f)$. Every nondegenerate $M_E(\R)$-orbit contains a reduced cube. See \cite{GanSavin2022-TwistedCompositionAlgebrasArthurPacketsTrialitySpin8}*{Section~2.9, Section~5, and the proof of Proposition 12.3}.

\emph{The composition algebra and its fiber.}
Using the matrix Jordan operations of Section~\ref{theta-common-notation}, the fixed embedding $E\hookrightarrow J$ gives the orthogonal decomposition
\[
 J=E\oplus C_0,\quad C_0=E^\perp,
\]
with projections $\pr_E,\pr_{C_0}$. The $E$-module structure and tensors on $C_0$ are
\[
 \begin{gathered}
 e\cdot v=-e\times v,\quad
 Q(v)=-\pr_E(v^\#),\quad \beta_{C_0}(v)=\pr_{C_0}(v^\#),\\
 N_{C_0}(v)=\det v.
 \end{gathered}
\]
These formulas give the Springer construction. See \cite{GanSavin2022-TwistedCompositionAlgebrasArthurPacketsTrialitySpin8}*{Sections 4.1 and 4.5}. They give $v^\#=(-Q(v),\beta_{C_0}(v))$. The group $H=\Aut_E(C_0)$ identifies with the pointwise stabilizer of $E$ in $\Aut(J)$. Its identity component $E^\times/\R^\times$ acts by conjugation. An adjoint involution for the trace pairing on $E$ supplies the other component.

The algebra $C_\Sigma$ attached to a nondegenerate reduced cube $\Sigma=(1,0,f,b)$ is $E^2$ with
\[
 \begin{aligned}
 Q_\Sigma(u,v)&=-fu^2-buv+f^\#v^2,\\
 \beta_\Sigma(u,v)&=(-bv^\#-(fu)\times v,\ u^\#+fv^\#).
 \end{aligned}
\]
In particular, $\beta_\Sigma(1,0)=(0,1)$. Put $q=-f$ and $n=-b$. By \cite{GanSavin2022-TwistedCompositionAlgebrasArthurPacketsTrialitySpin8}*{Proposition 5.2 and Corollary 5.3}, we have
\begin{equation}\label{theta-A2:eq:real-fiber-isomorphisms}
 \begin{gathered}
 \operatorname{Isom}_E(C_\Sigma,C_0)\longrightarrow
 X_{q,n}:=\{v\in C_0:Q(v)=q,\ N_{C_0}(v)=n\},\\
 \phi\longmapsto\phi(1,0).
 \end{gathered}
\end{equation}
Its inverse sends $v$ to $(u_1,u_2)\mapsto u_1v+u_2\beta_{C_0}(v)$. Thus a nonempty real fiber is an $H(\R)$-torsor, meaning that $H(\R)$ acts freely and transitively on it. The fiber is empty precisely when $C_\Sigma\not\cong C_0$.

\emph{Heisenberg groups and the polarization.}
A real Heisenberg group is a connected simply connected two-step nilpotent group with one-dimensional center and nondegenerate commutator pairing on its quotient by the center. A \emph{polarization} decomposes this symplectic quotient into complementary Lagrangians $\Lambda\oplus\Lambda^*$. The Schr\"odinger model with fixed nontrivial central character is $L^2(\Lambda)$, where the Lagrangians act by translations and multiplication by characters. See \cite{Frahm2024-MinimalRepresentations}*{Sections 3.4--3.5}.

The Heisenberg radical $N_J$ in the ambient split group of type $E_6$ contains $N_E$, has center $Z_J=Z_E$, and satisfies
\[
 V_J=N_J/Z_J\cong\R\oplus J\oplus J\oplus\R.
\]
To specify the auxiliary radical $N_{\mathrm{aux}}$ used for the Schr\"odinger model, use the bigrading of \cite{Frahm2024-MinimalRepresentations}*{Sections 5.1--5.2}. Let $H_1,H_2$ denote his commuting grading elements $H_\alpha,H_\beta$, and put
\[
 \widetilde{\fg}_{(i,j)}
 =\{X:[H_1,X]=iX,\ [H_2,X]=jX\}.
\]
With his normalizations, put
\[
 \begin{gathered}
 \R F_0=\widetilde{\fg}_{(-1,-1)},\quad
 \R A_0=\widetilde{\fg}_{(1,-2)},\quad
 \R B_0=\widetilde{\fg}_{(-2,1)},\quad
 \R B_1=\widetilde{\fg}_{(-1,2)},\\
 J=\widetilde{\fg}_{(0,-1)},\quad J^*=\widetilde{\fg}_{(-1,0)}.
 \end{gathered}
\]
The symplectic form $\omega$ pairs $J$ with $J^*$, which is identified with $J$ by the trace form, and $\omega(A_0,B_0)=2$. The radicals and the chosen polarization are
\begin{equation}\label{theta-A2:eq:real-polarization}
 \begin{aligned}
 \Lambda&=\R A_0\oplus J,\quad \Lambda^*=\R B_0\oplus J^*,\\
 \fn_{\mathrm{aux}}&=\R F_0\oplus\Lambda\oplus\Lambda^*,\\
 \fn_J&=\R B_0\oplus\R F_0\oplus J^*
       \oplus\widetilde{\fg}_{(-1,1)}\oplus\R B_1.
 \end{aligned}
\end{equation}
Thus the auxiliary center is $\R F_0$, whereas the center of $\fn_J$ is $\R B_0$.

Frahm's model is $L^2(\R^\times_\lambda\times\Lambda)$. Write a point of $\Lambda$ as $tA_0+x$ with $x\in J$. The parameter $\lambda$ is the auxiliary central frequency. We have $d\Omega(F_0)=i\lambda$. The first scalar character coordinate $a$ of a cube corresponds to this root group, so reduction to $a=1$ fixes $\lambda=1$. The equation for $Z_J$ fixes $t=0$. Our $A_0,B_0,B_1,F_0$ are Frahm's $A,B,\overline B,F$, respectively. We use \cite{Frahm2024-MinimalRepresentations}*{Proposition 5.5.5, Lemma 7.1.3, and Theorem 7.1.4} throughout.

On the chart $a=1$, the minimal-orbit character coordinates in $V_J$ are $(1,x,x^\#,N_J(x))$, $x\in J$. Their restriction to $N_E$ is, in the root conventions of \cite{GanSavin2022-TwistedCompositionAlgebrasArthurPacketsTrialitySpin8}*{Proposition 8.1 and the proof of Proposition 12.3},
\begin{equation}\label{theta-A2:eq:real-character-restriction}
 (a,x,y,d)\longmapsto(a,-\pr_E x,\pr_E y,-d).
\end{equation}
Consequently $(1,0,f,b)$ forces $x\in C_0$, $Q(x)=-f$, and $N_{C_0}(x)=-b$.

Finally, $\mathcal S(W)$ denotes the Schwartz space of a real vector space $W$, and $\mathcal S'(W)$ its continuous dual, the tempered distributions. For a polynomial $p$ and a distribution $D$, $pD=0$ means $D(p\varphi)=0$ for every test function $\varphi$. All Hom spaces below consist of continuous functionals. Equivariance means $\ell(\Omega(u)\varphi)=\psi_\Sigma(u)\ell(\varphi)$.

\subsubsection{The real fibers}\label{theta-A2:sec:real-fibers}

Fix a nondegenerate reduced cube $\Sigma=(1,0,f,b)$, and use $q,n,X_{q,n}$ from \eqref{theta-A2:eq:real-fiber-isomorphisms}. The character restriction \eqref{theta-A2:eq:real-character-restriction} identifies $X_{q,n}$ with the fiber of minimal-orbit characters above $\Sigma$.

It suffices to use reduced representatives for which $q\in E^\times$. Indeed, the construction in \cite{GanSavin2022-TwistedCompositionAlgebrasArthurPacketsTrialitySpin8}*{Proposition 5.2(ii)} gives reduced representatives from vectors $v\in C_\Sigma$ with $N_{C_\Sigma}(v)^2-4N_{E/\R}(Q_\Sigma(v))\ne0$. Intersecting this nonempty Zariski-open set with $N_{E/\R}(Q_\Sigma(v))\ne0$ gives the required choice. The latter polynomial is nonzero because each absolute component of $Q_\Sigma$ is a nondegenerate quadratic form. Then $q=Q_\Sigma(v)\in E^\times$. Set
\[
 \Delta=n^2-4N_{E/\R}(q)\ne0.
\]

If $E=\R^3$, embed $E$ diagonally and write
\[
 v=\begin{pmatrix}0&a&b'\\c&0&d\\e'&f'&0\end{pmatrix}.
\]
An adjugate and determinant calculation gives
\[
 Q(v)=(df',b'e',ac),\quad N_{C_0}(v)=ade'+b'cf'.
\]
Put $r=ade'$ and $s=b'cf'$. On the fiber,
\[
 r+s=n,\quad rs=q_1q_2q_3,\quad \Delta=(r-s)^2.
\]
There are no real points if $\Delta<0$. If $\Delta>0$, the two roots $r$ of
$r^2-nr+q_1q_2q_3=0$ are nonzero, and each gives the parametrization
\[
 a=u,\quad d=w,\quad e'=\frac r{uw},\quad
 c=\frac{q_3}u,\quad f'=\frac{q_1}w,\quad
 b'=\frac{q_2uw}r,
 \quad (u,w)\in(\R^\times)^2.
\]
Thus $X_{q,n}$ consists of two copies of $(\R^\times)^2$. Diagonal conjugation is simply transitive on each copy, and transpose exchanges them. This realizes the torsor under
\[
 H(\R)=\bigl((\R^\times)^3/\R^\times\bigr)\rtimes\Z/2\Z.
\]
For example, the Jacobian minor of $(df',b'e',ac,N_{C_0})$ in the variables $(b',c,e',f')$ is $ad(s-r)$, which is nonzero on the fiber. Haar measure on each copy is a multiple of
$\rd u\,\rd w/|uw|$. Invariance under the components of the torus and the involution leaves one scalar.

If $E=\R\times\C$, diagonalize $E$ after complexification. Complex conjugation exchanges the last two basis vectors, so that the real points of $C_0$ have the form
\[
 v=\begin{pmatrix}0&a&\overline a\\b'&0&c\\
 \overline{b'}&\overline c&0\end{pmatrix},
 \quad a,b',c\in\C.
\]
Write $q=(q_0,\overline z,z)$. The equations are
\[
 |c|^2=q_0,\quad ab'=z,\quad
 2\operatorname{Re}(ac\overline{b'})=n.
\]
For $r=ac\overline{b'}$ they imply
\[
 |r|^2=q_0|z|^2,\quad 2\operatorname{Re}r=n,
 \quad \Delta=-4(\operatorname{Im}r)^2.
\]
Hence the fiber is empty if $\Delta>0$. If $\Delta<0$, its two components are indexed by
\[
 r=\frac{n\pm i\sqrt{4q_0|z|^2-n^2}}2,
\]
and each is parametrized by
\[
 a=u,\quad b'=\frac zu,\quad
 c=\frac{r\overline u}{u\overline z},\quad u\in\C^\times.
\]
The torus $(\R^\times\times\C^\times)/\R^\times\cong\C^\times$ acts simply transitively on each component. The adjoint involution exchanges $r$ and $\overline r$. Haar measure is a multiple of $\rd\operatorname{Re}u\,\rd\operatorname{Im}u/|u|^2$. The Jacobian of the four real equations has rank four, since the parametrization gives the fiber smoothly for all nearby $(q_0,z,n)$ with $\Delta<0$. In particular, again there is only one invariant distribution for the full torsor action.

\subsubsection{A distribution calculation in the minimal representation}\label{theta-A2:sec:real-distributions}

\begin{lemma}\label{theta-A2:lem:real-minimal-bound}
For every nondegenerate $E$-twisted cube $\Sigma$,
\[
 \dim\Hom_{H(\R)\times N_E(\R)}
       (\Omega^\infty,\mathbf1\boxtimes\psi_\Sigma)
 \leq
 \begin{cases}1,&C_\Sigma\cong C_0,\\0,&C_\Sigma\not\cong C_0.\end{cases}
\]
\end{lemma}
\begin{proof}
By Section~\ref{theta-A2:sec:real-fibers}, we may choose a reduced representative $\Sigma=(1,0,f,b)$ with $q=-f\in E^\times$. We use the Schr\"odinger model of \cite{Frahm2024-MinimalRepresentations}*{Proposition 5.5.5, Lemma 7.1.3, and Theorem 7.1.4}. Its Hilbert model in the present case is
\[
 L^2(\R^\times_\lambda\times\R_t\times J_x).
\]
Use the polarization \eqref{theta-A2:eq:real-polarization}. The auxiliary center $\exp(\R F_0)$ acts by $e^{is\lambda}$. The center $\exp(\R B_0)$ of $N_J$ acts by multiplication by an exponential in $t$, and $J^*\subset\fn_J$ acts by multiplication by exponentials linear in $x$. The remaining $J$-coordinates act, at $t=0$, by multiplication by the quadratic adjoint $x^\# /\lambda$. The last scalar coordinate acts there by the cubic norm $N_J(x)/\lambda^2$.

To obtain these operators from \cite{Frahm2024-MinimalRepresentations}, write its variable $x$ as $tA_0+x'$ with $x'\in J$. For $T\in\widetilde{\fg}_{(-1,1)}$ and $B_1\in\widetilde{\fg}_{(-1,2)}$, the relevant operators, in its normalization, are
\[
 d\Omega(T)=-t\partial_{TA_0}-\frac{\omega(Tx',x')}{2i\lambda},
 \quad
 d\Omega(B_1)=-\omega(tA_0,B_0)\partial_\lambda+2i\lambda^{-2}n_J^{\mathrm{Fr}}(x').
\]
Here $\partial_{TA_0}$ is directional differentiation in $J$, and $n_J^{\mathrm{Fr}}$ is Frahm's cubic Jordan norm, denoted $n$ in \cite{Frahm2024-MinimalRepresentations}. We reserve $n=-b$ for the prescribed scalar. The quadratic and cubic polynomials are the adjoint and norm of the cubic Jordan algebra. With the root-coordinate normalization of the reduced cube, they give $x^\#$ and $N_J(x)$ as above. Both derivative terms vanish on restriction to $t=0$. These normalizations are determined by the orbit of the character at the origin. At $x=0$ the character is $(1,0,0,0)$, and the translation subgroup indexed by $J$ gives its orbit
$(1,x,x^\#,N_J(x))$, exactly as in \cite{GanSavin2022-TwistedCompositionAlgebrasArthurPacketsTrialitySpin8}*{Proposition 8.1}.

On a compact $\lambda$-interval disjoint from zero, the smooth-vector seminorms are equivalent to those of smooth functions of $\lambda$ with values in $\mathcal S(\R\times J)$. The auxiliary Heisenberg action gives all coordinate multiplications and derivatives in $(t,x)$, and its split Levi gives $\lambda\partial_\lambda$ modulo these operators. Conversely, the infinitesimal operators have polynomial coefficients in $(t,x)$ and Laurent-polynomial coefficients in $\lambda$, by \cite{Frahm2024-MinimalRepresentations}. Thus compactly supported smooth functions of $\lambda$ with Schwartz values in $(t,x)$ are smooth vectors, and the seminorms control one another. Cutoffs in $(t,x)$ converge in these seminorms.

Let $\ell$ be a continuous $(N_E,\psi_\Sigma)$-functional. A convolution in the auxiliary center localizes it to any sufficiently small neighborhood of $\lambda=1$. Multiplication by a smooth Fourier cutoff equal to one at $1$ leaves $\ell$ unchanged. There is therefore no contribution from $\lambda=0$. The equations for this center, for $Z_J=Z_E$, and for the first $E$-coordinate give, respectively,
\[
 (\lambda-1)\ell=0,\quad t\ell=0,\quad e\ell=0,
 \quad x=e+v\in E\oplus C_0.
\]
These are multiplication equations, so they exclude derivatives normal to the slice $\lambda=1$, $t=0$, $e=0$. Restriction to this slice is onto $\mathcal S(C_0)$. A continuous section is obtained by multiplying a Schwartz function of $v$ by fixed cutoffs in $\lambda,t,e$. Consequently $\ell$ is represented by a tempered distribution $D$ on $C_0$. The remaining character equations become
\begin{equation}\label{theta-A2:eq:real-distribution-equations}
 (Q-q)D=0,\quad (N_{C_0}-n)D=0.
\end{equation}
The $E$-valued equation gives three real equations and the scalar equation one.

Since $\Sigma$ is nondegenerate, the four defining equations of $X_{q,n}$ have independent differentials along it, by Section~\ref{theta-A2:sec:real-fibers}. By the implicit function theorem and Hadamard's lemma, \eqref{theta-A2:eq:real-distribution-equations} implies that $D$ is the pushforward of a distribution on $X_{q,n}$.

The group $H$ preserves the auxiliary polarization and acts through a homogeneous line bundle on $X_{q,n}$. If $X_{q,n}$ is empty then $D=0$. Otherwise it is an $H(\R)$-torsor, and trivializing this bundle identifies equivariant distributions with scalar multiples of Haar measure. This proves the upper bound.
\end{proof}

\Needspace{10\baselineskip}
\subsubsection{Passage to the spherical theta lift}

\begin{proposition}\label{theta-A2:prop:real-fourier-model}
For $E=\R^3$ or $\R\times\C$ and every nondegenerate cube $\Sigma$, one has
\[
 \Hom_{N_E(\R)}(\pi_1,\psi_\Sigma)
 \cong
 \begin{cases}\C,&C_\Sigma\cong C_0,\\0,&C_\Sigma\not\cong C_0,\end{cases}
\]
and the stabilizer $M_{E,\Sigma}(\R)$ acts trivially on the nonzero space.
\end{proposition}
\begin{proof}
We first construct a continuous theta map. Choose an auxiliary cubic field $L/\Q$ with $L\otimes_\Q\R\cong E$. For $E=\R^3$ one may take $L=\Q[t]/(t^3-3t+1)$, and for $E=\R\times\C$ one may take $L=\Q[t]/(t^3-t-1)$. These polynomials are irreducible with discriminants $81$ and $-23$, respectively. Embed $L$ in $M_3(\Q)$ by its regular representation and denote the corresponding composition algebra by $C_L$. Then $T_{C_L}\cong\Res_{L/\Q}\mathbf G_m/\mathbf G_m$ is anisotropic over $\Q$, so $[H_{C_L}]=H_{C_L}(\Q)\backslash H_{C_L}(\A_\Q)$ is compact.

The ordinary global theta lift of $\mathbf1$ is nonzero and square-integrable by \cite{GanSavin2022-TwistedCompositionAlgebrasArthurPacketsTrialitySpin8}*{Proposition 14.5 and Corollary 14.6(i)}. The real Harish-Chandra module of this lift is $\pi_1$ by \cite{GanLokeEtAl2025-FamilySpinEightDualPairsRealGroups}*{Theorems 1 and 3}.

Fix finite-adelic data for which the theta map is nonzero. Integration over the compact quotient $[H_{C_L}]$ gives an $H(\R)$-invariant continuous $G_E(\R)$-map $T$ from $\Omega^\infty$ into smooth automorphic functions, with their topology of locally uniform convergence of all derivatives. In particular, $\ker T$ is closed. Endow $W=\Omega^\infty/\ker T$ with its quotient topology. Restriction from the ambient real group and passage to a closed quotient preserve smoothness and moderate growth, so $W$ is a smooth Fr\'echet representation of moderate growth.

Choose the compatible maximal compact subgroups $K_E\subset\widetilde K$ of \cite{GanLokeEtAl2025-FamilySpinEightDualPairsRealGroups}, and put $V=(\Omega^\infty)_{\widetilde K\text{-finite}}$. This space is dense in $\Omega^\infty$ and stable under projection to each $K_E$-type. By $H(\R)$-invariance, $T|_V$ factors through the maximal trivial-$H$ quotient, identified with the irreducible module $(\pi_1)_{K_E\text{-finite}}$ in \cite{GanLokeEtAl2025-FamilySpinEightDualPairsRealGroups}. Density makes $T|_V$ nonzero, so its image in $W$ is isomorphic to that module.

For each irreducible $K_E$-type $\tau$, continuous $K_E$-projection shows that the finite-dimensional $\tau$-isotypic image of $V$ is dense in $W[\tau]$, hence equals it. Thus $W_{K_E\text{-finite}}\cong(\pi_1)_{K_E\text{-finite}}$.
The Casselman--Wallach globalization theorem now identifies $W$ with $\pi_1^\infty$. See \cite{BernsteinKrotz2014-SmoothGlobalizations}*{Section~1 and Theorem 1.1}. Composing the quotient map with this identification gives an $H(\R)$-invariant continuous surjection $A:\Omega^\infty\longrightarrow\pi_1^\infty$.

Pullback by $A$ therefore gives an injection
\[
 \Hom_{N_E(\R)}(\pi_1,\psi_\Sigma)
 \hookrightarrow
 \Hom_{H(\R)\times N_E(\R)}
       (\Omega^\infty,\mathbf1\boxtimes\psi_\Sigma).
\]
Lemma~\ref{theta-A2:lem:real-minimal-bound} proves vanishing outside the orbit of $C_0$ and the upper bound one on that orbit.

For the reverse inequality, take the nonzero $C_L$-Fourier coefficient of the global theta lift given by \cite{GanSavin2022-TwistedCompositionAlgebrasArthurPacketsTrialitySpin8}*{Proposition 14.5}. Fixing the finite-adelic arguments gives a nonzero continuous Fourier functional at the real place. Its cube belongs to the orbit of $C_0$. Conjugation by $M_E(\R)$ gives nonvanishing throughout this orbit.

Finally, at every finite place the stabilizer acts trivially on the corresponding one-dimensional Fourier space by \cite{GanSavin2022-TwistedCompositionAlgebrasArthurPacketsTrialitySpin8}*{Proposition 12.3}, since the quadratic algebra attached to $M_3(\Q)^+$ is split. Factor the nonzero global Fourier functional into local functionals. Since the rational stabilizer fixes the global functional and all its finite factors, the real factor is invariant under $H_{C_L}(\Q)$. Continuity gives invariance under the real stabilizer.
\end{proof}

\subsection{Associated varieties of local theta lifts}\label{theta-A2:sec:local-AV}

Retain the notation of Sections~\ref{theta-A2:sec:residue} and~\ref{theta-A2:sec:detection}.

\begin{lemma}\label{theta-A2:lem:residual-AV}
At every real place $w$, the local representation $\pi_{1,w}$ in \eqref{theta-A2:eq:nonfield-residue} satisfies
\[
 \AV(\Ann\pi_{1,w})=\overline{\cO_{A_2}}.
\]
\end{lemma}
\begin{proof}
The representation $\pi_{1,w}$ is a subquotient of $I_{E_w}(1/2)$. Its infinitesimal character is the Weyl orbit of
\[
 \rho_{M_E}+\tfrac12\beta
 =\tfrac12(\alpha_1+\alpha_3+\alpha_4)+\tfrac12\beta
 =(1,0,1,0)\sim(1,1,0,0)=:\lambda.
\]
The orbit $\cO_{A_2}$ has neutral element $2\beta^\vee=(2,2,0,0)$, so $\lambda$ is half this neutral element.

Let $J_{\max}(\lambda)$ be the unique maximal primitive ideal with infinitesimal character $\lambda$. By \cite{BarbaschVogan1985-UnipotentRepresentationsComplexSemisimpleGroups}*{Appendix, Definition A1 and Proposition A2(d)},
\[
 \AV(J_{\max}(\lambda))=\overline{\cO_{A_2}}.
\]
Since $\Ann\pi_{1,w}\subset J_{\max}(\lambda)$, this gives
$\overline{\cO_{A_2}}\subset\AV(\Ann\pi_{1,w})$.

For the opposite inclusion, use the bound for the associated variety of the annihilator of a scalar parabolically induced representation. It lies in the closure of the Richardson orbit for $P_E$ by \cite{BorhoBrylinski1989-RelativeEnvelopingAlgebras}*{Theorem 2}. The Richardson orbit is $\cO_{A_2}$, since its neutral element $2\beta^\vee$ defines the even grading with parabolic $P_E$. The annihilator can only increase on passing to the subquotient $\pi_{1,w}$, so
\[
 \AV(\Ann\pi_{1,w})\subset\overline{\cO_{A_2}},
\]
as required.
\end{proof}

\bibliographystyle{plain}
\bibliography{bibliography}

\end{document}